\documentclass[a4paper, 11pt]{article}

\pdfoutput=1
\usepackage{fullpage}

\RequirePackage[OT1]{fontenc}
\RequirePackage{amsmath, amsthm}
\RequirePackage[numbers]{natbib}
\usepackage{natbib}
\RequirePackage[colorlinks, citecolor=blue, urlcolor=blue]{hyperref}
\usepackage{chapterbib}
\usepackage{amsfonts}
\usepackage{algorithmic}
\usepackage{amssymb, bbm}
\newboolean{ARXIVVERSION}
\setboolean{ARXIVVERSION}{true}
\usepackage[normalem]{ulem}
\usepackage{mathrsfs}
\usepackage{enumerate}
\usepackage{float}
\usepackage[linesnumbered,ruled]{algorithm2e}

\usepackage{graphicx}
\usepackage[dvipsnames]{xcolor}
\usepackage{url}
\usepackage{tikz}
\usepackage{cancel}
\usetikzlibrary{arrows}
\usepackage{paralist}
\newdimen\arrowsize
\pgfarrowsdeclare{arcsq}{arcsq}
{
  \arrowsize=0.2pt
  \advance\arrowsize by .5\pgflinewidth
  \pgfarrowsleftextend{-4\arrowsize-.5\pgflinewidth}
  \pgfarrowsrightextend{.5\pgflinewidth}
}
{
  \arrowsize=1.5pt
  \advance\arrowsize by .5\pgflinewidth
  \pgfsetdash{}{0pt} %
  \pgfsetroundjoin   %
  \pgfsetroundcap    %
  \pgfpathmoveto{\pgfpoint{0\arrowsize}{0\arrowsize}}
  \pgfpatharc{-90}{-140}{4\arrowsize}
  \pgfusepathqstroke
  \pgfpathmoveto{\pgfpointorigin}
  \pgfpatharc{90}{140}{4\arrowsize}
  \pgfusepathqstroke
}

\newcommand{\independent}{\mbox{${}\perp\mkern-11mu\perp{}$}}
\newcommand{\notindependent}{\mbox{${}\not\!\perp\mkern-11mu\perp{}$}}

\DeclareMathOperator*{\argmin}{argmin}

\newcommand{\pr}{{\mathbb P}}
\newcommand{\R}{{\mathbb R}}

\newcommand{\E}{{\mathbb E}}
\newcommand{\Cov}{\mathrm{Cov}}
\newcommand{\Var}{\mathrm{Var}}

\newcommand{\tr}{{\mathrm{tr}}}

\newcommand{\cond}{\,|\,}

\newcommand{\inprob}{\stackrel{P}{\to}}

\newcommand{\ind}{\mathbbm{1}}

\newcommand{\dist}{P}
\newcommand{\HS}{\textrm{HS}}
\newcommand{\TR}{\textrm{TR}}
\newcommand{\op}{\textrm{op}}
\newcommand{\OLS}{\textrm{OLS}}
\newcommand{\BL}{\textrm{BL}}
\newtheorem{theorem}{Theorem}

\newtheorem{lemma}{Lemma}
\newtheorem{corollary}{Corollary}
\newtheorem{definition}{Definition}
\newtheorem{proposition}{Proposition}
\newtheorem{remark}{Remark}

\tikzset{every picture/.style={line width=0.6pt}}
\tikzset{every picture/.style={outer sep=.4mm}}
\tikzstyle{graphnode} = 
   [circle, draw=black, minimum size=25pt, text centred, inner sep=0.2mm] 
\tikzstyle{observed}   =[graphnode, fill=white, text=black]
\tikzstyle{unobserved}   =[graphnode, fill=white, text=black, style=dashed]
\tikzstyle{graphnodesmall} = 
   [circle, draw=black, minimum size=20pt, text centred, inner sep=0.2mm] 
\tikzstyle{observedsmall}   =[graphnodesmall, fill=white, text=black]
\tikzstyle{unobservedsmall}   =[graphnodesmall, fill=white, text=black, style=dashed]
\tikzstyle{graphnodestiny} = 
   [circle, draw=black, minimum size=15pt, text centred, inner sep=0.2mm] 
\tikzstyle{observedtiny}   =[graphnodestiny, fill=white, text=black]
\tikzstyle{unobservedtiny}   =[graphnodestiny, fill=white, text=black, style=dashed]

\newcommand{\RN}[1]{%
  \textup{\uppercase\expandafter{\romannumeral#1}}%
}
\makeatletter
\newtheorem*{rep@theorem}{\rep@title}
\newcommand{\newreptheorem}[2]{%
\newenvironment{rep#1}[1]{%
 \def\rep@title{#2 \ref{##1}}%
 \begin{rep@theorem}}%
 {\end{rep@theorem}}}
\makeatother

\newreptheorem{theorem}{Theorem}
\newreptheorem{lemma}{Lemma}

\title{Conditional Independence Testing in Hilbert Spaces with Applications to Functional Data Analysis}

\author{Anton Rask Lundborg\thanks{Part of this work was done while ARL was at the University of Copenhagen. ARL was supported by the Cantab Capital Institute for the Mathematics of Information.
}\\
University of Cambridge, UK\\
\url{a.lundborg@statslab.cam.ac.uk}
\and
Rajen D.\ Shah\thanks{RDS was supported by an EPSRC Programme Grant EP/N031938/1 and an EPSRC First Grant EP/R013381/1.}\\
University of Cambridge, UK\\
\url{r.shah@statslab.cam.ac.uk}
\and
Jonas Peters\thanks{JP was supported by the Carlsberg Foundation and a research grant (18968) from the VILLUM Foundation.} \\
University of Copenhagen, Denmark\\
\url{jonas.peters@math.ku.dk}}
\date{\today}

\begin{document}
%\begin{bibunit}[abbrvnat]

\maketitle
\begin{cbunit}
\begin{abstract}
	We study the problem of testing the null hypothesis that $X$ and $Y$ are conditionally independent given $Z$, where each of $X$, $Y$ and $Z$ may be functional random variables. This generalises testing the significance of $X$ in a
	%scalar-on-function
	regression model of scalar response $Y$ on functional regressors $X$ and $Z$. We show however that even in the idealised setting where additionally $(X, Y, Z)$ has a Gaussian distribution, the power of any test cannot exceed its size.
	Further modelling assumptions are needed and we argue that a convenient way of specifying these assumptions is based on choosing methods for regressing each of $X$ and $Y$ on $Z$.
	%We thus propose as a test statistic
	We propose a test statistic involving inner products of the resulting residuals that is simple to compute and calibrate: type I error is controlled uniformly when the in-sample prediction errors are sufficiently small. We show this requirement is met by ridge regression  in functional linear model settings without requiring any eigen-spacing conditions or lower bounds on the eigenvalues of the covariance of the functional regressor.
	We apply our test in constructing confidence intervals for truncation points in truncated functional linear models and testing for edges in a functional graphical model for EEG data.
%  We study the problem of testing the null hypothesis that $X$ and $Y$ are conditionally independent given $Z$, where each of $X$, $Y$ and $Z$ may be functional random variables. This generalises \Jonas{for example} \Rajen{okay, but shall we see what happens with the word count?} testing the significance of $X$ in a scalar-on-function linear regression model of response $Y$ on functional regressors $X$ and $Z$. We show however that even in the idealised setting where \Rm{additionally} \Rajen{I would keep this as we still need $Y$ scalar for the hardness result?} $(X, Y, Z)$ have a non-singular Gaussian distribution, the power of any test cannot exceed its size. Given that type I error cannot be controlled across the entire null, we argue that tests must be designed so their effective null hypotheses are interpretable. \Jonas{unclear to me.}
\end{abstract}

\section{Introduction}
In a variety of application areas, such as meteorology, neuroscience, linguistics, and chemometrics, we observe samples containing random functions \citep{ullah2013applications,Ramsay2005}. The field of functional data analysis (FDA) has a rich toolbox of methods for the study of such data. For instance, there are a number of regression methods for different functional data types, including
linear function-on-scalar \citep{Reiss2010}, scalar-on-function \citep{hall2007methodology, Goldsmith2011, Shin2009, Reiss2007, Yuan2010, delaigle2012methodology} and function-on-function \citep{Ivanescu2015, Scheipl2015} regression; there are also nonlinear and nonparametric variants \citep{ferraty2006, Ferraty2011, Fan2015, Yao2010}, and versions able to handle potentially large numbers of functional predictors \citep{fan2015functional}, to give a few examples; see \citet{Wang2016, Morris2015} for  helpful reviews and a more extensive list of relevant references.
The availability of software packages for functional regression methods, such as the \texttt{R}-packages \texttt{refund} \citep{refund} and \texttt{FDboost} \citep{FDboost}, allow practitioners to easily adopt the FDA framework for their particular data.

One area of FDA that has received less attention is that of conditional independence testing.
Given random elements $X, Y, Z$, the conditional independence $X \independent Y \cond Z$ formalises the idea that $X$ contains no further information about $Y$ beyond that already contained in $Z$. A precise definition is given in Section~\ref{sec:prelim}.
%For the simple setting where $(X, Y, Z)$ are multivariate Gaussian, this conditional independence is equivalent to a partial correlation of zero, that is the residuals from linearly regressing (in a population sense) each of $X$ and $Y$ on $Z$ are uncorrelated. More generally, a partial correlation of zero neither implies nor is implied by $X \independent Y \cond Z$, but there is a related characterisation of this conditional independence due to \citet{Daudin1980}:
%\begin{equation} \label{eq:daudin}
%\E \{\Cov (f(X, Z), g(Y, Z) \cond Z)\} = 0 \text{ for all } f, g \text{ such that } \E \{f(X, Z)^2\}, \, \E\{g(Y, Z)^2\} < \infty.
%\end{equation}
Inferring conditional independence from observed data is of central importance
in causal inference \citep{Pearl2009, Spirtes2000, Peters2017}, graphical 
modelling \citep{lauritzen1996, Koller2009}
and
variable selection. %\citep{Hastie2009} \Rajen{need to cite? probably not?}.
For example, consider the linear scalar-on-function regression model
\begin{equation} \label{eq:fun_reg}
Y = \int_0^1 \theta_X(t) X(t) dt + \int_0^1 \theta_Z(t) Z(t) dt + \varepsilon,
\end{equation}
where $X, Z$ are random covariate functions taking values in $L^2([0, 1], \mathbb{R})$,
%\Rajen{too specific with $[0,1]$? Interval $\mathcal{I}$? Or general index set?}
$\theta_X, \theta_Z$ are unknown parameter functions, $Y\in \R$ is a scalar response and $\varepsilon \in \mathbb{R}$ satisfying $\varepsilon \independent (X, Z)$  represents stochastic error. In this model, conditional independence $X \independent Y \cond Z$ is equivalent to $\theta_X = 0$, i.e., whether the functional predictor $X$ is significant.

For nonlinear regression models, the conditional independence $X \independent Y \cond Z$ still characterises whether $X$ is useful for predicting $Y$ given $Z$. Indeed, consider a more general setting where $Y$ is a potentially infinite-dimensional response, and $X_1, \ldots, X_p$ are predictors, some or all of which may be functional. Then a set of predictors $S \subseteq \{1,\ldots,p\}$ that contain all useful information for predicting $Y$, that is such that $Y \independent \{X_j\}_{j \notin S} \cond \{X_j\}_{j \in S}$, is known as a Markov blanket of $Y$ in the graphical modelling literature \citep[Sec.~3.2.1]{pearl2014probabilistic}. If $Y \notindependent X_j \cond \{X_k\}_{k \neq j}$, then $j$ is contained in every Markov blanket, 
%by the weak union property
and under mild conditions (e.g., %faithfulness 
the intersection property \citep{Pearl2009, Peters2014jci}),
the smallest Markov blanket (sometimes called the Markov boundary) is unique
%\citep{Pearl2009} 
%\Rajen{(e.g., faithfulness \citep{Pearl2009}), the Markov blanket is unique --
% Jonas, check please?} \Jonas{[I believe the Markov blanket as defined above is hardly ever unique. I wrote down the proof to convince myself. It's really the intersection property that we need here (this is also the claim in 
% \url{https://www.jmlr.org/papers/volume14/statnikov13a/statnikov13a.pdf});
%I would not use faithfulness since we then need to talk about graphs I believe)]}
and coincides exactly with those variables $j$ satisfying this conditional dependence. 
This set may thus be inferred by applying conditional independence tests. Conditional independence tests may also be used to test for edge presence in conditional independence graphs and are at the heart of several
	methods for causal discovery \citep{Spirtes2000, Peters2016jrssb}.

Recent work \citep{GCM} however has shown that in the setting where $X, Y$ and $Z$ are random vectors where $Z$ is absolutely continuous (i.e., has a density with respect to Lebesgue measure), testing the conditional independence $X \independent Y \cond Z$ is fundamentally hard in the sense that any test for conditional independence must have power at most its size. Intuitively, the reason for this is that given any test, there are potentially highly complex joint distributions for the triple $(X, Y, Z)$ that maintain conditional independence but yield rejection rates as high as for any alternative distribution. Lipschitz constraints on the joint density, for example, preclude the presence of such distributions \citep{neykov2020minimax}.

In the context of functional data however, the problem can be more severe, and we show in this work that even in the idealised setting where $(X, Y, Z)$ are jointly Gaussian in the functional linear regression model \eqref{eq:fun_reg}, testing for $X \independent Y \cond Z$ is fundamentally impossible: any test must have power at most its size. In other words, any test with power $\beta$ at some alternative cannot hope to control type I error at level $\alpha < \beta$ across the entirety of the null hypothesis, even if we are willing to assume Gaussianity. Perhaps more surprisingly, this fundamental problem persists even if additionally we allow ourselves to know the precise null distribution of the infinite-dimensional $Z$.

Consequently, there is no general purpose conditional independence test even for Gaussian functional data, and we must necessarily make some additional modelling assumptions to proceed. We argue that this calls for the need of conditional independence tests whose suitability for any functional data setting can be judged more easily.

Motivated by the Generalised Covariance Measure \citep{GCM}, we propose a simple test we call the Generalised Hilbertian Covariance Measure (GHCM) that involves regressing $X$ on $Z$ and $Y$ on $Z$ (each of which may be functional or indeed collections of functions),
and computing a test statistic formed from inner products of pairs of residuals.
%and then computing the Hilbert--Schmidt norm of the outer product between the resulting residuals.
We show that the validity of this form of test relies primarily on the relatively weak requirement that the regression procedures have sufficiently small in-sample prediction errors.
% We show that the validity of this form of test relies primarily on the relatively weak requirement that the
%regression procedures are able to estimate the conditional means $X$ given $Z$, and $Y$ given
%$Z$, at a slow rate.
We thus aim to convert the problem of conditional independence testing into the
more familiar task of regression with functional data, for which well-developed methods are readily available.
%Thus if appropriate regression methods can be chosen, conditional independence testing is possible.
%We see in Figure~\ref{fig:ex} that for the particular regression example studied here, the GHCM which we applied using ??, is able to successfully maintain type I error control.
These features mark out our test as rather different from existing approaches for assessing conditional independence in FDA, which we review in the following. 

%\subsection{\Jonas{Further?} Related work \Jonas{or remove this subsection?}} \label{sec:rel_work}
One approach to measuring conditional dependence with functional data is based on the Gaussian graphical model. \citet{zhu2016bayesian} propose a Bayesian approach for learning a graphical model for jointly Gaussian multivariate functional data. \citet{Qiao2019} and \citet{Zapata2019} study approaches based on generalisations of the graphical Lasso \citep{yuan2007model}. These latter methods do not aim to perform statistical tests for conditional independence, but rather provide a point estimate of the graph, for which the authors establish consistency results valid in potentially high-dimensional settings.

As discussed earlier, conditional independence testing is related to significance testing in regression models. There is however a paucity of literature on formal significance tests for functional predictors. The \texttt{R} implementation \citep{refund} of the popular functional regression methodology of \citet{Greven2017} produces $p$-values for the inclusion of a functional predictor based on significance tests for generalised additive models developed in \citet{wood2013p}. These tests, whilst being computationally efficient, however do not have formal uniform level control guarantees.

\subsection{Our main contributions and organisation of the paper}
%The rest of the paper is organised as follows.
%Our first key result is that testing conditional independence with Gaussian functional data is hard.
\subsubsection*{It is impossible to test conditional independence with  Gaussian functional data.}
In Section~\ref{sec:hardness} we present our formal hardness result on conditional independence testing for Gaussian functional data. The proof rests on a new result on the maximum power attainable at any alternative when testing for conditional independence with multivariate Gaussian data. The full technical details are given in Section~\ref{app:hardness} of the supplementary material. As we cannot hope to have level control uniformly over the entirety of the null of conditional independence, it is important to establish, for any given test, subsets $\tilde{\mathcal{P}}_0$ of null distributions $\mathcal{P}_0$ over which we do have uniform level control. 

\subsubsection*{We provide new tools allowing for the development of uniform results in FDA.}

%\Jonas{maybe move the preceding two sentence into the first point and call this point: `We provide tools allowing for developing uniform results in fda' or sth similar?}
Uniform results are scarce in functional data analysis; we develop the tools for deriving such results in
	Section~\ref{app:uniform-convergence} of the supplementary material which studies uniform convergence of Hilbertian and Banachian random variables.

\subsubsection*{Given sufficiently good methods for regressing each of $X$ and $Y$ on $Z$, the GHCM can test conditional independence with certain uniform level guarantees.}
In Section~\ref{sec:method_all} we describe our new GHCM testing framework for testing $X \independent Y \cond Z$, where each of $X$, $Y$ and $Z$ may be collections of functional and scalar variables.
%In view of the negative result on conditional independence testing with Gaussian data, any guarantees on type I error control must necessarily only hold on some subset \Rajen{$\tilde{\mathcal{P}}_0$} of the null hypothesis of conditional independence.
In Section~\ref{sec:theory} we show that for the GHCM, an effective null hypothesis $\tilde{\mathcal{P}}_0$ may be characterised as one where in addition to some tightness and moment conditions, the conditional expectations $\E(X \cond Z)$ and $\E(Y \cond Z)$ can be estimated at sufficiently fast rates, such that the product of the corresponding in-sample mean squared prediction errors (MSPEs) decay faster than $1/n$ uniformly, where $n$ is the sample size. Note that this does not contradict the hardness result: it is well known that there do not exist regression methods with risk converging to zero uniformly over all distributions for the data \citep[Thm.~3.1]{gyorfi}. Thus, the regression methods must be chosen appropriately in order for the GHCM to perform well.
%Whereas While in practice, our method requires each functional variable to be observed on an equally-spaced grid, importantly the validity of our theoretical guarantees on type I error control and power do not rely on the finiteness of the dimensions of the observed vectors of function evaluations, which can be arbitrarily large.
%and also cover the hypothetical case where the data is truly infinite dimensional.
%\Jonas{what does this mean? is this referring to situations where the data are symbolic functions, for example? (maybe add an explaining half-sentence)}
%	We believe our tests are the first for which such formal results are available.
In Section~\ref{sec:power} we show that a version of the GHCM incorporating sample-splitting has uniform power against alternatives where the expected conditional covariance operator $\E \{\Cov(X, Y \cond Z)\}$ has Hilbert--Schmidt norm of order $n^{-1/2}$, and is thus rate-optimal.
%Such uniform guarantees are especially important in the case of conditional independence testing since in particular they reveal the form of the null hypothesis actually being tested.
%Our proofs rely on new results on uniform convergence of Hilbertian and Banachian random variables which may be useful in other settings, too. These and related results we use are given in Section~\ref{app:uniform-convergence} of the supplementary material.

\subsubsection*{The regression methods are only required to perform well on the observed data.}
The fact that control of the type I error of the GHCM depends on an in-sample MSPE rather than a more conventional out-of-sample MSPE, has important consequences. Whilst in-sample and out-of-sample errors may be considered rather similar, in the context of function regression, they are substantially different. We demonstrate in Section~\ref{sec:ex} that bounds on the former are achievable under significantly weaker conditions than equivalent bounds on the latter by considering ridge regression in the functional linear model. In particular the required prediction error rates are satisfied over classes of functional linear models where the eigenvalues of the covariance operator of the functional regressor are dominated by a summable sequence; no additional eigen-spacing conditions, or lower bounds on the decay of the eigenvalues are needed, in contrast to existing results on out-of-sample error rates \citep{cai2006,hall2007methodology,Crambes2013}.

\subsubsection*{The GHCM has several uses.}
Section~\ref{sec:experiments} presents the results of numerical experiments on the GHCM. We study the following use cases.  
(i) Testing for significance of functional predictors in functional regression models.
We are not aware of other approaches that provide significance statements in functional regression models and come with statistical guarantees. 
For example, in comparison to the $p$-values from \texttt{pfr}, which are highly anti-conservative in challenging setups, the type I error of the GHCM test is 
well-controlled (see Figure~\ref{fig:level-rejection-rates}).
(ii) Deriving confidence intervals for truncation points in truncated functional linear model.
%and its use in assessing variable importance in regression models, in comparison with competitors.
We demonstrate in Section~\ref{sec:ci-truncated} the use of the GHCM in the construction of a confidence interval for the truncation point in a truncated functional linear model, a problem which we show may be framed as one of testing certain conditional independencies.
(iii) Testing for edge presence in functional graphical models.
In Section~\ref{sec:real-data}, we use the GHCM to learn functional graphical models for EEG data from a study on alcoholism.\\

We conclude with a discussion in Section~\ref{sec:conclusion} outlining potential follow-on work and open problems. The supplementary material contains the proofs of all results presented in the main text and some additional numerical experiments, as well as the uniform convergence results mentioned above. An \texttt{R}-package \texttt{ghcm} \citep{ghcm} implementing the methodology is available on CRAN. %\url{https://github.com/ARLundborg/ghcm}.  
%\end{sloppypar}
%\Rajen{
%\subsection{Functional data compared to multivariate data} \label{sec:func_multi}
%In practice all functional data is multivariate data in that its dimension is necessarily finite. Methods and theory for such data however must be able to accommodate 
%}

\subsection{Preliminaries and notation} \label{sec:prelim}
For three random elements $X$, $Y$ and $Z$ defined on the same probability space $(\Omega, \mathcal{F}, \mathbb{P})$ with values in measurable spaces $(\mathcal{X}, \mathcal{A})$, $(\mathcal{Y}, \mathcal{G})$ and $(\mathcal{Z}, \mathcal{K})$  respectively, we say that $X$ is conditionally independent of $Y$ given $Z$ and write $X \independent Y \cond Z$ when
\[
\mathbb{E}(f(X)g(Y) \cond Z) \overset{a.s}{=} \mathbb{E}(f(X) \cond Z) \mathbb{E}(g(Y) \cond Z) 
\]
for all bounded and Borel measurable $f: \mathcal{X} \to \mathbb{R}$ and $g:\mathcal{Y} \to \mathbb{R}$. Several equivalent definitions are given in \citet[][Proposition~2.3]{Constantinou2017}.
As with Euclidean variables, the interpretation of $X \independent Y \cond Z$ is that `knowing $Z$ renders $X$ irrelevant for predicting $Y$' \citep{lauritzen1996}.

Throughout the paper we consider families of probability distributions $\mathcal{P}$ of the triplet $(X, Y, Z)$, which we partition into the null hypothesis  $\mathcal{P}_0$  of those $P \in \mathcal{P}$ satisfying $X \independent Y \cond Z$, and set of alternatives $\mathcal{Q} := \mathcal{P} \setminus \mathcal{P}_0$ where the conditional independence relation is violated.
 We consider data $(x_i, y_i, z_i)$, $i=1,\ldots,n$, consisting of i.i.d.\ copies of $(X, Y, Z)$, and write $X^{(n)} := (x_i)_{i=1}^n$ and similarly for $Y^{(n)}$ and $Z^{(n)}$. We apply to this data a test $\psi_n : (\mathcal{X} \times \mathcal{Y} \times \mathcal{Z})^n \to \{0,1\}$, with a value of $1$ indicating rejection. We will at times write $\E_P(\cdot)$ for expectations of random elements whose distribution is determined by $P$, and similarly $\pr_P(\cdot) = \E_P(\ind_{\{\cdot\}})$. Thus, the size of the test $\psi_n$ may be written as $\sup_{P \in \mathcal{P}_0} \pr_P (\psi_n = 1)$. 

%The set $\mathcal{P}$ will at times consist of combinations of real-valued and functional random variables and at other times consist of purely functional random variables depending on the context.
We  always take $\mathcal{X}= \mathcal{H}_X$ and $\mathcal{Y}=\mathcal{H}_Y$ for separable Hilbert spaces $\mathcal{H}_X$ and $\mathcal{H}_Y$ and write $d_X$ and $d_Y$ for their dimensions, which may be $\infty$. When these are finite-dimensional, as will typically be the case in practice, $X^{(n)}$ will be a $n \times d_X$ matrix and similarly for $Y^{(n)}$. Similarly, we will take $\mathcal{Z} = \R^{d_Z}$ in the finite-dimensional case and then $Z^{(n)} \in \R^{n \times d_Z}$.
However, in order for our theoretical results to be relevant for settings where $d_X$ and $d_Y$ may be arbitrarily large compared to $n$, our theory must also accommodate infinite-dimensional settings, for which we introduce the following notation.

For $g$ and $h$ in a Hilbert space $\mathcal{H}$, we write $\langle g , h \rangle$ for the inner product of $g$ and $h$
and $\|g\|$ for its norm; note we suppress dependence of the norm and inner product on the Hilbert space.
%Note we will therefore suppress the dependence of Hilbert space norms and inner products on the underlying.
The bounded linear operator on $\mathcal{H}$ given by $x \mapsto \langle x, g \rangle h$ is the outer product of $g$ and $h$ and is denoted by $g \otimes h$. A bounded linear operator $\mathcal{A}$ on $\mathcal{H}$ is compact if it has a singular value decomposition, i.e., there exists two orthonormal bases $(e_{1,k})_{k \in \mathbb{N}}$ and $(e_{2,k})_{k \in \mathbb{N}}$ of $\mathcal{H}$ and a non-increasing sequence $(\lambda_k)_{k \in \mathbb{N}}$ of singular values such that 
\[
  \mathscr{A} h = \sum_{k=1}^\infty \lambda_k (e_{1,k} \otimes e_{2,k}) h =  \sum_{k=1}^\infty \lambda_k \langle e_{1,k}, h \rangle e_{2,k}
\]
for all $h \in \mathcal{H}$. For a compact linear operator $\mathcal{A}$ as above, we denote by $\lVert \mathscr{A} \rVert_{\op}$, $\lVert \mathscr{A} \rVert_{\HS}$ and $\lVert \mathscr{A} \rVert_{\TR}$ the operator norm, Hilbert--Schmidt norm and trace norm, respectively, of $\mathscr{A}$, which equal the $\ell^\infty$, $\ell^2$ and $\ell^1$ norms, respectively, of the sequence of singular values $(\lambda_k)_{k \in \mathbb{N}}$. 

A random variable on a separable Banach space $\mathcal{B}$ is a mapping $X: \Omega \to \mathcal{B}$ defined on a probability space $(\Omega, \mathcal{F}, \mathbb{P})$ which is measurable with respect to the Borel $\sigma$-algebra on $\mathcal{B}$, $\mathbb{B}(\mathcal{B})$.
Integrals with values in Hilbert or Banach spaces, including expectations, are Bochner integrals throughout. For a random variable $X$ on Hilbert space $\mathcal{H}$, we define the covariance operator of $X$ by
\[
\Cov(X) := \mathbb{E}\left[(X - \mathbb{E}(X)) \otimes (X - \mathbb{E}(X)) \right] = \mathbb{E}(X \otimes X) - \mathbb{E}(X) \otimes \mathbb{E}(X)
\]
whenever $\mathbb{E}\lVert X \rVert^2 < \infty$. For $h \in \mathcal{H}$ we thus have
\[
  \Cov(X)h = \mathbb{E}\left(\langle X, h \rangle^2 \right) - \mathbb{E}(\langle X, h \rangle)^2.
\]
For another random variable $Y$ with $\mathbb{E}\lVert Y \rVert^2 < \infty$, we define the cross-covariance operator of $X$ and $Y$ by
\[
\Cov(X, Y) := \mathbb{E}\left[(X - \mathbb{E}(X)) \otimes (Y - \mathbb{E}(Y)) \right] = \mathbb{E}(X \otimes Y) - \mathbb{E}(X) \otimes \mathbb{E}(Y) .
\]
We define conditional variants of the covariance operator and cross-covariance operator by replacing expectations with conditional expectations given a $\sigma$-algebra or random variable.

\section{The hardness of conditional independence testing with Gaussian functional data} \label{sec:hardness}
In this section we present a negative result on the possibility of testing for conditional independence with functional data in the idealised setting where all variables are Gaussian. We take $\mathcal{P}$ to consist of distributions of $(X, Y, Z)$ that are jointly Gaussian with injective covariance operator, where $X$ and $Z$ take values in separable Hilbert spaces $\mathcal{H}_X$ and $\mathcal{H}_Z$ respectively with $\mathcal{H}_Z$ infinite-dimensional, and $Y \in \R^{d_Y}$.
%for some $d_Y \in \mathbb{N}$.
We note that in the case where $d_Y=1$ and $\mathcal{H}_X = \mathcal{H}_Z = L^2([0, 1], \mathbb{R})$, each $P \in \mathcal{P}$ admits a representation as a Gaussian scalar-on-function linear model \eqref{eq:fun_reg} where $Y$ is the scalar response, and functional covariates $X, Z$ and error $\varepsilon$ are all jointly Gaussian with $\varepsilon \independent (X, Z)$ (see Proposition~\ref{prop:gaussian-conditional-distribution} in the supplementary material); the settings with $d_Y  > 1$ may be thought of equivalently as multi-response versions of this.

For each $Q$ in the set of alternatives $\mathcal{Q}$, we further define $\mathcal{P}_0^Q \subset\mathcal{P}_0$ by
\[
\mathcal{P}_0^Q := \{P \in \mathcal{P}_0 : \text{ the marginal distribution of } Z \text{ under }P \text{ and } Q \text{ is the same} \}.
\]
Theorem~\ref{thm:main-hardness-theorem} below shows that not only is it fundamentally hard to test the null hypothesis of   $\mathcal{P}_0$ against $\mathcal{Q}$ for all dataset sizes $n$, but restricting to the null $\mathcal{P}_0^Q$ for $Q \in \mathcal{Q}$ presents an equally hard problem.
%let $\mathcal{P}_0^Q \subset \mathcal{P}_0$ consist of all distributions $P \in \mathcal{P}_0$ where the marginal distributions of $Z$ under $P$ and $Q$ are the same.
%Given a sample of $n$ i.i.d.\ observations from $\mathcal{P}$, we denote the problem of testing the null hypothesis that the data-generating distribution is in $\mathcal{P}_0^Q$ against the alternative hypothesis that the distribution is $Q$ \emph{the Gaussian conditional independence testing problem on $(\mathcal{H}_X, \mathcal{H}_Y, \mathcal{H}_Z)$ against $Q$} or simply \emph{the Gaussian conditional independence testing problem against $Q$} when $\mathcal{H}_X$, $\mathcal{H}_Y$ and $\mathcal{H}_Z$ are clear from the context. 
%
%It has already been shown that conditional independence testing is hard in the scalar and multivariate setting \citep[][Theorem 2]{GCM} without additional structural assumptions. An assumption of joint Gaussianity of $(X, Y, Z)$ is sufficient to ensure that conditional independence testing is possible in the real-valued setting. Thereom~\ref{thm:main-hardness-theorem} shows that this is not the case when $Z$ is a functional random variable and either $X$ or $Y$ is multivariate. In fact, we show the stronger result that even when knowing the marginal distribution of $Z$ a priori, the Gaussian conditional independence testing problem is still hard when one of $X$ and $Y$ is multivariate.

\begin{theorem}
\label{thm:main-hardness-theorem}
Given alternative $Q \in \mathcal{Q}$ and $n \in \mathbb{N}$, let $\psi_n$ be a test for null hypothesis $\mathcal{P}_0^Q$ against $Q$. Then we have that the power is at most the size:
\[
\pr_Q(\psi_n = 1) \leq \sup_{P \in \mathcal{P}_0^Q} \pr_P(\psi_n=1).
\]
%For any $Q \in \mathcal{Q}$ the Gaussian conditional independence testing problem on $(\mathcal{H}_X, \mathcal{H}_Y, \mathcal{H}_Z)$ against $Q$ is hard when $\mathcal{H}_Z$ is infinite-dimensional and at least one of $\mathcal{H}_X$ and $\mathcal{H}_Y$ is finite-dimensional. In other words, any test with size $\alpha$ over $\mathcal{P}_0^Q$ has power at most $\alpha$ against $Q$.
\end{theorem}
An interpretation of this statement in the context of the functional linear model is that regardless of the number of observations $n$, there is no non-trivial test for the significance of the functional predictor $X$, even if the marginal distribution of the additional infinite-dimensional predictor $Z$ is known exactly.
It is clear that the size of a test over $\mathcal{P}_0$ is at least as large as that over the null $\mathcal{P}_0^Q$, so testing the larger null is of course at least as hard. 

It is known that testing conditional independence in simple multivariate (finite-dimensional) settings is hard in the sense of Theorem~\ref{thm:main-hardness-theorem} when the conditioning variable is continuous. In such settings, restricting the null to include only distributions with Lipschitz densities, for example, allows for the existence of tests with power against large classes of the alternative. The functional setting is however very different, simply removing pathological distributions from the entire null of conditional independence does not make the problem testable. Even with the parametric restriction of Gaussianity,
%It is perhaps initially surprising that the additional parametric restrictions placed on the general null of conditional independence do not render it testable, but ultimately
the null is still too large for the existence of non-trivial hypothesis tests. Indeed, the starting point of our proof is a result due to \citet{Kraft1955} that the hardness in the statement of Theorem~\ref{thm:main-hardness-theorem} is equivalent to the $n$-fold product $Q^{\otimes n}$ lying in the convex closure in total variation distance of the set of $n$-fold products of distributions in $\mathcal{P}_0^Q$.

%The proof is based on a new result concerning conditional independence testing for multivariate Gaussian data. Specifically, our result gives a bound on the maximum power attainable when testing $X \independent Y  \cond Z$ for Gaussian $(X, Y, Z) \in \R \times \R \times \R^d$ as a function of $d$. 
%on maximum power at any alternative attainable 
%It is clear that any test of size $\alpha$ over $\mathcal{P}_0$ must also be size $\alpha$ over $\mathcal{P}_0^Q$ for any $Q \in \mathcal{Q}$ hence the Gaussian conditional independence testing problem must also be hard when we do not restrict the marginal distribution of $Z$.
%These results may seem surprising, given the

A consequence of Theorem~\ref{thm:main-hardness-theorem} is that we need to make strong modelling assumptions in order to test for conditional independence in the functional data setting. Given the plethora of regression methods for functional data, we argue that it can be convenient to frame these modelling assumptions in terms of regression models for each of $X$ and $Y$ on $Z$, or more generally, in terms of the performances of methods for these regressions.
The remainder of this paper is devoted to developing a family of conditional independence tests whose validity rests primarily on the prediction errors of these regressions.

\section{GHCM methodology} \label{sec:method_all}
In this section we present the Generalised Hilbertian Covariance Measure (GHCM) for testing conditional independence with functional data.
%Our approach can be applied to any triplet $(X, Y, Z)$ where $X$ and $Y$ take values in separable Hilbert spaces, but for ease of exposition we  initially only consider the simple functional case where $X$ and $Y$ take values in $L^2[0,1]$.
To motivate the approach we take, it will be helpful to first review the construction of the Generalised Covariance Measure (GCM) developed in \citet{GCM} for univariate $X$ and $Y$, which we do in the next section. In Section~\ref{sec:GHCM} we then define the GHCM.
\subsection{Motivation}
\label{sec:method}
Consider first therefore the case where $X$ and $Y$ are real-valued random variables, and $Z$ is a random variable with values in some space $\mathcal{Z}$. We can always write $X = f(Z) + \varepsilon$ where $f(z) := \mathbb{E}(X \cond Z=z)$ and similarly $Y=g(Z) + \xi$ with $g(z) := \mathbb{E}(Y \cond Z=z)$. The conditional covariance of $X$ and $Y$ given $Z$,
\[
\Cov(X, Y \cond Z)  := \mathbb{E}\left[\{X- \mathbb{E}(X \cond Z)\} \{Y-\mathbb{E}(Y \cond Z)\} \cond Z \right] = \E(\varepsilon \xi \cond Z),
\]
has the property that $\Cov(X, Y \cond Z) = 0$ and hence $\E(\varepsilon \xi)=0$ whenever $X \independent Y \cond Z$.
%Further, note that 
%\[
%\mathbb{E}(\Cov(X, Y \cond Z)) = \mathbb{E}(\varepsilon \xi).
%\]
The GCM forms an empirical version of $\E(\varepsilon \xi)$ given data $(x_i, y_i, z_i)_{i=1}^n$ by first regressing each of $X^{(n)}$ and $Y^{(n)}$ onto $Z^{(n)}$
%$X^{(n)} := (x_i)_{i=1}^n$ and $Y^{(n)} := (y_i)_{i=1}^n$ onto $Z^{(n)} := (z_i)_{i=1}^n$
to give estimates $\hat{f}$ and $\hat{g}$ of $f$ and $g$ respectively. Using the corresponding residuals $\hat{\varepsilon}_i := x_i - \hat{f}(z_i)$ and $\hat{\xi}_i := y_i - \hat{g}(z_i)$, the product $R_i:=\hat{\varepsilon}_i \hat{\xi}_i$ is computed for each $i=1,\ldots,n$ and then averaged to give $\bar{R} :=\sum_{i=1}^n R_i /n$, an estimate of $\E(\varepsilon \xi)$. The standard deviation of $\bar{R}$ under the null $X \independent Y \cond Z$ may also be estimated, and it can be shown \citep[Thm~8]{GCM} that under some conditions, $\bar{R}$ divided by its estimated standard deviation converges uniformly to a standard Gaussian distribution.

%
%
%Hence, given $n$ i.i.d.\ observations of $(X, Y, Z)$, $(x_i, y_i, z_i)$ for $i \in \{1, \dots, n\}$, if we regress $X$ on $Z$ and $Y$ on $Z$ forming estimates $\hat{f}$ and $\hat{g}$ with corresponding residuals $\hat{\varepsilon}$ and $\hat{\xi}$, we would expect that under the null $\bar{R} := \frac{1}{n} \sum_{i=1}^n R_i \approx 0$ where $R_i := \hat{\varepsilon}_i \hat{\xi}_i$ for $i \in \{1, \dots, n\}$.
%There are several ways to extend this basic approach to multivariate settings. One such option is 
This basic approach can be extended to the case where $X$ and $Y$ take values in $\mathbb{R}^{d_X}$ and $\mathbb{R}^{d_Y}$ respectively,
%for $d_X, d_Y \in \mathbb{N}$,
by considering a multivariate conditional covariance,
\[
\Cov(X, Y \cond Z)  := \mathbb{E}\left[\{X- \mathbb{E}(X \cond Z)\}\{Y-\mathbb{E}(Y \cond Z)\}^\top  \cond Z \right]  = \E(\varepsilon \xi^\top  \cond Z) \in \R^{d_X \times d_Y}.
\]
This is a zero matrix when $X \independent Y \cond Z$, and hence $\E(\varepsilon \xi^\top )=0$ under this null. Thus, $\bar{R}$ defined as before but where $R_i := \hat{\varepsilon}_i \hat{\xi}_i^\top $ can form the basis of a test of conditional independence. There are several ways to construct a final test statistic using $\bar{R} \in \R^{d_X \times d_Y}$. The approach taken in \citet{GCM} involves taking the maximum absolute value of a version of $\bar{R}$ with each entry divided by its estimated standard deviation. This, however, does not generalise easily to the functional data setting we are interested in here; we now outline  an alternative that can be extended to handle functional data.

%Again, $X \independent Y \cond Z$ implies $\Cov(X, Y \cond Z) = 0$, $\mathbb{E}(\Cov(X, Y \cond Z)) = \mathbb{E}(\varepsilon \xi^\top )$ and performing regressions as before, we would expect $\bar{R} \approx 0$ where $R_i := \hat{\varepsilon}_i \hat{\xi}_i^\top $ for $i \in \{1, \dots, n\}$ and $0$ denotes the $d_X \times d_Y$ zero matrix.

To motivate our approach, consider multiplying $\bar{R}$ by $\sqrt{n}$:
\begin{equation} \label{eq:expansion}
	\begin{split}
\sqrt{n} \bar{R} &= \frac{1}{\sqrt{n}} \sum_{i=1}^n \hat{\varepsilon}_i \hat{\xi}_i^\top  = \frac{1}{\sqrt{n}} \sum_{i=1}^n (f(z_i) - \hat{f}(z_i) + \varepsilon_i) (g(z_i) - \hat{g}(z_i) + \xi_i)^\top  \\
& = \frac{1}{\sqrt{n}} \underbrace{\sum_{i=1}^n \varepsilon_i \xi_i^\top }_{U_n} + \underbrace{\frac{1}{\sqrt{n}} \sum_{i=1}^n  (f(z_i)-\hat{f}(z_i))  (g(z_i)-\hat{g}(z_i))^\top }_{a_n}\\
& \qquad + \underbrace{\frac{1}{\sqrt{n}} \sum_{i=1}^n (f(z_i)-\hat{f}(z_i)) \xi_i^\top  }_{b_n} + \underbrace{\frac{1}{\sqrt{n}} \sum_{i=1}^n \varepsilon_i  (g(z_i)-\hat{g}(z_i))^\top }_{c_n} .
\end{split}
\end{equation}
Observe that $U_n$ is a sum of i.i.d.\ terms and so the multivariate central limit theorem dictates that $U_n / \sqrt{n}$ converges to a $d_X \times d_Y$-dimensional Gaussian distribution.
%By the above reasoning, under the null $\varepsilon \xi^\top $ is mean zero, hence the multivariate CLT yields that $U_n$ converges in distribution to a mean zero Gaussian on $\mathbb{R}^{d_X \times d_Y}$.
Applying the Frobenius norm $\lVert \cdot \rVert_{F}$ to the $a_n$ term, we get by submultiplicativity and the Cauchy--Schwarz inequality,
\begin{align}
\lVert a_n \rVert_{F} &\leq \frac{1}{\sqrt{n}} \sum_{i=1}^n  \lVert f(z_i)-\hat{f}(z_i) \rVert_2  \lVert g(z_i)-\hat{g}(z_i) \rVert_2 \notag\\
&\leq \sqrt{n} \bigg(\frac{1}{n} \sum_{i=1}^n  \lVert f(z_i)-\hat{f}(z_i)\rVert_2^2\bigg)^{1/2} \bigg(\frac{1}{n} \sum_{i=1}^n  \lVert g(z_i)-\hat{g}(z_i)\rVert_2^2\bigg)^{1/2}, \label{eq:fg_prod}
\end{align}
where $\lVert \cdot \rVert_2$ denotes the Euclidean norm. The right-hand-side here is a product of in-sample mean squared prediction errors for each of the regressions performed.
% where $\lVert \cdot \rVert_{F}$ denotes the Frobenius norm
%and $\lVert \cdot \rVert$ denotes the Euclidean norm.
Under the null of conditional independence, each term of $b_n$ and $c_n$ is mean zero conditional on $(X^{(n)}, Z^{(n)})$ and $(Y^{(n)}, Z^{(n)})$, respectively. Thus, so long as both of the regression functions are estimated at a sufficiently fast rate, we can expect $a_n, b_n, c_n$ to be small so the distribution of $\sqrt{n} \bar{R}$ can be well-approximated by the Gaussian limiting distribution of $U_n/\sqrt{n}$. As in the univariate setting, it is crucially the product of the prediction errors in \eqref{eq:fg_prod} that is required to be small, so each root mean squared prediction error term can decay at relatively slow $o(n^{-1/4})$
rates.
%This property is known as double robustness \citep{}

Unlike the univariate setting however, $\sqrt{n} \bar{R}$ is now a matrix and hence we need to choose some sensible aggregator function $t: \mathbb{R}^{d_X \times d_Y} \to \mathbb{R}$ such that we can threshold $t(\sqrt{n}\bar{R})$ to yield a $p$-value. One option is as follows; we take a different approach as the basis of the GHCM for reasons which will become clear in the sequel.
If we vectorise $\bar{R}$, i.e., view the matrix as a $d_X d_Y$-dimensional vector, then under the assumptions required for the above heuristic arguments to formally hold, $\sqrt{n} \textrm{Vec}(\bar{R})$ converges to a Gaussian with mean zero and some covariance matrix $C \in \R^{d_Xd_Y \times d_Xd_Y}$ if $X \independent Y \cond Z$. Provided $C$ is invertible, $\sqrt{n} C^{-1/2} \bar{R}$ therefore converges to a Gaussian with identity covariance under the null and hence $\lVert C^{-1/2} \sqrt{n} \bar{R} \rVert_2^2$ converges to a $\chi^2$-distribution with $d_Xd_Y$ degrees of freedom. Replacing $C$ with an estimate $\hat{C}$ then yields a test statistic from which we may derive a $p$-value.

\subsection{The GHCM} \label{sec:GHCM}
We now turn to the setting where $X$ and $Y$ take values in separable Hilbert spaces $\mathcal{H}_X$ and $\mathcal{H}_Y$ respectively. These could for example be $L^2([0, 1], \mathbb{R})$, or %Euclidean spaces
$\R^{d_X}$ and $\R^{d_Y}$ respectively,
%of dimensions $d_X$ and $d_Y$ as before
%with the standard inner product,
but where $X$ and $Y$ are vectors of function evaluations.
%on equally spaced grids.
The latter case, which we will henceforth refer to as the finite-dimensional case, corresponds
to how data would often be received in practice with the observation vectors consisting of function evaluations on fixed grids (which are not necessarily equally spaced). However, it is important to recognise that the dimensions $d_X$ and $d_Y$ of the grids may be arbitrarily large, and it is necessary for the methodology to accommodate this; as we will see, the approach for the multivariate setting described in the previous section does not satisfy this requirement whereas our proposed GHCM will do so.

In some settings, our observed vectors of function evaluations will not be on fixed grids, and the numbers of function evaluations may vary from observation to observation. In Section~\ref{sec:unequal_grids} we set out a scheme to handle this case and bring it within our framework here.

Similarly to the approach outlined in Section~\ref{sec:method}, we propose to first regress each of $X^{(n)}$ and $Y^{(n)}$ onto $Z^{(n)}$ to give residuals $\hat{\varepsilon}_i \in \mathcal{H}_X$, $\hat{\xi}_i \in \mathcal{H}_Y$ for $i=1,\ldots,n$. (In practice, these regressions could be performed by \texttt{pfr} or \texttt{pffr} in the \texttt{refund} package \citep{Goldsmith2011,Ivanescu2015} or boosting \citep{FDboost}, for instance.) We centre the residuals, as these and other functional regression methods do not always produce mean-centred residuals. With these residuals we proceed as in the multivariate case outlined above but replacing matrix outer products in the multivariate setting with outer products in the Hilbertian sense, that is we define for $i=1,\ldots,n$,
\begin{gather}
\mathscr{R}_i:= \hat{\varepsilon}_i \otimes \hat{\xi}_i, \;\text{ and }\;\mathscr{T}_n := \sqrt{n} \bar{\mathscr{R}} \label{eq:scrT_n_def} \\
\text{where } \; \bar{\mathscr{R}} := \frac{1}{n} \sum_{i=1}^n \mathscr{R}_i. \notag
\end{gather}
%\Rajen{There are several potential obstacles when attempting to transfer the reasoning of the previous Section relating to the multivariate $\bar{R}$ of fixed dimension to $\bar{\mathscr{R}}$ which may have arbitrarily large dimension.
%	\begin{enumerate}[(i)]
%		\item The argument that $\bar{R}$ converges (uniformly) to a Gaussian distribution relies on the central limit theorem
%	\end{enumerate}
%}

We can show (see Theorem~\ref{thm:asymptotic normality and covariance estimation}) that under the null, provided the analogous prediction error terms in \eqref{eq:fg_prod} decay sufficiently fast and additional regularity conditions hold, $\mathscr{T}_n$ above converges uniformly to a Gaussian distribution in the space of Hilbert--Schmidt operators. This comes as a consequence of new results we prove on uniform convergence of Banachian random variables.
Moreover, the covariance operator of this limiting Gaussian distribution can be estimated by the empirical covariance operator
\begin{equation} \label{eq:hat_scrC}
\hat{\mathscr{C}} := \frac{1}{n-1} \sum_{i=1}^n (\mathscr{R}_i - \bar{\mathscr{R}}) \otimes_{\HS} (\mathscr{R}_i - \bar{\mathscr{R}})
\end{equation}
where $\otimes_{\HS}$ denotes the outer 
product in the space of Hilbert--Schmidt operators.

An analogous approach to that outlined above for the multivariate setting would involve attempting to whiten this limiting distribution using the square-root of the inverse of $\hat{\mathscr{C}}$.
However, here we hit a clear obstacle: even in the finite-dimensional setting, whenever $d_Xd_Y \geq n$, the inverse of $\hat{\mathscr{C}}$ or $\hat{C}$ from the previous section, cannot  exist.
Moreover, as indicated by \citet{bai1996effect}, who study the problem of testing whether a finite-dimensional Gaussian vector has mean zero, even when the inverses do exist, the estimated inverse covariance may not approximate its population level counterpart sufficiently well. Instead, \citet{bai1996effect} advocate using a test statistic based on the squared $\ell_2$-norm of the Gaussian vector.

%however this inverse does not exist due to the compactness of covariance operators.
%Furthermore, even when these are testing whether a finite-dimensional Gaussian vector $W$ with unknown variance has mean zero (rather than the infinite-dimensional setting here), it is not recommended to attempt to whiten the distribution as the estimated inverse covariance may not approximate its population level counterpart sufficiently well, as shown by \citet{bai1996effect}. Instead, \citet{bai1996effect} advocate using a test statistic based on the squared norm $\| W\|_2^2$.

We take an analogous approach here, and use as our test statistic
\begin{equation} \label{eq:T_n_def}
	T_n := \|\mathscr{T}_n\|_{\HS}^2
\end{equation}
where $\lVert \cdot \rVert_{\HS}$ denotes the Hilbert--Schmidt norm.
A further advantage of this test statistic is that it admits an alternative representation given by
\begin{equation}
	\label{eq:T_n-inner-product-form}
	T_n = \frac{1}{n}\sum_{i=1}^n \sum_{j=1}^n \langle \hat{\varepsilon}_i , \hat{\varepsilon}_j \rangle \langle \hat{\xi}_i , \hat{\xi}_j \rangle;
\end{equation}
see Section~\ref{sec:inner-product-version-of-test} for a derivation.
Only inner products between residuals need to be computed, and so
%when, as in practice, the residuals $\hat{\varepsilon}_i$ and $\hat{\varepsilon}_j$ are in fact $d_X$ and $d_Y$-dimensional respectively for potentially large but finite $d_X$ and $d_Y$,
in the finite-dimensional case with the standard inner product, the computational burden is only $O(\max(d_X, d_Y)n^2)$.
%Moreover, in principle, different inner products can be chosen to yield different test statistics, and we discuss this further below.

As $\mathscr{T}_n$ has an asymptotic Gaussian distribution under the null with an estimable covariance operator, we can deduce the asymptotic null distribution of $T_n$ as a function of $\mathscr{T}_n$. This leads to the $\alpha$-level test function $\psi_n$ given by
\begin{equation} \label{eq:GHCM_test}
	\psi_n := \ind_{\{T_n \geq q_\alpha\}}
\end{equation}
where $q_\alpha$ is the $1-\alpha$ quantile of a weighted sum
\[
\sum_{k=1}^{d} \lambda_k W_k 
\]
of independent $\chi^2_1$ distributions $(W_k)_{k=1}^d$ with weights given by the $d$ non-zero eigenvalues $(\lambda_k)_{k=1}^{d}$ of $\hat{\mathscr{C}}$. Note that $d \leq \min(n-1, d_Xd_Y)$.

These eigenvalues may also be derived from inner products of the residuals: they are equal to the eigenvalues of the $n \times n$ matrix
\[
\frac{1}{n-1}(\Gamma - J\Gamma - \Gamma J + J\Gamma J)
\]
where $J \in \R^{n \times n}$ is a matrix with all entries equal to $1/n$, and $\Gamma \in \R^{n \times n}$ has $ij$th entry given by
\begin{equation} \label{eq:Gamma}
\Gamma_{ij} :=  \langle \hat{\varepsilon}_i, \hat{\varepsilon}_j \rangle \langle \hat{\xi}_i , \hat{\xi}_j \rangle ;
\end{equation}
see Section~\ref{sec:inner-product-version-of-test} for a derivation. Thus, in the finite-dimensional case, the computation of the eigenvalues requires $O(n^2\max(d_X,d_Y,n))$ operations. In typical usage therefore, the cost for computing the test statistic given the residuals is dominated by the cost of performing the initial regressions, particularly those corresponding to function-on-function regression. Note that there are several schemes for approximating $q_\alpha$ \citep{Imhof1961,Liu2009,Farebrother1984}; we use the approach of \citet{Imhof1961} as implemented in the \texttt{QuadCompForm} package in \texttt{R} \citep{QuadCompForm} in all of our numerical experiments. 
We summarise the above construction of our test function for the finite-dimensional case with the standard inner product in Algorithm~\ref{alg:ghcm}.
% in the finite-dimensional case with the standard in\Jonas{with equally-spaced observations[?]}.
%\Jonas{this indicates the option for the infinite-dimensional case, maybe add a comment?}
%when each $x_i \in R^{d_X}$ and $y_i \in \R^{d_Y}$, and where $\langle \cdot, \cdot \rangle$ and  $\langle \cdot, \cdot \rangle$ are the usual Euclidean dot product.

In principle, different inner products may be chosen, to yield different test functions. However, the theoretical properties of the test function rely on the prediction errors of the regressions, measured in terms of the norm corresponding to the inner product used, being small. In the common case where the observed data are finite vectors of function evaluations, i.e.,
%This for example corresponds to the case where
for each $i=1,\ldots,n$, $x_{ik} = W_{X,i}(k/d_{X})$ for a function $W_{X,i} \in L_2([0,1],\R)$, and similarly for $y_i$, our default recommendation is to use the standard inner product. The residuals, $\hat{\varepsilon}_i \in \R^{d_X}$ and $\hat{\xi}_i \in \R^{d_Y}$, would then similarly correspond to underlying functional residuals via $\hat{\varepsilon}_{ik} = W_{\hat{\varepsilon},i}(k/d_{X})$ for $W_{\hat{\varepsilon},i} \in L_2([0,1],\R)$, and similarly for $\hat{\xi}_i$.
We may compare
%It is interesting to compare
the test function computed based on the computed residuals $\hat{\varepsilon}_i$ and $\hat{\xi}_i$ with that which would be obtained when replacing these with the underlying functions $W_{\hat{\varepsilon},i}$ and $W_{\hat{\xi},i}$. As the test function depends entirely on inner products between residuals, it suffices to compare
\begin{equation} \label{eq:Riemann}
\hat{\varepsilon}_i^\top  \hat{\varepsilon}_j = \sum_{k=1}^{d_X} W_{\hat{\varepsilon},i}(k/d_X) W_{\hat{\varepsilon},i}(k/d_X)  \qquad \text{and} \qquad  \int_0^1 W_{\hat{\varepsilon},i}(t) W_{\hat{\varepsilon},j}(t) \,\mathrm{d}t.
\end{equation}
We see that the LHS is $d_X$ times a Riemann sum approximation to the integral on the RHS. The $p$-value computed is invariant to multiplicative scaling of the test statistic, and so in the so-called densely observed case where $d_X$ is large, the $p$-value from the finite-dimensional setting would be a close approximation to that which would be obtained with the true underlying functions.

Other numerical integration schemes could be used to make the approximation even more precise. However, the theory we present in Section~\ref{sec:theory} that guarantees uniform asymptotic level control and power over certain classes of nulls and alternatives applies directly to the finite-dimensional or infinite-dimensional settings, and so there is no requirement that the approximation error above is small. In particular, there is no strict requirement that the residuals computed correspond to function evaluations on equally spaced grids. However, in that case $\hat{\varepsilon}_i^\top  \hat{\varepsilon}_j$ will not necessarily approximate a scaled version of the RHS of \eqref{eq:Riemann}, and an inner product that maintains this approximation may be more desirable from a power perspective. 

\begin{algorithm}[ht!] \caption{Generalised Hilbertian Covariance Measure (GHCM)} \label{alg:ghcm}
  \textbf{input}: $X^{(n)} \in \R^{n \times d_X}$, $Y^{(n)} \in \R^{n \times d_Y}$, $Z^{(n)} \in \R^{n \times d_Z}$ \;
  \textbf{options}: regression methods for each of the regressions \; %
  \Begin{
    regress $X^{(n)}$ on $Z^{(n)}$ producing residuals $\hat{\varepsilon}_i \in \mathbb{R}^{d_X}$ for $i=1, \dots, n$ \;
    regress $Y^{(n)}$ on $Z^{(n)}$ producing residuals $\hat{\xi}_i \in \mathbb{R}^{d_Y}$ for $i=1, \dots, n$ \;
    construct $\Gamma \in \R^{n \times n}$ with entries $\Gamma_{ij} \leftarrow \hat{\varepsilon}_i^\top  \hat{\varepsilon}_j \hat{\xi}_i^\top  \hat{\xi}_j$ (or more generally via \eqref{eq:Gamma}) \;
    compute test statistic $T_n \leftarrow \frac{1}{n} \sum_{i=1}^n \sum_{j=1}^n \Gamma_{ij}$ \;
    set $A \leftarrow \frac{1}{n-1} \left(\Gamma - J\Gamma - \Gamma J + J\Gamma J \right)$ where $J \in \R^{n \times n}$ has all entries equal to $1/n$ \;
    compute the non-zero eigenvalues $\lambda_1, \dots, \lambda_{d}$ of $A$ (there are at most $n-1$)\;
    compute by numerical integration $p$-value
    $p \leftarrow \mathbb{P}\left( \sum_{k=1}^{d} \lambda_k \zeta_k^2 > T_n  \right)$,    
    where $\zeta_1, \dots, \zeta_{d}$ are independent standard Gaussian variables \;
  }
  \textbf{output}: $p$-value $p$\;
\end{algorithm} 
%\begin{sloppypar} % This allows line break for the url.
In the following section we explain how when the residuals $\hat{\varepsilon}_i$ and $\hat{\xi}_i$ correspond to function evaluations on different grids for each $i$, we can preprocess these to obtain residuals corresponding to fixed grids, which may then be fed into our algorithm.

An \texttt{R}-package \texttt{ghcm} \citep{ghcm} implementing the methodology is available on CRAN.
%\end{sloppypar}
% currently being developed for densely-observed functional data. The package will be submitted to CRAN soon and in the meantime a development version can be downloaded at \url{https://github.com/ARLundborg/ghcm}. 

\subsubsection{Data observed on irregularly spaced grids of varying lengths} \label{sec:unequal_grids}
We now consider the case where $\hat{\varepsilon}_{i} \in \R^{d_{X,i}}$ with its $k$th component given by $\hat{\varepsilon}_{ik} = W_{\hat{\varepsilon},i}(t_{ik})$ for $t^X_{ik} \in [0,1]$, and similarly for $\hat{\xi}_{i}$. Such residuals would typically be output by regression methods when supplied with functional data $x_i \in \R^{d_{X,i}}$ and $y_i \in \R^{d_{Y,i}}$ corresponding to functional evaluations on grids $(t_{ik})_{k=1}^{d_{X,i}}$ and $(t_{ik})_{k=1}^{d_{Y,i}}$ respectively.

In order to apply our GHCM methodology, we need to represent these residual vectors by vectors of equal lengths corresponding to fixed grids. Our approach is to construct for each $i$, natural cubic interpolating splines $\hat{W}_{\hat{\varepsilon},i}$ and $\hat{W}_{\hat{\xi},i}$ corresponding to $\hat{\varepsilon}_i$ and $\hat{\xi}_i$ respectively. We may compute the inner product between these functions in $L_2([0,1], \R)$ exactly and efficiently as it is the integral of a piecewise polynomial with the degree in each piece at most $6$. This gives us the entries of the matrix $\Gamma$ \eqref{eq:Gamma} which we may then use in lines 7 and following in Algorithm~\ref{alg:ghcm}. Furthermore, Theorems~\ref{thm:level of test} and \ref{thm: consistency of test} apply equally well to the setting considered here provided the residuals are understood as the interpolating splines described above, and the fitted regression functions are defined accordingly as the difference between the observed functional responses these functional residuals.
%
%The theory that we present below should be understood as applying to the 
%\jonas{We see empirically in Section~XX that in many settings this procedure still holds the level guarantees that we now formally prove for the equal grid case. [?]}

\section{Theoretical properties of the GHCM} \label{sec:theory}
%We now turn to the theoretical analysis of the performance of the GHCM.
In this section, we provide uniform level control guarantees for the GHCM, and uniform power guarantees for a version incorporating sample-splitting; note that we do not recommend the use of the latter in practice but consider it a proxy for the GHCM that is more amenable to theoretical analysis in non-null settings.
Before presenting these results, we explain the importance of uniform results in this context, and set out some notation relating to uniform convergence.

\subsection{Background on uniform convergence}
%We consider functional data that is fully observed functions here, and so our analysis is most relevant for the practical setting where functions are observed densely.
In Section~\ref{sec:hardness} we saw that even when $\mathcal{P}$ consists of Gaussian distributions over $\mathcal{H}_X \times \R^{d_Y} \times \mathcal{H}_Z$, we cannot ensure that our test has both the desired size $\alpha$ over $\mathcal{P}_0$ and also non-trivial power properties against alternative distributions in $\mathcal{Q}$. We also have the following related result.
\begin{proposition} \label{prop:gaussian-conditional-independence-testing-is-hard-when-dependence-is-never-in-tail}
	Let $\mathcal{H}_Z$ be a separable Hilbert space with orthonormal basis $(e_k)_{k \in \mathbb{N}}$. Let $\mathcal{P}$ be the family of Gaussian distributions for $(X, Y, Z) \in \R \times \R \times \mathcal{H}_Z$ with injective covariance operator and where $(X, Y) \independent (Z_{r+1}, Z_{r+2},\ldots) \cond Z_1,\ldots,Z_r$ for some $r \in \mathbb{N}$ and $Z_k := \langle e_k, Z \rangle$ for all $k \in \mathbb{N}$. Let $Q \in \mathcal{Q}$ and recall the definition of $\mathcal{P}_0^Q$ from Section~\ref{sec:hardness}. Then, for any test $\psi_n$,
	\[
	\pr_Q(\psi_n = 1) \leq \sup_{P \in \mathcal{P}_0^Q} \pr_P(\psi_n=1).
	\]
\end{proposition}
In other words, even if we know a basis $(e_k)_{k \in \mathbb{N}}$ such that in particular the conditional expectations $\E(X \cond Z)$ and $\E(Y \cond Z)$ are sparse in that they depend only on finitely many components $Z_1,\ldots,Z_r$ (with $r \in \mathbb{N}$ unknown), and the marginal distribution of $Z$ is known exactly, there is still no non-trivial test of conditional independence.

In this specialised setting, it is however possible to give a test of conditional independence that will, for each \emph{fixed} null hypothesis $P \in \mathcal{P}_0$, yield exact size control and power against all alternatives $\mathcal{Q}$ for $n$ sufficiently large. These properties are for example satisfied by the nominal $\alpha$-level $t$-test $\psi_n^{\OLS}$ for $Y$ in a linear model of $X$ on $Y, Z_1,\ldots,Z_{a(n)}$ and an intercept term, for some sequence $a(n) < n-1$ with $a(n) \to \infty$ and $n - a(n) \to \infty$ as $n \to \infty$. Indeed,
%if $\psi_n^{\OLS}$ is a nominal $\alpha$-level test we have
\begin{equation} \label{eq:OLStest}
\sup_{P \in \mathcal{P}_0} \lim_{n \to \infty} \pr_P( \psi_n^{\OLS} = 1) =\alpha \qquad \text{ and } \qquad  \inf_{Q \in \mathcal{Q}} \lim_{n \to \infty} \pr_Q( \psi_n^{\OLS} = 1) =1;
\end{equation}
see Section~\ref{sec:OLS-proof} in the supplementary material for a derivation. This illustrates the difference between pointwise asymptotic level control in the left-hand side of \eqref{eq:OLStest}, and uniform asymptotic level control given by interchanging the limit and the supremum.

%Often, pointwise asymptotic level control is easily achieveable, but for testing, uniform control is desired.
%We argue \Jonas{where? Maybe: "This sort of..."?} \Rajen{I was trying to use `we argue' in place of something like `our view is', which sounds a bit funny, as we} that here a pointwise analysis, which would purport to justify the use of  a test such as $\psi_n^{\OLS}$  and indeed suggest that $a(n)$ should be chosen as small as possible whilst still maintaining $a(n) \to \infty$ to maximise power, is not very informative. 

%We have seen that such uniform control is impossible to attain whilst maintaining power even in the Gaussian setting

Our analysis instead focuses on proving that the GHCM asymptotically maintains its level uniformly over a subset of
the conditional independence null.
%$\mathcal{P}_0$.
%governed largely by the prediction properties of the regression methods it is based on.
In order to state our results we first introduce some definitions and notation to do with uniform stochastic convergence.
%To do so, we use the notion of uniform stochastic convergence.
Throughout the remainder of this section we tacitly assume the existence of a measurable space $(\Omega, \mathcal{F})$ whereupon all random quantities are defined. The measurable space is equipped with a family of probability measures $(\mathbb{P}_\dist)_{\dist \in \mathcal{P}}$ such that the distribution of $(X, Y, Z)$ under $\mathbb{P}_\dist$ is $\dist$. For a subset $\mathcal{A} \subseteq \mathcal{P}$, we say that a sequence of random variables $W_n$ \emph{converges uniformly in distribution to $W$ over $\mathcal{A}$} and write if
\[
W_n \underset{\mathcal{A}}{\overset{\mathcal{D}}{\rightrightarrows}} W \qquad \text{if} \;\;\; \lim_{n \to \infty} \sup_{\dist \in \mathcal{A}} d_{{\BL}}(W_n, W) = 0,
\]
where $d_{\BL}$ denotes the bounded Lipschitz metric.
%, and we write $W_n \underset{\mathcal{A}}{\overset{\mathcal{D}}{\rightrightarrows}} W$ or simply $W_n \overset{\mathcal{D}}{\rightrightarrows} W$ when $\mathcal{A}$ is clear from the context.
We say, \emph{$W_n$ converges uniformly in probability to $W$ over $\mathcal{A}$} and write
%if for any $\varepsilon > 0$
\[
W_n \underset{\mathcal{A}}{\overset{P}{\rightrightarrows}} W \qquad \text{if for any }\varepsilon > 0, \;\;\; \lim_{n \to \infty} \sup_{\dist \in \mathcal{A}} \mathbb{P}_{\dist}( \lVert W_n - W \rVert \geq \varepsilon) = 0.
\]
We sometimes omit the subscript $\mathcal{A}$ when it is clear from the context.
%and we write $W_n \underset{\mathcal{A}}{\overset{P}{\rightrightarrows}} W$ or simply $W_n \overset{P}{\rightrightarrows} W$ when $\mathcal{A}$ is clear from the context.
A full treatment of uniform stochastic convergence in a general setting is given in Section~\ref{app:uniform-convergence} of the supplementary material. Throughout this section we emphasise the dependence of many of the quantities in Section~\ref{sec:method} on the distribution of $(X, Y, Z)$ with a subscript $\dist$, e.g.\ $f_\dist$, $\varepsilon_\dist$ etc.

In Sections~\ref{sec:size} and \ref{sec:power} we present general results on the size and power of the GHCM.
We take $\mathcal{P}$ to be the set of  all distributions over $\mathcal{H}_X \times \mathcal{H}_Y \times \mathcal{Z}$, and $\mathcal{P}_0$ to be the corresponding conditional independence null. We however show properties of the GHCM under smaller sets of distributions $\tilde{\mathcal{P}} \subset \mathcal{P}$ with corresponding null distributions $\tilde{\mathcal{P}}_0 \subset \mathcal{P}_0$, where in particular certain conditions on the quality of the regression procedures on which the test is based are met. 
In Section~\ref{sec:ex} we consider the special case where the regressions of each of $X$ and $Y$ on $Z$ are given by functional linear models and show that Tikhonov regularised regression can satisfy these conditions.
We note that throughout, the dimensions $d_X$ and $d_Y$ may be finite or infinite.

\subsection{Size of the test} \label{sec:size}

In order to state our result on the size of the GHCM, we introduce the following quantities. Let
\begin{equation*}
u_\dist(z) := \mathbb{E}_\dist\left(\lVert \varepsilon_\dist \rVert^2 \cond Z = z\right), \quad v_\dist(z) := \mathbb{E}_\dist\left(\lVert \xi_\dist \rVert^2 \cond Z = z\right).
\end{equation*}
We further define the in-sample unweighted and weighted mean squared prediction errors of the regressions as follows:
\begin{align}
M_{n, \dist}^f &:= \frac{1}{n} \sum_{i=1}^n \left\lVert f_\dist(z_i) - \hat{f}^{(n)}(z_i) \right\rVert^2, &\quad  &&
M_{n, \dist}^g &:= \frac{1}{n} \sum_{i=1}^n \left\lVert g_\dist(z_i) - \hat{g}^{(n)}(z_i) \right\rVert^2, \label{eq:unweighted_err} \\
 \tilde{M}_{n, \dist}^f &:= \frac{1}{n} \sum_{i=1}^n \left\lVert f_\dist(z_i) - \hat{f}^{(n)}(z_i) \right\rVert^2 v_\dist(z_i),  &\quad  &&
 \tilde{M}_{n, \dist}^g &:= \frac{1}{n} \sum_{i=1}^n \left\lVert g_\dist(z_i) - \hat{g}^{(n)}(z_i) \right\rVert^2 u_\dist(z_i). \label{eq:weighted_err}
\end{align}
The result below shows that on a subset $\tilde{\mathcal{P}}_0$ of the null distinguished primarily by the product of the prediction errors in \eqref{eq:unweighted_err} being small, the operator-valued statistic $\mathscr{T}_n$ converges in distribution uniformly to a mean zero Gaussian whose covariance can be estimated consistently. We remark that prediction error quantities in \eqref{eq:unweighted_err} and \eqref{eq:weighted_err} are ``in-sample'' prediction errors, only reflecting the quality of estimates of the conditional expectations $f$ and $g$ at the observed values $z_1,\ldots,z_n$.
%In Theorem~\ref{thm:asymptotic normality and covariance estimation} we show that under the null, $\mathscr{T}_n$ \eqref{eq:scrT_n_def} converges in distribution to a Gaussian with a $\dist$-dependent covariance uniformly over a subset of the null hypothesis $\mathcal{P}_0$, and that $\hat{\mathscr{C}}$ \eqref{eq:hat_scrC} is a consistent estimator of this covariance. 
%Define further the mean squared prediction error and the weighted mean squared prediction error as
%$$
%M_{n, \dist}^f := \frac{1}{n} \sum_{i=1}^n \left\lVert f_\dist(z_i) - \hat{f}^{(n)}(z_i) \right\rVert^2 \quad \text{and} \quad \tilde{M}_{n, \dist}^f := \frac{1}{n} \sum_{i=1}^n \left\lVert f_\dist(z_i) - \hat{f}^{(n)}(z_i) \right\rVert^2 v_\dist(z_i) ,
%$$
%and, 
%$$
%M_{n, \dist}^g := \frac{1}{n} \sum_{i=1}^n \left\lVert g_\dist(z_i) - \hat{g}^{(n)}(z_i) \right\rVert^2 \quad \text{and} \quad \tilde{M}_{n, \dist}^g := \frac{1}{n} \sum_{i=1}^n \left\lVert g_\dist(z_i) - \hat{g}^{(n)}(z_i) \right\rVert^2 u_\dist(z_i) ,
%$$
%for $f$ and $g$, respectively.

\begin{theorem}
\label{thm:asymptotic normality and covariance estimation}
Let $\tilde{\mathcal{P}}_0 \subseteq \mathcal{P}_0$ be such that uniformly over $\tilde{\mathcal{P}}_0$,
\begin{enumerate}[(i)]
	\item $n M_{n, \dist}^f M_{n, \dist}^g \overset{P}{\rightrightarrows} 0$,
	\item $\tilde{M}_{n, \dist}^f \overset{P}{\rightrightarrows} 0$, $\tilde{M}_{n, \dist}^g \overset{P}{\rightrightarrows} 0$,
	\item $\inf_{P \in \tilde{\mathcal{P}}_0} \mathbb{E}_\dist \left(\lVert \varepsilon_\dist \rVert^2 \lVert \xi_\dist \rVert^2\right) > 0$ and $\sup_{\dist \in \tilde{\mathcal{P}}_0} \mathbb{E}_\dist\left(\lVert \varepsilon_\dist \rVert^{2+\eta} \lVert \xi_\dist \rVert^{2+\eta}\right) < \infty $ for some $\eta > 0$, and
	\item for some orthonormal bases $(e_{X, i})_{i=1}^{d_X}$ and $(e_{Y, j})_{j=1}^{d_Y}$ of $\mathcal{H}_X$ and $\mathcal{H}_Y$, respectively, writing $\varepsilon_{P,i} := \langle e_{X,i}, \varepsilon_P \rangle$ and $\xi_{P,j} := \langle e_{Y,j}, \xi_P \rangle$, we have 
	\[
	\lim_{K \to \infty} \sup_{\dist \in \tilde{\mathcal{P}}_0} \sum_{ (i,j): i+j \geq K} \E_P (\varepsilon_{P,i}^2 \xi_{P,j}^2)=0,
  \]
  where we interpret an empty sum as $0$.
\end{enumerate}
Then uniformly over $\tilde{\mathcal{P}}_0$ we have
\[
\mathscr{T}_n {\overset{\mathcal{D}}{\rightrightarrows}} \mathcal{N}(0, \mathscr{C}_\dist ) \quad \text{and} \quad   \lVert \hat{\mathscr{C}} - \mathscr{C}_\dist \rVert_{\TR} {\overset{P}{\rightrightarrows}} 0,
\]
where
\[
\mathscr{C}_P := \E \{(\varepsilon_P \otimes \xi_P) \otimes_{\HS} (\varepsilon_P \otimes \xi_P)\}.
\]
\end{theorem}
Condition (i) is the most important requirement, and says that the regression methods must perform sufficiently well, uniformly on $\tilde{\mathcal{P}}_0$. It is satisfied if $\sqrt{n} M_{n, \dist}^f, \, \sqrt{n} M_{n, \dist}^g\overset{P}{\rightrightarrows} 0$, and so allows for relatively slow $o(\sqrt{n})$ rates for the mean squared prediction errors. Moreover, if one regression yields a faster rate, the other can go to zero more slowly. These properties are shared with the regular generalised covariance measure and more generally doubly robust procedures popular in the literature on causal inference and semiparametric statistics \citep{robins1995semiparametric, scharfstein1999adjusting, chernozhukov2018double}. Condition (ii) is much milder, and if the conditional variances $u_P$ and $v_P$ are bounded almost surely, it is satisfied when simply $M_{n, \dist}^f, \, M_{n, \dist}^g\overset{P}{\rightrightarrows} 0$.
We note that importantly, the regression methods are not required to extrapolate well beyond the observed data. We show in Section~\ref{sec:ex} that when the regression models are functional linear models and ridge regression is used for the functional regressions, (i) and (ii) hold under much weaker conditions than are typically required for out-of-sample prediction error guarantees in the literature.

Conditions (iii) and (iv) imply that the family $\{ \varepsilon_P \otimes \xi_P : P \in \tilde{\mathcal{P}}_0\}$ is uniformly tight. Similar tightness conditions are required in \citet[Lem.~3.1]{chen1998central} in the context of functional central limit theorems. Note that if $d_X$ and $d_Y$ are both finite, this condition is always satisfied.
%\Rajen{Condition (iv) is related to uniform tightness of the family $\{ \varepsilon_P \otimes \xi_P : P \in \tilde{\mathcal{P}}_0\}$, which is trivially satisfied if $\tilde{\mathcal{P}}_0$ is finite, or if $\mathcal{H}_X$ and $\mathcal{H}_Y$ are finite-dimensional; similar conditions occur in \citet[Lem.~3.1]{chen1998central} in the context of functional central limit theorems.}

The result below shows that the GHCM test $\psi_n$ \eqref{eq:GHCM_test}
%based on the test statistic $T_n := \|\mathscr{T}_n\|_{\HS}$
has type I error control uniformly over $\tilde{\mathcal{P}}_0$ given in Theorem~\ref{thm:asymptotic normality and covariance estimation}, provided an additional assumption of non-degeneracy of the covariance operators is satisfied.
\begin{theorem} \label{thm:level of test}
	Let $\tilde{\mathcal{P}}_0 \subseteq \mathcal{P}_0$ satisfy the conditions stated in Theorem~\ref{thm:asymptotic normality and covariance estimation}, and in addition suppose
	\begin{equation} \label{eq:degen}
	\inf_{\dist \in \tilde{\mathcal{P}}_0} \lVert \mathscr{C}_\dist \rVert_{\op} > 0.
	\end{equation}
	Then for each $\alpha \in (0, 1)$, the $\alpha$-level GHCM test $\psi_n$ \eqref{eq:GHCM_test} satisfies
	\begin{equation} \label{eq:level}
	\lim_{n \to \infty} \sup_{\dist \in \tilde{\mathcal{P}}_0} | \mathbb{P}_\dist(\psi_n = 1) - \alpha | = 0.
	\end{equation}
%\Rajen{Can we instead show that for the GHCM $p$-value $p_n$, we have
%\[
%\lim_{n \to \infty} \sup_{\dist \in \tilde{\mathcal{P}}_0} \sup_{\alpha \in [0,1]} | \mathbb{P}_\dist(p_n \leq \alpha) - \alpha | = 0
%\]
%}
%Let
%$\psi_n$ denote the test function of the GHCM with desired level $\alpha$,
%
%
%i.e.\ the function that is $1$ when $T_n$ exceeds the empirical $1-\alpha$ quantile as descrived in Section~\ref{sec:method} and $0$ otherwise.
%\begin{enumerate}
%\item Under the pointwise conditions of Theorem~\ref{thm:asymptotic normality and covariance estimation}, for every $\dist \in \mathcal{P}_0$ 
%\[
%\lim_{n \to \infty} \mathbb{P}_\dist(\psi_n = 1) = \alpha .
%\]
%\item Under the uniform conditions of Theorem~\ref{thm:asymptotic normality and covariance estimation} and the additional assumption that 
%\[
%0 < \inf_{\dist \in \tilde{\mathcal{P}}_0} \lVert \mathscr{C}_\dist \rVert_{\op},
%\]
%\[
%\lim_{n \to \infty} \sup_{\dist \in \tilde{\mathcal{P}}_0} | \mathbb{P}_\dist(\psi_n = 1) - \alpha | = 0.
%\]
%\end{enumerate}
\end{theorem}

\subsection{Power of the test} \label{sec:power}
We now study the power of the GHCM.
It is not straightforward to analyse what happens to the test statistic $T_n$ when the null hypothesis is false in the setup we have considered so far. However, if we modify the test such that the regression function estimates $\hat{f}$ and $\hat{g}$ are constructed using an auxiliary dataset independent of the main data $(x_i, y_i, z_i)_{i=1}^n$, the behaviour of $T_n$ is more tractable. Given a single sample, this could be achieved through sample splitting, and cross-fitting \citep{chernozhukov2018double} could be used to recover the loss in efficiency from the split into smaller datasets.
However, we do not recommend such sample-splitting in practice here and view this as more of a technical device that facilitates our theoretical analysis. As we require $\hat{f}$ and $\hat{g}$ to satisfy (i) and (ii) of Theorem~\ref{thm:asymptotic normality and covariance estimation}, these estimators would need to perform well out of sample rather than just on the observed data, which is typically a harder task.

Given that our test is based on an empirical version of $\mathbb{E}(\Cov(X, Y \cond Z)) = \E (\varepsilon \otimes \xi)$, we can only hope to have power against alternatives where this is non-zero. For such alternatives however, we have positive power whenever the Hilbert--Schmidt norm of the expected conditional covariance operator is at least $c/\sqrt{n}$ for a constant $c>0$, as the following result shows.
%we have an auxiliary dataset $A$ at our disposal to construct the regression estimates $\hat{f}$ and $\hat{g}$, we can show that the test is consistent whenever $\mathbb{E}(\Cov(X, Y \cond Z)) \neq 0$. 
\begin{theorem}
\label{thm: consistency of test}
Consider a version of the GHCM test $\psi_n$ where $\hat{f}$ and $\hat{g}$ are constructed on independent auxiliary data. Let $\tilde{\mathcal{P}} \subset \mathcal{P}$ be the set of distributions for $(X, Y, Z)$ satisfying (i)-(iv) of Theorem~\ref{thm:asymptotic normality and covariance estimation} and \eqref{eq:degen} with $\tilde{\mathcal{P}}$ in place of $\tilde{\mathcal{P}}_0$. Then writing $\mathscr{K}_\dist := \mathbb{E}_\dist(\varepsilon_\dist \otimes \xi_\dist) = \mathbb{E}_\dist(\Cov_\dist(X, Y\cond  Z))$, we have, uniformly over $\tilde{\mathcal{P}}$,
\[
\tilde{\mathscr{T}}_n := \frac{1}{\sqrt{n}} \sum_{i=1}^n (\mathscr{R}_i - \mathscr{K}_\dist) \overset{\mathcal{D}}{\rightrightarrows}  \mathcal{N}(0, \mathscr{C}_\dist ) \qquad \text{and} \qquad
\lVert \hat{\mathscr{C}} - \mathscr{C}_\dist \rVert_{\TR} \overset{P}{\rightrightarrows}  0.
\]
Furthermore, an $\alpha$-level GHCM test $\psi_n$ (constructed using independent estimates $\hat{f}$ and $\hat{g}$) satisfies the following two statements.
\begin{enumerate}[(i)]
	\item Redefining $\tilde{\mathcal{P}}_0 = \tilde{\mathcal{P}} \cap \mathcal{P}_0$, we have that \eqref{eq:level} is satisfied, and so an $\alpha$-level GHCM test has size converging to $\alpha$ uniformly over $\tilde{\mathcal{P}}_0$.
	\item  For every $0 < \alpha < \beta < 1$ there exists $c > 0$  and $N \in \mathbb{N}$ such that for any $n \geq N$,
	\[
	\inf_{\dist \in \mathcal{Q}_{c, n}} \mathbb{P}_{\dist}(\psi_n = 1) \geq \beta,
	\]
	where $\mathcal{Q}_{c, n} := \{\dist \in \tilde{\mathcal{P}} \ : \ \lVert  \mathscr{K}_\dist \rVert_{\HS} > c / \sqrt{n}  \}$.
\end{enumerate}
\end{theorem}
In a setting where $X$, $Y$ and $Z$ are related by linear regression models, we can write down $\|\E \Cov(X, Y \cond Z)\|_{\HS}$ more explicitly. Suppose
 $Z$, $\varepsilon$ and $\xi$ are independent random variables in $L^2([0, 1], \mathbb{R})$, with $X$ and $Y$ determined by
\begin{align*}
	X(t) &= \int \beta^X(s, t) Z(s) \, \mathrm{d}s + \varepsilon(t) \\
	Y(t) &= \int \beta^Y(s, t) Z(s) \, \mathrm{d}s + \int \theta(s, t) X(s) \, \mathrm{d}s + \varepsilon+ \xi(t).
\end{align*}
Then $\E \Cov(X, Y \cond Z)$ is an  
 integral operator with kernel 
$$
\phi(s, t) = \int_0^1 \theta(u, s) v(t, u) \, \mathrm{d}u,
$$
where $v(t, u)$ denotes
%the kernel of the covariance operator of $\varepsilon$, i.e.
the covariance function of $\varepsilon$. The Hilbert--Schmidt norm $\|\E \Cov(X, Y \cond Z)\|_{\HS}$ is then given by the $L^2([0, 1]^2, \mathbb{R})$-norm of $\phi$. We investigate the empirical performance of the GHCM in such a setting in Section~\ref{sec:functional-level-power-sim}.
%\Anton{Could you edit / double check my writing here, Rajen?}
%\Anton{Add comment referencing new figure in appendix about empirical performance.}
% \Jonas{move the following sentence into introduction/related work} \Anton{I've just commented out this line since I'm not even sure it is going to be relevant in the intro/relevant work}
% A similar analysis was performed by \citet[][Theorem~8]{GCM} and the same caveats apply. 
%While sample-splitting simplifies the analysis of the power properties of the test, it also affects the theoretical assumptions on the regression methods: instead of bounding an in-sample error, one now needs to bound an out-of-sample error, which usually comes with stronger requirements. 

\subsection{GHCM using linear function-on-function ridge regression} \label{sec:ex}
Here we consider a special case of the general setup used in Sections~\ref{sec:size} and \ref{sec:power} where we assume that $\mathcal{Z}$ is a Hilbert space $\mathcal{H}_Z$ and that, under the null of conditional independence, the Hilbertian $X$ and $Y$ are related to Hilbertian $Z$ via linear models:
\begin{align}
	X &= \mathscr{S}^X_\dist Z + \varepsilon_\dist  \label{eq:x_mod}\\
	Y &= \mathscr{S}^Y_\dist Z + \xi_\dist. \label{eq:y_mod}
\end{align}
Here $\mathscr{S}^X_\dist$ is a Hilbert--Schmidt operator such that $\mathscr{S}^X_\dist Z = f(Z) := \E(X \cond Z)$, with analogous properties holding for $\mathscr{S}^Y_\dist$, and it is assumed that $\E Z = 0$. If $X$, $Y$ and $Z$ are elements of $L^2([0, 1], \mathbb{R})$, this is equivalent to 
\begin{equation}
  \label{eq:process-linear-operator}
  X(t) = \int_0^1 \beta^X_P(s, t) Z(s) \, \mathrm{d}s + \varepsilon_P(t),
\end{equation}
where $\beta^X_\dist$ is a square-integrable function, and similarly for the relationship between $Y$ and $Z$. Such functional response linear models have been discussed by \citet[Chap.~16]{Ramsay2005}, and studied by \citet{chiou2004functional, yao2005functional, Crambes2013}, for example. \citet{Benatia2017} propose a Tikhonov regularised estimator analogous to ridge regression \citep{hoerl1970ridge}; applied to the regression model \eqref{eq:x_mod}, this estimator takes the form
\begin{equation}
	\label{eq:penalised-likelihood-criterion-hilbertian-linear-model}
	\hat{\mathscr{S}} = \argmin_{\mathscr{S}} \sum_{i=1}^n \lVert x_i - \mathscr{S}(z_i) \rVert^2  + \gamma \lVert \mathscr{S} \rVert_{\HS}^2,
\end{equation}
where $\gamma >0 $ is a tuning parameter.

We now consider a specific instance of the general GHCM framework using regression estimates based on \eqref{eq:penalised-likelihood-criterion-hilbertian-linear-model}. Specifically, we form estimate $\hat{\mathscr{S}}^X$ of $\mathscr{S}^X$ by solving the optimisation in \eqref{eq:penalised-likelihood-criterion-hilbertian-linear-model} with regularisation parameter
\begin{equation}
  \label{eq:gamma_hat}
\hat{\gamma} :=  \argmin_{\gamma > 0} \left( \frac{1}{\gamma n} \sum_{i=1}^n \min(\hat{\mu}_i/4, \gamma)  + \frac{\gamma}{4} \right),
\end{equation}
where $\hat{\mu}_1 \geq \hat{\mu}_2 \geq \dots \geq \hat{\mu}_n \geq 0$ are the ordered eigenvalues of the $n \times n$ matrix $K$ with $K_{ij}=\langle z_i, z_j \rangle/n$.
We form estimate $\hat{\mathscr{S}}^Y$ of $\mathscr{S}^Y$ analogously but with the $x_i$ replaced by $y_i$ in \eqref{eq:penalised-likelihood-criterion-hilbertian-linear-model}.
Note that in the case where $K=0$ and so $\hat{\gamma}$ does not exist, we simply take $\hat{\mathscr{S}}^X$ and $\hat{\mathscr{S}}^Y$ to be $0$ operators, i.e., no regression is performed.

The data-driven choice of $\hat{\gamma}$ above is motivated by an upper bound on the in-sample MSPE
of the estimators $\hat{\mathscr{S}}^X$ and $\hat{\mathscr{S}}^Y$ (see Lemma \ref{lem:deterministic-regression-mse-bound} in the supplementary material) where we have omitted some distribution-dependent factors of $\lVert \mathscr{S}_\dist^X \rVert_{\HS}^2$ or $\lVert \mathscr{S}_\dist^Y \rVert^2_{\HS}$ and a variance factor; a similar strategy was used in an analysis of kernel ridge regression \citep{GCM} which closely parallels ours here.
This choice allows us to conduct a theoretical analysis that we present below. In practice, other choices of regularisation parameter such as cross validation-based approaches may perform even better and so could alternative methods that are not based on Tikhonov regularisation.

In the following result, we take $\psi_n$ to be the $\alpha$-level GHCM test \eqref{eq:GHCM_test} with estimated regression functions $\hat{f}$ and $\hat{g}$ yielding fitted values given by
\begin{equation} \label{eq:hatf_hatg}
	\hat{f}(z_i) = \hat{\mathscr{S}}^X z_i \qquad \text{and} \qquad \hat{g}(z_i) = \hat{\mathscr{S}}^Y z_i, \qquad \text{for all } i=1,\ldots,n.
\end{equation}
Note that in the finite dimensional setting where $X^{(n)} \in \R^{n \times d_X}$ (which is also covered by the result below), we have that the matrix of fitted values $(\hat{f}(z_i))_{i=1}^n \in \R^{n \times d_X}$ is given by
\[
K(K+\gamma I)^{-1} X^{(n)},
\]
and similarly for the $Y^{(n)}$ regression.

\begin{theorem}
\label{thm:kernel-regression-ghcm}
Let $\tilde{\mathcal{P}}_0 \subset \mathcal{P}_0$ be such that \eqref{eq:x_mod} and \eqref{eq:y_mod}  are satisfied, and moreover (iii) and (iv) of Theorem~\ref{thm:asymptotic normality and covariance estimation} and \eqref{eq:degen} hold when $\hat{f}$ and $\hat{g}$ are as in \eqref{eq:hatf_hatg}. Suppose further that
%Recall the definitions of $\mathscr{C}_\dist$, $u_\dist$ and $v_\dist$ from Section \ref{sec:size} and for $\dist \in \mathcal{P}$ let $(\mu_{i, \dist})_{i \in \mathbb{N}}$ denote the ordered eigenvalues of the covariance operator of $Z$ under $\dist$. Let $\tilde{\mathcal{P}}_0 \subset \mathcal{P}_0$ satisfy all of the following:
\begin{enumerate}[(i)]
  \item $\sup_{\dist \in \tilde{\mathcal{P}}_0} \max(\lVert \mathscr{S}^X_\dist \rVert_{\HS}, \lVert \mathscr{S}^Y_\dist \rVert_{\HS}) < \infty$,
  \item $\sup_{\dist \in \tilde{\mathcal{P}}_0} \max(u_\dist(Z), v_\dist(Z)) < \infty$ almost surely,
  \item $\sup_{\dist \in \tilde{\mathcal{P}}_0} \mathbb{E}\lVert Z \rVert^2 < \infty$ and $\lim_{\gamma \downarrow 0} \sup_{\dist \in \tilde{\mathcal{P}}_0} \sum_{k=1}^\infty \min(\mu_{k, \dist}, \gamma)  = 0$ where $(\mu_{k, \dist})_{k \in \mathbb{N}}$ denote the ordered eigenvalues of the covariance operator of $Z$ under $\dist$.
%  \item $\inf_{\dist \in \tilde{\mathcal{P}}_0} \lVert \mathscr{C}_\dist \rVert_{\op} > 0$,
%  \item there exists $\eta >0$ such that $\sup_{\dist \in \tilde{\mathcal{P}}_0} \mathbb{E}_\dist \left(\lVert \varepsilon_\dist \rVert^{2+\eta} \lVert \xi_\dist \rVert^{2+\eta} \right) < \infty$ and
%  \item for some orthonormal bases $(e_{X, k})_{k \in \mathbb{N}}$ and $(e_{Y, k})_{k \in \mathbb{N}}$ of $\mathcal{H}_X$ and $\mathcal{H}_Y$, respectively, we have
%  \[
%  \lim_{K \to \infty} \sup_{\dist \in \tilde{\mathcal{P}}_0} \sum_{k=K}^\infty \langle \mathscr{C}_\dist (e_{X, k} \otimes e_{Y, k}) , e_{X, k} \otimes e_{Y, k} \rangle = 0.
%  \]
\end{enumerate}
Then the $\alpha$-level GHCM test $\psi_n$ satisfies
\[
	\lim_{n \to \infty} \sup_{\dist \in \tilde{\mathcal{P}}_0} | \mathbb{P}_\dist(\psi_n = 1) - \alpha | = 0.
\]
%using $\hat{\mathscr{S}}^X$ and $\hat{\mathscr{S}}^Y$ satisfies the uniform conditions of Theorem~\ref{thm:level of test} over $\tilde{\mathcal{P}}_0$.
\end{theorem}
Condition (iii) is generally satisfied, by the dominated convergence theorem, for any family $\tilde{\mathcal{P}}_0$ for which %the family of covariance operators of $Z$ is finite or
the sequence of eigenvalues of the covariance operators are uniformly bounded above by a summable sequence.
As a very simple example where all the remaining conditions of Theorem~\ref{thm:kernel-regression-ghcm} are satisfied, we may consider the family of distribution $\tilde{\mathcal{P}}_0$ where $Z$, $\varepsilon_P$ in \eqref{eq:X_mod} and $\xi_P$ in \eqref{eq:Y_mod} are independent, and the latter two are Brownian motions with variances $\sigma_{\varepsilon,P}^2$ and $\sigma_{\xi,P}^2$ respectively. If the coefficient functions $\beta_P^X$ corresponding to $X$ in \eqref{eq:process-linear-operator} are in $L_2([0,1]^2, \mathbb{R})$ with norms bounded above for all $P \in \mathcal{P}_0$, and an equivalent assumption for the coefficient functions relating to $Y$ holds, and $\sigma_{\varepsilon,P}^2$ and $\sigma_{\xi,P}^2$ are bounded from above and below uniformly, we have that $\mathcal{P}_0$ satisfies all the requirements of Theorem~\ref{thm:kernel-regression-ghcm}.
%
%A very simple example where the above conditions are satisfied is where $Z$, $\varepsilon$ and $\xi$ are independent Brownian motions with variances $\sigma_{Z,P}^2$, $(\sigma^\varepsilon_P)^2$ and $(\sigma^\xi)^2$, respectively, and
%$X$ and $Y$ are functions of $Z$ as in \eqref{eq:process-linear-operator} with coefficient functions $\beta_P^X$ and $\beta_P^Y$, respectively. The conditions of Theorem~\ref{thm:kernel-regression-ghcm} are then satisfied if the $L_2([0,1]^2, \mathbb{R})$ norms of $\beta_P^X$ and $\beta_P^Y$ have a uniform upper bound and all of $(\sigma^Z_P)^2$, $(\sigma^\varepsilon_P)^2$ and $(\sigma^\xi)^2$ are uniformly upper and lower bounded.

The proof of Theorem~\ref{thm:kernel-regression-ghcm} relies on Lemma~\ref{lem:deterministic-regression-mse-bound} in Section~\ref{sec:4.3-proofs} of the supplementary material, which gives a bound on the in-sample MSPE of ridge regression in terms of the decay of the eigenvalues $\mu_{k,P}$, which may be of independent interest. For example, we have that if these are dominated by an exponentially decaying sequence, the in-sample MSPE is $o(\log n / n)$ as $n \to \infty$ (see Corollary~\ref{cor:exponential-decay-mse}).
 This matches the out-of-sample MSPE bound obtained in \citet[Corollary 5]{Crambes2013} in the same setting as that described, but the out-of-sample result additionally requires convexity and lower bounds on the decay of the sequence of eigenvalues of the covariance operator, and stronger moment assumptions on the norm of the predictor. Similarly, other related results \citep[e.g.,][]{cai2006,hall2007methodology} require additional eigen-spacing conditions in place of convexity, and upper and lower bounds on the decay of the eigenvalues.  Furthermore, while some of these bounds are uniform over values of the linear coefficient operator for fixed distributions of the predictors, our in-sample MSPE bound is uniform over both the coefficients and distributions of the predictor. This illustrates how in-sample and out-of-sample prediction are very different in the functional data setting, and reliance on the former being small, as we have with the GHCM, is desirable due to the weaker conditions needed to guarantee this.

\section{Experiments} \label{sec:experiments}
In this section we present the results of numerical experiments that investigate the performance of our proposed GHCM methodology. We implement the GHCM as described in Algorithm~\ref{alg:ghcm} with scalar-on-function and function-on-function regressions performed using the \texttt{pfr} and \texttt{pffr} functions respectively from the \texttt{refund} package \citep{refund}. These are functional linear regression methods which rely on fitting smoothers implemented in the \texttt{mgcv} package \citep{Wood2017}; we choose the tuning parameters for these smoothers (dimension of the basis expansions of the smooth terms) as per the standard guidance such that a further increase does not decrease the deviance.
In Section~\ref{sec:real-data} in the supplement, we study high-dimensional EEG data using the GHCM with regressions performed using \texttt{FDboost}.

We note that, to the best of our knowledge, neither \texttt{FDboost} nor the
regression methods in \texttt{refund} come with prediction error
bounds (such as the ones derived in Section~\ref{sec:ex}) that are required
for obtaining formal guarantees for the GHCM; nevertheless they are well-developed and well-used functional regression methods and our aim here is to demonstrate empirically that they perform suitably well in terms of prediction such that when used with the GHCM, type I error is maintained across a variety of settings.
In Section~\ref{app: additional-sim-results} of the supplementary material, we include additional simulations that consider among others, settings with heavy tailed errors, test the GHCM with \texttt{FDboost} in further settings and examine the local power of the GHCM.
 %that use the \texttt{FDboost} \citep{FDboost} package to do the required regressions.
%\Anton{Should we describe tuning? (We haven't really done any)}
%\Rajen{And add other references below with a brief description of the conclusions}
%error of the model fit

%All of the upcoming simulation studies are performed using densely \Rajen{should we say how dense?} observed functional random variables on an equidistant grid on $[0,1]$. In each simulation we follow the procedure described in Section \ref{sec:practice} using the \texttt{fpca.sc} method of the \texttt{refund} package to perform FPCA. Scalar-on-function linear regressions are performed using the \texttt{pfr} function and function-on-function linear regressions are performed using \texttt{pffr} function from the \texttt{refund} package. These regression methods are wrappers for the \texttt{mgcv} package and for all regression we ensured that we chose tuning parameters (dimension of the basis expansions of the smooth terms) large enough that a further increase did not decrease the error of the model fit\Anton{Cite mgcv?}.

\subsection{Size and power simulation} \label{sec:level-power-sim}
In this section we examine the size and power properties of the GHCM when testing the conditional independence $X \independent Y \cond Z$.  We take $X, Z \in L^2([0, 1], \mathbb{R})$, and first consider the setting where $Y$ is scalar. In Section~\ref{sec:functional-level-power-sim} we present experiments for the case where $Y \in L^2([0, 1], \mathbb{R})$, so all variables are functional. All simulated functional random variables
are sampled
%we consider are in $L^2([0, 1], \mathbb{R})$ and we sample these
on an equidistant grid of $[0, 1]$ with $100$ grid points. 
\subsubsection{\texorpdfstring{Scalar $Y$, functional $X$ and $Z$}{Scalar Y, functional X and Z}}
\label{sec:scalar-level-power-sim}
%We take $X, Z \in L^2[0,1]$ and $Y \in \R$ where
Here we consider the setup where $Z$ is standard Brownian motion and $X$ and $Y$ are related to $Z$ through the functional linear models
\begin{align}
X(t) &= \int_0^1 \beta_a(s, t) Z(s) \, \mathrm{d}s + N_X(t), \label{eq:X_mod}\\
Y &= \int_0^1 \alpha_a(t) Z(t) \, \mathrm{d}t + N_Y. \label{eq:Y_mod}
\end{align}
The variables $N_X, N_Y$ and $Z$ are independent with $N_X$ a Brownian motion with variance $\sigma_X^2$, $N_Y \sim \mathcal{N}(0,1)$, so $X \independent Y \cond Z$. Nonlinear coefficient functions $\beta_a$ and $\alpha_a$ are given by
\begin{align} \label{eq:beta_alpha}
	\beta_a(s, t) = a \exp(-(st)^2/2) \sin(ast), \qquad \alpha_a(t) = \int_0^1 \beta_a(s, t) \, \mathrm{d}s.
\end{align}
We vary the parameters $\sigma_X \in \{0.1, 0.25, 0.5, 1\}$ and $a \in \{2, 6, 12\}$. We generate $n$ i.i.d.\ observations from each of the $4 \times 3=12$ models given by \eqref{eq:X_mod}, \eqref{eq:Y_mod}, for sample sizes $n \in \{100, 250, 500, 1000\}$.
Increasing $a$ or decreasing $\sigma_X$ increase the difficulty of the testing problem: for large $a$, $\beta_a$ oscillates more, making it harder to remove the dependence of $X$ on $Z$. A smaller $\sigma_X$ makes $Y$ closer to the integral of $X$, and so increases the marginal dependence of $X$ and $Y$. 

We apply the GHCM and compare the resulting tests to those corresponding to the significance test for $X$ in a regression of $Y$ on $(X, Z)$ implemented in \texttt{pfr}. The rejection rates of the two tests at the $5\%$ level, averaged over $100$ simulation runs, can be seen in Figure \ref{fig:level-rejection-rates}. We see that the \texttt{pfr} test has size greatly exceeding its level in the more challenging large $a$, small $\sigma_X$ settings, with large values of $n$ exposing most clearly the miscalibration of the test statistic. In these settings, $Y$ may be approximated simply by the integral of $X$ reasonably well, and is also well-approximated by the true regression function that features only $Z$. Regularisation encourages \texttt{pfr} to fit a model where $X$ determines the response, rather than $X$, and the $p$-values reflect this. 
%As discussed earlier, in these settings, the $p$-values from the \texttt{pfr} function do not come with formal guarantees. 
On the other hand, the GHCM tests maintain reasonable type I error control across the settings considered here.
%, though it does exceed the expected maximum rejection rate across all settings of $0.104$ when the underlying $p$-values are uniform in two settings. \Anton{Should we even mention this expected maximum rejection rate here? I'm not sure it really adds much and one of the reviewers seemed confused about it.}

%is a standard Gaussian random variable. The functions are observed on an equidistant grid with $100$ grid points and $n$ observations are produced. For $a \in \{2, 6, 12\}$, we let
%\[
%\beta(s, t) = a \exp(-(st)^2/2) \sin(ast)
%\]
%and $\alpha(t) = \int_0^1 \beta(s, t) \, \mathrm{d}t$. For $\sigma_X \in \{0.1, 0.25, 0.5, 1\}$ and $n \in \{100, 250, 500, 1000\}$ we repeat the experiment $100$ times. We compare the GHCM to performing a single scalar-on-function regression where we regress $Y^{(n)}$ on $(X^{(n)}, Z^{(n)})$ and use the $p$-value when testing whether the regression function for $X$ is zero as a test of conditional independence. 
\begin{figure}
\centering
\includegraphics[scale=0.8]{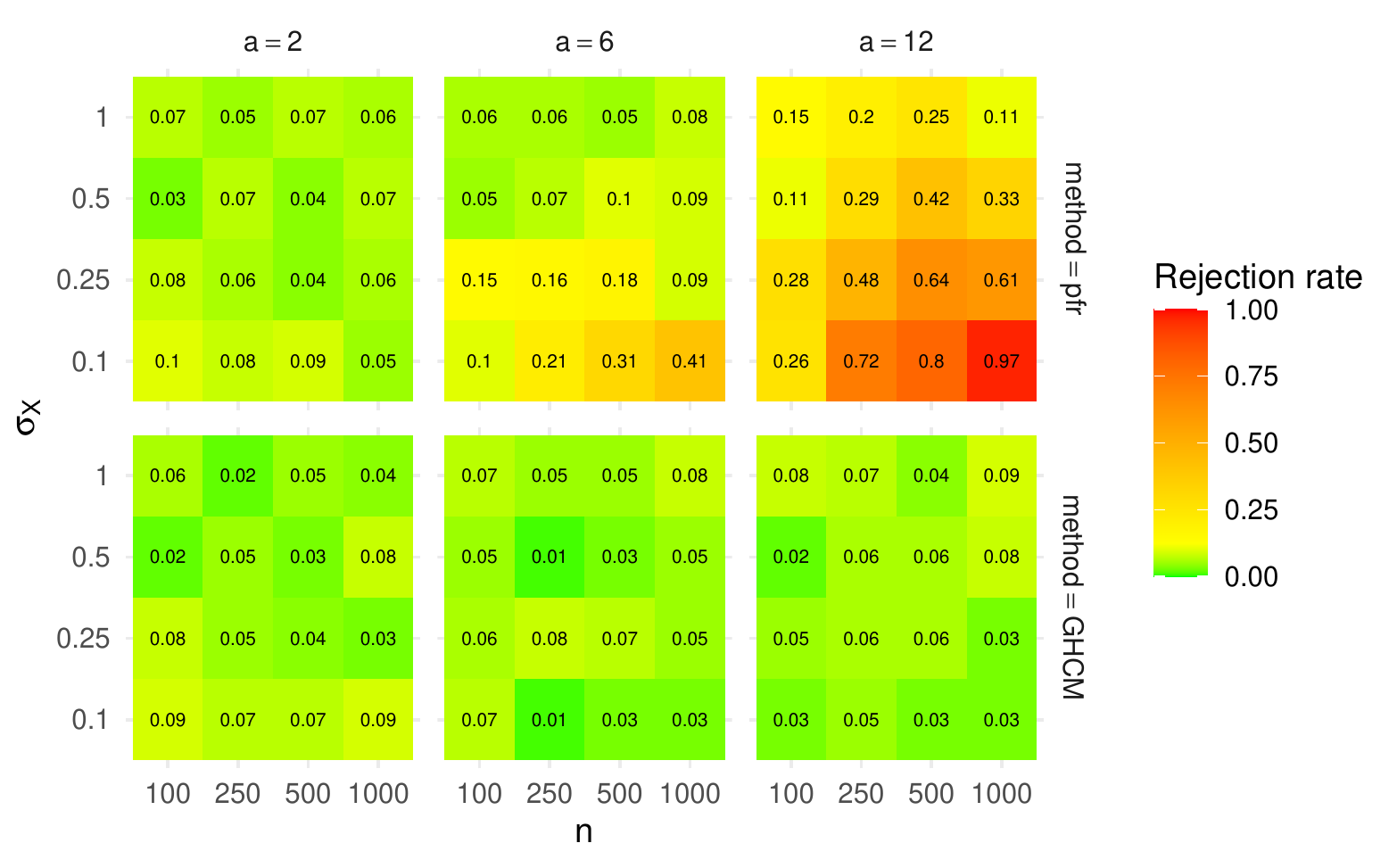}
\caption{Rejection rates in the various null settings considered in Section~\ref{sec:scalar-level-power-sim} for the nominal 5\%-level \texttt{pfr} test (top) and GHCM test (bottom).}
\label{fig:level-rejection-rates}
\end{figure}

%It is clear from the figure that with increasing complexity of the testing problem (increasing $a$ or decreasing $\sigma_X$ or $n$) that the \texttt{pfr} $p$-value begins to fail to holding level and in the most extreme case will incorrectly reject the null hypothesis almost all the time. In contrast, the GHCM is closer to the nominal level in all situations. 

To investigate the power properties of the test, we simulate $Z$ as before with $X$ also generated according to \eqref{eq:X_mod}. We replace the regression model \eqref{eq:Y_mod} for $Y$ with
\begin{equation}
\label{eq:sim_alt}  
Y = \int_0^1 \alpha_a(t) Z(t) \, \mathrm{d}t + \int_0^1 \frac{\alpha_a(t)}{a} X(t) \, \mathrm{d}t + N_Y,
\end{equation}
where $N_Y \sim \mathcal{N}_Y(0,1)$ as before.
Note that the coefficient function for $X$ oscillates more as $a$ increases.
 The rejection rates at the $5\%$ level can be seen in Figure \ref{fig:power-rejection-rates}. 
\begin{figure}
\centering
\includegraphics[scale=0.8]{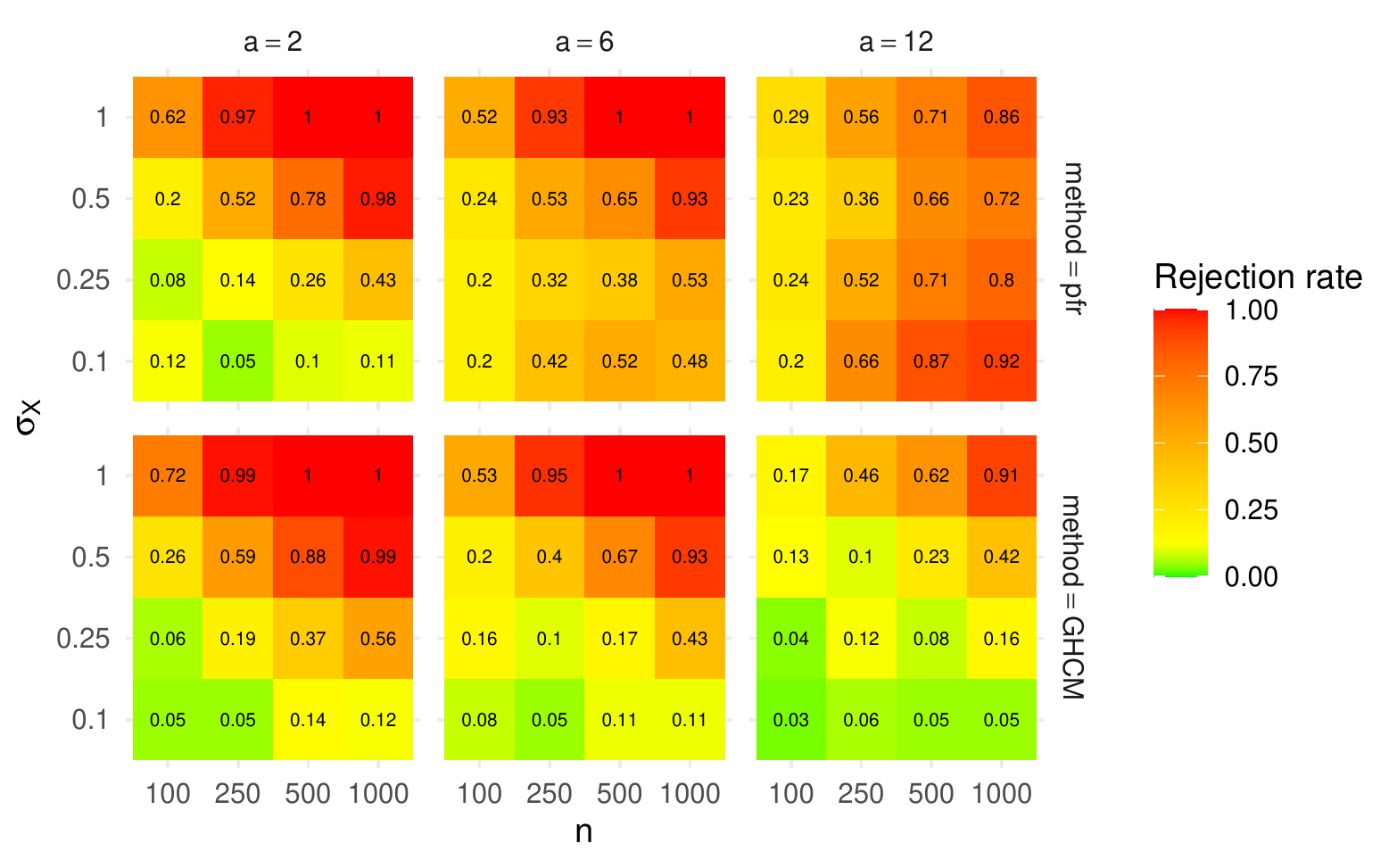}
\caption{
	Rejection rates in the various alternative settings considered in Section~\ref{sec:scalar-level-power-sim} (see \eqref{eq:sim_alt}) for the nominal 5\%-level \texttt{pfr} test (top) and GHCM test (bottom).}
\label{fig:power-rejection-rates}
\end{figure}
While the two approaches perform similarly when $a=2$, the \texttt{pfr} test has higher power in the more complex cases. However, as the results from the size analysis in Figure \ref{fig:level-rejection-rates} show, null cases are also rejected in the analogous settings.

To illustrate the full distribution of $p$-values from the two methods under the null and the alternative, we plot false positive rates and true positive rates in each setting as a function of the chosen significance level of the test $\alpha$. The full set of results can be seen in Section~\ref{app: additional-sim-results} of the supplementary material and a plot for a subset of the simulations settings where $n=500$ and $\sigma_X \in \{0.1, 0.25, 0.5\}$ is presented in Figure \ref{fig:FPR-TPR}. 
\begin{figure}
\centering
\includegraphics[scale=0.8]{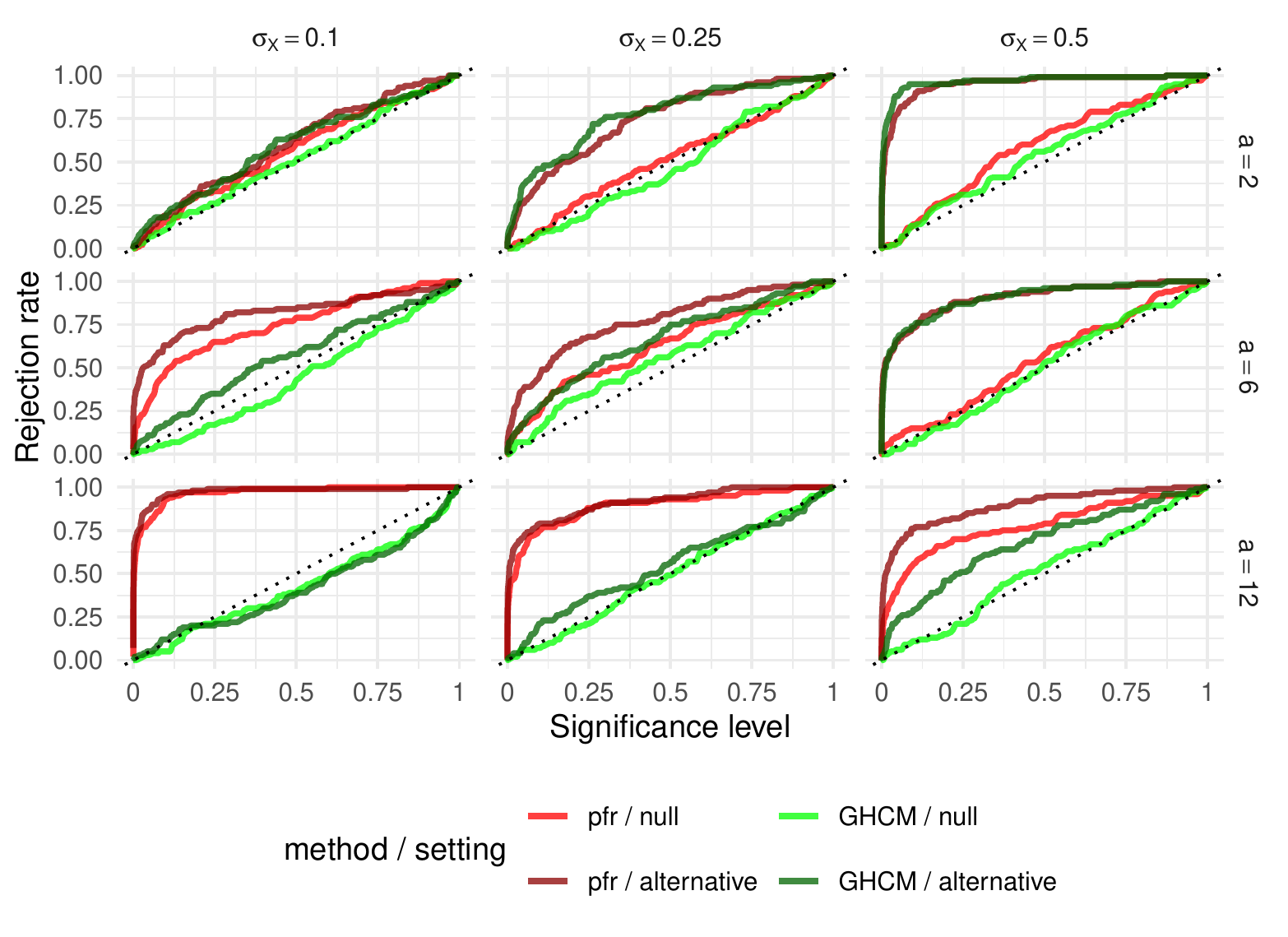}
\caption{Rejection rates against significance level for the $\texttt{pfr}$ (red) and GHCM (green) tests under null (light) and alternative (dark) settings when $n=500$.}
\label{fig:FPR-TPR}
\end{figure}
We see that both tests distinguish null from alternative well in the cases with $a$ small and $\sigma_X$ large. The $p$-values of the GHCM are close to uniform in the settings considered, whereas the distribution of the \texttt{pfr} $p$-values is heavily dependent on the particular null setting, illustrating the difficulty with calibrating this test.

	In Section~\ref{app: additional-sim-results} of the supplementary material we also present the results of two additional sets of experiments. We repeat the experiments above using the \texttt{FDboost} package for regressions in place of the \texttt{refund} package. We see that the performance of the GHCM with \texttt{FDboost} is broadly similar to that displayed in Figures~\ref{fig:level-rejection-rates} and \ref{fig:power-rejection-rates}, supporting our theoretical results which indicate that provided the prediction errors of the regression methods used are sufficiently small, the test will perform similarly.
	
	We also consider the case where the noise is heavy-tailed. Specifically, we present analogous plots for setting where $N_Y$ is $t$-distributed with different degrees of freedom, $n=500$ and $\sigma_X = 0.25$; the results are similar to Figure~\ref{fig:FPR-TPR}, with the GHCM maintaining type I error control, and \texttt{pfr} tending to be anti-conservative in the more challenging settings.

%\Anton{
%	In Section~\ref{app: additional-sim-results} of the supplementary material we repeat the experiments above using the \texttt{FDboost} package for regressions instead of the \texttt{refund} package. We see that the performance of the GHCM with \texttt{FDboost} is broadly similar to that displayed in Figures~\ref{fig:level-rejection-rates} and \ref{fig:power-rejection-rates} and could perhaps be further enhanced by more delicate tuning of the regressions.
%}
%
%In Section~\ref{app: additional-sim-results} of the supplementary material, we also consider the case where the noise is heavy-tailed. Specifically, we present analogous plots for setting where $N_Y$ is $t$-distributed with different degrees of freedom, $n=500$ and $\sigma_X = 0.25$; the results are similar to Figure~\ref{fig:FPR-TPR}, with the GHCM maintaining type I error control, and \texttt{pfr} tending to be anti-conservative in the more challenging settings.
%The results in Figure \ref{fig:FPR-TPR} reiterate the points given above -- in the hard cases it appears that the regression $p$-value based approach fails to hold level while in the simpler cases the methods perform similarly. A naive power analysis of the \texttt{pfr} $p$-value approach would make it seem like the method was efficient at distinguishing nulls from alternatives in the complex settings, however when looking at the $a=12$, $\sigma_X =0.1$ graph, we see that both methods fail to distinguish nulls from alternatives but the GHCM does so while maintaining level.

\subsubsection{\texorpdfstring{Functional $X$, $Y$ and $Z$}{Functional X, Y and Z}}
\label{sec:functional-level-power-sim}
In this section we modify the setup and consider functional $Y \in L^2([0, 1], \mathbb{R})$. We take $X$ and $Z$ as in Section \ref{sec:scalar-level-power-sim} but in the null settings we let
\[
Y(t) = \int_0^1 \beta_a(s, t) Z(s) \, \mathrm{d}s + N_Y(t) ,
\]
where $N_Y$ is a standard Brownian motion. Note that this is a particularly challenging setting to maintain type I error control as $X$ and $Y$ are then highly correlated, and moreover the biases from regressing each of $X$ and $Y$ on $Z$ will tend to be in similar directions making the equivalent of the term $a_n$ in \eqref{eq:expansion} potentially large.

In the alternative settings, we take
\[
  Y(t) = \int_0^1 \beta_a(s, t) Z(s) \, \mathrm{d}s + \int_0^1 \frac{\beta_a(s, t)}{a} X(s) \, \mathrm{d}s  + N_Y(t) 
\]
with $N_Y$ again being a standard Brownian motion. 

%\Rajen{We removed this, should we add?? Unlike the setting with scalar $Y$, there are no established competing methods and the $p$-values produced by a function-on-function regression using the \texttt{refund} package exhibit erratic behaviour even for smooth functions, tending to reject the null consistently even in simple models.}
The rejection rates at the $5\%$ level, averaged over $100$ simulation runs, can be seen in Figure~\ref{fig:functional-rejection-rates}. We see that, as in the case where $Y \in \mathbb{R}$, the GHCM maintains good type I error control in the settings considered, and has power increasing with $n$ and $\sigma_X$ as expected. We note that a comparison with the $p$-values from $\texttt{ff}$-terms in the $\texttt{pffr}$-function of the \texttt{refund} package here does not seem helpful. In our experiments the corresponding tests consistently reject in true null settings even for simple models.

In Section~\ref{app: additional-sim-results} of the supplementary material we look at the subset of the settings considered above with $n=500$ and $\sigma_X=0.25$ but where $X$ and $Y$ are observed on irregular grids of varying length grids. We first preprocess the residuals output by the regression method as described in Section~\ref{sec:unequal_grids} and then apply the GHCM. We observe that the performance is similar to that in the fixed grid setting, though the power is lower when the average grid length is smaller, and type I error increases slightly above nominal levels in the most challenging $a=12$ setting.

\begin{figure}
\centering
\includegraphics[scale=0.8]{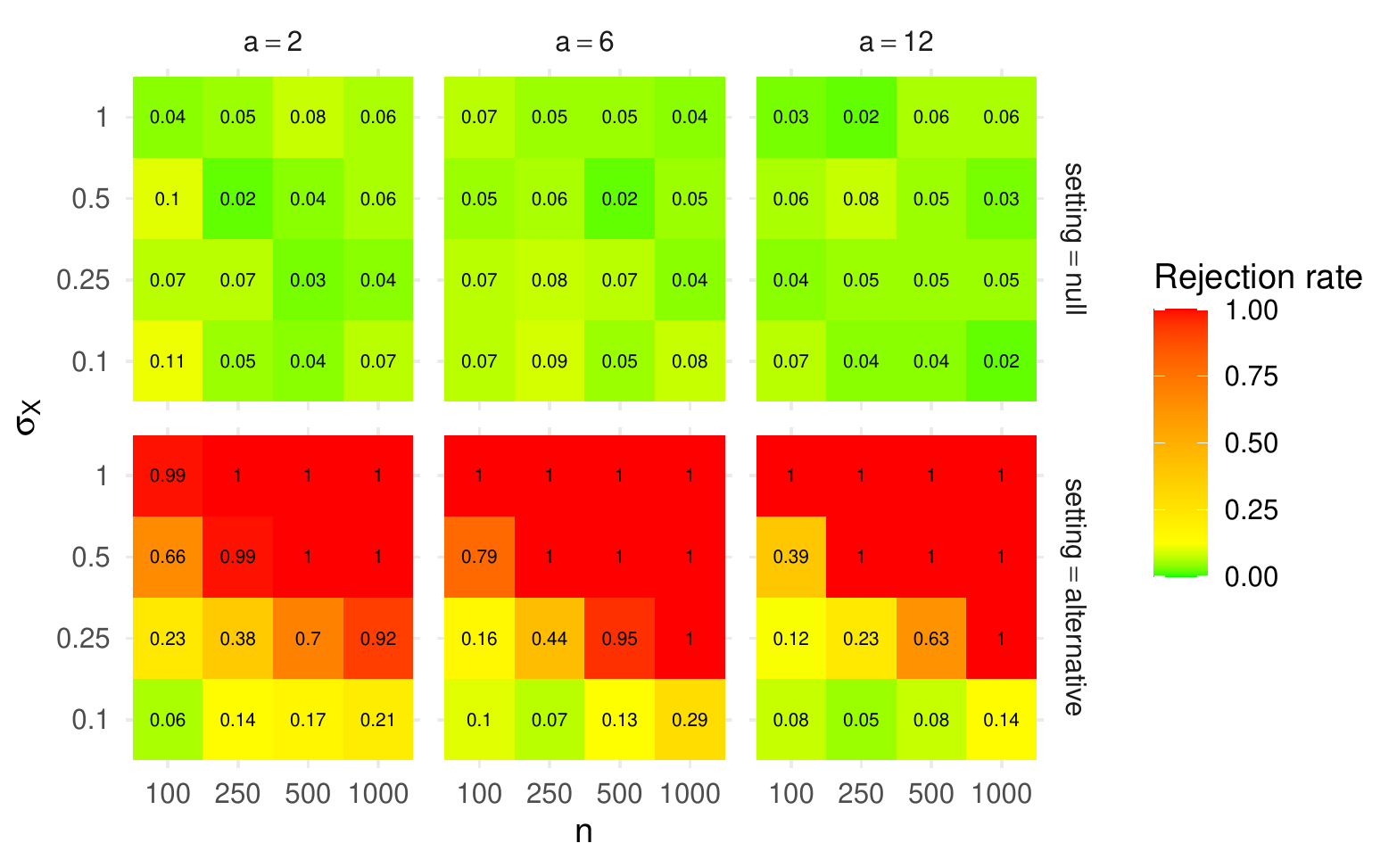}
\caption{Rejection rates in the various null (top) and alternative (bottom) settings considered in Section~\ref{sec:functional-level-power-sim} for the nominal 5\%-level GHCM test.}
\label{fig:functional-rejection-rates}
\end{figure}

\subsection{Confidence intervals for truncated linear models} \label{sec:ci-truncated}
In this section we consider an application of the GHCM in constructing a confidence interval for the truncation point $\theta \in [0, 1]$ in a truncated functional linear model \citep{Hall2016}
%Consider the standard functional linear model
\begin{equation} \label{eq:trunc}
Y = \int_0^\theta \alpha(t) X(t) \, \mathrm{d}t + \varepsilon,
\end{equation}
where the predictor $X \in L^2([0, 1], \mathbb{R})$, $Y \in \R$ is a response and $\varepsilon \independent X$ is stochastic noise. To frame this as a conditional independence testing problem, observe that \eqref{eq:trunc} implies that defining the null hypotheses
\begin{equation} \label{eq:trunc_null}
H_{\tilde{\theta}} : \;\;\; Y \independent \{X(t)\}_{t > \tilde{\theta}} \cond \{X(t)\}_{t \leq \tilde{\theta}}
\end{equation}
for $\tilde{\theta} \in (0, 1)$, we have that $H_{\tilde{\theta}}$ is true for all $\theta \leq \tilde{\theta} \leq 1$.

Given an $\alpha$-level conditional independence test $\psi$, we may thus form a one-sided confidence interval for $\theta$ using
\begin{equation} \label{eq:conf_int}
\left[ \inf \left\{\tilde{\theta} \in (0, 1)  \,: \, \psi \text{ accepts null } H_{\tilde{\theta}} \right\}, \, 1\right].
\end{equation}
Indeed, with probability $1-\alpha$, $\psi$ will not reject the true null $H_\theta$, and so with probability $1-\alpha$ the infimum above will be at most $\theta$.

To approximate \eqref{eq:conf_int} we initially consider the null hypothesis $H_{\tilde{\theta}}$ at $5$ equidistant values of $\tilde{\theta}$ and then employ a bisection search between the smallest of these points $\tilde{\theta}$ at which $H_{\tilde{\theta}}$ is accepted by a 5\% level GHCM, and the point immediately before it or $0$.
We consider two instances of the model \eqref{eq:trunc} with $\theta = 0.275, 0.675$ and with $\alpha(t) := 10(t+1)^{-1/3}$, $X$ a standard Brownian motion and $\varepsilon \sim \mathcal{N}(0,1)$. The simulated functional variables are observed on an equidistant grid of $[0,1]$ with $121$ grid points. The results across $500$ simulations are given in Figure~\ref{fig:confidencce-interval-histograms}. We see that the empirical coverage probabilities are close to the nominal coverage of 95\%.
% 
%It might be that $\alpha(t) = 0$ for all $t > b$ for some $b \in (0, 1)$. This is the truncated functional linear model studied in \citet{Hall2016}. This would imply that $Y \independent X(t)_{t > b} \cond X(t)_{t \leq b}$ which we can test using the GHCM. We simulate from the model where $\varepsilon$ is standard Gaussian, $X$ is a standard Brownian motion observed on an equidistant grid of length $121$ on $[0,1]$ and $\alpha(t) = 10(t+1)^{-1/3}$ when $t \leq b$ and $0$ otherwise. Initially we test the null hypothesis for $5$ equidistant values of $b$ and then perform bisection on the grid from $0$ to the right-most value of $b$ that was not rejected at the $5\%$ level. We repeat the experiment $500$ times for true values $b=0.275$ and $b=0.675$ to produce the results seen in Figure \ref{fig:confidencce-interval-histograms}.

\begin{figure}
\centering
\includegraphics[scale=0.8]{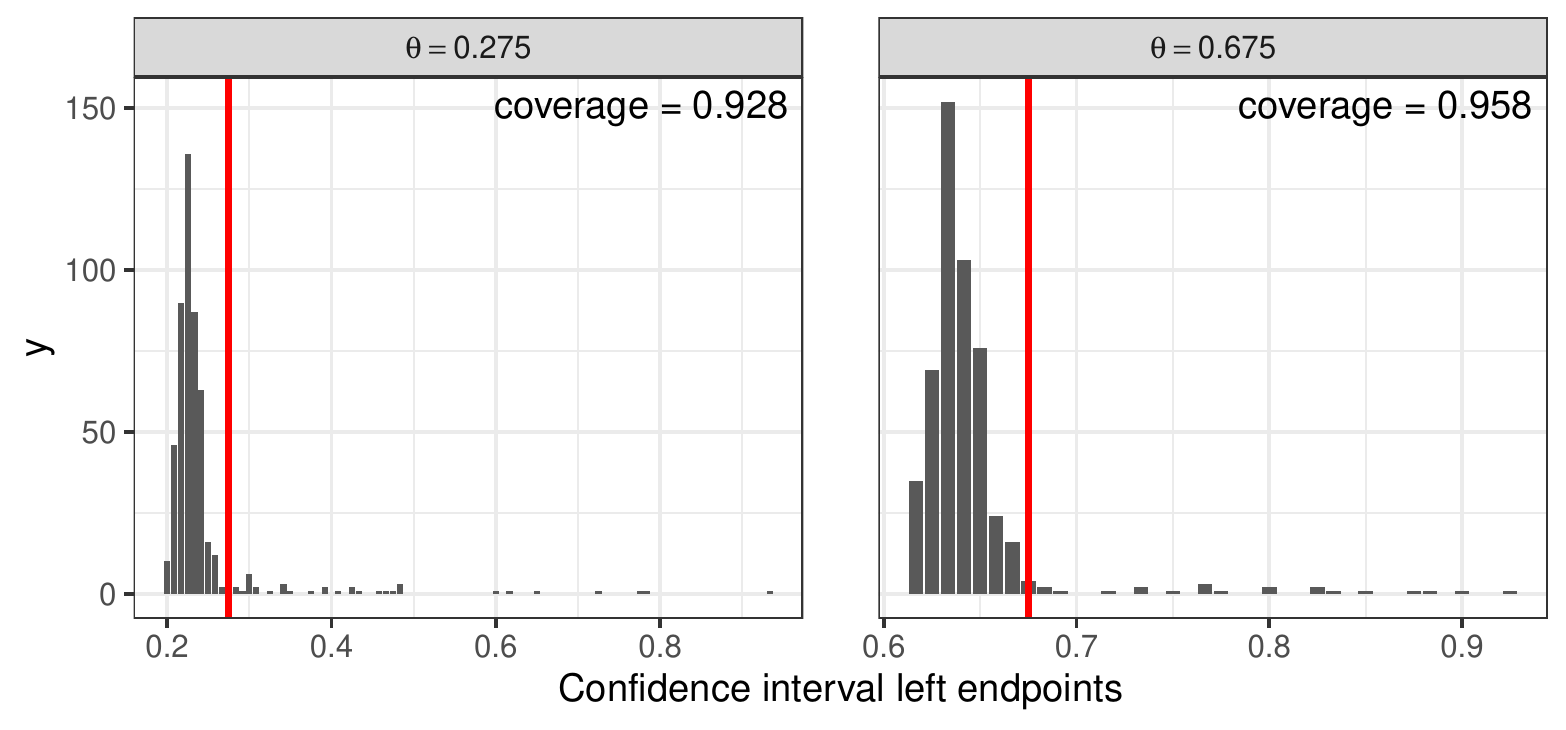}
\caption{Histograms of the left endpoints of 95\% confidence intervals for truncation points $\theta = 0.275$ (left) and $\theta=0.675$ (right), given by red vertical lines, in model \eqref{eq:trunc} across $500$ simulations. } 
\label{fig:confidencce-interval-histograms}
\end{figure}

\subsection{EEG data analysis} \label{sec:real-data}
In this section we demonstrate the application of our GHCM methodology to the problem of learning functional graphical models. In contrast to existing work \citep{Qiao2019, Qiao2020} which typically assumes a Gaussian functional graphical model and outputs a point estimate of the conditional independence graph, here we are able to test for the presence of each edge, with type I error control guaranteed for data generating processes where our regression methods perform suitably well as indicated by Theorem~\ref{thm:level of test}.

We illustrate this on an EEG dataset from a study on alcoholism \citep{Zhang1995,Ingber1997,Ingber1998}. The study participants were shown one of three visual stimuli repeatedly and simultaneous EEG activity was measured across $64$ channels over the course of $1$ second at $256$ measurements per second.  While the study included both a control group and an alcoholic group we will restrict our analysis to the alcoholic group consisting of $77$ subjects and further restrict ourselves to a single type of visual stimulus.
We preprocess the data as in \citet{Qiao2019}, averaging across the repetitions of the experiment for each subject and using an order $96$ FIR filter implemented in the $\texttt{eegkit}$ \texttt{R}-package \citep{Helwig2018} to filter the averaged curves at the $\alpha$ frequency bands (between $8$ and $12.5$ Hz). We thus obtain $64$ $\alpha$-filtered frequency curves for each of the $77$ subjects.

Given the low number of observations compared to the $64$ functional variables, there is not enough data to reject the null of edge absence even if a true edge were to be present. We therefore aim for a coarser analysis by grouping the variables by brain region and then further according to whether the variable corresponded to the right or left hemispheres of the brain. This yields disjoint groups $G_1,\ldots,G_{24}$ comprising $52$ variables in total after omitting reference channels and midline channels that could not easily be classified as being in either hemisphere, that is, $G_1 \cup \ldots \cup G_{24} = \{1, \ldots, 52\}$. We suppose the observed data are i.i.d.\ copies functional variables $(X_1,\ldots,X_{52})$, and then test the null hypothesis
\begin{equation} \label{eq:group}
X_{G_j} \independent X_{G_k} \cond \{ X_{G_m} : m \in \{1,\ldots,24\} \setminus \{j,k\}\},
\end{equation}
for each $j,k \in \{1,\ldots,24\}$ with $j \neq k$; that is, we test for edge presence in the conditional independence graph of the grouped variables.
Here, the conditional independence graph over the grouped variables is defined as an undirected graph over $G_{1}, \ldots, G_{24}$, in which the edge between $G_j$ and $G_k$, $j \neq k$ is missing if and only if~\eqref{eq:group} holds; that is, rejection of the null in \eqref{eq:group} for $k$ and $j$ indicates that the conditional independence graph has an edge between $G_k$ and $G_j$.
%\rmj{Note that \Rajen{by the global Markov property,} rejection of the null in \eqref{eq:group} indicates that the conditional independence graph on the set of $52$ variables has at least one edge between a variable in $G_k$ and one in $G_j$.}

To construct $p$-values for the null in \eqref{eq:group} using the GHCM, we must regress for each $l \in G_j$ and $r \in G_k$, each of the functional variables $X_l$ and $X_r$ on to the set of variables in the conditioning set.
%\Anton{I modified the next two sentences to address the refund comments}
Since the regressions will involve large numbers of functional predictors, the \texttt{refund} package is not suitable to perform the regressions. Instead, we use the \texttt{FDboost} package in \texttt{R}, which is well-suited to high-dimensional functional regressions \citep{FDboost}. We fit a concurrent functional model \citep[][Section 16]{Ramsay2005} of the form
\[
X_l(t) = \sum_m \beta_m(t) X_m(t) ;
\]
the inclusion of additional functional linear terms did not improve the fit. We assessed the appropriateness of this regression method to data of the sort studied here through simulations described in Section~\ref{app: additional-sim-results} of the supplement.
%\Anton{Here we should perhaps add some comments about the validity of the $p$-values / regressions and the relationship to the hardness result. We should also reference the new Figure~\ref{fig:eeg-sim} in section D of the supplement}.
%with regularisation. We also investigated adding in additional functional linear terms but these did not improve the fit of the models hence we omitted them in our final analysis. \Anton{Rajen, can you double check / edit this?}
\begin{figure}
	\centering
	\includegraphics[scale=0.8]{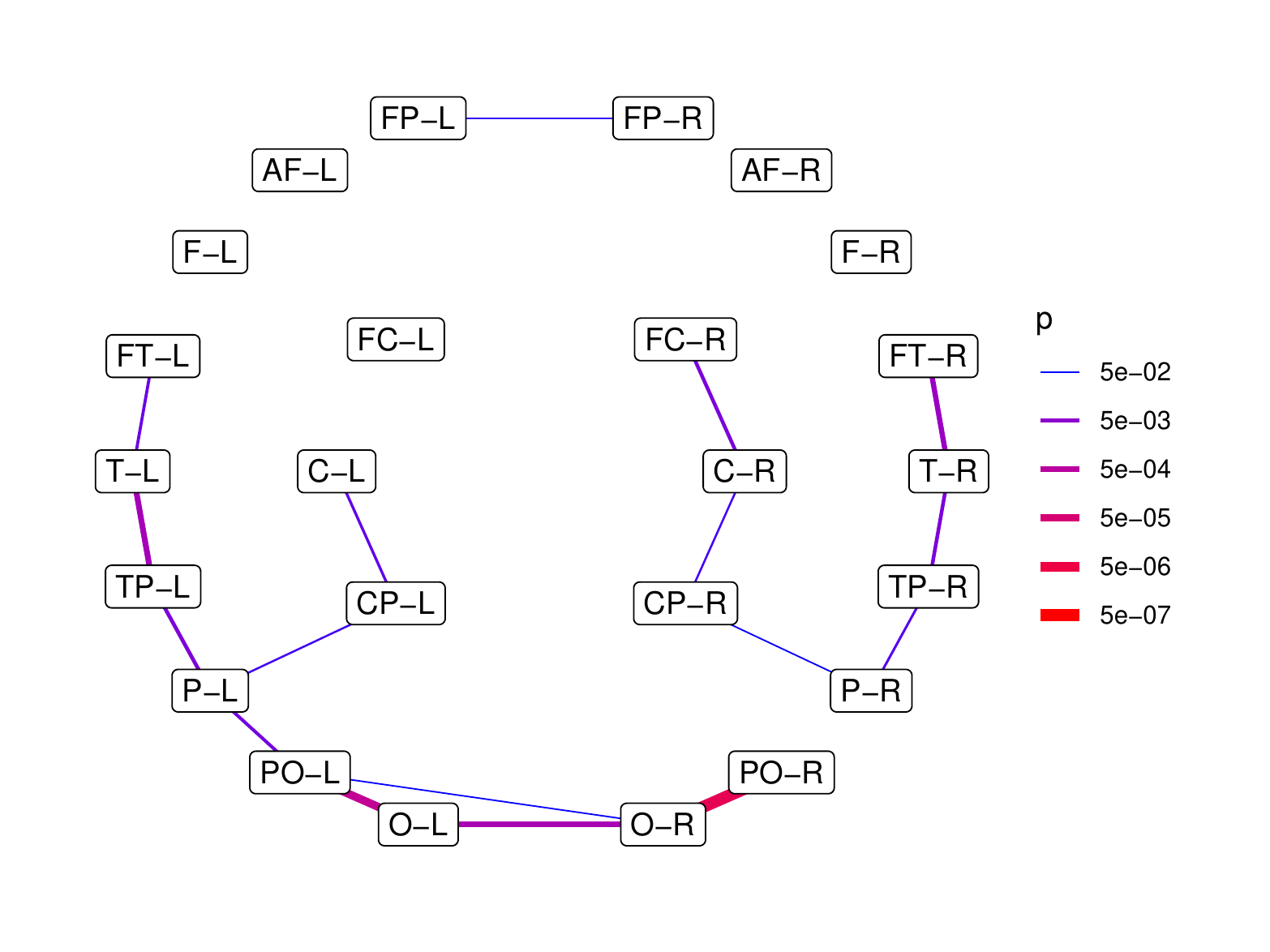}
	\caption{Network summarising the output of conditional independence tests for each pair of groups. Only edges with $p$-values of less than 5\% are shown with thicker lines indicating smaller $p$-values.} 
	\label{fig:eeg-network}
\end{figure}

Figure~\ref{fig:eeg-network} summarises the results of GHCM applied to test the presence of each edge in the conditional independence graph. We see that some of the brain regions located close to each other appear to be connected, as one might expect.
%\Anton{Added the next sentence but it is perhaps a bit too brief?} Our estimated network is significantly sparser than the ones estimated in both \citet{Qiao2019} and \citet{Qiao2020}.
Note that the network presented includes all edges that had a $p$-value less than $5\%$. The edge PO-R---O-R has a Bonferroni-corrected $p$-value of $0.0027$, and is the only edge yielding a corrected $p$-value less than $5\%$. Applying the Benjamini--Hochberg procedure \citep{FDR} to control the false discovery rate at the 5\% level selects this edge and also PO-L---O-L. We may compare these results with those of \citet{Qiao2019} and \citet{Qiao2020} who study the same dataset but consider the different problem of estimation of the conditional independence graph rather than testing of edge presence as we do here. We see that our results are broadly in line with their estimates: for example, there are edges estimated between the groups represented by PO-R and O-R (the group pair which yields the lowest $p$-value) even in some of their sparsest estimated graphs.

%Our aim is to ... \Anton{Again, I'm a bit unsure how to phrase what we're trying to achieve here, formally speaking.}. To simplify our analysis we group the EEG channels by which region of the brain was measured and whether the channel was placed on the right or left half of the brain. To further simplify we also omit the reference EEG channels and the EEG channels along the midline of the brain from the analysis since these cannot easily be categorised into left or right side of the brain.
%\Anton{I'm not really satisfied with the argument for omission in the previous sentences. If you have a better argument, please do tweak it.} This leaves us with $24$ groups of between $1$ and $4$ channels and we now seek to test conditional independence of each pair of groups of channels given the remaining groups. \Anton{We just need to describe the procedure now but I wanted to reference the practicals section that we have yet to write.}

%The resulting network can be seen in Figure \ref{fig:eeg-network} where we've included edges with a $p$-value of less than $0.05$. Only the two edges between PO-L and O-L and PO-R and O-R retain a $p$-value of less than $0.05$ after applying the Benjamini-Hochberg procedure to the full set of $p$-values. \Anton{Add more comments.}

\section{Conclusion} \label{sec:conclusion}
Testing the conditional independence $X \independent Y \cond Z$ has been shown to be a hard problem in the setting where $X, Y, Z$ are all real-valued and $Z$ is absolutely continuous with respect to Lebesgue measure \citep{GCM}. This hardness takes a more extreme form in the functional setting: even when $(X, Y, Z)$ are jointly Gaussian with non-degenerate covariance and $Z$ and at most one of $X$ and $Y$ are infinite-dimensional, there is no non-trivial test of conditional independence. This requires us to (i) understand the form of an `effective null hypothesis' for a given hypothesis test, and (ii) develop tests where these effective nulls are somewhat interpretable so that domain knowledge can more easily inform the choice of a conditional independence test to use on any given dataset.

In order to address these two needs, we introduce here a new family of tests for functional data and develop the necessary uniform convergence results to understand the forms of null hypotheses that we can have type I error control over. We see that for our proposed GHCM tests, error control is guaranteed under conditions largely determined by the in-sample prediction error rate of regressions upon which the test is based. Whilst in-sample and more common out-of-sample results share similarities in some settings, the lack of a need to extrapolate beyond the data in the former lead to important differences when regressing on functional data. In particular, no eigen-spacing conditions or lower bounds on the eigenvalues of the covariance of the regressor are required for the in-sample error to be controlled when ridge regression is used. It would be interesting to investigate the in-sample MSPE properties of other regression methods and understand whether such conditions can be avoided more generally.

One attractive feature of the GHCM is that it only depends on inner products between the residuals produced by the regression methods.
%Thus, in the case where these correspond to function evaluations on unequally spaced grids, potentially with the grid lengths varying with the observation index, it should be possible to design inner products that can approximate the inner product between the underlying unobserved functions.
An interesting question is whether different inner products can be constructed to have power against different sets of alternatives, by emphasising certain regions of the function domains, for example.

%In the case where the grids are not only unequally spaced, but also vary with the observation index, one would need to decide upon an appropriate inner product to use to construct the test statistic; we do not cover this here and leave it for further work. %\Jonas{[maybe move this to future work section?]}

Another direction which may be fruitful to pursue is to adapt the GHCM so that it has power against alternatives where $\E \Cov (X, Y \cond Z) = 0$. It is likely that further conditions will be required of the regression methods than simply that their in-sample prediction errors are small, and so some interpretability of the effective null hypotheses, and indeed its size compared to the full null of conditional independence, will need to be sacrificed. There are however settings where the severity of type I versus type II errors may be balanced such that this is an attractive option.

It would also be interesting to investigate the hardness of conditional independence in the setting where all of $X$, $Y$ and $Z$ are infinite-dimensional. For our hardness result here, at least one of $X$ and $Y$ must be finite-dimensional. It may be the case that requiring two infinite-dimensional variables to be conditionally independent is such a strong condition that the null is not prohibitively large compared to the entire space of Gaussian measures, and so genuine control of the type I error while maintaining power is in fact possible. Such a result, or indeed a proof that hardness persists, would certainly be of interest.

\section*{Acknowledgements}
We thank Yoav Zemel, Alexander Aue, Sonja Greven and Fabian Scheipl for helpful discussions.

%\putbib[bibliography.bib]
%\end{bibunit}

\end{cbunit}

\newpage
\setcounter{page}{1}

%\begin{bibunit}[abbrvnat]
\begin{cbunit}
\appendix

\section*{Supplementary material for `Conditional Independence Testing in Hilbert Spaces with Applications to Functional Data Analysis'}

Section~\ref{app:hardness} is a self-contained presentation of the theory and proofs of Section~\ref{sec:hardness} in the paper. Section~\ref{app:uniform-convergence} contains much of the background on uniform stochastic convergence that is used for the technical results of the paper. This includes an account of previously established results for real-valued random variables and new results for Hilbertian and Banachian random variables. Section~\ref{app:proofs} contains the proofs of the results in Sections~\ref{sec:GHCM}~and~\ref{sec:theory} in the paper. Section~\ref{app: additional-sim-results} contains some additional simulation results.

\section{Hardness of functional Gaussian independence testing}
\label{app:hardness}
In this section we provide the necessary background and prove the hardness result in Section~\ref{sec:hardness}. We use the notation and terminology described in the setup of Section~\ref{sec:hardness} with the exception that $\mathcal{P}$, $\mathcal{P}_0$ and $\mathcal{Q}$ will consist of $n$ i.i.d.\ copies of jointly Gaussian $(X, Y, Z)$ rather than a single copy. For a bounded linear operator $\mathscr{A}$ on a Hilbert space $\mathcal{H}$, we let $\mathscr{A}^*$ denote the adjoint of $\mathscr{A}$. For two orthogonal subspaces $\mathcal{A}$ and $\mathcal{B}$ of a Hilbert space $\mathcal{H}$, we write $\mathcal{A} \oplus \mathcal{B}$ for the orthogonal direct sum of $\mathcal{A}$ and $\mathcal{B}$.

In Section~\ref{sec:power-finite-dimensional-gaussian-ci-testing} we consider the  setup of Section~\ref{sec:hardness} in the specific case where all the Hilbert spaces are finite-dimensional. We show that for any $Q \in \mathcal{Q}$, sample size $n$ and $\varepsilon > 0$, we can find a sufficiently large dimension of $\mathcal{H}_Z$ such that any test of size $\alpha$ over $\mathcal{P}_0^Q$ has power at most $\alpha+\varepsilon$ against any alternative. In Section~\ref{sec:hardness-infinite-dimensional-gaussian-testing} we use this to prove Theorem~\ref{thm:main-hardness-theorem}. 
% In Section~\ref{sec:hardness-infinite-dimensional-gaussian-testing} we use \rmj{the previous results} \Rajen{this} to prove \rmj{the main hardness result} \ref{thm:main-hardness-theorem}\Rajen{REF result}.
In Section~\ref{sec:conditional-distributions-on-hilbert-spaces} we review the theory of regular conditional probabilities and conditional distributions of Hilbertian random variables and prove several Hilbertian analogues of well-known multivariate Gaussian results. Sections~\ref{sec:power-finite-dimensional-gaussian-ci-testing}~and~\ref{sec:hardness-infinite-dimensional-gaussian-testing} with the exception of Lemma~\ref{lemma:size-power-TV-distance} contain new material while Section~\ref{sec:conditional-distributions-on-hilbert-spaces} is primarily a review of relatively well-known results.

\subsection{Power of finite-dimensional Gaussian conditional independence testing} \label{sec:power-finite-dimensional-gaussian-ci-testing}

Before we consider Gaussian conditional independence testing, we present the following general result from \citet{Kraft1955}. A summary is given in \citet{LeCam1973}. 
\begin{lemma}
\label{lemma:size-power-TV-distance}
Let $\mathcal{P}$ and $\mathcal{Q}$ denote two families of probability measures on some measurable space $(\mathcal{X}, \mathcal{A})$ and assume that both families are dominated by a $\sigma$-finite measure. Consider the problem of testing the null hypothesis that the given data is from a distribution in $\mathcal{P}$ against the alternative that the distribution is in $\mathcal{Q}$. Let $d_{\textrm{TV}}$ denote the total variation distance and $\widetilde{\mathcal{P}}$ and $\widetilde{\mathcal{Q}}$ the closed convex hulls of $\mathcal{P}$ and $\mathcal{Q}$. Then
\[ 
\inf_{\psi: \mathcal{X} \to [0,1]} \sup_{P \in \mathcal{P}, Q \in \mathcal{Q}} \left[ \int \psi \, \mathrm{d}P + \int (1-\psi) \, \mathrm{d}Q \right] = 1 - \inf_{P \in \widetilde{\mathcal{P}}, Q \in \widetilde{\mathcal{Q}}} d_{\textrm{TV}}(P, Q) .
\]
An immediate consequence of this is that for any test $\psi$ that has size $\alpha$ and power function $\beta: \mathcal{Q} \to [0,1]$, $\beta(Q) = \int \psi \, \mathrm{d}Q$, we have
\[
\inf_{Q \in \mathcal{Q}} \beta(Q) \leq \alpha + \inf_{P \in \widetilde{\mathcal{P}}, Q \in \widetilde{\mathcal{Q}}} d_{\textrm{TV}}(P, Q) \leq \alpha + \inf_{P \in \widetilde{\mathcal{P}}, Q \in \mathcal{Q}} d_{\textrm{TV}}(P, Q).
\]
\end{lemma}
In most practical situations both $\mathcal{P}$ and $\mathcal{Q}$ will consist of product measures on a product space corresponding to a situation where we observe a sample of $n$ i.i.d.\ observations of some random variable. The theorem states that a lower bound on the sum of the type I and type II error probabilities of testing the null that data is from a distribution in $\mathcal{P}$ against the alternative that the distribution is in $\mathcal{Q}$ is given by $1$ minus the total variation distance between the closed convex hulls of $\mathcal{P}$ and $\mathcal{Q}$. As a consequence we see that the power of a test is upper bounded by the size plus the total variation distance between the closed convex hull of $\mathcal{P}$ and $\mathcal{Q}$.

In the remainder of this section we will consider the testing problem described in Section~\ref{sec:hardness} with $\mathcal{H}_X = \mathbb{R}^{d_X}$ and $\mathcal{H}_Z=\mathbb{R}^{d_Z}$ for $d_X, d_Z \in \mathbb{N}$.
%\rmj{various choices of} $d_X, d_Z \in \mathbb{N}$.
To produce bounds on the power of a test in this setting, we will construct an explicit TV-approximation to a family of particularly simple distributions in $\mathcal{Q}$ using a distribution in the convex hull of the null distributions. We will need the following upper bound on the total variation distance between measures.
\begin{lemma}
\label{lem:TV-bound}
Let $P$ and $Q$ be probability measures where $P$ has density $f$ with respect to $Q$. Then 
\[
4d_{\mathrm{TV}}(P, Q)^2 \leq \int f^2 \, \mathrm{d}Q - 1 .
\]
\end{lemma}
\begin{proof}
We may assume that the integral of $f^2$ with respect to $Q$ is finite, otherwise the inequality is trivially valid. Then by Jensen's inequality, we get
\begin{equation*}
d_{\mathrm{TV}}(P, Q)^2 = \frac{1}{4} \left( \int |f-1| \, \mathrm{d}Q \right)^2 \leq \frac{1}{4} \int (f-1)^2 \, \mathrm{d}Q = \frac{1}{4}  \int f^2 \, \mathrm{d}Q - \frac{1}{4}. \qedhere
\end{equation*}
\end{proof}

Using this bound and Lemma~\ref{lemma:size-power-TV-distance}, we can show the following result. 
\begin{theorem}
\label{thm:gaussian-conditional-independence-simple-power-bound}
Let $Q$ be a distribution consisting of $n$ i.i.d.\ copies of jointly Gaussian $(X, Y, Z)$ on $(\mathbb{R}, \mathbb{R}, \mathbb{R}^d)$ for some $d \in \mathbb{N}$, where $X$ and $Y$ are standard Gaussian, $Z$ is mean zero with identity covariance matrix, $\Cov(X, Z) = \Cov(Y, Z) = 0$ and $\Cov(X, Y) = \rho \in (0, 1)$. Consider the testing problem described in Section~\ref{sec:hardness} with $\mathcal{H}_X = \mathbb{R}$ and $\mathcal{H}_Z=\mathbb{R}^d$ and let $\psi$ be the test function of a size $\alpha$ test over $\mathcal{P}_0^Q$. Writing $\beta$ for the power of $\psi$ against $Q$, we have 
\[
\beta \leq \alpha + \frac{1}{2} \sqrt{-1 + (1+\rho)^n \sum_{k=0}^d \frac{\binom{d}{k}}{2^d (1+(3-4k/d)\rho)^n} }.
\]
In particular, for fixed $n$ the upper bound converges to $\alpha$ as $d$ increases.
\end{theorem}
\begin{proof}
Let $\tau \in \{-1, 1 \}^d$ and let $P_\tau$ denote the Gaussian distribution consisting of $n$ i.i.d.\ copies of jointly Gaussian $(X, Y, Z)$ where $X$ and $Y$ are standard Gaussian, $Z$ is mean zero with identity covariance matrix, $\Cov(X, Y) = \rho$ and $\Cov(X,Z) = \Cov(Y,Z) = \sqrt{\frac{\rho}{d}} \tau^\top $. For every $\tau \in \{-1, 1 \}^d$, it is clear that $X \independent Y \cond Z$ under $P_\tau$ and thus forming
\[
P := \frac{1}{2^d} \sum_{\tau \in \{-1, 1\}^d} P_\tau
\]
we note that $P$ is in the closed convex hull of the set of null distributions. Let $\Gamma_\tau$ and $\Gamma_Q$ denote the $n(d+2)$-dimensional covariance matrices of the $n$ i.i.d.\ copies of $(X, Y, Z)$ under $P_\tau$ and $Q$ respectively. These are block-diagonal, and we let $\Sigma_\tau$ and $\Sigma_Q$ respectively denote the matrices in the diagonal, corresponding to the covariance of a single observation of $(X, Y, Z)$ under $P_\tau$ and $Q$. 
By standard manipulations of densities, the density of $P$ with respect to $Q$ is simply the ratio of their respective densities with respect to the Lebesgue measure. We have
\[
\Sigma_\tau = \begin{pmatrix}
\begin{pmatrix}
1 & \rho \\
\rho & 1
\end{pmatrix} & \sqrt{\frac{\rho}{d}}  \begin{pmatrix}
\tau^\top \\ \tau^\top 
\end{pmatrix} \\
\sqrt{\frac{\rho}{d}}  \begin{pmatrix}
\tau &  \tau
\end{pmatrix} & I_d
\end{pmatrix}
\]
and, letting $I_d$ denote the $d$-dimensional identity matrix,
\[
\Sigma_Q = \begin{pmatrix}
\begin{pmatrix}
1 & \rho \\
\rho & 1
\end{pmatrix} & 0 \\
 0 & I_d
\end{pmatrix}.
\]
The determinant of $\Sigma_Q$ is $1-\rho^2$ by Laplace-expanding the first row. Letting $J_2$ denote the $2$-dimensional matrix of ones, we have
\[
\det(\Sigma_\tau) = \det(I_d) \det\left( \begin{pmatrix}
1 & \rho\\
\rho & 1
\end{pmatrix}
-
\rho J_2
\right) = (1-\rho)^2
\]
by Schur's formula. Defining $f$ to be the density of $P$ with respect to $Q$, we see that
\[
f(v) = \frac{1}{2^d} \frac{(1+\rho)^{n/2}}{(1-\rho)^{n/2}} \sum_{\tau \in \{-1, 1 \}^d} \exp\left( -\frac{1}{2}v^\top (\Gamma_\tau^{-1} - \Gamma_Q^{-1})v \right)
\]
since the determinants of $\Gamma_\tau$ and $\Gamma_Q$ are the determinants of $\Sigma_\tau$ and $\Sigma_Q$ to the $n$th power. From this we get that
\begin{align*}
&\int f^2 \, \mathrm{d}Q = \frac{1}{2^{2d}} \frac{(1+\rho)^n}{(1-\rho)^n} \sum_{\tau, \tau' \in \{ -1, 1\}^d}  \int \exp\left(-\frac{1}{2}v^\top (\Gamma_\tau^{-1} + \Gamma_{\tau'}^{-1} - 2 \Gamma_Q^{-1})v \right) \, \mathrm{d}Q(v) =\\
 & \frac{1}{2^{2d}} \frac{(1+\rho)^n}{(1-\rho)^n} \frac{1}{\sqrt{(2\pi)^{n(d+2)}(1-\rho^2)^n}}\sum_{\tau, \tau' \in \{ -1, 1\}^d}\int \exp\left(-\frac{1}{2}v^\top (\Gamma_\tau^{-1} + \Gamma_{\tau'}^{-1} -  \Gamma_Q^{-1})v \right) \, \mathrm{d}\lambda_{n(d+2)}(v),
\end{align*}
where $\lambda_{n(d+2)}$ denotes the $n(d+2)$-dimensional Lebesgue measure. Each integral is the integral of an unnormalised Gaussian density in $\mathbb{R}^{n(d+2)}$, and thus we can simplify further to get 
\begin{align*}
\int f^2 \, \mathrm{d}Q &=  \frac{1}{2^{2d}} \frac{(1+\rho)^n}{(1-\rho)^n} \frac{1}{(1-\rho^2)^{n/2}}\sum_{\tau, \tau' \in \{ -1, 1\}^d} \sqrt{ \det \left[ (\Gamma_\tau^{-1} + \Gamma_{\tau'}^{-1} -  \Gamma_Q^{-1})^{-1} \right] } \\
&= \frac{1}{2^{2d}} \frac{(1+\rho)^n}{(1-\rho)^n} \frac{1}{(1-\rho^2)^{n/2}}\sum_{\tau, \tau' \in \{ -1, 1\}^d}  \det(\Gamma_\tau^{-1} + \Gamma_{\tau'}^{-1} -  \Gamma_Q^{-1})^{-1/2} \\
&= \frac{1}{2^{2d}} \frac{(1+\rho)^n}{(1-\rho)^n} \frac{1}{(1-\rho^2)^{n/2}}\sum_{\tau, \tau' \in \{ -1, 1\}^d} \det(\Sigma_\tau^{-1} + \Sigma_{\tau'}^{-1} -  \Sigma_Q^{-1})^{-n/2} ,
\end{align*}
by again using the block diagonal structure of $\Gamma_Q$ and the $\Gamma_\tau$'s. Recall that for a symmetric block matrix 
\[
\begin{pmatrix}
A & B^\top \\
B & C
\end{pmatrix}^{-1} = \begin{pmatrix}
(A-B^\top C^{-1}B)^{-1} & -(A-B^\top C^{-1}B)^{-1} B^\top C^{-1}\\
-C^{-1}B (A-B^\top C^{-1}B^\top )^{-1} & C^{-1} + C^{-1}B(A-B^\top C^{-1}B)^{-1} B^\top  C^{-1}
\end{pmatrix}.
\]
Using this, we see that
\[
\Sigma_Q^{-1} = \begin{pmatrix}\frac{1}{1-\rho^2}\begin{pmatrix}
1 & -\rho \\
-\rho & 1
\end{pmatrix} & 0\\
0 & I_d
\end{pmatrix}
\]
and
\[
\Sigma_\tau^{-1} =\begin{pmatrix}\frac{1}{1-\rho} I_2 & -\frac{1}{1-\rho} \sqrt{\frac{\rho}{d}} \begin{pmatrix}
\tau^\top  \\ \tau^\top 
\end{pmatrix}\\
-\frac{1}{1-\rho} \sqrt{\frac{\rho}{d}} \begin{pmatrix}
\tau & \tau
\end{pmatrix} & I_d + \frac{2\rho}{(1-\rho)d} \tau \tau^\top 
\end{pmatrix} .
\]
Further, 
\[
\Sigma_\tau^{-1} + \Sigma_{\tau'}^{-1} - \Sigma_Q^{-1} =\begin{pmatrix}
A & B^\top \\
B & C
\end{pmatrix},
\]
where
\begin{align*}
A &:= \frac{1}{1-\rho^2}\begin{pmatrix}
2\rho +1 & \rho\\
 \rho & 2\rho +1 
\end{pmatrix}\\
B &:= -\frac{1}{1-\rho}\sqrt{\frac{\rho}{d}} \begin{pmatrix}
\tau + \tau' & \tau + \tau'
\end{pmatrix} \\
C &:= I_d + \frac{2\rho}{(1-\rho)d}(\tau\tau^\top  + \tau'\tau'^\top ) .
\end{align*}
We may once more use Schur's formula for the determinant of a block matrix to find that
\[
\det ( \Sigma_\tau^{-1} + \Sigma_{\tau'}^{-1} - \Sigma_Q^{-1} ) = \det(C ) \det (A - B^\top  C^{-1} B ) .
\]
Defining $V=\begin{pmatrix}
\tau & \tau'
\end{pmatrix}$, we note that $C= I_d + \frac{2\rho}{(1-\rho)d}VV^\top  $ and defining further
\[
M:= I_2 + \frac{2\rho}{(1-\rho)d} V^\top  V = \frac{1}{d(1-\rho)} \begin{pmatrix}
d(1+\rho) & 2\rho \langle \tau , \tau' \rangle\\
2\rho \langle \tau , \tau' \rangle & d(1+\rho),
\end{pmatrix}
\]
the Weinstein--Aronszajn identity yields that
\[
\det(C) = \det(M) = \frac{(d(1+\rho)+2\rho \langle \tau, \tau' \rangle)(d(1+\rho)-2\rho \langle \tau, \tau' \rangle)}{d^2(1-\rho)^2}.
\]
The Woodbury matrix identity yields that
\[
C^{-1} = I_d - \frac{2\rho}{(1-\rho)d} V M^{-1} V^\top  .
\]
Hence,
\[
\det (A - B^\top  C^{-1} B ) = \det \left(A - B^\top  B+ \frac{2\rho}{(1-\rho)d} B^\top  V M^{-1} V^\top  B \right) .
\]
Now
\[
M^{-1} = \frac{(1-\rho)d}{(d(1+\rho)+2\rho \langle \tau, \tau' \rangle)(d(1+\rho)-2\rho \langle \tau, \tau' \rangle)} \begin{pmatrix}
d(1+\rho) & -2\rho \langle \tau , \tau' \rangle \\
 -2\rho \langle \tau , \tau' \rangle & d(1+\rho)
\end{pmatrix}
\]
and
\[
B^\top  V = -\frac{1}{1-\rho} \sqrt{\frac{\rho}{d}} (d+\langle \tau , \tau' \rangle) J_2,
\]
where $J_2$ is the $2$-dimensional matrix of ones. Thus,
\[
 \frac{2\rho}{(1-\rho)d} B^\top  V M^{-1} V^\top  B  = \frac{2 \rho^2(d+\langle \tau , \tau' \rangle)^2}{(1-\rho)^3 d^2}  J_2 M^{-1} J_2 =  \frac{4 \rho^2(d+\langle \tau , \tau' \rangle)^2}{(1-\rho)^2 d(d(1+\rho) + 2\rho \langle \tau, \tau' \rangle)}  J_2 .
\]
Since
\[
B^\top  B = \frac{2 \rho}{(1-\rho)^2d} (d+\langle \tau , \tau' \rangle) J_2
\]
we get that
\begin{align*}
&\det (A - B^\top  C^{-1} B ) = \det \left(A + \left( \frac{4 \rho^2(d+\langle \tau , \tau' \rangle)^2}{(1-\rho)^2 d(d(1+\rho) + 2\rho \langle \tau, \tau' \rangle)}  - \frac{2 \rho}{(1-\rho)^2d} (d+\langle \tau , \tau' \rangle) \right) J_2  \right) \\
 &=  \det \left(A -  \frac{2\rho(d+\langle \tau , \tau' \rangle)}{(1-\rho)(d(1+\rho)+2\rho\langle \tau , \tau' \rangle)}J_2  \right)  \\
 &= \frac{\det \left( (d(1+\rho)+2\rho \langle \tau , \tau' \rangle ) \begin{pmatrix}
 2\rho +1 & \rho \\
 \rho & 2\rho +1
\end{pmatrix} -  2\rho(d+\langle \tau , \tau' \rangle ) (1+\rho) J_2 \right) }{(1-\rho)^2(1+\rho)^2(d(1+\rho)+2\rho \langle \tau , \tau' \rangle )^2} \\
&= \frac{\det   \begin{pmatrix}
(d(1+\rho)+2\rho\langle \tau , \tau' \rangle)\rho + (1+\rho)(1-\rho)d &  (1+\rho)(1-\rho)d -(d(1+\rho)+2\rho\langle \tau , \tau' \rangle)  \\
(1+\rho)(1-\rho)d -(d(1+\rho)+2\rho\langle \tau , \tau' \rangle)  & (d(1+\rho)+2\rho\langle \tau , \tau' \rangle)\rho + (1+\rho)(1-\rho)d 
\end{pmatrix}  }{(1-\rho)^2(1+\rho)^2(d(1+\rho)+2\rho \langle \tau , \tau' \rangle )^2} \\
&= \frac{(d(1+\rho) + 2\rho \langle \tau , \tau' \rangle)(1+\rho)(\rho-1) + 2(1+\rho)^2(1-\rho)d  }{(1-\rho)^2(1+\rho)^2(d(1+\rho)+2\rho \langle \tau , \tau' \rangle )} \\
&= \frac{d(1+\rho) - 2\rho \langle \tau , \tau' \rangle}{(1-\rho)(1+\rho)(d(1+\rho)+2\rho \langle \tau , \tau' \rangle )} 
\end{align*}
and thus
\[
\det ( \Sigma_\tau^{-1} + \Sigma_{\tau'}^{-1} - \Sigma_Q^{-1} ) = \frac{(d(1+\rho) - 2\rho \langle \tau , \tau' \rangle)^2}{d^2(1-\rho)^3(1+\rho)} .
\]
Returning to the squared integral of $f^2$ with respect to $Q$, we get that
\begin{align*}
\int f^2 \, \mathrm{d}Q 
&= \frac{1}{2^{2d}} \frac{(1+\rho)^n}{(1-\rho)^n} \frac{1}{(1-\rho^2)^{n/2}}\sum_{\tau, \tau' \in \{ -1, 1\}^d}  \frac{d^n \sqrt{
(1-\rho)^{3n}
(1+\rho)^n}}{|d(1+\rho) - 2\rho \langle \tau , \tau' \rangle|^n}\\
 &=  \frac{1}{2^{2d}} (1+\rho)^n \sum_{\tau, \tau' \in \{ -1, 1\}^d}  \frac{d^n} {|d(1+\rho) - 2\rho \langle \tau , \tau' \rangle|^n} .
\end{align*}
For $\tau, \tau' \in \{-1, 1\}^d$,  $\langle \tau, \tau' \rangle = 2k-d$ where $k$ is the number of indices where $\tau_i = \tau_i'$. 
Thus instead of summing over $\tau, \tau' \in \{-1, 1\}^d$, we can count the number of $(\tau, \tau')$-pairs where $\tau$ and $\tau'$ agree in exactly $k$ positions. For each $\tau$, there are $\binom{d}{k}$ other elements in $\{-1, 1\}^d$ agreeing
in exactly $k$ positions and there are $2^d$ different $\tau$'s, hence
\begin{align*}
\int f^2 \, \mathrm{d}Q  
 &=  \frac{1}{2^{2d}} (1+\rho)^n \sum_{k=0}^d \frac{d^n \binom{d}{k} 2^d} {|d(1+\rho) - 2\rho (2k-d)|^n} \\
 &=   (1+\rho)^n \sum_{k=0}^d   \frac{d^n \binom{d}{k} } {2^d(d+\rho(3d-4k))^n} =   (1+\rho)^n \sum_{k=0}^d  \frac{ \binom{d}{k}} {2^d(1+\rho(3-4k/d))^n}  .
\end{align*}
The result now follows from Proposition~\ref{lem:TV-bound} and Lemma~\ref{lemma:size-power-TV-distance}.

To see this for each $n$ the bound converges to $\alpha$ as $d$ increases, let $W_d$ be a random variable with a binomial distribution with probability parameter $1/2$ and with $d$ trials and note that
\[
\sum_{k=0}^d  \frac{ \binom{d}{k}} {2^d(1+\rho(3-4k/d))^n}   = \mathbb{E}\left( (1+\rho(3-4W_d/d))^{-n} \right) .
\]
By the Strong Law of Large Numbers (SLLN), $W_d/d \overset{a.s.}{\to} 1/2$ and thus $ (1+\rho(3-4W_d/d))^{-n} \overset{a.s.}{\to} (1+\rho)^{-n}$. Since $(1+\rho(3-4W_d/d))^{-n} \leq (1-\rho)^{-n}$, we get by the bounded convergence theorem that
\[
\lim_{d \to \infty} \mathbb{E}\left( (1+\rho(3-4W_d/d))^{-n} \right) = \mathbb{E}\left( (1+\rho)^{-n}\right)  = (1+\rho)^{-n},
\]
and hence the upper bound on the power converges to $\alpha$.
\end{proof}

We can generalise the previous result to the situation where $X$ and $Y$ are of arbitrary finite dimension.

\begin{theorem}
\label{thm:gaussian-conditional-independence-simple-multivariate-power-bound}
Let $Q$ be a distribution consisting of $n$ i.i.d.\ copies of jointly Gaussian $(X, Y, Z)$ on $(\mathbb{R}^{d_X}, \mathbb{R}^{d_Y}, \mathbb{R}^{d_Z})$ for some $d_X, d_Y, d_Z \in \mathbb{N}$ where $X$, $Y$ and $Z$ are all mean zero with identity covariance matrix, $\Cov(X, Z) = \Cov(Y, Z) = 0$ and $\Cov(X, Y) = R$ for some rectangular diagonal matrix $R$ with diagonal entries $\rho_1, \dots, \rho_r \in (0, 1)$, where $r = \min(d_X, d_Y)$. Consider the testing problem described in Section~\ref{sec:hardness} with $\mathcal{H}_X = \mathbb{R}^{d_X}$ and $\mathcal{H}_Z=\mathbb{R}^{d_Z}$ and let $\psi$ be the test function of a size $\alpha$ test over $\mathcal{P}_0^Q$. Assume that $d_Z \geq r$ and let $d = \lfloor d_Z/r \rfloor$. Letting $\beta$ denote the power of $\psi$ against $Q$, we have
\[
 \beta \leq \alpha + \frac{1}{2}\sqrt{-1 + \prod_{i=1}^r \left( (1+\rho_i)^n \sum_{k=0}^d \frac{\binom{n}{k}}{2^d (1+(3-4k/d)\rho_i)^n} \right) }.
\]
In particular for fixed $n$ the upper bound converges to $\alpha$ as $d_Z$ increases. 
\end{theorem}
\begin{proof}
%\Jonas{to myself: check}
Assume without loss of generality that $d_X \geq d_Y$. The proof follows a similar idea to the proof of Theorem~\ref{thm:gaussian-conditional-independence-simple-power-bound}. In what follows we consider a different ordering of the variables than the natural one given by $(X, Y, Z)$. We consider $r+1$ blocks, where the first $r$ blocks are $(X_i, Y_i, Z_{(i-1)d+1}, \dots , Z_{id})$ for $i \in \{1, \dots, r \}$ and the final block consists of the remaining components of $X$ and $Z$. When we consider $n$ i.i.d.\ copies, we will again reorder the variables such that we consider each block separately. As a consequence of doing this, the covariance matrix of $n$ i.i.d.\ copies under $Q$, $\Xi_Q$, can be written as a block-diagonal matrix with $r$ $n(d+2) \times n(d+2)$ blocks $\Gamma_{Q, i}$ and a final identity matrix block. Each of the $\Gamma_{Q, i}$'s is again a block-diagonal matrix consisting of $n$ identical blocks $\Sigma_{Q, i}$ of the form
\[
\Sigma_{Q, i} = \begin{pmatrix}
\begin{pmatrix}
1 & \rho_i \\
\rho_i & 1
\end{pmatrix} & 0 \\
0 & I_d
\end{pmatrix} .
\]
Let now $\mathcal{T}= (\{-1, 1\}^d)^r$ and for each $\tau = (\tau_1, \dots, \tau_r) \in \mathcal{T}$ let $P_\tau$ denote the Gaussian distribution consisting of $n$ i.i.d.\ copies of jointly Gaussian $(X, Y, Z)$ where $X$, $Y$ and $Z$ are mean zero with identity covariance, $\Cov(X, Y) = R$ and $\Cov(X, Z) = \Cov(Y, Z)= 0$ except for 
\[
\Cov(X_i, (Z_{(i-1)d+1}, \dots , Z_{id}) ) = \Cov(Y_i, (Z_{(i-1)d+1}, \dots , Z_{id}) ) = \sqrt{\frac{\rho_i}{d}} \tau_i^\top 
\]
for $i \in \{1, \dots , r \}$. Arranging the random variables as before, the covariance matrix of $n$ i.i.d.\ copies under $P_\tau$, $\Xi_\tau$, is a block-diagonal matrix with $r$ $n(d+2) \times n(d+2)$ blocks $\Gamma_{\tau, i}$ and a final identity matrix block. Each of the $\Gamma_{\tau, i}$'s is again a block-diagonal matrix consisting of $n$ identical blocks $\Sigma_{\tau, i}$ of the form
\[
\Sigma_{\tau, i} = \begin{pmatrix}
\begin{pmatrix}
1 & \rho_i \\
\rho_i & 1
\end{pmatrix} & \sqrt{\frac{\rho_i}{d}} \tau_i^\top  \\
\sqrt{\frac{\rho_i}{d}} \tau_i & I_d
\end{pmatrix} .
\]
Clearly $X \independent Y \cond Z$ under $P_\tau$ for every $\tau \in \mathcal{T}$ and thus letting
\[
P := \frac{1}{2^{dr}} \sum_{\tau \in \mathcal{T}} P_\tau,
\]
we note that $P$ is in the closed convex hull of the null distributions. Letting $f$ be the density of $P$ with respect to $Q$, we see that 
\[
f(v) = \frac{1}{2^{dr}} \left(\prod_{i=1}^r \frac{1+\rho_i}{1-\rho_i}\right)^{n/2} \sum_{\tau \in \mathcal{T}} \exp\left(- \frac{1}{2} v^\top  (\Xi_\tau^{-1} - \Xi_Q^{-1})v \right),
\]
since this is simply the ratio of their respective densities with respect to the Lebesgue measure. We can now repeat the argument of the proof of Theorem~\ref{thm:gaussian-conditional-independence-simple-power-bound} to obtain
\[
\int f^2 \, \mathrm{d}Q = \frac{1}{2^{2dr}} \prod_{i=1}^r \left( \frac{1+\rho_i}{(1-\rho_i)\sqrt{1-\rho_i^2}} \right)^n \sum_{\tau, \tau' \in \mathcal{T}} \left( \sqrt{\det(\Xi_\tau^{-1} + \Xi_{\tau'}^{-1} - \Xi_Q^{-1}) } \right)^{-1} .
\]
The determinant can be written as 
\[
\det(\Xi_\tau^{-1} + \Xi_{\tau'}^{-1} - \Xi_Q^{-1})  = \prod_{i=1}^r\det(\Gamma_{\tau, i}^{-1} + \Gamma_{\tau', i}^{-1} - \Gamma_{Q, i}^{-1}) 
\]
by the block-diagonal structure of the $\Xi$'s. In the proof of Theorem~\ref{thm:gaussian-conditional-independence-simple-power-bound}, we derive that
\[
\det(\Gamma_{\tau, i}^{-1} + \Gamma_{\tau', i}^{-1} - \Gamma_{Q, i}^{-1}) = \left( \frac{(d(1+\rho_i) - 2\rho_i \langle \tau_i, \tau_i' \rangle)^2}{d^2(1+\rho_i)(1-\rho_i)^3} \right)^n .
\]
Therefore,
\[
\int f^2 \, \mathrm{d}Q  = \frac{1}{2^{2dr}} \left( \prod_{j=1}^r (1+\rho_j)^n  \right) \sum_{\tau, \tau' \in \mathcal{T}} \prod_{i=1}^r \frac{d^n}{|d(1+\rho_i)-2\rho_i \langle \tau_i, \tau_i' \rangle|^n} .
\]
Since each factor of the second product only depends on the $i$th component of $\tau$ and $\tau'$, we can interchange the product and sum and apply the same counting arguments as in Theorem~\ref{thm:gaussian-conditional-independence-simple-power-bound} to get that
\[
\int f^2 \, \mathrm{d}Q  = \prod_{i=1}^r \left( (1+\rho_i)^n \sum_{k=0}^d \frac{\binom{n}{k}}{2^d (1+(3-4k/d)\rho_i)^n} \right)
\]
as desired. We can repeat the same SLLN-based limiting arguments as in Theorem~\ref{thm:gaussian-conditional-independence-simple-power-bound} to show that as $d$ increases the integral will converge to $1$ and hence the power is bounded by the size in the limit. 
\end{proof}

Having shown that for each $n$ and $d$, we have an upper bound on the power of a Gaussian conditional independence test against a simple alternative, we can now show this also holds for Gaussian conditional independence testing problems against other $Q$.

\begin{lemma}
\label{lem:gaussian-conditional-independence-reduction-to-simple-distribution}
Let $Q \in \mathcal{Q}$ be a distribution consisting of $n$ i.i.d.\ copies of jointly Gaussian and injective $(X, Y, Z)$ on $(\mathbb{R}^{d_X}, \mathbb{R}^{d_Y}, \mathbb{R}^{d_Z})$ with non-singular covariance for some $d_X, d_Y, d_Z \in \mathbb{N}$. Consider the testing problem described in Section~\ref{sec:hardness} with $\mathcal{H}_X = \mathbb{R}^{d_X}$ and $\mathcal{H}_Z=\mathbb{R}^{d_Z}$ and let $\psi$ be the test function of a size $\alpha$ test over $\mathcal{P}_0^Q$ with power $\beta$ against $Q$. Then there exists a $d_X \times d_Y$-rectangular diagonal matrix $R$ with diagonal entries $\rho_1, \dots, \rho_r \in (0, 1)$, a distribution $\tilde{Q}$ consisting of $n$ i.i.d.\ copies of jointly Gaussian $(\tilde{X}, \tilde{Y}, \tilde{Z})$ where $\tilde{X}$, $\tilde{Y}$ and $\tilde{Z}$ are all mean zero with identity covariance matrix, $\Cov(\tilde{X}, \tilde{Z}) = \Cov(\tilde{Y}, \tilde{Z}) = 0$ and $\Cov(\tilde{X}, \tilde{Y}) = R$ and a test size $\alpha$ test over $\mathcal{P}_0^{\tilde{Q}}$ with power $\beta$ against $\tilde{Q}$.
\end{lemma}
\begin{proof}
Let $\psi$ denote the test function of the test with power $\beta$ against $Q$ and $\mu$ and $\Sigma$ denote the mean and covariance matrix of $(X, Y, Z)$ under $Q$. We construct a new test with test function $\tilde{\psi}$ performed by first applying a transformation $f$ to each sample of the data and then applying $\psi$. The transformation $f: \mathbb{R}^{d_X+d_Y+d_Z} \to \mathbb{R}^{d_X+d_Y+d_Z}$ is an affine transformation given by $f(v) = Ax + \mu$ where
\[
A = \begin{pmatrix}
D & M \\
0 & B
\end{pmatrix}
\]
for a block-diagonal matrix $D$ consisting of a $d_X \times d_X$ matrix $D_X$ and $d_Y \times d_Y$ matrix $D_Y$, a $(d_X + d_Y) \times d_Z$ matrix $M$ and a full rank $d_Z \times d_Z$ matrix $B$. 

Note first that such a transformation preserves conditional independence. Let $(X^0, Y^0, Z^0)$ be jointly Gaussian with $X^0 \independent Y^0 \cond Z^0$,  joint mean $\mu^0$ and covariance matrix $\Sigma^0$. The distribution of $(\check{X}_0, \check{Y}_0, \check{Z}_0) := f(X^0, Y^0, Z^0)$ is again Gaussian by the finite-dimensional version of Proposition~\ref{prop:linear-transformation-gaussian} and has mean $A \mu^0 + \mu$ and covariance 
\[
A \Sigma^0 A^\top  = \begin{pmatrix}
D \Sigma_{XY}^0 D^\top  + M \Sigma_{Z, XY}^0 D^\top  + D \Sigma_{XY, Z}^0 M^\top  + M \Sigma_Z^0 M^\top  & D \Sigma_{XY, Z}^0 A^\top  + M \Sigma_Z^0 B^\top  \\
B \Sigma_{Z, XY}^0 D^\top  + B \Sigma_Z^0 M^\top  & B \Sigma_Z^0 B^\top 
\end{pmatrix} ,
\]
where $\Sigma_{XY}^0 = \Cov((X^0, Y^0))$, $ \Sigma_{XY, Z}^0 = \Sigma_{Z , XY}^0  = \Cov((X^0, Y^0), Z^0)$ and $\Sigma_Z^0 = \Cov(Z_0)$. Using the finite-dimensional version
of Proposition~\ref{prop:gaussian-conditional-distribution}, we get that the conditional distribution of $(\check{X}_0, \check{Y}_0)$ given $\check{Z}_0$ is again Gaussian with covariance matrix
\begin{align*}
&D \Sigma_{XY}^0 D^\top  + M \Sigma_{Z, XY}^0 D^\top  + D \Sigma_{XY, Z}^0 M^\top  + M \Sigma_Z^0 M^\top  \\
 &\qquad - (D \Sigma_{XY, Z}^0 B^\top  + M \Sigma_Z^0 B^\top )( B\Sigma_Z^0 B^\top )^{-1} (B \Sigma_{Z, XY}^0 D^\top  + B \Sigma_Z^0 M^\top )\\
  = &D(\Sigma^0_{XY} - \Sigma_{XY, Z}^0 \Sigma_Z^0 \Sigma_{Z, XY}^0) D^\top  .
\end{align*}
The matrix $\Sigma^0_{XY} - \Sigma_{XY, Z}^0 \Sigma_Z^0 \Sigma_{Z, XY}^0$ is the conditional covariance matrix of $(X^0, Y^0)$ given $Z^0$ and is block-diagonal since $X^0 \independent Y^0 \cond Z^0$ by the multivariate analogue of Proposition~\ref{prop:uncorrelated-independence-gaussian}. By the same proposition, since $D$ is block-diagonal, we see that the conditional covariance of $(\check{X}_0, \check{Y}_0)$ given $\check{Z}_0$ is block-diagonal and hence $\check{X}_0 \independent \check{Y}_0 \cond \check{Z}_0$ as desired. 

Let now
\[
\Sigma_{X|Z}^{-1/2}\Sigma_{XY \cond Z} \Sigma_{Y \cond Z}^{-1/2}= U S V^\top 
\]
be the singular-value decomposition of the normalised conditional covariance of $X$ and $Y$ given $Z$ under $Q$. The normalisation ensures that $S$ is a rectangular diagonal matrix with diagonal entries in the open unit interval.
If we let
\begin{align*}
B &:= \Sigma_Z^{1/2}, \quad
M := \begin{pmatrix}
\Sigma_{X, Z} \Sigma_Z^{-1/2} \\
\Sigma_{Y, Z} \Sigma_Z^{-1/2}
\end{pmatrix}\\
 D & := \begin{pmatrix}
\Sigma_{X \cond Z}^{1/2} U & 0 \\
0 & \Sigma_{Y \cond Z}^{1/2} V 
\end{pmatrix}, \quad R := S
\end{align*}
and $(\check{X}, \check{Y}, \check{Z}) = f((\tilde{X}, \tilde{Y}, \tilde{Z}))$ where $(\tilde{X}, \tilde{Y}, \tilde{Z}) \sim \tilde{Q}$, then Proposition~\ref{prop:linear-transformation-gaussian} yields that $(\check{X}, \check{Y}, \check{Z}) \sim Q$ and hence when applying $\psi$, we have power $\beta$ by assumption. Since $A$ also transforms a null distribution with identity covariance into a null distribution with where $Z$ has mean $\mu_Z$ and covariance $\Sigma_Z$, we have the desired result.
\end{proof}

\subsection{Hardness of infinite-dimensional Hilbertian Gaussian conditional independence testing}
\label{sec:hardness-infinite-dimensional-gaussian-testing}
In this section we consider the testing problem described in Section~\ref{sec:hardness} with $\mathcal{H}_X$ and $\mathcal{H}_Z$ infinite-dimensional and separable. We will show that the testing problem against $Q$ is hard for any $Q \in \mathcal{Q}$. In particular, this includes the typical functional data setting where $\mathcal{H}_Z = L^2([0, 1], \mathbb{R})$. It follows that the Gaussian conditional independence problem is hard in the same settings when the null distributions are not restricted to match the marginals of $Q$. 

\subsubsection{Preliminary results}

In this section, we consider finite-dimensional $\mathcal{H}_X$ and infinite-dimensional $\mathcal{H}_Z$. We will need a lemma using the theory of conditional Hilbertian Gaussian distributions from Section \ref{sec:conditional-distributions-on-hilbert-spaces}.

\begin{lemma}
\label{lem:gaussian-conditional-dependence-basis}
Let $(X, Y, Z)$ be jointly Gaussian on $\mathbb{R}^{d_X} \times \mathbb{R}^{d_Y} \times \mathcal{H}$ and assume that the covariance operator of $Z$ is injective. Then there exists a basis 
$(e_k)_{k \in \mathbb{N}}$
 of $\mathcal{H}$ such that
\[
(X, Y) \independent Z_{d_X+d_Y+1}, \ldots \cond Z_1, \dots Z_{d_X}, Z_{d_X+1}, \ldots , Z_{d_X+d_Y}
\]
where $Z_k := \langle Z, e_k \rangle$. 
\end{lemma}
\begin{proof}
Note that $\mathbb{R}^{d_X} \times \mathbb{R}^{d_Y} \times \mathcal{H}_Z$ is itself a Hilbert space and decompose it as $(\mathbb{R}^{d_X} \times \mathbb{R}^{d_Y}) \oplus \mathcal{H}_Z$. 
Let $\mathscr{C}_{Z} := \Cov(Z)$, 
$\mathscr{C}_{(X,Y)} := \Cov((X,Y))$ (the covariance of the joint vector $(X, Y)$) and $\mathscr{C}_{(X,Y),  Z} := \Cov((X,Y), Z)$.
We can apply Proposition~\ref{prop:gaussian-conditional-distribution} to see that $(X,Y)$ conditional on $Z$ is Gaussian with mean $\mathscr{C}_{(X,Y), Z} \mathscr{C}_{Z}^\dag Z$ and covariance operator $\mathscr{C}_{(X,Y)} - \mathscr{C}_{(X,Y), Z} \mathscr{C}_{Z}^\dag \mathscr{C}_{(X,Y), Z}^*$. 
The operator $\mathscr{A} := \mathscr{C}_{(X,Y), Z} \mathscr{C}_{Z}^\dag$ maps from $\mathcal{H}$ to $\mathbb{R}^{d_X} \times \mathbb{R}^{d_Y}$ and thus is at most a rank $d_X+d_Y$ operator. By \citet[][Theorem~3.3.7 7.]{Hsing2015} this implies that the rank of $\mathscr{A}^*$ is also at most $d_X+d_Y$. Furthermore, \citet[][Theorem~3.3.7 6.]{Hsing2015} yields that $\mathcal{H}= \textrm{Ker}(\mathscr{A}) \oplus \textrm{Im}(\mathscr{A}^*)$. 
%\Jonas{again, it is important that these are orthogonal, correct? ($\langle v, A^* w\rangle = \langle Av, w\rangle = 0$)} \Anton{Yep, it is crucial but the Theorem states that that is indeed the case since $\oplus$ denotes orthogonal sum.}
Using this decomposition we can write $Z = (Z_{\textrm{Ker}(\mathscr{A})} , Z_{\textrm{Im}(\mathscr{A}^*)})$ and note that by construction $\mathscr{A} Z= \mathscr{A} Z_{\textrm{Im}(\mathscr{A}^*)}$ thus the conditional distribution of $(X,Y)$ given $Z$ only depends on $Z_{\textrm{Im}(\mathscr{A}^*)}$. In total, we have shown by Proposition~\ref{prop:conditional-independence-from-conditional-distribution} that $(X,Y) \independent Z_{\textrm{Ker}(\mathscr{A})} \cond   Z_{\textrm{Im}(\mathscr{A}^*)}$. Letting $r$ denote the rank of $\mathscr{A}^*$, if we start with a basis for $\textrm{Im}(\mathscr{A}^*)$ and append vectors to form a basis for $\mathcal{H}$ using the Gram--Schmidt procedure, we get a basis where
\[
(X, Y) \independent Z_{r+1}, \dots \cond Z_1, \dots Z_r.
\]
Since $r \leq d_X + d_Y$, the weak union property of conditional independence yields
\[
(X, Y) \independent Z_{d_X+d_Y+1}, \dots \cond Z_1, \dots Z_{d_X}, Z_{d_X+1}, \dots , Z_{d_Y+d_X},
\]
as desired.
% \Jonas{The basis for $Im(A^*)$ that we start with has at most $d_X + d_Y$ elements, but potentially less, correct? Thus, do we not need the weak union property of cond independence to deduce the above statement?}
% \Anton{Has this comment been misplaced? I don't see any mention of the weak union property here?}
% \Jonas{Maybe, I am missing sth. But what if the dimension of 
% $Im(A^*)$ 
% is $d_X + d_Y -1$? IMO, the argument above yields 
% \[
% (X, Y) \independent Z_{d_X+d_Y}, Z_{d_X+d_Y+1} \dots \cond Z_1, \dots Z_{d_X}, Z_{d_X+1}, \dots , Z_{d_Y+d_X-1},
% \]
% but the weak union property yields the desired result?
% }
\end{proof}

Using this lemma and Lemma~\ref{lem:gaussian-conditional-independence-reduction-to-simple-distribution} and Theorem~\ref{thm:gaussian-conditional-independence-simple-multivariate-power-bound} from the previous section, we can prove the hardness result for finite-dimensional $\mathcal{H}_X$ and $\mathcal{H}_Y$. 

\begin{theorem}
\label{thm:gaussian-finite-dimensional-conditional-independence-testing-is-hard}
Let $Q \in \mathcal{Q}$ be a distribution consisting of $n$ i.i.d.\ copies of jointly Gaussian and injective $(X, Y, Z)$ on $(\mathbb{R}^{d_X}, \mathbb{R}^{d_Y}, \mathcal{H}_Z)$ for $d_X, d_Y \in \mathbb{N}$ and any infinite-dimensional and separable $\mathcal{H}_Z$. Consider the testing problem described in Section~\ref{sec:hardness} with $\mathcal{H}_X = \mathbb{R}^{d_X}$ and $\mathcal{H}_Z$ as above and let $\psi$ be the test function of a size $\alpha$ test over $\mathcal{P}_0^Q$. Then $\psi$ has power at most $\alpha$ against $Q$.
\end{theorem}
\begin{proof}
Assume for contradiction that $\psi$ is a test of size $\alpha$ over $\mathcal{P}_0^\mathcal{Q}$ with power $\alpha + \varepsilon$ for some $\varepsilon > 0$ against $Q$. Let $(X, Y, Z)$ be distributed as one of the $n$ i.i.d.\ copies constituting $Q$. By Lemma~\ref{lem:gaussian-conditional-dependence-basis}, we can express $Z$ in a basis $(e_k)_{k \in \mathbb{N}}$ such that defining $Z_k = \langle Z , e_k \rangle$, 
we have
\[
(X, Y) \independent Z_{d_X+d_Y+1}, \dots \cond Z_1, \dots , Z_{d_X}, Z_{d_X+1}, \dots , Z_{d_X+d_Y}.
\]
By the weak union property of conditional independence, this implies that
\[
(X, Y) \independent Z_{d+1}, \dots \cond Z_1, \dots , Z_d
\]
for any $d \geq d_X + d_Y$. 

Choose now an arbitrary $d \geq d_X + d_Y$ and let $\tilde{Q}$
denote the distribution of $n$ i.i.d.\ copies of $(X, Y, Z_1, \dots, Z_d)$ under $Q$. Consider the testing problem described in Section~\ref{sec:hardness} with $\mathcal{H}_X = \mathbb{R}^{d_X}$ and $\mathcal{H}_Z = \mathbb{R}^d$. We can construct a test in this setting by defining new observations $(\check{X}, \check{Y}, \check{Z})$ with values in $(\mathbb{R}^{d_X}, \mathbb{R}^{d_Y}, \mathcal{H}_Z)$ and applying $\psi$. We form the new observations by setting $\check{X} := \tilde{X}$, $\check{Y} := \tilde{Y}$ and $\check{Z}:=(\tilde{Z}_1, \dots, \tilde{Z}_d, Z^\circ_{d+1}, Z^\circ_{d+2}, \dots)$, where $Z^\circ_{d+1}, Z^\circ_{d+2}, \dots$ are sampled from the conditional distribution $Z_{d+1}, Z_{d+2}, \dots \cond Z_1 = \tilde{Z}_1, \dots Z_d = \tilde{Z}_d$. If the original sample is from a distribution in $\mathcal{P}_0^{\tilde{Q}}$ then the modified sample will be from a null distribution in $\mathcal{P}_0^Q$, thus the test has size $\alpha$ over $\mathcal{P}_0^{\tilde{Q}}$. Similarly, if $(\tilde{X}, \tilde{Y}, \tilde{Z}) \sim \tilde{Q}$, the modified sample will have distribution $Q$ and hence the test has power $\alpha+\varepsilon$ against $\tilde{Q}$.

By Lemma~\ref{lem:gaussian-conditional-independence-reduction-to-simple-distribution} this implies the existence of a $d_X \times d_Y$ block-diagonal matrix $R$ with diagonal entries in the open unit interval, a Gaussian distribution $Q'$ on $(\mathbb{R}^{d_X}, \mathbb{R}^{d_Y}, \mathbb{R}^d)$ where if $(X' , Y', Z') \sim Q'$, $X'$, $Y'$ and $Z'$ are mean zero with identity covariance matrix, $\Cov(X', Z') = \Cov(Y', Z') = 0$ and $\Cov(X', Y') = R$, and a test with size $\alpha$ over $\mathcal{P}_0^{Q'}$ with power $\alpha + \varepsilon$ against $Q'$. Since $d$ was arbitrary, this contradicts Theorem~\ref{thm:gaussian-conditional-independence-simple-multivariate-power-bound}.
\end{proof}

\subsubsection{Proofs of Theorem~\ref{thm:main-hardness-theorem} and Proposition~\ref{prop:gaussian-conditional-independence-testing-is-hard-when-dependence-is-never-in-tail}}
\label{sec:hardness-proofs}
In this section we prove Theorem~\ref{thm:main-hardness-theorem} and Proposition~\ref{prop:gaussian-conditional-independence-testing-is-hard-when-dependence-is-never-in-tail}. We do this by extending the results from the previous section to the situation where at most one of $X$ and $Y$ are infinite-dimensional.

\begin{lemma}
\label{lem:gaussian-conditional-dependence-basis-variant}
Let $(X, Y, Z)$ be jointly Gaussian on $\mathbb{R}^{d_X} \times \mathcal{H}_Y \times \mathcal{H}_Z$ and assume that the covariance operator of $(Y, Z)$ is injective. Then there exists a basis $(e_k)_{k\in \mathbb{N}}$ of $\mathcal{H}_Y$ such that
\[
X \independent Y_{d_X+1}, \dots \cond Y_1, \dots , Y_{d_X} , Z
\]
where $Y_k := \langle Y, e_k \rangle$. 
\end{lemma}
\begin{proof}
Note that $\mathbb{R}^{d_X} \times \mathcal{H}_Y \times \mathcal{H}_Z$ is again a Hilbert space and decompose it as $\mathbb{R}^{d_X} \oplus (\mathcal{H}_Y \times \mathcal{H}_Z)$. Let $\mathscr{C}_{(Y,Z)} := \Cov((Y, Z))$ (the covariance of the joint vector $(Y, Z)$), $\mathscr{C}_{X} := \Cov(X)$ and $\mathscr{C}_{X, (Y, Z)} := \Cov(X, (Y, Z))$. 
We can apply Proposition~\ref{prop:gaussian-conditional-distribution} to see that $X$ conditional on $(Y, Z)$ is Gaussian with mean $\mathscr{C}_{X, (Y,Z)} \mathscr{C}_{(Y,Z)}^\dag (Y,Z)$ and covariance operator $\mathscr{C}_{X} - \mathscr{C}_{X, (Y,Z)} \mathscr{C}_{(Y, Z)}^\dag \mathscr{C}_{X, (Y,Z)}^*$. 
The operator $\mathscr{A} = \mathscr{C}_{X, (Y,Z)} \mathscr{C}_{(Y, Z)}^\dag $ maps from $\mathcal{H}_Y \times \mathcal{H}_Z$ to $\mathbb{R}^{d_X}$ and thus is at most a rank $d_X$ operator. By \citet[][Theorem~3.3.7 7.]{Hsing2015} this implies that the rank of $\mathscr{A}^*$ is also at most $d_X$. Furthermore, \citet[][Theorem~3.3.7 6.]{Hsing2015} yields that $\mathcal{H}_Y \times \mathcal{H}_Z= \textrm{Ker}(\mathscr{A}) \oplus \textrm{Im}(\mathscr{A}^*)$. 

Using this decomposition we can write $(Y, Z) = ((Y, Z)_{\textrm{Ker}(\mathscr{A})} , (Y, Z)_{\textrm{Im}(\mathscr{A}^*)})$ and note that by construction $\mathscr{A} (Y, Z)= \mathscr{A} (Y, Z)_{\textrm{Im}(\mathscr{A}^*)}$ thus the conditional distribution of $X$ given $(Y, Z)$ only depends on $(Y, Z)_{\textrm{Im}(\mathscr{A}^*)}$.
In total, we have shown by Proposition~\ref{prop:conditional-independence-from-conditional-distribution} that $X \independent (Y, Z)_{\textrm{Ker}(\mathscr{A})} \cond  (Y, Z)_{\textrm{Im}(\mathscr{A}^*)}$ which implies by the weak union property of conditional independence that $X \independent Y_{\textrm{Ker}(\mathscr{A})} \cond Y_{\textrm{Im}(\mathscr{A}^*)}, Z$.

Any basis of $\textrm{Im}(\mathscr{A}^*)$ will consist of at most $d_X$ elements. Forming the span of the $\mathcal{H}_Y$-components of the basis vectors will yield a subspace of $\mathcal{H}_Y$ that contains the projection onto $\mathcal{H}_Y$ of $\textrm{Im}(\mathscr{A}^*)$. Thus, letting $r$ denote the rank of $\mathscr{A}^*$, we can append vectors and form a basis for $\mathcal{H}_Y$ using the Gram--Schmidt procedure to get a basis where 
\[
X \independent Y_{r+1}, \dots \cond Y_1, \dots , Y_r , Z .
\]
Since $r \leq d_X$, the weak union property of conditional independence yields
\[
X \independent Y_{d_X+1}, \dots \cond Y_1, \dots , Y_{d_X} , Z
\]
as desired.
\end{proof}

We are now ready to prove Theorem~\ref{thm:main-hardness-theorem}.

\begin{proof}[Proof of Theorem~\ref{thm:main-hardness-theorem}]
Assume without loss of generality that $\mathcal{H}_X$ is finite-dimensional and thus $\mathcal{H}_X$ is isomorphic to a real vector space, and we will instead denote $\mathcal{H}_X = \mathbb{R}^{d_X}$ where $d_X$ is the dimension of $\mathcal{H}_X$.

Assume for contradiction that $\psi$ is a test of size $\alpha$, with power $\alpha + \epsilon$ for some $\epsilon > 0$ against $Q$. Let $(X, Y, Z)$ be distributed as one of the $n$ i.i.d.\ copies constituting $Q$. By Lemma~\ref{lem:gaussian-conditional-dependence-basis-variant} we can express $Y$ in a basis $(e_k)_{k \in \mathbb{N}}$ such that defining $Y_k = \langle Y , e_k \rangle$, we have
\[
X \independent Y_{d_X+1}, \dots \cond Y_1, \dots, Y_{d_X}, Z.
\]
By the weak union property of conditional independence, this implies that
\[
X \independent Y_{d+1}, \dots \cond Y_1, \dots , Y_d, Z
\]
for any $d \geq d_X$.

Choose now an arbitrary $d \geq d_X + d_Y$ and let $\tilde{Q}$ denote the distribution of $n$ i.i.d.\ copies of $(X, Y_1, \dots, Y_d, Z)$ under $Q$. Consider the testing problem described in Section~\ref{sec:hardness} with $\mathcal{H}_X = \mathbb{R}^{d_X}$ and $\mathcal{H}_Z$ as above. We can construct a test in this setting by defining new observations $(\check{X}, \check{Y}, \check{Z})$ with values in $(\mathbb{R}^{d_X}, \mathcal{H}_Y, \mathcal{H}_Z)$ and applying $\psi$. We form the new observations by setting $\check{X} := \tilde{X}$, $\check{Z}:=\tilde{Z}$ and $\check{Y} := (\tilde{Y}_1, \dots, \tilde{Y}_d, Y^\circ_{d+1}, Y^\circ_{d+2}, \dots)$, where $Y^\circ_{d+1}, Y^\circ_{d+2}, \ldots$ are sampled from the conditional distribution $Y_{d+1}, Y_{d+2}, \dots \cond Y_1 = \tilde{Y}_1, \ldots, Y_d = \tilde{Y}_d, Z=\tilde{Z}$. If the original sample is from a distribution in $\mathcal{P}_0^{\tilde{Q}}$, then the modified sample will be from a null distribution in $\mathcal{P}_0^Q$, thus the test has size $\alpha$ over $\mathcal{P}_0^{\tilde{Q}}$. Similarly, if $(\tilde{X}, \tilde{Y}, \tilde{Z}) \sim \tilde{Q}$, the modified sample will have distribution $Q$ and hence the test has power $\alpha+\epsilon$ against the distribution of $(X, Y_1, \dots, Y_d, Z)$. But this contradicts Theorem~\ref{thm:gaussian-finite-dimensional-conditional-independence-testing-is-hard}.
\end{proof}

A similar strategy can be employed to prove Proposition~\ref{prop:gaussian-conditional-independence-testing-is-hard-when-dependence-is-never-in-tail}.

\begin{proof}[Proof of Proposition~\ref{prop:gaussian-conditional-independence-testing-is-hard-when-dependence-is-never-in-tail}]
  We can repeat the arguments of Theorem~\ref{thm:gaussian-finite-dimensional-conditional-independence-testing-is-hard} and Theorem~\ref{thm:main-hardness-theorem} without using Lemma~\ref{lem:gaussian-conditional-dependence-basis} and Lemma~\ref{lem:gaussian-conditional-dependence-basis-variant} since we can use the basis $(e_k)_{k \in \mathbb{N}}$ instead. 
\end{proof}

\subsection{Auxiliary results about conditional distributions on Hilbert spaces} \label{sec:conditional-distributions-on-hilbert-spaces}

Let us first recall how to formally define a conditional distribution. We follow \citet[][Chapter 10.2]{Dudley2002} and \citet{Roenn-Nielsen2014}. 
\begin{definition}
Let $(\Omega, \mathcal{F}, \mathbb{P})$ be a probability space, let $\mathcal{D}$ be a sub-$\sigma$-algebra of $\mathcal{F}$ and let $\mathbb{P}_{|\mathcal{D}}$ denote the restriction of $\mathbb{P}$ to $\mathcal{D}$. Let $X$ be a random variable defined on $(\Omega, \mathcal{F}, \mathbb{P})$ mapping into a measurable space $(\mathcal{X}, \mathcal{A})$. We say that a function $P_{X \cond \mathcal{D}}: \mathcal{A} \times \Omega \to [0,1]$ is a \emph{conditional distribution for $X$ given $\mathcal{D}$} if the following two conditions hold.
\begin{enumerate}[(i)]
\item For each $A \in \mathcal{A}$, $P_{X \cond \mathcal{D}}(A, \cdot) = \mathbb{E}(\ind_{\{(X \in A)\}}\cond \mathcal{D}) = \mathbb{P}(X \in A \cond \mathcal{D})$  $\mathbb{P}_{|\mathcal{D}}$-a.s.
\item For $P_{|\mathcal{D}}$ almost every $\omega \in \Omega$, $P_{X \cond \mathcal{D}}(\cdot, \omega)$ is a probability measure on $(\mathcal{X}, \mathcal{A})$. 
\end{enumerate}
\end{definition}
We are mainly interested in conditioning on the value of some random variable which leads to the following definition.
\begin{definition}
Consider random variables $X$ and $Y$ defined on the probability space $(\Omega, \mathcal{F}, \mathbb{P})$ with values in the measurable spaces $(\mathcal{X}, \mathcal{A})$ and $(\mathcal{Y}, \mathcal{G})$, respectively. We say that a function $P_{Y \cond X}:\mathcal{G} \times \mathcal{X} \to [0,1]$ is a \emph{conditional distribution for $Y$ given $X$} if the following conditions hold.
\begin{enumerate}[(i)]
\item For each $x \in \mathcal{X}$, $P_{Y \cond X}(\cdot, x)$ is a probability measure on $(\mathcal{Y}, \mathcal{G})$.
\item For each $G \in \mathcal{G}$, $P_{Y \cond X}(G, \cdot)$ is $\mathcal{A}$-$\mathbb{B}$ measurable, where $\mathbb{B}$ denotes the Borel $\sigma$-algebra on $\mathbb{R}$.
\item For each $A \in \mathcal{A}$ 
\[
\mathbb{P}(X \in A, Y \in G) = \int_{(X \in A)} P_{Y \cond X}(G, X(\omega)) \, \mathrm{d}\mathbb{P}(\omega) = \int_{A} P_{Y \cond X}(G, x) \, \mathrm{d}X(\mathbb{P})(x),  
\]
where $X(\mathbb{P})$ is the push-forward measure of $X$ under $\mathbb{P}$, i.e. the measure on $(\mathcal{X}, \mathcal{A})$ such that $X(\mathbb{P})(A) = \mathbb{P}(X \in A)$ for $A \in \mathcal{A}$. 
\end{enumerate}
Informally, we write $Y \cond X$ for the conditional distribution of $Y$ given $X$ and $Y \cond X=x$ for the measure $P_{Y \cond X}(\cdot, x)$. If a function $Q: \mathcal{G} \times \mathcal{X} \to [0, 1]$ only satisfies the first two conditions, we say that $Q$ is a \emph{$(\mathcal{X}, \mathcal{A})$-Markov kernel on $(\mathcal{Y}, \mathcal{G})$}. 
\end{definition}
The connection between the previous two definitions can be seen by viewing $X$ and $Y$ as random variables on the probability space $(\mathcal{X} \times \mathcal{Y}, \mathcal{A} \otimes \mathcal{G}, (X, Y)(\mathbb{P}))$ where $(X, Y)(\mathbb{P})$ is the joint push-forward measure of $X$ and $Y$ under $\mathbb{P}$. If we then let $\mathcal{D}$ be the smallest $\sigma$-algebra making the projection onto the $\mathcal{X}$-space measurable, 
we see by letting $P_{Y \cond \mathcal{D}}(G, (x,y)) = P_{Y \cond X}(G, x)$ that $P_{Y \cond X}$ also satisfies the conditions of the first definition. For more on this perspective, see \citet[][Theorem~10.2.1]{Dudley2002}. It is non-trivial to show the existence of conditional distributions, however, we do have the following result from \citet[][Theorem~10.2.2]{Dudley2002}.
\begin{lemma}
Consider random variables $X$ and $Y$ defined on the probability space $(\Omega, \mathcal{F}, \mathbb{P})$ with values in the measurable spaces $(\mathcal{X}, \mathcal{A})$ and $(\mathcal{Y}, \mathcal{G})$ respectively. If $\mathcal{X}$ and $\mathcal{Y}$ are Polish spaces and $\mathcal{A}$ and $\mathcal{G}$ are their respective Borel $\sigma$-algebras then the conditional distribution for $Y$ given $X$ exists.
\end{lemma}
We will consider real-valued and Hilbertian random variables in the following, thus we are free to assume the existence of conditional distributions wherever needed. Before we delve into the main preliminary results about Hilbertian conditional distributions, we present some fundamental results from the theory of regular conditional distributions. For measurable spaces $(\mathcal{X}, \mathcal{A})$ and $(\mathcal{Y}, \mathcal{G})$, we let $i_x: \mathcal{Y} \to \mathcal{X} \times \mathcal{Y}$ denote the inclusion map, i.e.\ $i_x(y)=(x, y)$. This is a $\mathcal{G}-\mathcal{A} \otimes \mathcal{G}$ measurable mapping for each fixed $x$. The following four results are included for completeness and can be found in \citet[][Lemma~1.1.4, Theorem~1.2.1, Theorem~2.1.1 \& Theorem~3.5.5]{Roenn-Nielsen2014}. Unless otherwise specified, for these results $X$, $Y$ and $Z$ are random variables on measurable spaces $(\mathcal{X}, \mathcal{A})$, $(\mathcal{Y}, \mathcal{G})$ and $(\mathcal{Z}, \mathcal{K})$ respectively. 

\begin{lemma}
\label{lem:measurability-of-image-measure-of-kernel-under-inclusion-mapping}
Let $Q$ be a $(\mathcal{X}, \mathcal{A})$-Markov kernel on $(\mathcal{Y}, \mathcal{G})$ and let $\mathbb{B}$ denote the Borel $\sigma$-algebra on $\mathbb{R}$. For each $C \in \mathcal{A} \otimes \mathcal{G}$ the map
\[
x \mapsto Q(i_x^{-1}(C), x)
\]
is $\mathcal{A}$-$\mathbb{B}$ measurable.
\end{lemma}
\begin{proof}
Let
\[
\mathcal{D} = \{C \in \mathcal{A} \otimes \mathcal{G} \cond \text{$x \mapsto Q(i_x^{-1}(C), x)$ is $\mathcal{A}$-$\mathbb{B}$ measurable} \}  
\]
and consider a product set $A \times G \in \mathcal{A} \otimes \mathcal{G}$. Clearly, 
\[
i_x^{-1}(A \times G)  = \begin{cases}
  \emptyset & \text{if $x \not\in A$}\\
  B & \text{if $x \in A$}
\end{cases}
\]
and therefore
\[
Q(i_x^{-1}(A \times G), x)  = \begin{cases}
  0 & \text{if $x \not\in A$}\\
  Q(G, x) & \text{if $x \in A$}
\end{cases} = \ind_A(x) Q(G, x).
\]
This is a product of two $\mathcal{A}$-$\mathbb{B}$ measurable functions and is thus also $\mathcal{A}$-$\mathbb{B}$ measurable. This shows that $\mathcal{D}$ contains all product sets and since the product sets are an intersection-stable generator of $\mathcal{A} \otimes \mathcal{G}$, we are done if we can show that $\mathcal{D}$ is a Dynkin class by \citet[][Theorem~5.5]{schilling}. 

We have already shown that product sets are in $\mathcal{D}$ which includes $\mathcal{X} \times \mathcal{Y}$. If $C_1, C_2 \in \mathcal{D}$ where $C_1 \subseteq C_2$ then clearly also $i_x^{-1}(C_1) \subseteq i_x^{-1}(C_2)$ and further $i_x^{-1}(C_2 \setminus C_1) = i_x^{-1}(C_2) \setminus i_x^{-1}(C_1)$. This implies that
\[
Q(i_x^{-1}(C_2 \setminus C_1), x) = Q(i_x^{-1}(C_2), x) - Q(i_x^{-1}(C_1), x)
\]
which is the difference of two $\mathcal{A}$-$\mathbb{B}$ measurable functions and is thus also $\mathcal{A}$-$\mathbb{B}$ measurable. Hence, $C_2 \setminus C_1 \in \mathcal{D}$. Finally, assume that $C_1 \subseteq C_2 \subseteq \cdots $ is an increasing sequence of $\mathcal{D}$-sets. Similarly to above we have $i_x^{-1}(C_1) \subseteq i_x^{-1}(C_2) \subseteq \cdots $ and
\[
i_x^{-1}\left( \bigcup_{n=1}^\infty C_n \right) = \bigcup_{n=1}^\infty i_x^{-1}\left(  C_n \right) .
\]
Then
\[
Q \left( i_x^{-1}\left( \bigcup_{n=1}^\infty C_n \right), x \right) = Q \left( \bigcup_{n=1}^\infty i_x^{-1}\left(  C_n \right), x \right) = \lim_{n \to \infty} Q(i_x^{-1}(C_n), x) .
\]
The limit is $\mathcal{A}$-$\mathbb{B}$ measurable since each of the functions $x \mapsto Q(i_x^{-1}(C_n), x)$ are measurable. Hence, $\mathcal{D}$ is a Dynkin class, and we have the desired result.
\end{proof}

\begin{proposition}
\label{prop:existence-of-integration-measure}
Let $\mu$ be a probability measure on $(\mathcal{X}, \mathcal{A})$ and let $Q$ be a $(\mathcal{X}, \mathcal{A})$-Markov kernel on $(\mathcal{Y}, \mathcal{G})$. There exists a uniquely determined probability measure $\lambda$ on $(\mathcal{X} \times \mathcal{Y}, \mathcal{A} \otimes \mathcal{G})$ satisfying
\[
\lambda(A \times G) = \int_A Q(G, x) , \mathrm{d}\mu(x)  
\]
for all $A \in \mathcal{A}$ and $G \in \mathcal{G}$. Furthermore, for $C \in \mathcal{A} \otimes \mathcal{G}$
\[
\lambda(C) = \int Q(i_x^{-1}(C), x) , \mathrm{d}\mu(x) .
\]
\end{proposition}
\begin{proof}
Uniqueness follows from \citet[][Theorem~5.7]{schilling} since $\lambda$ is determined on the product sets which form an intersection-stable generator of $\mathcal{A} \otimes \mathcal{G}$.

For existence, we show that $\lambda$ as defined for general $C \in \mathcal{A} \otimes \mathcal{G}$ is a measure. The integrand is measurable by Lemma~\ref{lem:measurability-of-image-measure-of-kernel-under-inclusion-mapping} and since $Q$ is non-negative, the integral is well-defined with values in $[0, \infty]$. Let $C_1, C_2 \dots $ be a sequence of disjoint sets in $\mathcal{A} \otimes \mathcal{G}$. Then for each $x \in \mathcal{X}$ the sets $i_x^{-1}(C_1), i_x^{-1}(C_2), \dots$ are disjoint as well. Hence,
\begin{align*}
\lambda \left(\bigcup_{n=1}^\infty C_n \right) &= \int Q \left( i_x^{-1} \left( \bigcup_{n=1}^\infty C_n \right)  , x\right) \, \mathrm{d} \mu(x) =  \int \sum_{n=1}^\infty Q \left( i_x^{-1} (C_n)  , x\right) \, \mathrm{d} \mu(x) \\
&= \sum_{n=1}^\infty \int  Q \left( i_x^{-1} (C_n)  , x\right) \, \mathrm{d} \mu(x) = \sum_{n=1}^\infty \lambda(C_n)
\end{align*}
where the second equality uses that $Q(\cdot, x)$ is a measure and the third uses monotone convergence to interchange integration and summation. Since also
\[
\lambda(\mathcal{X} \times \mathcal{Y}) = \int Q( i_x^{-1}(\mathcal{X} \times \mathcal{Y}), x) \, \mathrm{d}\mu(x) = \int Q( \mathcal{Y}, x) \, \mathrm{d}\mu(x) = \int 1 \, \mathrm{d}\mu(x) = 1 
\]
$\lambda$ is a probability measure, and it follows that 
\[
\lambda(A \times G) = \int Q(i_x^{-1}(A \times G), x) \, \mathrm{d}\mu(x) = \int_A Q(G, x) \, \mathrm{d}\mu(x) 
\]
for all $A \in \mathcal{A}$ and $G \in \mathcal{G}$ as desired.
\end{proof}

\begin{proposition}
\label{prop:substitution-theorem-for-conditional-distributions}
Assume that $P_{Y \cond X}$ is the conditional distribution of $Y$ given $X$. Let $(\mathcal{Z}, \mathcal{K})$ be another measurable space and let $\phi: \mathcal{X} \times \mathcal{Y} \to \mathcal{Z}$ be a measurable mapping. Define $Z= \phi(X, Y)$. Then the conditional distribution of $Z$ given $X$ exists and for $K \in \mathcal{K}$ and $x \in \mathcal{X}$ is given by
\[
P_{Z \cond X}(K, x) = P_{Y \cond X}((\phi \circ i_x)^{-1}(K), x).
\]
\end{proposition}
\begin{proof}
Clearly $P_{Z \cond X}(\cdot , x)$ is a probability measure for every $x \in \mathcal{X}$ and Lemma~\ref{lem:measurability-of-image-measure-of-kernel-under-inclusion-mapping} yields that $P_{Z \cond X}(K, \cdot)$ is $\mathcal{A}$-$\mathbb{B}$ measurable for every $K \in \mathcal{K}$. It remains to show that $P_{Z \cond X}$ satisfies the third condition required to be the conditional distribution of $Z$ given $X$. For $A \in \mathcal{A}$ and $K \in \mathcal{K}$ we get that
\[
\mathbb{P}(X \in A, Z \in K) = \mathbb{P}((X, Y) \in (A \times \mathcal{Y}) \cap \phi^{-1}(K))
\]
and hence by Proposition~\ref{prop:existence-of-integration-measure}, we get that
\[
\mathbb{P}(X \in A, Z \in K) = \int P_{Y \cond X}(i_x^{-1}((A \times \mathcal{Y}) \cap \phi^{-1}(K)), x) \, \mathrm{d}X(\mathbb{P})(x) .
\]
Since 
\[
i_x^{-1}((A \times \mathcal{Y}) \cap \phi^{-1}(K))  = \begin{cases}
  \emptyset & \text{if $x \not \in A$}\\
  i_x^{-1}(\phi^{-1}(K)) & \text{if $x \in A$}
\end{cases},
\]
we get
\[
\mathbb{P}(X \in A, Z \in K) = \int_A P_{Y \cond X}(i_x^{-1}(\phi^{-1}(K)), x) \, \mathrm{d}X(\mathbb{P})(x) = \int_A P_{Z \cond X}(K, x) \, \mathrm{d}X(\mathbb{P})(x),
\]
proving the desired result.
\end{proof}

\begin{proposition}
\label{prop:conditional-independence-from-conditional-distribution}
Suppose that conditional distribution $P_{Y \cond (X,Z)}$ of $Y$ given $(X, Z)$ has the structure
\[
P_{Y \cond (X, Z)}(G, (x, z)) = Q(G, z)
\]
for some $Q: \mathcal{G} \times \mathcal{Z}$ where for every $z \in \mathcal{Z}$, $Q(\cdot, z)$ is a probability measure. Then $Q$ is a Markov kernel, $Q$ is the conditional distribution of $Y$ given $Z$ and $X \independent Y \cond Z$.
\end{proposition}
\begin{proof}
That $Q$ is a Markov kernel follows immediately from the fact that $P_{Y \cond (X,Z)}$ is a Markov kernel. To see that $Q$ is the conditional distribution of $Y$ given $Z$, note that defining $\pi_Z: \mathcal{X} \times \mathcal{Z} \to \mathcal{Z}$ to be the projection onto $\mathcal{Z}$, we get
\begin{align*}
\mathbb{P}(Z \in K, Y \in G) &= \mathbb{P}((X, Z) \in \pi_Z^{-1}(K), Y \in G) = \int_{\pi_Z^{-1}(K)} P_{Y \cond (X, Z)}(G, (x, z)) \, \mathrm{d}(X, Z)(P)(x, z)\\
&= \int_{\pi_Z^{-1}(K)} Q(G, \pi_Z(x, z)) \, \mathrm{d}(X, Z)(P)(x, z) =  \int_K Q(G, z) \, \mathrm{d}Z(P)(z),
\end{align*}
by viewing $Z(P)$ as the image measure of $(X,Z)(P)$ under $\pi_Z$ and applying \citet[][Theorem~14.1]{schilling}. For every $G \in \mathcal{G}$, $Q(G, Z)$ is a version of the conditional probability $\mathbb{P}(Y \in G \cond Z) = \mathbb{E}(1_{(Y \in G)} \cond Z)$ since $Q(G, Z)$ is clearly measurable with respect to $\sigma(Z)$ and 
\[
\int_{(Z \in K)} 1_{(Y \in G)}  \, \mathrm{d}P = \mathbb{P}(Z \in K, Y \in G) = \int_{(Z \in K)} Q(G, Z) \, \mathrm{d}P.
\]
The same argument applies to show that $P_{Y \cond (X,Z)}(G, (X,Z))$ is a version of $\mathbb{P}(Y \in G \cond X, Z)$. Hence, for every $G \in \mathcal{G}$
\[
\mathbb{P}(Y \in G  \cond Z) = Q(G, Z) = P_{Y \cond (X, Z)}(G, (X, Z)) = \mathbb{P}(Y  \cond X, Z)  
\]
and thus $X \independent Y \cond Z$ as desired.
\end{proof}

With these results we are ready to start considering Hilbertian conditional distributions. 

\begin{remark}
\label{remark:hilbert-space-orthogonal-decomposition}
In the following we will repeatedly consider orthogonal decompositions of Hilbert spaces. We write $\mathcal{H} = \mathcal{H}_1 \oplus \mathcal{H}_2$ if every $h \in \mathcal{H}$ can be written as $h = h_1 + h_2$ where $h_1 \in \mathcal{H}_1$ and $h_2 \in \mathcal{H}_2$ and $\mathcal{H}_1 \perp \mathcal{H}_2$. If an operator $\mathscr{A}$ is defined on $\mathcal{H}$, the decomposition induces four operators: $\mathscr{A}_{11}$ and $\mathscr{A}_{21}$, the $\mathcal{H}_1$ and $\mathcal{H}_2$ components of the restriction of $\mathscr{A}$ to $\mathcal{H}_1$ and similarly $\mathscr{A}_{12}$ and $\mathscr{A}_{22}$, the $\mathcal{H}_1$ and $\mathcal{H}_2$ components of the restriction of $\mathscr{A}$ to $\mathcal{H}_2$. We can write $\mathscr{A}$ as the sum of these four operators. 
If $X$ is a random variable on $\mathcal{H}$ 
and $\mathcal{H}_1$ and $\mathcal{H}_2$ are as above, we can similarly decompose $X$ into $(X_1, X_2)$ where $X_1 \in \mathcal{H}_1$ and $X_2 \in \mathcal{H}_2$. If $\mathscr{C}$ is the covariance operator of $X$, we can decompose it as mentioned above and, in particular, we have $\mathscr{C}_{11}= \Cov(X_1)$, $\mathscr{C}_{22}= \Cov(X_2)$ and $\mathscr{C}_{12} = \mathscr{C}_{21}^* = \Cov(X_1, X_2)$, where $ \mathscr{C}_{21}^*$ denotes the adjoint of $\mathscr{C}_{21}$. 
This is analogous to the usual block matrix decomposition of the covariance matrix of multivariate random variables.
\end{remark}
We will need two results that are fundamental in the theory of the multivariate Gaussian distribution.

%\Jonas{call this lemma?}
\begin{proposition}
\label{prop:uncorrelated-independence-gaussian}
Let $X$ be Gaussian on $\mathcal{H}$ and assume that $\mathcal{H} = \mathcal{H}_1 \oplus \mathcal{H}_2$. Define $(X_1, X_2)$ to be the corresponding decomposition of $X$. Then $X_1 \independent X_2$ if and only if $\Cov(X_1, X_2) = 0$. 
\end{proposition}
\begin{proof}
We show that $\Cov(X_1, X_2) = 0$ implies independence since the other direction is trivial. We will use the approach of characteristic functionals as described in detail in \citet[][Chapter IV]{Vakhania1987}. The characteristic functional of a random variable (technically, the distribution of the random variable) is the mapping defined on $\mathcal{H}$ where $h \mapsto \mathbb{E}[\exp(i \langle X, h \rangle)]$. 
\citet[][Theorem~IV.2.4]{Vakhania1987} state that for Gaussian $X$ with mean $\mu$ and covariance operator $\mathscr{C}$ the characteristic functional is
\[
\phi_X(h) = \exp\left(i \langle \mu , h \rangle - \frac{1}{2} \langle \mathscr{C} h, h \rangle  \right) . 
\]
\citet[][Chapter~IV, Proposition~2.2 + Corollary]{Vakhania1987} state that $X_1$ and $X_2$ are independent if the characteristic functional of $X$ factorises into the product of their respective characteristic functionals. By the assumption that $\mathscr{C}_{12} = \Cov(X_1, X_2) = 0$, we can write the covariance as $\mathscr{C} = \mathscr{C}_{1} + \mathscr{C}_{2}$ where $\mathscr{C}_i$ is the covariance of $X_i$. The result then follows by factorising the characteristic functional appropriately. 
\end{proof}

\begin{proposition}
\label{prop:linear-transformation-gaussian}
Let $X$ be Gaussian on $\mathcal{H}_1$ with mean $\mu$ and covariance operator $\mathscr{C}$ and let $\mathscr{A}$ be a bounded linear operator from $\mathcal{H}_1$ to $\mathcal{H}_2$ and $z \in \mathcal{H}_2$. Then $Y = \mathscr{A} X + z$ is Gaussian on $\mathcal{H}_2$ with mean $\mathscr{A}\mu + z$ and covariance operator $\mathscr{A}\mathscr{C}\mathscr{A}^*$ where $\mathscr{A}^*$ is the adjoint of $\mathscr{A}$. 
\end{proposition}
\begin{proof}
Throughout, we let $\langle \cdot , \cdot \rangle_1$ and  $\langle \cdot , \cdot \rangle_2$ denote the inner products of $\mathcal{H}_1$ and $\mathcal{H}_2$ respectively. By definition, for every $h_1 \in \mathcal{H}_1$, $\langle X , h_1 \rangle$ is Gaussian on $\mathbb{R}$. For every $h_2 \in \mathcal{H}_2$ we have
\[
\langle Y , h_2 \rangle_2 = \langle \mathscr{A}X, h_2 \rangle_2 + \langle z , h_2 \rangle_2 =  \langle X, \mathscr{A}^* h_2 \rangle_1 + \langle z , h_2 \rangle_2 
\]
thus $Y$ is also Gaussian. Using the interchangeability of the Bochner integral and linear operators (see \citet[][Theorem~3.1.7]{Hsing2015}), we get the mean of $Y$ immediately. By noting that for any $h, k \in \mathcal{H}_1$, we have
\[
(\mathscr{A}h) \otimes k = \langle \mathscr{A}h , \cdot \rangle_2  k  = \langle h , \mathscr{A}^* \cdot \rangle_1 k = (h \otimes k) \mathscr{A}^*,
\]
the covariance result then follows by the same argument as for the mean.
%\Jonas{is the bochner integration really necessary? i believe one could also use the characterizations based on dot products, but maybe that does not matter...}
%\Anton{One could but I don't believe that it is as straightforward to get the nice commutativity with linear operators if you go through that definition (which would be the Pettis integral).}
\end{proof}

With these results we can now show that conditioning on an injective part of a Gaussian distribution on a Hilbert space yields another Gaussian distribution with mean and covariance given by the Hilbertian analogue of the well-known Gaussian conditioning formula.

%\Jonas{call this lemma?}
\begin{proposition}
\label{prop:gaussian-conditional-distribution}
Let $X$ be mean zero Gaussian on $\mathcal{H}$ with covariance operator $\mathscr{C}$ and assume that $\mathcal{H} = \mathcal{H}_1 \oplus \mathcal{H}_2$. 
Let $(X_1, X_2)$ denote the corresponding decomposition of $X$. As discussed in Remark~\ref{remark:hilbert-space-orthogonal-decomposition}, we then set $\mathscr{C}_{11} := \Cov(X_1)$, $\mathscr{C}_{22} := \Cov(X_2)$ and $\mathscr{C}_{12} = \mathscr{C}_{21}^*  := \Cov(X_1, X_2)$,
 where  $\mathscr{C}_{21}^*$ denotes the adjoint of $\mathscr{C}_{21}$. If $\mathscr{C}_{22}$ is injective, i.e.
\[
\textrm{Ker}(\mathscr{C}_{22}) = \{ h \in \mathcal{H}_2 \cond \mathscr{C}_{22} h = 0 \} = \{0\}
\]
then the conditional distribution of $X_1$ given $X_2$ is Gaussian on $\mathcal{H}_1$ with 
\[
\mathbb{E}(X_1 \cond X_2) = \mathscr{C}_{12} \mathscr{C}_{22}^{\dag} X_2
\]
and
\[
\Cov(X_1 \cond X_2) = \mathscr{C}_{11} - \mathscr{C}_{12} \mathscr{C}_{22}^{\dag} \mathscr{C}_{21}, 
\]
where $\mathscr{C}_{22}^{\dag}$ is the generalised inverse (or Moore--Penrose inverse) of $\mathscr{C}_{22}$.
\end{proposition}
\begin{proof}
Define $Z := X_1 - \mathscr{C}_{12} \mathscr{C}_{22}^{\dag} X_2$. Note that since $(Z, X_2)$ is a bounded linear transformation of $(X_1, X_2)$, $(Z, X_2)$ must be jointly Gaussian by Proposition~\ref{prop:linear-transformation-gaussian}. By Proposition~\ref{prop:uncorrelated-independence-gaussian}, $Z$ and $X_2$ are independent if $\Cov(Z, X_2) = 0$. We calculate the covariance and get
\[
\Cov(Z, X_2) = \mathscr{C}_{12} - \mathscr{C}_{12} \mathscr{C}_{22}^{\dag}  \mathscr{C}_{22} = 0
\]
by \citet[][Theorem~3.5.8 (3.18)]{Hsing2015} since $\textrm{Ker}(\mathscr{C}_{22}) = {0}$.  This implies that the conditional distribution of $Z$ given $X_2$ is simply the distribution of $Z$. We can find the complete distribution of $Z$ by calculating the mean and covariance of $Z$, since $Z$ is Gaussian. We get by Proposition~\ref{prop:linear-transformation-gaussian},
\[
\mathbb{E}(Z) = \mathbb{E}(X_1) -  \mathscr{C}_{12} \mathscr{C}_{22}^{\dag} \mathbb{E}(X_2) = 0
\]
and
\[
\Cov(Z) = \mathscr{C}_{11} -  \mathscr{C}_{12} \mathscr{C}_{22}^{\dag} \mathscr{C}_{21} .
\]
By Proposition~\ref{prop:substitution-theorem-for-conditional-distributions}, since we can write $X_1 = Z + \mathscr{C}_{12} \mathscr{C}_{22}^{\dag}  X_2$, the conditional distribution of $X_1$ given $X_2$ is as desired.
\end{proof}

\section{Uniform convergence of random variables}
\label{app:uniform-convergence}
In this section we develop some background theory that will be useful when considering simultaneous convergence of sequences with varying distributions. In particular, we are interested the convergence of a sequence of random variables $(X_n)_{n \in \mathbb{N}}$ defined on a measurable space $(\Omega, \mathcal{F})$ with a family of probability measures $(\mathbb{P}_\theta)_{\theta \in \Theta}$. For each $\theta \in \Theta$ the distribution of $(X_n)_{n \in \mathbb{N}}$ will change as the background measure $\mathbb{P}_\theta$ changes. We are also interested in the convergence of $\theta$-dependent functions of $X_n$ such as the conditional expectation with respect to $\mathbb{P}_\theta$ of $X_n$ given a sub-$\sigma$-algebra $\mathcal{D}$ of $\mathcal{F}$. To allow for such considerations, the definitions given here will be more general than in Section~\ref{sec:theory} and will allow for a family of random variables $(X_{n, \theta})_{n \in \mathbb{N}, \theta \in \Theta}$ to converge to a family of random variables $(X_\theta)_{\theta \in \Theta}$.

The material in this section extends the work of \citet{Kasy2019} and \citet{Bengs2019} to Hilbertian and Banachian random variables and also adds further characterisations of their central assumptions for families of real-valued random variables.

Unless stated otherwise, we consider the following setup for the remainder of this section. Let $(\Omega, \mathcal{F})$ be a measurable space, $(\mathbb{P}_\theta)_{\theta \in \Theta}$ a family of probability measure on $(\Omega, \mathcal{F})$ where $\Theta$ is any set and $(\mathcal{B}, \mathbb{B}(\mathcal{B}))$ a separable Banach space with its Borel $\sigma$-algebra.  Let $(X_{n, \theta})_{n \in \mathbb{N}, \theta \in \Theta}$ and $(X_\theta)_{\theta \in \Theta}$ be families of random variables defined on $(\Omega, \mathcal{F})$ with values in $\mathcal{B}$. All additional random variables are also defined on $(\Omega, \mathcal{F})$. We write $\mathbb{E}_\theta$ for the expectation with respect to $\mathbb{P}_\theta$. 

\begin{definition}[Uniform convergence of random variables] 
\label{def:conv}
\begin{enumerate}[(i)]
\item We say that \emph{$X_{n, \theta}$ converges uniformly in distribution over $\Theta$ to $X_\theta$ and write $X_{n, \theta} \underset{\Theta}{\overset{\mathcal{D}}{\rightrightarrows}} X_\theta$} if
\[
\lim_{n \to \infty} \sup_{\theta \in \Theta} d_{{\BL}}^\theta(X_{n, \theta}, X_\theta) = 0,
\]
where 
\[
d_{\BL}^\theta(X_{n, \theta}, X_\theta) := \sup_{f \in {\BL}_1} \left| \mathbb{E}_\theta(f(X_{n, \theta})) -  \mathbb{E}_\theta(f(X_\theta)) \right| ,
\]
and ${\BL}_1$ denotes the set of all functions $f:\mathcal{B} \to [-1,1]$ that are Lipschitz with constant at most $1$. We write $X_{n, \theta} \overset{\mathcal{D}}{\rightrightarrows} X_\theta$ and simply say that $X_{n, \theta}$ converges uniformly in distribution to $X_\theta$ when $\Theta$ is clear from the context. When considering collections of random variables that do not depend on $\theta$ except through the measure on the domain of the random variables, we simply write $X_n \overset{\mathcal{D}}{\rightrightarrows} X$.
\item We say that \emph{$X_{n, \theta}$ converges uniformly in probability over $\Theta$ to $X_\theta$ and write \\
${X_{n, \theta} \underset{\Theta}{\overset{P}{\rightrightarrows}} X_\theta}$} if, for any $\epsilon > 0$,
\[
\lim_{n \to \infty} \sup_{\theta \in \Theta} \mathbb{P}_\theta( \lVert X_{n, \theta} - X_\theta \rVert \geq \epsilon) = 0.
\]
We write $X_{n, \theta} \overset{P}{\rightrightarrows} X_\theta$ and simply say that $X_{n, \theta}$ converges uniformly in probability to $X_\theta$ when $\Theta$ is clear from the context. When considering collections of random variables that do not depend on $\theta$ except through the measure on the domain of the random variables, we simply write $X_n \overset{P}{\rightrightarrows} X$.
\end{enumerate} 
\end{definition}
Using a slight abuse of notation, we write $X_{n, \theta} \overset{\mathcal{D}}{\rightrightarrows} 0$ and $X_{n, \theta} \overset{P}{\rightrightarrows} 0$ to mean that $X_{n, \theta}$ converges uniformly to the family of random variables $X_\theta$ that is equal to $0$ for all $\omega \in \Omega$ and any $\theta \in \Theta$. Note that if $(\mathbb{P}_\theta)_{\theta \in \Theta}$ contains a single element, we recover the standard definitions of convergence in distribution and probability. We have the following helpful characterisations of the two modes of uniform convergence.

\begin{proposition}
\label{prop:alternative-definition-uniform-convergence}
\begin{enumerate}[(i)]
\item $X_{n, \theta} \overset{\mathcal{D}}{\rightrightarrows} X_\theta$ if and only if for any sequence $(\theta_n)_{n \in \mathbb{N}} \subset \Theta$
\[
\lim_{n \to \infty} d_{\BL}^{\theta_n}(X_{n, \theta_n}, X_{\theta_n}) = 0.
\]
\item $X_{n, \theta} \overset{P}{\rightrightarrows} X_\theta $ if and only if for any sequence $(\theta_n)_{n \in \mathbb{N}} \subset \Theta$ and any $\varepsilon > 0$
\[
\lim_{n \to \infty} \mathbb{P}_{\theta_n}( \lVert X_{n, \theta_n} - X_{\theta_n} \rVert \geq \varepsilon) = 0.
\]
\end{enumerate}
\end{proposition}
\begin{proof}
The proof given in \citet[][Lemma~1]{Kasy2019} also works in the Banachian case. 
\end{proof}

In the remainder of this section we derive various properties of uniform convergence in probability and distribution that are analogous to the well-known properties of non-uniform convergence. In particular, we first consider a uniform version of the continuous mapping theorem which relies on stronger versions of continuity.

\begin{proposition}
\label{prop:uniform-continuous-mapping-theorem}
Let $\psi: \mathcal{B} \to \tilde{\mathcal{B}}$ where $\tilde{\mathcal{B}}$ is another separable Banach space.
\begin{enumerate}[(i)]
\item If $X_{n, \theta} \overset{\mathcal{D}}{\rightrightarrows} X_\theta$ and $\psi$ is Lipschitz-continuous then  $\psi(X_{n, \theta}) \overset{\mathcal{D}}{\rightrightarrows} \psi(X_\theta) $.
\item If $X_{n, \theta} \overset{P}{\rightrightarrows} X_\theta $ and $\psi$ is uniformly continuous then $\psi(X_{n, \theta} ) \overset{P}{\rightrightarrows} \psi( X_\theta )$.
\end{enumerate}
%\Rajen{At least in the real-valued case where $\psi(0)=0$ and $X_\theta=0$, only need continuity.} \Anton{Yes but we're not claiming any complete characterisation here. Proposition 10 deals with the case you're suggesting here (actually $\psi(0) = 0$ is not necessary). Looking at Proposition 10 again though, I'm a bit confused why I decided to include both X and Y as variables...}
\end{proposition}
\begin{proof}
The proof in \citet[][Theorem~1]{Kasy2019} also works in the Banachian case.
\end{proof}

In what follows we will investigate different alternative assumptions such that continuity of $\psi$ suffices. One such assumption is tightness of the family of pushforward measures $(X_\theta(\mathbb{P}_\theta))_{\theta \in \Theta}$.

\begin{definition} \label{def:tight}
Let $(\mu_\theta)_{\theta \in \Theta}$ be a family of probability measures on $\mathcal{B}$.
\begin{enumerate}[(i)]
\item $(\mu_\theta)_{\theta \in \Theta}$ is said to be \emph{tight} if for any $\varepsilon > 0$, there exists a compact set $K$ such that
$\sup_{\theta \in \Theta} \mu_\theta(K^c)  < \varepsilon$.
$(X_\theta)_{\theta \in \Theta}$ is said to be \emph{uniformly tight with respect to $\Theta$} if the family of pushforward measures $(X_\theta(\mathbb{P}_\theta))_{\theta \in \Theta}$ is tight. If $\Theta$ is clear from the context we simply say that $(X_\theta)_{\theta \in \Theta}$ is uniformly tight.
\item $(X_{n, \theta})_{n \in \mathbb{N}, \theta \in \Theta}$ is said to be \emph{sequentially tight with respect to $\Theta$} if for any sequence $(\theta_n)_{n \in \mathbb{N}} \subset \Theta$ the sequence of pushforward measures $(X_{n, \theta_n}(\mathbb{P}_{\theta_n}))_{n \in \mathbb{N}}$ is tight. If $\Theta$ is clear from the context we simply say that $(X_{n, \theta})_{n \in \mathbb{N}, \theta \in \Theta}$ is sequentially tight.
\item $(\mu_\theta)_{\theta \in \Theta}$ is said to be \emph{relatively compact} if for any sequence $(\theta_n)_{n \in \mathbb{N}}$ there exists a subsequence $(\theta_{k(n)})_{n \in \mathbb{N}}$, where $k: \mathbb{N} \to \mathbb{N}$ is strictly increasing, such that $\mu_{\theta_{k(n)}}$ converges weakly to some measure $\mu$, which is not necessarily in the family $(\mu_\theta)_{\theta \in \Theta}$. 
\end{enumerate}
\end{definition}

Prokhorov's theorem states that tightness implies relative compactness and that they are equivalent on separable and complete metric spaces; in this work, we therefore use the terms interchangeably since we only consider separable Banach and Hilbert spaces. With a uniform tightness assumption, we can perform continuous operations and preserve uniform convergence in probability just as in the non-uniform setting.

\begin{proposition}
\label{prop:uniform-convergence-in-probability-preserved-under-continuous-operations-when-tight}
Let $(X_{n, \theta})_{n \in \mathbb{N}, \theta \in \Theta}$ and $(X_\theta)_{\theta \in \Theta}$ be random variables taking values in $\mathcal{B}$. Assume that $X_{n, \theta} \overset{P}{\rightrightarrows} X_\theta $ and $X_\theta$ is uniformly tight. Then, for any continuous function $\psi: \mathcal{B} \to \tilde{\mathcal{B}}$, where $\tilde{\mathcal{B}}$ is another separable Banach space, we have
\[
\psi(X_{n, \theta}) \overset{P}{\rightrightarrows} \psi(X_\theta) .
\]
\end{proposition}
\begin{proof}
Let $\epsilon > 0$ be given. We need to show that 
\[
\sup_{\theta \in \Theta} \mathbb{P}_\theta(\lVert \psi(X_{n, \theta}) - \psi(X_\theta) \rVert \geq \epsilon) \to 0
\]
As $X_\theta$ is uniformly tight, for $\eta >0$ there exists a compact set $K$ such that 
\[
\sup_{\theta \in \Theta} \mathbb{P}_\theta(X_\theta \not \in K_X) < \eta/2 .
\]
By the Heine--Cantor theorem, $\psi$ is uniformly continuous on $K_X$, so there exists $\delta > 0$ such that $\lVert x-x' \rVert < \delta $ implies that $\lVert \psi(x) - \psi(x') \rVert < \epsilon$. We thus have
\begin{align*}
\sup_{\theta \in \Theta} \mathbb{P}_\theta(\lVert \psi(X_{n, \theta}) &- \psi(X_\theta) \rVert \geq \epsilon ) \leq \sup_{\theta \in \Theta} \mathbb{P}_\theta(X_\theta \not \in K) + \sup_{\theta \in \Theta} \mathbb{P}_\theta(\lVert X_{n, \theta} - X_\theta\rVert \geq \delta) .
\end{align*}
By assumption, we can choose $N$ sufficiently large such that for all $n \geq N$, the final term is less than $\eta/2$, resulting in the whole expression being less than $\eta$. As $\eta$ was arbitrary, this proves the result. 
\end{proof}

\citet{Bengs2019} make repeated use of an alternative assumption for many of their results for real-valued random variables.

\begin{definition}
A family of probability measures $(\mu_\theta)_{\theta \in \Theta}$ is \emph{uniformly absolutely continuous with respect to the measure $\mu$} if for any $\varepsilon > 0$, there exists $\delta > 0$ such that for any Borel set $B$
\[
\mu(B) < \delta \Longrightarrow \sup_{\theta \in \Theta} \mu_\theta(B) < \varepsilon  .
\]
A family of random variables $(X_\theta)_{\theta \in \Theta}$ is \emph{uniformly absolutely continuous over $\Theta$ with respect to the measure $\mu$} if the family of pushforward measures $(X_\theta(\mathbb{P}_\theta))_{\theta \in \Theta}$ is uniformly absolutely continuous with respect to $\mu$. When $\Theta$ is clear from the context we simply say that $X_\theta$ is uniformly absolutely continuous with respect to $\mu$.
\end{definition}
Uniform absolute continuity has previously been studied in other works such as the ones by \citet[][Section 5.6]{Bogachev2018} and \citet[][Chapter IX, Section 4]{Doob1994}. An intuitive view of uniform absolute continuity can be given when $\mu$ is a finite measure. In this case, we can define a pseudometric $d_\mu$ on the Borel sets with $d_\mu(A,B)=\mu(A \triangle B)$, where $A \triangle B$ is the symmetric difference. Uniform absolute continuity is then uniform $d_\mu$-continuity over $\theta$ of the collection of push-forward measures $(X_\theta(\mathbb{P}_\theta))_{\theta \in \Theta}$ viewed as mappings from the Borel sets into $\mathbb{R}$.

Another helpful perspective is in the case where for each $\theta$, $X_\theta$ has a density $f_\theta$ with respect to a common measure $\mu$. The following proposition shows that $X_\theta$ is uniformly absolutely continuous with respect to $\mu$ if and only if for each $\theta$, $X_\theta$ has a density $f_\theta$ with respect to $\mu$ and the family of densities is uniformly integrable. A convenient sufficient condition for uniform integrability is the existence of $r > 0$ such that $\sup_{\theta \in \Theta} \int f_\theta^{1+r} \, \mathrm{d}\mu < \infty$.

\begin{proposition}
\label{prop:uniform-absolute-continuity-and-uniform-integrability}
If $(X_\theta)_{\theta \in \Theta}$ is uniformly absolutely continuous with respect to $\mu$, then for each $\theta$ $X_\theta$ has a density $f_\theta$ with respect to $\mu$ and the family $(f_\theta)_{\theta \in \Theta}$ is uniformly integrable with respect to $\mu$. Conversely, if for each $\theta$ $X_\theta$ has a density $f_\theta$ with respect to $\mu$ and the family $(f_\theta)_{\theta \in \Theta}$ is uniformly integrable then $(X_\theta)_{\theta \in \Theta}$ is uniformly absolutely continuous with respect to $\mu$.
\end{proposition}
\begin{proof}
For the first statement, note that by the Radon--Nikodym theorem, we need to show that for each $\theta$, $\mu(B)=0$ implies that $\mathbb{P}_\theta(X_\theta \in B) = 0$ for every Borel measurable $B$. This is immediate from the assumption of uniform absolute continuity (by negation) and so is the uniform integrability of the family $(f_\theta)_{\theta \in \Theta}$.
The second statement follows immediately from the definitions of uniform integrability and uniform absolute continuity.
\end{proof}

In \citet{Bengs2019} uniform absolute continuity is 
assumed with respect to a probability measure. For uniformly tight Banachian random variables that are uniformly absolutely continuous with respect to a $\sigma$-finite measure $\mu$, we can show that the family is also uniformly absolutely continuous with respect to any $\sigma$-finite measure $\nu$ such that $\mu$ has a continuous density with respect to $\nu$. 

\begin{proposition}
\label{prop:uniformly-absolutely-continuous-under-domination}
Assume that $(X_\theta)_{\theta \in \Theta}$ is uniformly absolutely continuous with respect to some $\sigma$-finite measure $\mu$. If $\nu$ is another $\sigma$-finite measure dominating $\mu$ and there exists a continuous Radon-Nikodym derivative of $\mu$ with respect to $\nu$, then $X$ is uniformly absolutely continuous with respect to $\nu$. 
\end{proposition}
\begin{proof}
Let $\varepsilon > 0$ be given. Because $(X_\theta)_{\theta \in \Theta}$ is uniformly tight, we can choose a compact set $K$, such that
\[
\sup_{\theta \in \Theta} \mathbb{P}_\theta(X_\theta \not\in K) < \varepsilon/2.
\]
Then note that for any Borel measurable set $B$
\[
\sup_{\theta \in \Theta} \mathbb{P}_\theta(X_\theta \in B) < \varepsilon/2 + \sup_{\theta \in \Theta} \mathbb{P_\theta}(X_\theta \in B \cap K).
\]
We thus need to find $\delta$ so that $\nu(B \cap K) < \delta$ implies $ \sup_{\theta \in \Theta} \mathbb{P}_\theta(X_\theta \in B \cap K) < \varepsilon/2$. Letting $g$ denote the continuous Radon-Nikodym derivative of $\mu$ with respect to $\nu$, we see that
\[
\mu(B \cap K) = \int_{B \cap K} g \, \mathrm{d}\nu \leq \left(\sup_{x \in K} g(x) \right) \nu(B \cap K) .
\]
The supremum is finite by the extreme value theorem for continuous functions since $K$ is compact. If $\sup_{x \in K} g(x)>0$ choose $\delta'$ from the uniform absolute continuity of $X$ with respect to $\mu$ matching $\varepsilon/2$ and set $\delta = \delta' / (\sup_{x \in K} g(x))$. Then for all $B$ with $\nu(B) < \delta$, we have
\[
\delta > \nu(B) \geq \nu(B \cap K) \geq \frac{
\mu(B \cap K)}{\sup_{x \in K} g(x) } \Longrightarrow \mu(B \cap K) < \delta' 
\]
and thus 
\[
\sup_{\theta \in \Theta} \mathbb{P}_\theta(X_\theta \in B \cap K) < \varepsilon /2
\]
proving the result. If $\sup_{x \in K} g(x)=0$ any $\delta$ works since $\mu(B \cap K) = 0$ implies \\ ${\sup_{\theta \in \Theta} \mathbb{P}(X_\theta \in B \cap K) = 0}$.
\end{proof}

A consequence of the above result is that uniform absolute continuity with respect to the Lebesgue measure implies uniform absolute continuity with respect to the standard Gaussian measure. This lets us immediately apply many of the results of \citet{Bengs2019} such as Theorem~4.1, when we consider a uniformly tight real-valued random variable that is uniformly absolutely continuous with respect to the Lebesgue measure. 

\begin{corollary}
\label{corollary:uniform-absolute-continuity-wrt-lebesgue-iff-wrt-gaussian}
A real-valued family of random variables $(X_\theta)_{\theta \in \Theta}$ is uniformly absolutely continuous with respect to the Lebesgue measure if and only if it is uniformly absolutely continuous with respect to the standard Gaussian measure.
\end{corollary}
\begin{proof}
The statement follows immediately by the equivalence of the standard Gaussian measure and the Lebesgue measure, by the continuity of the Gaussian density and its reciprocal, and Proposition~\ref{prop:uniformly-absolutely-continuous-under-domination}.
\end{proof}

We will consider sums of real-valued random variables and thus need to consider when such sums are uniformly absolutely continuous with respect to a measure. It turns out that when the random variables are independent and one of the families is uniformly absolutely continuous with respect to the Lebesgue measure, the same is true for the family of sums.

\begin{theorem}
\label{thm:independent-convolution-preserves-uniform-absolute-continuity}
Let $(X_\theta)_{\theta \in \Theta}$ and $(Y_\theta)_{\theta \in \Theta}$ be two real-valued random variables such that for any $\theta \in \Theta$ $X_\theta$ and $Y_\theta$ are independent under $\mathbb{P}_\theta$. Assume that $(X_\theta)_{\theta \in \Theta}$ is uniformly absolutely continuous with respect to the Lebesgue measure. Then $(X_\theta+Y_\theta)_{\theta \in \Theta}$ is uniformly absolutely continuous with respect to the Lebesgue measure.
\end{theorem}
\begin{proof}
Let $\varepsilon > 0$ be given and let $\lambda$ denote the Lebesgue measure. We need to find $\delta > 0$ such that for any Borel measurable $B$ with $\lambda(B) < \delta$, we have $\sup_{\theta \in \Theta} \mathbb{P}_\theta(X_\theta + Y_\theta \in B) < \varepsilon$. We can use the independence of $X_\theta$ and $Y_\theta$ to write the probability as a double-integral with respect to the pushforward measures $X_\theta(\mathbb{P}_\theta)$ and $Y_\theta(\mathbb{P}_\theta)$ as follows:
\[
\mathbb{P}_\theta(X_\theta + Y_\theta \in B) = \int \ind_B(X_\theta(\omega)+Y_\theta(\omega)) \, \mathrm{d}\mathbb{P}_\theta(\omega) = \int \int  \ind_B(x+y) \, \mathrm{d} X_\theta(\mathbb{P}_{\theta} )(x)  \, \mathrm{d} Y_\theta(\mathbb{P}_\theta )(y). 
\]
Note that $\ind_B(x+y) = \ind_{B-y}(x)$ where $B-y := \{b-y \,:\, b \in B\}$ and that, by the translation invariance of the Lebesgue measure,  $\lambda(B) = \lambda(B-y)$.
As $X_\theta$ is uniformly absolutely continuous with respect to $\lambda$, there exists $\delta$ such that if $\lambda(B) < \delta$ we have
\begin{align*}
\sup_{\theta \in \Theta} \mathbb{P}_\theta(X_\theta + Y_\theta \in B)  &\leq \sup_{ \theta \in  \Theta} \int\left(  \sup_{\theta \in \Theta} \int  \ind_{B-y}(x) \, \mathrm{d} X_\theta(\mathbb{P}_\theta )(x) \right)  \, \mathrm{d} Y_\theta(\mathbb{P}_\theta )(y)\\
& <\sup_{ \theta \in  \Theta} \int \varepsilon  \, \mathrm{d} Y_\theta(\mathbb{P}_\theta )(y) < \varepsilon. & \qedhere
\end{align*}
\end{proof}

Thus far, we have not discussed when we can expect uniform convergence in distribution to imply uniform convergence of distribution functions. This is exactly where we need an assumption of uniform absolute continuity. The following result is a modified version of \citet[][Theorem~4.1]{Bengs2019}, where our condition includes uniform convergence in $x$, rather than convergence for all $x$.

\begin{proposition}
\label{prop:uniform-convergence-of-distribution-functions}
Let $(X_{n, \theta})_{n \in \mathbb{N}, \theta \in \Theta}$ and $(X_\theta)_{\theta \in \Theta}$ be real-valued random variables. Assume that $(X_\theta)_{\theta \in \Theta}$ is uniformly absolutely continuous with respect to a continuous probability measure $\mu$. Then $X_{n, \theta} \overset{\mathcal{D}}{\rightrightarrows} X_\theta $ if and only if 
\begin{equation} \label{eq:unif_conv_dist}
\lim_{n \to \infty} \sup_{x \in \mathbb{R}} \sup_{\theta \in \Theta} | \mathbb{P}_\theta(X_{n, \theta} \leq x) -  \mathbb{P}_\theta(X_\theta \leq x) |  = 0 .
\end{equation}
\end{proposition}
\begin{proof}
See \citet[][Theorem~4.1]{Bengs2019} for a proof that $X_{n, \theta} \overset{\mathcal{D}}{\rightrightarrows} X_\theta$ if and only if 
\[
\lim_{n \to \infty} \sup_{\theta \in \Theta} | \mathbb{P}_\theta(X_{n, \theta} \leq x) -  \mathbb{P}_\theta(X_\theta \leq x) |  = 0 
\]
for all $x \in \mathbb{R}$.
To show that the convergence of distribution functions is uniform, we proceed as follows. In view of the uniform absolute continuity of $(X_\theta)_{\theta \in \Theta}$ with respect to $\mu$, for all $\varepsilon > 0$ there exists $\delta > 0$ such that for Borel measurable $B$ with $\mu(B) < \delta$, we have
$\sup_{\theta \in \Theta} \mathbb{P}_\theta(X_\theta \in B) < \varepsilon$.
Let $-\infty = x_0 < x_1 < \dots < x_m = \infty$ such that for all $i \in \{1, \dots , m\}$, $0 < \mu((x_{i-1}, x_i]) < \delta$. We can find such a grid since $\mu$ is a continuous probability measure. For any $\theta$ and $i \in \{1, \dots , m\}$, we thus have
\[
\mathbb{P}_\theta(X_\theta \leq x_i) - \mathbb{P}_\theta(X_\theta\leq x_{i-1})  = \mathbb{P}_\theta(X_\theta \in (x_{i-1}, x_i]) < \varepsilon.
\]
For $x \in (x_{i-1}, x_i]$,
\begin{align*}
& \sup_{\theta \in \Theta} \{ \mathbb{P}_\theta (X_{n, \theta} \leq x) -  \mathbb{P}_\theta(X_\theta \leq x)\} \leq \sup_{\theta \in \Theta} \{\mathbb{P}_\theta(X_{n, \theta} \leq x_i) -  \mathbb{P}_\theta(X_\theta \leq x_{i-1})\} \\
 &\leq \sup_{\theta \in \Theta} \{\mathbb{P}_\theta(X_{n, \theta} \leq x_i) -  \mathbb{P}_\theta(X_\theta \leq x_i)\}  + \varepsilon \leq  \sup_{\theta \in \Theta} | \mathbb{P}_\theta(X_{n, \theta} \leq x_i) -  \mathbb{P}_\theta(X_\theta \leq x_i)|  + \varepsilon,
\end{align*}
and, similarly,
\begin{align*}
& \sup_{\theta \in \Theta}  \{\mathbb{P}_\theta(X_\theta \leq x) - \mathbb{P}_\theta(X_{n, \theta} \leq x)\}  \leq \sup_{\theta \in \Theta}  \{\mathbb{P}_\theta(X_\theta \leq x_i) - \mathbb{P}_\theta(X_{n, \theta} \leq x_{i-1})\} \\
 &\leq \sup_{\theta \in \Theta}  \{\mathbb{P}_\theta(X_\theta \leq x_{i-1}) - \mathbb{P}_\theta(X_{n, \theta} \leq x_{i-1}) \}  + \varepsilon \leq  \sup_{\theta \in \Theta} | \mathbb{P}_\theta(X_\theta \leq x_{i-1}) - \mathbb{P}_\theta(X_{n, \theta} \leq x_{i-1})|  + \varepsilon.
\end{align*}
Thus,
\[
 \sup_{x \in \mathbb{R}} \sup_{\theta \in \Theta} | \mathbb{P}_\theta(X_{n, \theta} \leq x) - \mathbb{P}_\theta(X_\theta \leq x) | \leq  \sup_{i \in \{0, \dots, m\}} \sup_{\theta \in \Theta} | \mathbb{P}_\theta(X_{n, \theta} \leq x_i) -  \mathbb{P}_\theta(X_\theta \leq x_i) | + \varepsilon.
\]
The first term on the right-hand side goes to $0$ by assumption and $\varepsilon$ was arbitrary, thus proving the uniform convergence. 
\end{proof}

The final results of this section are uniform versions of Slutsky's lemma, the Weak Law of Large Numbers and the Central Limit Theorem. In the remaining results uniform tightness will play a crucial role. It is a standard result that if $(X_n)_{n \in \mathbb{N}}$ converges in distribution to $X$ then $(X_n)_{n \in \mathbb{N}}$ is tight. We can show that analogously if $X_{n, \theta} \overset{\mathcal{D}}{\rightrightarrows} X_\theta$ and $(X_\theta)_{\theta \in \Theta}$ is uniformly tight then $(X_{n, \theta})_{n \in \mathbb{N}, \theta \in \Theta}$ is sequentially tight.

\begin{proposition}
\label{prop:uniform-convergence-in-distribution-implies-sequential-tightness}
Assume that $(X_\theta)_{\theta \in \Theta}$ is uniformly tight. If $X_{n, \theta} \overset{\mathcal{D}}{\rightrightarrows} X_\theta$ then $(X_{n, \theta})_{n \in \mathbb{N}, \theta \in \Theta}$ is sequentially tight.
\end{proposition}
\begin{proof}
We prove the contrapositive statement. Assume that there exists a sequence $(\theta_n)_{n \in \mathbb{N}} \subset \Theta$ such that $(X_{n, \theta_n}(\mathbb{P}_{\theta_n}))_{n \in \mathbb{N}}$ is not tight. Let $Y_n$ be distributed as $X_{n, \theta_n}(\mathbb{P}_{\theta_n})$ and $Z_n$ distributed as $X(\mathbb{P}_{\theta_n})$ defined on a probability space $(\Omega, \mathcal{F}, \mathbb{P})$. Since $(Y_n)_{n \in \mathbb{N}}$ is not tight, there exists a subsequence $(k(n))_{n \in \mathbb{N}}$ with $k: \mathbb{N} \to \mathbb{N}$ strictly increasing such that any further subsequence of $(Y_{k(n)})_{n \in \mathbb{N}}$ does not
converge in distribution. Since $(Z_n)_{n \in \mathbb{N}}$ is tight, there exists a strictly increasing $k': \mathbb{N} \to \mathbb{N}$ and a random variable $Z$ such that writing $m = k \circ k'$, we have 
\[
d_{\BL}(Z_{m(n)}, Z) \to 0.
\]
However, since $Y_{k(n)}$ does not have a weakly convergent subsequence, we have
\[
d_{\BL}(Y_{m(n)},Z) \not\to 0.
\]
Thus, there exists $\varepsilon > 0$ and a strictly increasing $k'': \mathbb{N} \to \mathbb{N}$ such that writing $l=m \circ k''$, we have for all $n$
\[
d_{\BL}(Y_{l(n)}, Z)  \geq \varepsilon.
\]
Next choose $N$ such that for $n \geq N$ we have
\[
d_{\BL}(Z_{l(n)}, Z) < \varepsilon/2.
\]
Then by the reverse triangle inequality
\[
d_{\BL}(Z_{l(n)}, Y_{l(n)}) \geq \left| d_{\BL}(Z_{l(n)}, Z) - d_{\BL}(Z, Y_{l(n)}) \right| \geq \varepsilon/2
\]
for all $n \geq N$. Since 
\[
d_{\BL}(Z_{l(n)}, Y_{l(n)})  = d_{\BL}^{\theta_{l(n)}}\left( X_{l(n), \theta_{l(n)}}, X_{\theta_{l(n)}} \right)
\]
by Proposition~\ref{prop:alternative-definition-uniform-convergence} we cannot have $X_{n, \theta} \overset{\mathcal{D}}{\rightrightarrows} X_\theta$ proving the desired statement.
\end{proof}

The previous result will be required when proving the second part of the upcoming uniform version of Slutsky's lemma.

\begin{proposition}[Uniform Slutsky's lemma]
\label{prop:uniform-slutsky}
Let $(X_{n, \theta})_{n \in \mathbb{N}, \theta \in \Theta}$, $(Y_{n, \theta})_{n \in \mathbb{N}, \theta \in \Theta}$ and $(X_\theta)_{\theta \in \Theta}$ be Banachian random variables. Assume that $X_{n, \theta} \overset{\mathcal{D}}{\rightrightarrows} X_\theta $ and $Y_{n, \theta} \overset{P}{\rightrightarrows} 0 $. Then, the following two statements hold.
\begin{enumerate}[(i)]
\item $X_{n, \theta} + Y_{n, \theta}   \overset{\mathcal{D}}{\rightrightarrows} X_\theta$.
\item If $(Y_{n, \theta})_{n \in \mathbb{N}, \theta \in \Theta}$ is a family of real-valued random variables and $(X_\theta)_{\theta \in \Theta}$ is uniformly tight, then $Y_{n, \theta} X_{n, \theta} \overset{P}{\rightrightarrows} 0$. 
\end{enumerate}
\end{proposition}

\begin{proof}
We first prove (i), for which we need to show that 
\[
\sup_{\theta \in \Theta} d_{\BL}^\theta(X_{n, \theta} +Y_{n, \theta}, X_\theta) \to 0
\]
as $n \to \infty$. We have for any $\theta$
\begin{align*}
d_{\BL}^\theta(X_{n, \theta} +Y_{n, \theta}, X_\theta) \leq d_{\BL}^\theta(X_{n, \theta} +Y_{n, \theta}, X_{n,\theta}) +  d_{\BL}^\theta(X_{n, \theta}, X_\theta),
\end{align*}
where the second term goes to $0$ uniformly by assumption. It remains to show that the first term goes to $0$ uniformly. Now for $f \in {\BL}_1$ we have that for any $\varepsilon > 0$ and any $x, y \in \mathcal{B}$, $\lVert y \rVert < \varepsilon$ implies $\lVert f(x+y) - f(x) \rVert \leq \varepsilon$. Hence, by using the triangle inequality for the expectation, partitioning the integral and using the uniform continuity above, we get 
\begin{align*}
d_{\BL}^\theta(X_{n, \theta} +Y_{n, \theta}, X_{n,\theta})  \leq \varepsilon +  \sup_{f \in {\BL}_1}  \mathbb{E}_\theta \left| [f(X_{n, \theta}+Y_{n, \theta}) -  f(X_{n, \theta})] \ind_{\{\lVert Y_{n, \theta} \rVert > \varepsilon\}} \right| .
\end{align*}
We can again apply the triangle inequality and recall that $f$ is bounded by $1$, yielding 
\[
\sup_{\theta \in \Theta} \sup_{f \in {\BL}_1}  \mathbb{E}_\theta \left| [f(X_{n, \theta}+Y_{n, \theta}) -  f(X_{n, \theta})]\ind_{\{\lVert Y_{n, \theta} \rVert > \varepsilon\}} \right| \leq 2 \sup_{\theta \in \Theta}  \mathbb{P}_\theta(\lVert Y_{n, \theta} \rVert > \varepsilon),
\]
which goes to $0$ by assumption. Since $\varepsilon >0$ was arbitrary, we have proven the desired result.

We now turn to the proof of (ii). We will apply Proposition~\ref{prop:alternative-definition-uniform-convergence} and show that for any $(\theta_n)_{n \in \mathbb{N}} \subseteq \Theta$ and any $\varepsilon > 0$,
\[
\mathbb{P}_{\theta_n}(\lVert  Y_{n, \theta_n} X_{n, \theta_n} \rVert \geq \varepsilon)  \to 0
\]
as $n \to \infty$, which implies the desired result. Let $\delta > 0$ be given. By Proposition~\ref{prop:uniform-convergence-in-distribution-implies-sequential-tightness} there exists a compact set $K$ such that 
\[
\sup_{n \in \mathbb{N}} \mathbb{P}_{\theta_n}( X_{n, \theta_n} \in K^c) \leq \delta/2 .
\]
Since $K$ is compact, it is bounded and thus there exists $M > 0$ such that $\|x\| < M$ for all $x \in K$.
By the uniform convergence in probability of $Y_n$ to zero, we can find $N$ such that for all $n \geq N$,
\[
\mathbb{P}_{\theta_n}(|Y_{n, \theta_n}| \geq \varepsilon/M)  < \delta/2 .
\]
Putting things together, we get, for all $n \geq N$,
\begin{align*}
\mathbb{P}_{\theta_n}(\lVert X_{n, \theta_n} Y_{n, \theta_n} \rVert \geq \varepsilon) 
&
\leq 
\mathbb{P}_{\theta_n}(\lVert X_{n, \theta_n} Y_{n, \theta_n} \rVert \geq \varepsilon, X_{n, \theta_n} \in K)
+
\mathbb{P}_{\theta_n}(X_{n, \theta_n} \in K^c) 
\\
&\leq \mathbb{P}_{\theta_n}(|Y_{n, \theta_n}| \geq \varepsilon/M) +  \sup_{n \in \mathbb{N}} \mathbb{P}_{\theta_n}( X_{n, \theta_n} \in K^c) < \delta,
\end{align*}
proving the result.
\end{proof}

We will now consider the setting of uniform convergence of averages of i.i.d.\ random variables, i.e.\ we assume that for each $\theta \in \Theta$ the sequence $(X_{n, \theta})_{n \in \mathbb{N}}$ is i.i.d.\ and consider the convergence of $1/n \sum_{i=1}^n X_{i, \theta}$. We first prove a small technical lemma and then apply this lemma to prove an analogue of the Law of Large numbers for uniform convergence in probability for Hilbertian random variables. 

\begin{lemma}
\label{lem:hilbertian-expected-squared-norm}
Let $Y_1, \dots, Y_n$ be independent, mean zero random variables taking values in Hilbert space $\mathcal{H}$. Then
\[
\mathbb{E}\left( \left\lVert \sum_{i=1}^n Y_i \right\rVert^2 \right) = \sum_{i=1}^n \mathbb{E}\lVert Y_i \rVert^2 .
\]

\begin{proof}
Note first that
\[
\left\lVert \sum_{i=1}^n Y_i \right\rVert^2  =  \sum_{i=1}^n \sum_{j=1}^n \langle Y_i , Y_j \rangle .
\]
Let $(e_k)_{k \in \mathbb{N}}$ denote a basis of $\mathcal{H}$. Then for $i \neq j$
\[
\mathbb{E}(\langle Y_i , Y_j \rangle) = \mathbb{E}\left( \sum_{k=1}^\infty \langle Y_i , e_k \rangle \langle Y_j, e_k \rangle  \right) =  \sum_{k=1}^\infty \mathbb{E}\left( \langle Y_i , e_k \rangle \langle Y_j, e_k \rangle  \right) = \sum_{k=1}^\infty \mathbb{E}(\langle Y_i , e_k \rangle ) \mathbb{E}(\langle Y_j, e_k \rangle )
\]
but $\mathbb{E}(\langle Y_i, e_k \rangle) = 0$ for all $i$ and $k$ since $Y_i$ are mean zero. 
\end{proof}
\end{lemma}

\begin{proposition}
\label{prop:uniform-hilbert-lln}
Let $(X_\theta)_{\theta \in \Theta}$ be Hilbertian random variables with $\mathbb{E}_\theta(X_\theta) = 0$ for all $\theta \in \Theta$ and $\sup_ {\theta \in \Theta} \mathbb{E}_\theta(\lVert X_\theta \rVert^{1+\eta}) < C$ for some $C, \eta > 0$. Let $(X_{n, \theta})_{n \in \mathbb{N}, \theta \in \Theta}$ be random variables such that for $\theta \in \Theta$ $(X_{n, \theta})_{n \in \mathbb{N}}$ is i.i.d.\ with the same distribution as $X_\theta$ under $\mathbb{P}_\theta$. Then 
\[
\frac{1}{n} \sum_{i=1}^n X_{i, \theta} \overset{P}{\rightrightarrows} 0.
\]
\end{proposition}
\begin{proof}
We adapt the argument given in \citet[][Lemma~19]{GCM}. Defining $S_{n, \theta} := n^{-1} \sum_{i=1}^n X_{i, \theta}$, we need to show that for any $\varepsilon > 0$,
\[
\sup_{\theta \in \Theta} \mathbb{P}_\theta(\lVert S_{n, \theta} \rVert \geq \varepsilon) \to 0
\]
as $n \to \infty$. To this end, we let $M>0$ and define $X_\theta^< := \ind_{\{\lVert X_\theta \rVert \leq M\}} X_\theta$ and $X_\theta^> := \ind_{\{\lVert X_\theta \rVert > M\}} X_\theta$ and similarly $X_{i, \theta}^<$ and $X_{i, \theta}^>$ for $i \in \mathbb{N}$. We also define $S_{n, \theta}^< := n^{-1} \sum_{i=1}^n X_{i, \theta}^<$ and $S_{n, \theta}^> := n^{-1} \sum_{i=1}^n X_{i, \theta}^>$.
Note first that
\[
\sup_{\theta \in \Theta} \mathbb{P}_\theta(\lVert X_\theta \rVert > M) \leq \frac{\sup_{\theta \in \Theta} \mathbb{E}_\theta \lVert X_\theta \rVert}{M} \leq \frac{C}{M},  
\]
hence choosing $M$ large, we can make $\mathbb{P}_\theta(\lVert X_\theta \rVert > M)$ small uniformly in $\theta$. Combining this with the fact that $\mathbb{E}(X_\theta^<) = -\mathbb{E}(X_\theta^>)$, we get
\begin{equation}
  \label{eq:holder-lln-eq}
  \sup_{\theta \in \Theta} \lVert \mathbb{E}(X_\theta^<) \rVert =   \sup_{\theta \in \Theta}  \lVert \mathbb{E}(X_\theta^>) \rVert \leq \sup_{\theta \in \Theta}  \mathbb{E}\lVert X_\theta^> \rVert \leq \sup_{\theta \in \Theta}  \left(\mathbb{E}\lVert X_\theta \rVert^{1+\eta}\right)^{\frac{1}{1+\eta}} \mathbb{P}_\theta(\lVert X_\theta \rVert > M)^{\frac{\eta}{1+\eta}} \leq \frac{C^2}{M},
\end{equation}
by Hölder's inequality. This implies that choosing $M$ large we can ensure that $\sup_{\theta \in \Theta} \lVert \mathbb{E}(X_\theta^<) \rVert < \varepsilon/3$ and for these $M$, we have
\begin{align*}
  \sup_{\theta \in \Theta} \mathbb{P}_\theta(\lVert S_{n, \theta} \rVert > \varepsilon) &\leq  \sup_{\theta \in \Theta} \mathbb{P}_\theta(\lVert S_{n, \theta}^< \rVert > 2\varepsilon/3) +  \sup_{\theta \in \Theta} \mathbb{P}_\theta(\lVert S_{n, \theta}^> \rVert > \varepsilon/3) \\
  &\leq \sup_{\theta \in \Theta} \mathbb{P}_\theta(\lVert S_{n, \theta}^< - \mathbb{E}(X_\theta^<) \rVert > \varepsilon/3) +  \sup_{\theta \in \Theta} \mathbb{P}_\theta(\lVert S_{n, \theta}^> \rVert > \varepsilon/3) .
\end{align*}
By Markov's inequality and the triangle inequality
\[
  \sup_{\theta \in \Theta} \mathbb{P}_\theta(\lVert S_{n, \theta}^> \rVert > \varepsilon/3) \leq   \frac{3 \sup_{\theta \in \Theta} \mathbb{E}_\theta \lVert X_\theta^> \rVert}{\varepsilon} 
\]
which we have already shown in \eqref{eq:holder-lln-eq} is uniformly small when $M$ is sufficiently large. Finally, by Markov's inequality, the triangle inequality and Lemma \ref{lem:hilbertian-expected-squared-norm}, we have 
\[
  \sup_{\theta \in \Theta} \mathbb{P}_\theta(\lVert S_{n, \theta}^< - \mathbb{E}(X_\theta^<) \rVert > \varepsilon/3) \leq \frac{\sup_{\theta \in \Theta} \textrm{E}_\theta\lVert S_{n, \theta}^< - \mathbb{E}(X_\theta^<) \rVert^2}{t^2} = \frac{\sup_{\theta \in \Theta} \mathbb{E}_\theta \lVert X_\theta^< \rVert^2}{n t^2} \leq \frac{M^2}{n t^2}
\]
hence choosing $n$ sufficiently large, we can control the final term.
\end{proof}

We can extend the previous result to a special class of Banach spaces under an additional tightness assumption. Recall that a Banach space $\mathcal{B}$ has a \emph{Schauder basis} if there exists $(e_k)_{k \in \mathbb{N}}$ such that for every $v \in \mathcal{B}$ there exists a unique sequence of scalars $(\alpha_k)_{k \in \mathbb{N}}$ satisfying
\[
  \left\lVert v - \sum_{k=1}^K \alpha_k e_k \right\rVert \to 0
\]
as $K \to \infty$.

\begin{proposition}
\label{prop:uniform-lln}
Let $(X_\theta)_{\theta \in \Theta}$ be Banachian random variables taking values in $\mathcal{B}$ with $\mathbb{E}_\theta(X_\theta) = 0$ for all $\theta \in \Theta$ and $\sup_ {\theta \in \Theta} \mathbb{E}_\theta(\lVert X_\theta \rVert^{1+\eta}) < C$ for some $C, \eta > 0$. Let $(X_{n, \theta})_{n \in \mathbb{N}, \theta \in \Theta}$ be random variables such that for $\theta \in \Theta$ $(X_{n, \theta})_{n \in \mathbb{N}}$ is i.i.d.\ with the same distribution as $X_\theta$ under $\mathbb{P}_\theta$. Assume further that $\mathcal{B}$ has a Schauder basis and that $(X_\theta)_{\theta \in \Theta}$ is uniformly tight.
Then 
\[
\frac{1}{n} \sum_{i=1}^n X_{i, \theta} \overset{P}{\rightrightarrows} 0.
\]
\end{proposition}
\begin{proof}
For $K \in \mathbb{N}$ let $P_K$ denote the canonical projection of $v \in \mathcal{B}$ onto the first $K$ components of the Schauder basis, i.e. the mapping 
\[
v = \sum_{k=1}^\infty \alpha_k e_k \mapsto \sum_{k=1}^K \alpha_k e_k   .
\] 
This mapping is linear and satisfies that $\sup_{K \in \mathbb{N}} \lVert P_K \rVert_{\op} < \infty$ by \citet[][Theorem II.2 and II.3]{Li2017}.  By the triangle inequality
\[
\mathbb{P}_\theta \left( \left\lVert \frac{1}{n} \sum_{i=1}^n X_{i, \theta} \right\rVert \geq \varepsilon \right) \leq  \mathbb{P}_\theta \left(\left\lVert \frac{1}{n} \sum_{i=1}^n P_K X_{i, \theta} \right\rVert \geq \varepsilon/2 \right) + \mathbb{P}_\theta \left( \left\lVert \frac{1}{n} \sum_{i=1}^n (X_{i, \theta} -P_K X_{i, \theta}) \right\rVert \geq \varepsilon/2 \right) ,
\]
hence it is sufficient to show that the first term converges to $0$ uniformly as $n \to \infty$ for fixed $K$ and that the second term converges to $0$ uniformly as $K \to \infty$. By Proposition \ref{prop:uniform-hilbert-lln} the first term converges to $0$ for fixed $K$ since $(P_K X_\theta)_{\theta \in \Theta}$ are concentrated on a finite-dimensional subspace of $\mathcal{B}$ and since 
\[
  \sup_{\theta \in \Theta} \mathbb{E}_\theta \left( \lVert P_K X_\theta \rVert^{1+\eta} \right) \leq  \lVert P_K\rVert_{\op}^{1+\eta} \sup_{\theta \in \Theta} \mathbb{E}_\theta \left( \lVert  X_\theta \rVert^{1+\eta} \right) < \infty.
\]
It remains to show that when we choose $K$ large, the second term is small. \citet[][Theorem 2.7.10]{Bogachev2018} characterises tightness of families of random variables on Banach spaces with a Schauder basis. In particular, they satisfy
\begin{equation}
\label{eq:banach-tightness}
\lim_{K \to \infty} \sup_{\theta \in \Theta} \mathbb{P}_\theta(\lVert X_\theta - P_K X_\theta \rVert > \varepsilon )  = 0
\end{equation}
for every $\varepsilon > 0$. Applying Markov's inequality, partitioning the integral, applying Hölder's inequality and the triangle inequality yields that for any $t > 0$ and $\delta > 0$,
\begin{align*}
&\sup_{\theta \in \Theta} \mathbb{P}_\theta \left( \left\lVert \frac{1}{n} \sum_{i=1}^n (X_{i, \theta} -P_K X_{i, \theta}) \right\rVert \geq t \right) \leq \frac{1}{t} \sup_{\theta \in \Theta}  \mathbb{E}_\theta\lVert X_\theta - P_K X_\theta \rVert \\
& \leq \frac{1}{t} \sup_{\theta \in \Theta}  \left\{ \delta + \mathbb{E}_\theta \left( \lVert X_\theta - P_K X_\theta \rVert \ind_{\{\lVert X_\theta - P_K X_\theta \rVert > \delta\}} \right) \right\} \\
&\leq \frac{1}{t} \sup_{\theta \in \Theta}  \left\{ \delta + \left(\mathbb{E}_\theta \lVert X_\theta - P_K X_\theta \rVert^{1+\eta} \right)^{\frac{1}{1+\eta}} \left(\mathbb{P}_\theta(\lVert X_\theta - P_K X_\theta \rVert > \delta) \right)^{\frac{\eta}{1+\eta}}\right\} \\
&\leq  \frac{1}{t} \left\{\delta + (1+ \sup_{K \in \mathbb{N}} \lVert P_K \rVert_{\op}) C \sup_{\theta \in \Theta} \left(\mathbb{P}_\theta(\lVert X_\theta - P_K X_\theta \rVert > \delta) \right)^{\frac{\eta}{1+\eta}}\right\} .
\end{align*}
By \eqref{eq:banach-tightness}, we can choose $\delta$ and $K$ such that the upper bound is arbitrarily small, hence we have shown the desired result.
\end{proof}

For the uniform central limit theorem, we only consider the Hilbertian case since this is sufficient for our needs and avoids technical problems to do with tightness and the regular (non-uniform) central limit theorem on Banach spaces. We first give some sufficient conditions for uniform convergence in distribution of Hilbertian random variables.

\begin{proposition}
\label{prop:conditions-for-uniform-convergence-in-distribution}
Let $(X_{n, \theta})_{n \in \mathbb{N}, \theta \in \Theta}$ and $(X_\theta)_{\theta \in \Theta}$ be Hilbertian random variables. Assume that
\begin{enumerate}[(i)]
\item for all $h \in \mathcal{H}$, $\langle X_{n, \theta}, h \rangle \overset{\mathcal{D}}{\rightrightarrows} \langle X_\theta , h \rangle$,
\item $(X_{n, \theta})_{n \in \mathbb{N}, \theta \in \Theta}$ is sequentially tight, and
\item $(X_\theta)_{\theta \in \Theta}$ is uniformly tight.
\end{enumerate}
Then, $X_{n, \theta} \overset{\mathcal{D}}{\rightrightarrows}  X_\theta$.
\end{proposition}
\begin{proof}
Let $(\theta_n)_{n \in \mathbb{N}} \subseteq \Theta$ and let $Y_n$ have distribution $X_{n, \theta}(\mathbb{P}_{\theta_n})$ and $Z_n$ have distribution $X_\theta(\mathbb{P}_{\theta_n})$ defined on a probability space $(\Omega, \mathcal{F}, \mathbb{P})$. Suppose for contradiction that 
\[
d_{\BL}^{\theta_n}(X_{n, \theta_n}, X_{\theta_n})  =  d_{\BL}(Y_n, Z_n)  \not \to 0
\]
as $n \to \infty$. Then there exists a subsequence of $Y_n$ and $Z_n$ and an $\varepsilon > 0$ such that for all $n$
\[
d_{\BL}(Y_{k(n)}, Z_{k(n)})\geq \varepsilon, 
\]
where $k: \mathbb{N} \to \mathbb{N}$ is a strictly increasing function. By sequential tightness of $(X_{n, \theta})_{n \in \mathbb{N}, \theta \in \Theta}$, there exists a subsequence of $(Y_{k(n)})_{n \in \mathbb{N}}$, represented by the index function $m = k \circ k'$ for a strictly increasing $k': \mathbb{N} \to \mathbb{N}$ such that the subsequence $(Y_{m(n)})_{n \in \mathbb{N}}$ converges weakly to some random variable $Y$. By uniform tightness of $X$ there exists a further subsequence of $(Z_{m(n)})_{n \in \mathbb{N}}$, represented by the index function $l = m \circ k''$ for a strictly increasing $k'': \mathbb{N} \to \mathbb{N}$ such that $(Z_{l(n)})_{n \in \mathbb{N}}$ converges weakly to some random variable $Z$. Note that since the range of $l$ is a subset of the range of $m$, $(Y_{l(n)})_{n \in \mathbb{N}}$ also converges to $Y$. 

We intend to show that the distributions of $Z$ and $Y$ are equal. The distribution of a Hilbertian random variable is completely determined by the distribution of the linear functionals \citep[][Theorem~7.1.2]{Hsing2015}. However, for any $h \in \mathcal{H}$ and any $n$,
\begin{align*}
d_{\BL}\left(\langle Y, h \rangle, \langle Z, h \rangle \right) \leq d_{\BL}\left(\langle Y, h \rangle, \langle Y_{l(n)}, h \rangle \right)  + d_{\BL}\left(\langle Y_{l(n)}, h \rangle, \langle Z_{l(n)}, h \rangle \right) + d_{\BL}\left(\langle Z_{l(n)}, h \rangle, \langle Z, h \rangle \right) .
\end{align*}

The first and third term of the right-hand side go to zero by definition and the middle term goes to zero by assumption (i). Now, 
\[
d_{\BL}(Y_{l(n)}, Z_{l(n)}) \leq d_{\BL}(Y_{l(n)},Z) + d_{\BL}(Z, Z_{l(n)}).
\]
Hence, we can choose $N$ making $l(N)$ large enough that the RHS is smaller than $\varepsilon/2$. This is a contradiction since we chose $k$ such that $ d_{\BL}(Y_{k(n)}, Z_{k(n)}) \geq \varepsilon$ for all $n \in \mathbb{N}$ but $(l(n))_{n \in \mathbb{N}} \subseteq (k(n))_{n \in \mathbb{N}}$. 
\end{proof}

We can now prove a uniform central limit theorem in Hilbert spaces.

\begin{proposition}
\label{prop:uniform-clt}
Let $(X_\theta)_{\theta \in \Theta}$ be Hilbertian random variables with $\mathbb{E}_\theta( X_\theta) =0$ for all $\theta$ and $\sup_ {\theta \in \Theta} \mathbb{E}_\theta(\lVert X_\theta \rVert^{2+\eta}) < K$ for some $K, \eta > 0$.  Denote $(\mathscr{C}_\theta)_{\theta \in \Theta}$ the family of covariance operators of $X_\theta$ under each $\mathbb{P}_\theta$, i.e.\ $\mathscr{C}_\theta = \mathbb{E}_\theta(X_\theta \otimes X_\theta)$. Let $(X_{n, \theta})_{n \in \mathbb{N}, \theta \in \Theta}$ be random variables such that for $\theta \in \Theta$ $(X_{n, \theta})_{n \in \mathbb{N}}$ is i.i.d.\ with the same distribution as $X_\theta$ under $\mathbb{P}_\theta$. Assume further that for some orthonormal basis $(e_k)_{k=1}^\infty$ of $\mathcal{H}$ 
\begin{equation}
\label{eq:uniform-covariance-condition}
\lim_{K \to \infty} \sup_{\theta \in \Theta} \sum_{k=K}^\infty \langle \mathscr{C}_\theta e_k, e_k \rangle = 0.
\end{equation}
Then 
\[
\frac{1}{\sqrt{n}} \sum_{i=1}^n X_{i, \theta} \overset{\mathcal{D}}{\rightrightarrows}  Z
\]
where the distribution of $Z$ under $\mathbb{P}_\theta$ is $\mathcal{N}(0, \mathscr{C}_\theta)$.
\end{proposition}
\begin{proof}
We intend to apply Proposition~\ref{prop:conditions-for-uniform-convergence-in-distribution} and thus check the conditions. For the first condition, let $h \in \mathcal{H}$ be given and let $Y_n = \langle X_n, h \rangle$ and let $Y$ be distributed as $\langle \mathcal{N}(0, \mathscr{C}_\theta), h \rangle$ under $\mathbb{P}_\theta$, i.e.\ as $\mathcal{N}(0, \langle \mathscr{C}_\theta h , h \rangle)$. Note that
\[
\left\langle \frac{1}{\sqrt{n}} \sum_{i=1}^n X_{i, \theta}, h \right\rangle = \frac{1}{\sqrt{n}} \sum_{i=1}^n Y_{i, \theta}
\]
hence by Proposition~\ref{prop:alternative-definition-uniform-convergence} it is sufficient for the first condition that for any $(\theta_n)_{n \in \mathbb{N}} \subset \Theta$
\[
\lim_{n \to \infty} d_{\BL}^{\theta_n}(Y_n, Y) = 0.
\]
Suppose for contradiction that there exists a sequence $(\theta_n)_{n \in \mathbb{N}}$ such that the limit does not equal $0$. Then there exists an $\varepsilon >0$ and a strictly increasing function $m: \mathbb{N} \to \mathbb{N}$ such that
\[
  d_{\BL}^{\theta_{m(n)}}(Y_n, Y)  \geq \varepsilon
\]
for any $n \in \mathbb{N}$. Denoting for $n \in \mathbb{N}$, $\sigma^2_{\theta_{m(n)}} := \langle \mathscr{C}_{\theta_{m(n)}} h, h \rangle$, we note that the sequence $\left(\sigma^2_{\theta_{m(n)}}\right)_{n \in \mathbb{N}}$ is bounded by assumption and hence by the Bolzano--Weierstrass theorem it has a convergent subsequence, i.e.\ there exists $\sigma^2 \geq 0$ and a strictly increasing $m' : \mathbb{N} \to \mathbb{N}$ such that letting $l = m' \circ m$, $\sigma^2_{\theta_{l(n)}} \to \sigma^2$. Letting $W$ denote a random variable with distribution $\mathcal{N}(0, \sigma^2)$ for any $\mathbb{P}_\theta$, by
Scheff\'{e}'s
lemma this implies that
\[
\lim_{n \to \infty} d_{\BL}^{\theta_{l(n)}}(Y, W)  = 0.
\]
Further, the Lindeberg--Feller theorem \citep[][Theorem~3.4.10]{Durrett2019} yields that
\[
\lim_{n \to \infty} d_{\BL}^{\theta_{l(n)}}(Y_n, W)= 0,
\]
since Lyapunov's condition is fulfilled by the uniform bound on the $(2+\eta)$th moment of $X_\theta$. Because the range of $l$ is contained in the range of $m$, this is a contradiction, hence the first condition is fulfilled.

The third condition follows immediately from the assumption in \eqref{eq:uniform-covariance-condition} by \citet[][Proposition~2.5.2, Lemma~2.7.20]{Bogachev2018}. Define $S_{n, \theta} := \frac{1}{\sqrt{n}} \sum_{i=1}^n X_{i, \theta} $ for $n \in \mathbb{N}$. The second condition follows by the same assumption and theorems by observing that $\mathbb{E}_\theta \lVert S_{n, \theta} \rVert^2$ is bounded by the same constant bounding $\mathbb{E}_\theta \lVert X_\theta \rVert^2$ and that 
\[
\Cov_\theta(S_{n, \theta}) = \frac{1}{n} \sum_{i=1}^n \Cov_\theta(X_{i, \theta}) =  \mathscr{C}_\theta .
\]
This shows that the family of measures $(X_{n, \theta}(\mathbb{P}_\theta))_{n \in \mathbb{N}, \theta \in \Theta}$ is tight which implies the second condition. 
\end{proof}

\section{Proofs of results in Sections~\ref{sec:GHCM}~and~\ref{sec:theory}}
\label{app:proofs}
This section contains the proofs of all results in Sections~\ref{sec:GHCM}~and~\ref{sec:theory} except Proposition~\ref{prop:gaussian-conditional-independence-testing-is-hard-when-dependence-is-never-in-tail} which is proven in Section~\ref{sec:hardness-proofs}. The proofs are self-contained, but readers new to the field may find the following references helpful. For general results about random variables on metric spaces (Slutsky's theorem, etc.) see \citet[][Chapter 1]{Billingsley1999}. For more specific results about Hilbertian random variables, Bochner integrals and operators on Hilbert spaces, see \citet[][Chapter 2, 4, 7]{Hsing2015}. For existence and construction of conditional expectations on Hilbert spaces, see \citet[Chapter 2]{scalora1961}. In this section, we sometimes omit the subscript $\dist$ when it is clear from the context.

\subsection{Derivation of (\ref{eq:T_n-inner-product-form})}
\label{sec:inner-product-version-of-test}
We first prove a small lemma. 
\begin{lemma}
  Let $x_1, \dots, x_n$ be elements of a Hilbert space $\mathcal{H}$. Then
  $$
  \left\| \sum_{i=1}^n x_i \right\|^2 = \sum_{i=1}^n \sum_{j=1}^n \langle x_i, x_j \rangle 
  $$
  and the non-zero eigenvalues of the operator
  $$
  \mathscr{A} := \sum_{i=1}^n x_i \otimes x_i
  $$
  equal the eigenvalues of the matrix $A$ with entries
  $$
  A_{ij} := \langle x_i , x_j \rangle .
  $$
\end{lemma}
\begin{proof}
  The first claim is immediate, since
  $$
  \left\| \sum_{i=1}^n x_i \right\|^2 = \left\langle \sum_{i=1}^n x_i, \sum_{j=1}^n x_j \right\rangle  = \sum_{i=1}^n \sum_{j=1}^n \langle x_i, x_j \rangle.
  $$
  For the second claim, note that we can write the operator $\mathscr{A}$ as $\mathscr{B}^*\mathscr{B}$ where $\mathscr{B}: \mathcal{H} \to \mathbb{R}^n$ is an operator given by
  $$
  \mathscr{B} h = \begin{pmatrix}
      \langle x_1, h \rangle \\
      \vdots \\
      \langle x_n, h \rangle 
  \end{pmatrix} 
  $$
  with adjoint $\mathscr{B}^*$ given by
  $$
  \mathscr{B}^* v = \sum_{i=1}^n v_i  x_i .
  $$
  The result now follows since $A = \mathscr{B} \mathscr{B}^*$.
\end{proof}
Applying the first result of the lemma to the sequence $1/\sqrt{n} \mathscr{R}_i$ for $i=1, \dots, n$ viewed as Hilbert--Schmidt operators from $\mathcal{H}_X$ to $\mathcal{H}_Y$, we get that
$$
T_n = \frac{1}{n} \left\| \sum_{i=1}^n \mathscr{R}_i \right\|^2 = \frac{1}{n} \sum_{i=1}^n \sum_{j=1}^n \langle \mathscr{R}_i , \mathscr{R}_j  \rangle = \frac{1}{n} \sum_{i=1}^n \sum_{j=1}^n \langle \hat{\varepsilon}_i , \hat{\varepsilon}_j \rangle \langle \hat{\xi}_i , \hat{\xi}_j \rangle.
$$
Applying the second result of the lemma to the sequence $1/\sqrt{n-1} \mathscr{R}_i - \bar{\mathscr{R}}$, we get that the eigenvalues of
$$
\hat{\mathscr{C}} = \frac{1}{n-1} \sum_{i=1}^n  (\mathscr{R}_i - \bar{\mathscr{R}}) \otimes_{\HS}  (\mathscr{R}_i - \bar{\mathscr{R}})
$$
equal the eigenvalues of the matrix $A$ with entries
$$
A_{ij} := \frac{1}{n-1} \langle \mathscr{R}_i - \bar{\mathscr{R}} , \mathscr{R}_j - \bar{\mathscr{R}} \rangle
$$
Using bilinearity of the inner product, we can expand and see that 
$$
A = \Gamma - J\Gamma - \Gamma J + J\Gamma J 
$$
as desired.

\subsection{Derivation of (\ref{eq:OLStest})}
\label{sec:OLS-proof}
\begin{proof}
Fix $n \geq 2$ and write $p:=1 + a(n)$. Let $(\tilde{x}_i, \tilde{y}_i, \tilde{z}_i)_{i=1}^n$ denote mean-centred observations, so e.g.\ $\tilde{z}_i = z_i - \sum_{j=1}^nz_j/n$, and let $\tilde{X}^{(n)} = (\tilde{x}_1,\ldots,\tilde{x}_n)^\top  \in \R^n$.
Let $W_n \in \R^{n \times p}$ be the design matrix with $i$th row given by $(\tilde{y}_i, \tilde{z}_{i1},\ldots,\tilde{z}_{ia(n)})$, and let $\hat{\theta}_n \in \R$ be the first component of the coefficient vector from regressing $\tilde{X}^{(n)}$ onto $W_n$, so $\hat{\theta}_n := \{(W_n^\top  W_n)^{-1}W_n^\top  \tilde{X}^{(n)}\}_1$. Further, let $P_n \in \R^{n\times n}$ be the orthogonal projection onto the column space of $W_n$. Then
\[
\psi_n^{\OLS} = \ind_{\{|\hat{\theta}_n| \geq t_{n-p-1}(\alpha/2) \hat{\sigma}_{W,n} \|(I-P_n)\tilde{X}^{(n)}\|_2  / \sqrt{n-p-1}\}},
\]
where $t_{n-p}(\alpha/2)$ is the upper $\alpha/2$-point of a $t$ distribution on $n-p$ degrees of freedom, and
$\hat{\sigma}_{W,n}^2 := \{(W_n^\top  W_n)^{-1}\}_{11}$. Fix $Q \in \mathcal{Q}$; in the following we will suppress dependence on this for notational simplicity. Then there exists $r \in \mathbb{N}$ such that
\[
\theta := \frac{\Cov(Y, X \cond Z)}{\Var (X \cond Z)} = \frac{\Cov(Y, X \cond Z_1,\ldots,Z_r)}{\Var (X \cond Z_1,\ldots,Z_r)} >0,
\]
and so for $n$ such that $a(n) > r$, $\hat{\theta}_n \cond W_n \sim \mathcal{N}(\theta, \sigma^2 \hat{\sigma}_{W,n}^2)$ where
\[
\sigma^2 := \Var(X \cond Y, Z) = \Var(X \cond Y, Z_1,\ldots,Z_r)>0.
\] 

Note that $\|(I-P_n)\tilde{X}^{(n)}\|_2^2 / \sigma^2 \sim  \chi^2_{n-p-1}$, and so by the weak law of large numbers and the continuous mapping theorem, $\|(I-P_n)\tilde{X}^{(n)}\|_2 / \sqrt{n-p-1} \inprob \sigma$.
To show that $\pr(\psi_n^{\OLS} = 1) \to 1$, it therefore suffices to show that $\hat{\sigma}_{W,n}^2 \inprob 0$.

Now writing $\Sigma_n = \Cov(Y, Z_1,\ldots,Z_{a(n)})$, we have that $W_n^\top W_n$ has a Wishart distribution on $n-1$ degrees of freedom: $W_n^\top  W_n \sim W_p(\Sigma_n, n-1)$. Thus, $(\Sigma_n^{-1})_{11} / \hat{\sigma}_{W,n}^2 \sim \chi^2_{n-p}$ and $(\Sigma_n^{-1})_{11} = \Var(Y \cond Z_1,\ldots,Z_r) = \Var(Y \cond Z) < \infty$. We therefore see that as $n \to \infty$ and hence $n-p \to \infty$, we have $\hat{\sigma}_{W,n}^2 \inprob 0$ as required.
\end{proof}

\subsection{Proofs of results in Section~\ref{sec:size}}
In this section we provide proofs of Theorems~\ref{thm:level of test} and \ref{thm:asymptotic normality and covariance estimation}. The proofs rely heavily on the theory developed in Section~\ref{app:uniform-convergence}.

\subsubsection{Auxiliary lemmas}
We first prove some auxiliary lemmas that will be needed for the upcoming proofs.

\begin{lemma}
\label{lem: Convergence to zero if conditional convergence to zero}
Let  $(X_n)_{n \in \mathbb{N}}$ be a sequence of real-valued random variables defined on $(\Omega, \mathcal{F})$ equipped with a family of probability measures $(\mathbb{P}_\theta)_{\theta \in \Theta}$. Let $X$ be another real-valued random variable on the same space and let $(\mathcal{F}_n)_{n \in \mathbb{N}}$ be a sequence of sub-$\sigma$-algebras of $\mathcal{F}$. If $\mathbb{E}_{\theta}(|X_n| \cond \mathcal{F}_n)  \overset{P}{\rightrightarrows} 0$ then $X_n \overset{P}{\rightrightarrows} 0 $. 
\end{lemma}
\begin{proof}
Let $\epsilon > 0$ be given. By Markov's inequality
$$
\sup_{\theta \in \Theta} \mathbb{P}_\theta(|X_n| \geq \epsilon) \leq \sup_{\theta \in \Theta} \mathbb{P}_\theta(|X_n| \wedge \epsilon \geq \epsilon) \leq \sup_{\theta \in \Theta} \frac{\mathbb{E}_\theta(|X_n| \wedge \epsilon)}{\epsilon} .
$$
We will be done if we can show that $\sup_{\theta \in \Theta} \mathbb{E}_\theta(|X_n| \wedge \epsilon) \to 0$ as $n \to \infty$. Note that by monotonicity of conditional expectations, for each $\theta \in \Theta$ we have
$$
\mathbb{E}_\theta(|X_n| \wedge \epsilon \cond \mathcal{F}_n) \leq \mathbb{E}_\theta(\epsilon \cond \mathcal{F}_n) = \epsilon ,
$$
and
$$
\mathbb{E}_\theta(|X_n| \wedge \epsilon \cond \mathcal{F}_n) \leq \mathbb{E}_\theta(|X_n| \cond \mathcal{F}_n) .
$$
Combining both of the above expressions, we get
$$
\mathbb{E}_\theta(|X_n| \wedge \epsilon \cond \mathcal{F}_n) \leq  \mathbb{E}_\theta(|X_n| \cond \mathcal{F}_n) \wedge \epsilon .
$$
This lets us write by the tower property and monotonicity of integrals,
$$
\sup_{\theta \in \Theta} \mathbb{E}_\theta(|X_n| \wedge \epsilon) = \sup_{\theta \in \Theta} \mathbb{E}_\theta[\mathbb{E}_\theta(|X_n| \wedge \epsilon \cond \mathcal{F}_n)] \leq \sup_{\theta \in \Theta} \mathbb{E}_\theta[\mathbb{E}_\theta(|X_n| \cond \mathcal{F}_n) \wedge \epsilon] .
$$
Let $Y_n := \mathbb{E}_\theta(|X_n| \cond \mathcal{F}_n) \wedge \epsilon$ and let $\delta > 0$ be given. Then
\begin{align*}
\sup_{\theta \in \Theta} \mathbb{E}_\theta(Y_n) &\leq \sup_{\theta \in \Theta} \mathbb{E}_\theta \left(Y_n  \ind_{\{Y_n  < \delta / 2\}} \right) + \sup_{\theta \in \Theta} \mathbb{E}_\theta \left(Y_n \ind_{\{Y_n  \geq \delta / 2\}} \right) \\
&\leq \frac{\delta}{2} + \epsilon \sup_{\theta \in \Theta} 	 \mathbb{P}_\theta(Y_n \geq \delta/2) .
\end{align*}
By assumption, for any $\eta > 0$, we can choose $N \in \mathbb{N}$ so that for all $n \geq N$, we can make $\sup_{\theta \in \Theta}  \mathbb{E}_\theta(|X_n| \cond \mathcal{F}_n) \geq \delta/2) < \eta$. Thus, choosing $N$ to parry $\eta = \frac{\epsilon}{2\epsilon}$, we get
$$
\sup_{\theta \in \Theta} \mathbb{E}_\theta(|Y_n|) < \delta, 
$$
proving the desired result.
\end{proof}

\begin{lemma}
\label{lem: pulling out whats known}
Let $X$ and $Y$ be random variables defined on the probability space $(\Omega, \mathcal{F}, \mathbb{P})$ with values in a Hilbert space $\mathcal{H}$. Let $\mathcal{D}$ be a sub-$\sigma$-algebra of $\mathcal{F}$ so that $X$ is $\mathcal{D} $-measurable. Assume that $\mathbb{E}(\lVert X \rVert)$, $\mathbb{E}(\lVert Y \rVert)$ and $\mathbb{E}(\lVert X \rVert \lVert Y \rVert)$ all exist. Then
\[
\mathbb{E}(\langle X, Y\rangle \cond \mathcal{D}) = \langle X, \mathbb{E}(Y \cond \mathcal{D}) \rangle.
\]
\end{lemma}
\begin{proof}
To show the result, we need to show that $\langle X, \mathbb{E}(Y \cond \mathcal{D}) \rangle$ is $\mathcal{D}$-measurable and that integrals over $\mathcal{D}$-sets of $\langle X, Y\rangle$ and $\langle X, \mathbb{E}(Y \cond \mathcal{D})\rangle$ coincide.  $\langle X, \mathbb{E}(Y \cond \mathcal{D})\rangle$ is $\mathcal{D}$-measurable by continuity of the inner product and the fact that $X$ and $\mathbb{E}(Y \cond \mathcal{D})$ are $\mathcal{D}$-measurable by assumption and definition, respectively. By expanding the inner product in an orthonormal basis $(e_k)_{k \in \mathbb{N}}$ of $\mathcal{H}$, we get
\begin{align*}
&\int_D \langle X, Y \rangle \, \mathrm{d}\mathbb{P} =  \int_D \sum_{k=1}^\infty \langle X, e_k \rangle \langle Y, e_k\rangle \, \mathrm{d}\mathbb{P}   = \sum_{k=1}^\infty \int_D \mathbb{E}(\langle X, e_k\rangle \langle Y, e_k\rangle \cond \mathcal{D}) \, \mathrm{d}\mathbb{P} \\
& =\sum_{k=1}^\infty \int_D \langle X, e_k \rangle \langle\mathbb{E}( Y \cond \mathcal{D} ) , e_k\rangle \, \mathrm{d}\mathbb{P}  =  \int_D \sum_{k=1}^\infty \langle X, e_k \rangle \langle\mathbb{E}( Y \cond \mathcal{D} ) , e_k\rangle \, \mathrm{d}\mathbb{P}  = \int_D \langle X, \mathbb{E}( Y \cond \mathcal{D} ) \rangle \, \mathrm{d}\mathbb{P},
\end{align*}
by using the interchangeability of sums and integrals and the property 
$$
\mathbb{E}( \langle Y, e_i\rangle  \cond \mathcal{D} )  = \langle\mathbb{E}( Y \cond \mathcal{D} ) , e_i\rangle
$$
of conditional expectations on Hilbert spaces.

\end{proof}

\begin{lemma}
\label{lem:uniform-convergence-of-quantile-map}
Let $q$ denote the function that maps a self-adjoint, positive semidefinite, trace-class operator, $\mathscr{C}$ on a separable Hilbert space $\mathcal{H}$, to the $1-\alpha$ quantile of the $\lVert \mathcal{N}(0, \mathscr{C}) \rVert^2$ distribution. Then $q$ is continuous in trace norm and the restriction of $q$ to a bounded subset $\mathcal{C}$ of covariance operators satisfying 
\begin{equation}
\label{eq:covariance-compactness-lemma}
\lim_{N \to \infty} \sup_{\mathscr{C} \in \mathcal{C}} \sum_{k=K}^\infty \langle \mathscr{C} e_k, e_k \rangle = 0
\end{equation}
for some orthonormal basis $(e_k)_{k=1}^\infty$ of $\mathcal{H}$, is uniformly continuous in trace norm.
\end{lemma}
\begin{proof}
Let $\mathscr{C}_n$ be a sequence of self-adjoint, positive semidefinite, trace-class operators converging to $\mathscr{C}$ in trace norm. Then by \citet[][Theorem~2.7.21]{Bogachev2018} $\mathcal{N}(0, \mathscr{C}_n) \overset{\mathcal{D}}{\to} \mathcal{N}(0, \mathscr{C})$ and by the continuous mapping theorem we have $\lVert \mathcal{N}(0, \mathscr{C}_n) \rVert^2 \overset{\mathcal{D}}{\to} \lVert \mathcal{N}(0, \mathscr{C}) \rVert^2$. This implies the convergence of the quantile functions by the Portmanteau theorem and \citet[][Lemma~21.2]{Vandervaart1998} and hence $q$ is continuous. 

By the Heine--Cantor theorem, the restriction of $q$ to the closure of $\mathcal{C}$ is uniformly continuous if $\mathcal{C}$ is relatively compact. Restricting $q$ further to $\mathcal{C}$ preserves the uniform continuity. \citet[][Proposition~2.5.2]{Bogachev2018} states that equation \eqref{eq:covariance-compactness-lemma} exactly characterises the relatively compact sets of trace class operators. 
\end{proof}

\begin{lemma}
\label{lem:uniform-absolute-continuity-of-gamma-family}
Let $\Theta \subseteq \mathbb{R}_+$ and let $(\mu_\theta)_{\theta \in \Theta}$ be the family of probability distributions on $\mathbb{R}$ where for each $\theta \in \Theta$, $\mu_\theta$ denotes the distribution of $\theta Z$ where $Z \sim \chi^2_1$.
If $\Theta$ is bounded away from $0$, the family is uniformly absolutely continuous with respect to the Lebesgue measure $\lambda$. 
\end{lemma}
\begin{proof}
Note that the density $f_\theta$ of $\mu_\theta$ with respect to the Lebesgue measure is
\[
f_\theta(x) = \frac{1}{\sqrt{2\pi\theta}} \frac{e^{-\frac{1}{2\theta}x}}{\sqrt{x}}.
\]
We will apply Proposition~\ref{prop:uniform-absolute-continuity-and-uniform-integrability} by showing that $\sup_{\theta \in \Theta} \int f_\theta^{3/2} \, \mathrm{d}\lambda < \infty$, which is sufficient for uniform integrability by \citet[][Example~4.5.10]{Bogachev2007}. We see that
\[
\int f_\theta(x)^{3/2} \, \mathrm{d}\lambda = \frac{1}{\sqrt[4]{6\pi^3\theta^3}} \int  \frac{e^{-\frac{3}{4\theta}x}}{\sqrt[4]{x^3}} \, \mathrm{d}\lambda,
\]
and we recognise the final integral as the unnormalised density of a $\Gamma(1/4, 3/(4\theta))$ random variable. Thus,
\[
  \int f_\theta(x)^{3/2} \, \mathrm{d}\lambda  =  \frac{1}{\sqrt[4]{6\pi^3\theta^3}} \frac{\Gamma(1/4)\sqrt[4]{4\theta}}{\sqrt[4]{3}} = \frac{\Gamma(1/4)}{\sqrt[4]{6\pi^3 \theta^2}}.
\]
This is finite for all $\theta \in \Theta$ since $\Theta$ is bounded away from zero, proving the desired result.
\end{proof}

\begin{lemma}
\label{lem:uniform-absolute-continuity-preserved-under-square-root}
Let $X$ be a uniformly tight with respect to index family $\Theta$ (see Definition~\ref{def:tight}), real-valued and non-negative random variable that is uniformly absolutely continuous with respect to the Lebesgue measure. Then so is $\sqrt{X}$.
\end{lemma}
\begin{proof}
Let $\epsilon > 0$ be given and let $\lambda$ denote the Lebesgue measure. We need to find $\delta > 0$ such that for any Borel measurable $B$,
\[
\lambda(B)< \delta \Longrightarrow \sup_{\theta \in \Theta} \mathbb{P}_\theta(\sqrt{X} \in B) < \epsilon.
\]
For each measurable $B$, we define $B^2 := \{b^2 \ : \ b \in \mathbb{R} \}$. Then $\mathbb{P}_\theta(\sqrt{X} \in B) = \mathbb{P}_\theta(X \in B^2)$ and by the uniform tightness of $X$, we can find $M > 0$ such that 
\[
\sup_{\theta \in \Theta} \mathbb{P}_\theta(X \in B^2) \leq  \sup_{\theta \in \Theta} \mathbb{P}_\theta(X \in B^2 \cap [0, M])  + \epsilon/2.
\]
By the uniform absolute continuity of $X$ with respect to $\lambda$, we can find $\delta'$ such that $\lambda(B)< \delta'$ implies $ \sup_{\theta \in \Theta} \mathbb{P}_\theta(X \in B) < \epsilon/2$. Note that for any such $B$, by the regularity of the Lebesgue measure, we can find an open set $U \supseteq B$ such that $\lambda(U \setminus B) < \delta' - \lambda(B)$. This implies that $\lambda(U) < \delta'$. For every open $U$, by \citet[][Theorem~4.6]{Carothers2000}, we can find a countable union of disjoint open intervals $(I_j)_{j=1}^\infty$, where $I_j = (a_j, b_j)$, such that $U = \bigcup_{j=1}^\infty I_j$. 
Note that $U^2$ also covers $B^2$ since if $x \in U$, $x$ is in at least one of the intervals $I_j$, and thus $x^2$ is in $I_j^2$. Combining these observations, we get that
\begin{align*}
\lambda(B^2 \cap [0, M]) &\leq \lambda(U^2 \cap [0, M]) = \sum_{j=1}^\infty \lambda(I_j^2 \cap [0, M]) = \sum_{j=1}^\infty (\min(M, b_j^2) - a_j^2)\\
&= \sum_{j=1}^\infty (\min(\sqrt{M}, b_j)  + a_j)( \min(\sqrt{M}, b_j)-a_j) \leq  2 \sqrt{M} \sum_{j=1}^\infty  b_j-a_j < 2 \sqrt{M} \delta'.
\end{align*}
Thus letting $\delta = \delta'/(2\sqrt{M})$, we see that for all $B$ with $\lambda(B) < \delta$, we also have $\lambda(B^2 \cap [0, M]) < \delta'$, and hence
\[
  \sup_{\theta \in \Theta} \mathbb{P}_\theta(X \in B^2 \cap [0, M]) < \epsilon /2,
\]
proving the statement.
\end{proof}

\begin{lemma}
\label{lem:banachian-tightness-sufficient-condition}
Let $(X_\theta)_{\theta \in \Theta}$ be Hilbertian random variables with values in $\mathcal{H}$. Assume that for every $\theta \in \Theta$, $\mathbb{E}_\theta(X_\theta) = 0$, $\sup_{\theta \in \Theta} \allowbreak \mathbb{E} \lVert X_\theta \rVert^2 < \infty$ and that there exists a basis $(e_k)_{k \in \mathbb{N}}$ of $\mathcal{H}$ such that 
\[
\lim_{K \to \infty} \sup_{\theta \in \Theta} \sum_{k=K}^\infty \mathbb{E}(\langle X_\theta, e_k \rangle^2) = 0.  
\]
Then the family $(X_\theta \otimes X_\theta)_{\theta \in \Theta}$ is uniformly tight when viewed as random variables in the Banach space of trace-class operators on $\mathcal{H}$. 

\begin{proof}
By \citet[][Proposition 3.1]{Fugarolas1983} $(e_k \otimes e_j)_{(k,j) \in \mathbb{N}^2}$ is a Schauder basis for the Banach space of trace-class operators on $\mathcal{H}$. Thus, \citet[][Theorem 2.7.10]{Bogachev2018} yields that we need to show that
\[
  \lim_{r \to \infty} \sup_{\theta \in \Theta} \mathbb{P}_\theta(\lVert X_\theta \otimes X_\theta \rVert_{\TR} > r) = 0
\]
and for all $\varepsilon > 0$
\[
  \lim_{K \to \infty} \sup_{\theta \in \Theta} \mathbb{P}_\theta(\lVert X_\theta \otimes X_\theta - P_K(X_\theta \otimes X_\theta) \rVert_{\TR} > \varepsilon) = 0,
\]
where $P_K$ denotes the projection onto the $K$ first basis vectors in the space of trace-class operators. An application of Markov's inequality yields immediately that
\[
  \sup_{\theta \in \Theta} \mathbb{P}_\theta(\lVert X_\theta \otimes X_\theta \rVert_{\TR} > r) \leq \frac{  \sup_{\theta \in \Theta} \mathbb{E}_\theta\lVert X_\theta\rVert^2}{r}
\]
hence the first condition is satisfied by the assumed uniform upper bound on $\mathbb{E}_\theta\lVert X_\theta\rVert^2$. For $K = m^2$, $m \in \mathbb{N}$, note that
\begin{align*}
\lVert X_\theta \otimes X_\theta - P_K(X_\theta \otimes X_\theta) \rVert_{\TR} &=\left\lVert \left(\sum_{j=m}^\infty \langle X_\theta , e_j \rangle e_j \right) \otimes \left(\sum_{k=m}^\infty \langle X_\theta , e_k \rangle e_k \right) \right\rVert_{\TR}  \\
&= \left\lVert \sum_{j=m}^\infty \langle X_\theta , e_j \rangle e_j \right\rVert^2 = \sum_{j=m}^\infty \langle X_\theta , e_j \rangle^2,
\end{align*}
where the final equality is by Parseval's identity. Using this, the second condition is satisfied by assumption, since, by Markov's inequality, for all $\epsilon >0$,
\[
   \sup_{\theta \in \Theta} \mathbb{P}_\theta\left(\sum_{j=m}^\infty \langle X_\theta , e_j \rangle^2 > \epsilon\right) \leq \frac{\sup_{\theta \in \Theta} \sum_{j=m}^\infty \mathbb{E}_\theta \left( \langle X_\theta , e_j \rangle^2 \right)}{\epsilon} .
\]

\end{proof}

\end{lemma}

\subsubsection{Proof of Theorem~\ref{thm:asymptotic normality and covariance estimation}}
\begin{proof}
Throughout the proof we omit the subscript $\dist$ from $\varepsilon$, $\xi$, $f$, $g$. 

\paragraph*{Convergence of $\mathscr{T}_n$.}
We have that
\begin{align*}
\mathscr{T}_n = \frac{1}{\sqrt{n}}\sum_{i=1}^n \mathscr{R}_i &= \underbrace{\frac{1}{\sqrt{n}} \sum_{i=1}^n \varepsilon_i \otimes \xi_i}_{=:U_n} + \underbrace{\frac{1}{\sqrt{n}} \sum_{i=1}^n  (f(z_i)-\hat{f}(z_i)) \otimes (g(z_i)-\hat{g}(z_i))}_{=:a_n}\\
& + \underbrace{\frac{1}{\sqrt{n}} \sum_{i=1}^n (f(z_i)-\hat{f}(z_i)) \otimes \xi_i }_{=:b_n} + \underbrace{\frac{1}{\sqrt{n}} \sum_{i=1}^n (\varepsilon_i \otimes g(z_i)-\hat{g}(z_i))}_{=:c_n} .
\end{align*}
Since
\[
\mathbb{E}(\varepsilon_i \otimes \xi_i) =  \mathbb{E}((X-\mathbb{E}(X \cond Z)) \otimes (Y-\mathbb{E}(Y \cond Z))) = \mathbb{E}(\Cov(X, Y \cond Z)) = 0
\]
because $X \independent Y \cond Z$, Proposition~\ref{prop:uniform-clt} yields 
that $U_n$ converges uniformly in distribution to the desired Gaussian over $\tilde{\mathcal{P}_0}$. By Proposition~\ref{prop:uniform-slutsky}, if $a_n$, $b_n$ and $c_n$ all converge to $0$ uniformly in probability, we will have shown the desired result. We establish this by looking at the Hilbert--Schmidt norm of the sequences, since uniform convergence of the norms to $0$ implies uniform convergence of the sequences to $0$. For $a_n$, using properties of the Hilbert--Schmidt norm and the Cauchy--Schwarz inequality yields
\begin{align*}
\lVert a_n \rVert_{\HS} &= \left\lVert \frac{1}{\sqrt{n}} \sum_{i=1}^n  (f(z_i)-\hat{f}(z_i)) \otimes (g(z_i)-\hat{g}(z_i)) \right\rVert_{\HS}\\
 &\leq  \frac{1}{\sqrt{n}} \sum_{i=1}^n\lVert (f(z_i)-\hat{f}(z_i)) \otimes (g(z_i)-\hat{g}(z_i)) \rVert_{\HS}\\
  &=  \frac{1}{\sqrt{n}}  \sum_{i=1}^n \lVert (f(z_i)-\hat{f}(z_i)) \rVert \lVert g(z_i)-\hat{g}(z_i) \rVert  \\
  &\leq   \sqrt{\frac{1}{n} \sum_{i=1}^n\lVert (f(z_i)-\hat{f}(z_i)) \rVert^2 \sum_{i=1}^n \lVert g(z_i)-\hat{g}(z_i) \rVert^2}   = \sqrt{n M_{n, \dist}^f M_{n, \dist}^g}.
\end{align*}
By assumption $n M_{n, \dist}^f  M_{n, \dist}^g \overset{P}{\rightrightarrows} 0$ and Proposition~\ref{prop:uniform-convergence-in-probability-preserved-under-continuous-operations-when-tight} yields that the same is true for $\sqrt{n M_{n, \dist}^f  M_{n, \dist}^g}$. This implies that $\lVert a_n \rVert_{\HS} \overset{P}{\rightrightarrows} 0$ as desired.

To establish that $\lVert b_n \rVert_{\HS} \overset{P}{\rightrightarrows} 0$, we will instead show that the square of the Hilbert--Schmidt norm goes to $0$. This implies that $\lVert b_n \rVert_{\HS}\overset{P}{\rightrightarrows} 0$ by the same arguments about $x \mapsto \sqrt{x}$ as above. We will show that $\mathbb{E}_\dist(\lVert b_n \rVert_{\HS}^2 \cond X^{(n)}, Z^{(n)}) \overset{P}{\rightrightarrows} 0$, where $X^{(n)}=(x_1, \dots, x_n)$ and $Z^{(n)}=(x_1, \dots, x_n)$, which then implies the desired result by Lemma~\ref{lem: Convergence to zero if conditional convergence to zero}. 
For every $\dist \in \tilde{\mathcal{P}_0}$ we have
\begin{align}
& \mathbb{E}_\dist(\lVert b_n \rVert_{\HS}^2 \cond X^{(n)}, Z^{(n)}) = \frac{1}{n} \mathbb{E}_\dist\left( \left\lVert \sum_{i=1}^n  (f(z_i) - \hat{f}(z_i)) \otimes \xi \right\rVert^2_{\HS} \cond X^{(n)}, Z^{(n)} \right) \nonumber \\
&= \frac{1}{n}  \sum_{j=1}^n \sum_{i=1}^n \mathbb{E}_\dist\left(\langle (f(z_i) - \hat{f}(z_i)) \otimes \xi_i, (f(z_j) - \hat{f}(z_j)) \otimes \xi_j \rangle_{\HS} \cond X^{(n)}, Z^{(n)} \right) \nonumber \\
&= \frac{1}{n}  \sum_{j=1}^n \sum_{i=1}^n \mathbb{E}_\dist\left(\langle f(z_i) - \hat{f}(z_i), f(z_j) - \hat{f}(z_j) \rangle \langle \xi_i, \xi_j \rangle \cond X^{(n)}, Z^{(n)} \right) \nonumber \\
&= \frac{1}{n}  \sum_{j=1}^n \sum_{i=1}^n \langle f(z_i) - \hat{f}(z_i), f(z_j) - \hat{f}(z_j) \rangle\mathbb{E}_\dist\left(\langle \xi_i, \xi_j \rangle  \cond X^{(n)}, Z^{(n)} \right) \label{eq:b_n-decomposition},
\end{align}
where the penultimate equality uses the fact that for Hilbert--Schmidt operators $\langle x_1 \otimes y_1, x_2 \otimes y_2 \rangle_{\HS} = \langle x_1, x_2 \rangle \langle y_1, y_2 \rangle$. The final equality holds since the terms involving $f(z_i)-\hat{f}(z_i)$ are measurable with respect to the $\sigma$-algebra generated by $X^{(n)}$ and $Z^{(n)}$. 
The term $\langle \xi_i, \xi_j \rangle$ only depends on $Z_i$ and $Z_j$ of the conditioning variables, so we can omit the remaining variables from the conditioning expression. 
Recall that $\xi_i = Y_i - \mathbb{E}_\dist(Y_i \cond Z_i)$. 
For $i \neq j$, by using that $\mathbb{E}_\dist(Y_i \cond Z_i) = \mathbb{E}_\dist(Y_i \cond Z_i, Z_j)$ since $Z_j$ is independent of $(Y_i, Z_i)$ and Lemma~\ref{lem: pulling out whats known}, we get
\begin{align*}
\mathbb{E}_\dist\big[\langle \xi_i, \xi_j \rangle  \cond X^{(n)}, Z^{(n)} \big] &= \mathbb{E}_\dist\big[\langle Y_i, Y_j \rangle - \langle Y_i, \mathbb{E}_\dist(Y_j \cond Z_j) \rangle - \langle \mathbb{E}_\dist(Y_i \cond Z_i), Y_j \rangle \\
 &\qquad + \langle \mathbb{E}_\dist(Y_i \cond Z_i), \mathbb{E}_\dist(Y_j \cond Z_j) \rangle  \cond Z_i, Z_j \big]\\
 &= \mathbb{E}_\dist(\langle Y_i, Y_j \rangle \cond Z_i, Z_j) - \langle \mathbb{E}_\dist(Y_i  \cond Z_i, Z_j) , \mathbb{E}(Y_j  \cond Z_i, Z_j) \rangle.
\end{align*}
We will show that this is zero. By assumption $(Y_i, Z_i) \independent (Y_j, Z_j)$, so applying the usual laws of conditional independence, we get $Y_i \independent Y_j \cond (Z_i, Z_j)$. Take now some orthonormal basis for $\mathcal{H}_Y$, $(e_k)_{k \in \mathbb{N}}$, and expand $\langle Y_i, Y_j \rangle $ to get
$$
\mathbb{E}_\dist(\langle Y_i, Y_j \rangle \cond Z_i, Z_j) = \mathbb{E}_\dist\left( \sum_{k=1}^\infty \langle Y_i, e_k \rangle \langle Y_j, e_k \rangle \cond Z_i, Z_j\right) =  \sum_{k=1}^\infty  \mathbb{E}_\dist \left( \langle Y_i, e_k \rangle \langle Y_j, e_k \rangle \cond Z_i, Z_j\right).
$$
For all $k$, $\langle Y_i, e_k \rangle \independent \langle Y_j, e_k \rangle \cond (Z_i, Z_j)$, so $\mathbb{E}(\langle Y_i, e_k \rangle\langle Y_j, e_k \rangle \cond Z_i, Z_j)$ factorises, and we get
\begin{align*}
&\sum_{k=1}^\infty \mathbb{E}_\dist \left( \langle Y_i, e_k \rangle \langle Y_j, e_k \rangle \cond Z_i, Z_j\right) = \sum_{k=1}^\infty  \mathbb{E}_\dist ( \langle Y_i, e_k \rangle \cond Z_i, Z_j) \mathbb{E}_\dist( \langle Y_j, e_k \rangle \cond Z_i, Z_j )\\
& = \sum_{k=1}^\infty  \langle \mathbb{E}_\dist ( Y_i  \cond Z_i, Z_j) , e_k\rangle  \langle \mathbb{E}_\dist( Y_j \cond Z_i, Z_j), e_k  \rangle = \langle \mathbb{E}_\dist( Y_i  \cond Z_i, Z_j), \mathbb{E} ( Y_j  \cond Z_i, Z_j) \rangle ,
\end{align*}
where the second last equality follows from
$\mathbb{E}_\dist( \langle Y, e_k\rangle  \cond Z_i, Z_j )  = \langle\mathbb{E}_\dist( Y \cond Z_i, Z_j) , e_k\rangle$ by Lemma~\ref{lem: pulling out whats known}.
We can thus omit all terms from the sum in \eqref{eq:b_n-decomposition} where $i \neq j$ and get
\begin{align*}
&\mathbb{E}_\dist(\lVert b_n \rVert_{\HS}^2 \cond X^{(n)}, Z^{(n)}) = \frac{1}{n}  \sum_{i=1}^n \lVert f(z_i) - \hat{f}^{(n)}(z_i) \rVert^2_{X} \mathbb{E}_\dist\left(\lVert \xi_i \rVert_Y^2  \cond Z_i \right) = \tilde{M}_n^f \overset{P}{\rightrightarrows} 0 ,
\end{align*}
by assumption. An analogous argument can be repeated for $c_n$, thus proving the desired result.

\paragraph*{Convergence of $\hat{\mathscr{C}}$.}
For simplicity, we prove convergence where $\hat{\mathscr{C}}$ is instead defined as the estimate where we divide by $n$ instead of $n-1$ since this does not affect the asymptotics.

By the above and Proposition~\ref{prop:uniform-slutsky}, since $(\mathcal{N}(0, \mathscr{C}_\dist))_{\dist \in \tilde{\mathcal{P}}_0}$ is uniformly tight by \citet[][Proposition~2.5.2, Lemma~2.7.20]{Bogachev2018}, we have
\[
\frac{1}{n} \sum_{i=1}^n \mathscr{R}_i = \frac{1}{\sqrt{n}}  \mathscr{T}_n \overset{P}{\rightrightarrows} 0 .
\]
By Proposition~\ref{prop:uniform-convergence-in-probability-preserved-under-continuous-operations-when-tight}, this implies that the second term in the definition of $\hat{\mathscr{C}}$ 
converges to $0$ uniformly in probability since the mapping $(\mathscr{A},\mathscr{B}) \mapsto \mathscr{A} \otimes_{\HS} \mathscr{B}$ is continuous.
It remains to show that the first term in the definition of $\hat{\mathscr{C}}$ converges to $\mathscr{C}$. The proof is similar to the proof of Theorem~6 in \citep{GCM} and relies on expanding the first term $\frac{1}{n} \sum_{i=1}^n \mathscr{R}_i \otimes_{\HS} \mathscr{R}_i$ to yield
\[
\frac{1}{n} \sum_{i=1}^n [(f(z_i)-\hat{f}(z_i)) \otimes (g(z_i)-\hat{g}(z_i)) + (f(z_i)-\hat{f}(z_i)) \otimes \xi_i + \varepsilon_i \otimes (g(z_i)-\hat{g}(z_i)) + \varepsilon_i \otimes \xi_i]^{\otimes_{\HS}2},
\]
where $\mathscr{A}^{\otimes_{\HS}2} = \mathscr{A} \otimes_{\HS} \mathscr{A}$. Expanding this even further yields 16 terms of which 15 go to zero. The non-zero term is
\[
\RN{1}_n = \frac{1}{n} \sum_{i=1}^n (\varepsilon_i \otimes \xi_i)^{\otimes_{\HS}2} \overset{P}{\rightrightarrows} \mathbb{E}_{\dist} \left( (\varepsilon_i \otimes \xi_i)^{\otimes_{\HS}2} \right)=\mathscr{C},
\]
by Proposition~\ref{prop:uniform-lln} and Lemma~\ref{lem:banachian-tightness-sufficient-condition} and the assumed tightness condition. For the remaining 15 terms, we will argue by taking trace norms and applying the triangle inequality to reduce the number of cases. This leaves us with 8 terms and 5 cases (by symmetry of $f$ and $\varepsilon$, $g$ and $\xi$) that we need to argue converge to $0$ uniformly in probability.

The first case is
\begin{align*}
\RN{2}_n &= \left \lVert \frac{1}{n} \sum_{i=1}^n [(f(z_i)-\hat{f}(z_i)) \otimes (g(z_i)-\hat{g}(z_i))]^{\otimes_{\HS}2} \right\rVert_{\TR}\\
 &\leq \frac{1}{n} \sum_{i=1}^n \left\lVert [(f(z_i)-\hat{f}(z_i)) \otimes (g(z_i)-\hat{g}(z_i))]^{\otimes_{\HS}2} \right\rVert_{\TR}\\
 &= \frac{1}{n} \sum_{i=1}^n \left\lVert f(z_i)-\hat{f}(z_i) \otimes g(z_i)-\hat{g}(z_i) \right\rVert_{\HS}^2 = \frac{1}{n} \sum_{i=1}^n \lVert f(z_i)-\hat{f}(z_i) \rVert^2 \lVert g(z_i)-\hat{g}(z_i) \rVert^2\\
  &\leq n M_{n, \dist}^f M_{n, \dist}^g \overset{P}{\rightrightarrows} 0,
\end{align*} 
where the final inequality uses that for positive sequences $\sum a_n b_n \leq \sum a_n \sum b_n$, which can be seen by noting that every term on the left-hand side also appears on the right-hand side. For the second case we have, by applying the Cauchy--Schwarz inequality,
\begin{align*}
\RN{3}_n &= \left \lVert \frac{1}{n} \sum_{i=1}^n [(f(z_i)-\hat{f}(z_i)) \otimes \xi_i] \otimes_{\HS} [(g(z_i)-\hat{g}(z_i)) \otimes \varepsilon_i] \right\rVert_{\TR} \\
& \leq  \frac{1}{n} \sum_{i=1}^n \lVert f(z_i)-\hat{f}(z_i) \rVert \lVert g(z_i)-\hat{g}(z_i) \rVert \lVert \varepsilon_i \rVert \lVert \xi_i \rVert \\
& \leq  \sqrt{ \underbrace{\left( \frac{1}{n} \sum_{i=1}^n \lVert f(z_i)-\hat{f}(z_i) \rVert^2 \lVert g(z_i)-\hat{g}(z_i) \rVert^2 \right)}_{=:\tilde{a}_n} \underbrace{\left( \frac{1}{n} \sum_{i=1}^n \lVert \varepsilon_i \rVert^2 \lVert \xi_i \rVert^2 \right)}_{=:\tilde{U}_n} }. 
\end{align*}
By Cauchy--Schwarz, we have $\tilde{a}_n \leq n M_{n, \dist}^f M_{n, \dist}^g \overset{P}{\rightrightarrows} 0$. We have $\tilde{U}_n \overset{P}{\rightrightarrows} \lVert \mathscr{C} \rVert_{\TR}$ by Proposition~\ref{prop:uniform-hilbert-lln}. The family $(\lVert \mathscr{C} \rVert_{\TR})_{\dist \in \tilde{\mathcal{P}}_0}$ is uniformly tight by the assumption that $\mathbb{E}( \lVert \varepsilon_\dist \rVert^{2+ \eta} \lVert \xi_\dist \rVert^{2+\eta})$ is uniformly bounded, since this also yields a bound on $\mathbb{E}( \lVert \varepsilon_\dist \rVert^2 \lVert \xi_\dist \rVert^2) = \lVert \mathscr{C} \rVert_{\TR}$ thus Proposition~\ref{prop:uniform-convergence-in-probability-preserved-under-continuous-operations-when-tight} yields that $\sqrt{\tilde{a}_n \tilde{U}_n} \overset{P}{\rightrightarrows} 0$.

The remaining three cases have an $f$ and a $g$ variant where the roles of $f$ and $g$ and $\varepsilon$ and $\xi$ are swapped. We only show one variant of each, since the arguments are identical. The $f$-variant of the third case is 
\[
\RN{4}_n^f = \left \lVert \frac{1}{n} \sum_{i=1}^n [(f(z_i)-\hat{f}(z_i)) \otimes \xi_i]^{\otimes_{\HS}2} \right\rVert_{\TR} \leq \frac{1}{n} \sum_{i=1}^n \lVert f(z_i)-\hat{f}(z_i)\rVert^2 \lVert \xi_i \rVert^2 =: \tilde{b}_n.
\]
If we can show that $\mathbb{E}(\tilde{b}_n \cond X^{(n)}, Z^{(n)}) \overset{P}{\rightrightarrows} 0$, we have that $\tilde{b}_n \overset{P}{\rightrightarrows} 0$ by Lemma~\ref{lem: Convergence to zero if conditional convergence to zero} and hence $\RN{4}_n^f  \overset{P}{\rightrightarrows} 0$. This holds since
\[
\mathbb{E}_\dist ( \tilde{b}_n \cond X^{(n)}, Z^{(n)} ) = \frac{1}{n}  \sum_{i=1}^n \lVert f(z_i)-\hat{f}(z_i)\rVert^2   \mathbb{E}_\dist \left( \lVert \xi_i \rVert^2 \cond X^{(n)}, Z^{(n)} \right) = \tilde{M}^f_{n, \dist}  \overset{P}{\rightrightarrows} 0,
\]
by assumption.

The $f$-variant of the fourth case is, by applying the Cauchy--Schwarz inequality,
\begin{align*}
\RN{5}_n^f &=  \left \lVert \frac{1}{n} \sum_{i=1}^n [(f(z_i)-\hat{f}(z_i)) \otimes (g(z_i)-\hat{g}(z_i))] \otimes_{\HS} [(f(z_i)-\hat{f}(z_i)) \otimes \xi_i] \right\rVert_{\TR} \\
& \leq  \frac{1}{n} \sum_{i=1}^n \lVert f(z_i)-\hat{f}(z_i)\rVert^2 \lVert g(z_i)-\hat{g}(z_i) \rVert \lVert \xi_i \rVert \\
& \leq \sqrt{\underbrace{\left(\frac{1}{n} \sum_{i=1}^n \lVert f(z_i)-\hat{f}(z_i)\rVert^2 \lVert g(z_i)-\hat{g}(z_i) \rVert^2 \right)}_{\tilde{a}_n} \underbrace{\left( \frac{1}{n} \sum_{i=1}^n \lVert f(z_i)-\hat{f}(z_i)\rVert^2 \lVert \xi_i \rVert^2 \right)}_{\tilde{b}_n} } .
\end{align*}
We saw above that $\tilde{a}_n \overset{P}{\rightrightarrows} 0$ and $\tilde{b}_n \overset{P}{\rightrightarrows} 0$, hence by Proposition~\ref{prop:uniform-convergence-in-probability-preserved-under-continuous-operations-when-tight}, $\sqrt{\tilde{a}_n \tilde{b}_n} \overset{P}{\rightrightarrows} 0$.

For the $f$-variant of the fifth and final case, we get, by applying the Cauchy--Schwarz inequality again,
\begin{align*}
\RN{6}_n^f &=  \left \lVert \frac{1}{n} \sum_{i=1}^n [(f(z_i)-\hat{f}(z_i)) \otimes \xi_i] \otimes_{\HS} [\varepsilon_i \otimes \xi_i] \right\rVert_{\TR} \leq  \frac{1}{n} \sum_{i=1}^n \lVert f(z_i)-\hat{f}(z_i) \rVert \lVert \varepsilon_i \rVert  \lVert \xi_i \rVert^2 \\
& \leq  \sqrt{\underbrace{\left( \frac{1}{n} \sum_{i=1}^n \lVert f(z_i)-\hat{f}(z_i) \rVert^2 \lVert \xi_i \rVert^2 \right)}_{\tilde{b}_n}  \underbrace{\left( \frac{1}{n} \sum_{i=1}^n \lVert \varepsilon_i \rVert^2 \lVert \xi_i \rVert^2 \right)}_{\tilde{U}_n}}.
\end{align*}
We can repeat the arguments used above yielding $\sqrt{\tilde{a}_n \tilde{U}_n} \overset{P}{\rightrightarrows} 0$ to show that $\sqrt{\tilde{b}_n \tilde{U}_n} \overset{P}{\rightrightarrows} 0$ hence $\RN{6}_n^f \overset{P}{\rightrightarrows} 0$ as desired.
\end{proof}

\subsubsection{Proof of Theorem~\ref{thm:level of test}}
\begin{proof}
Let $W$ be distributed as $\lVert \mathcal{N}(0, \mathscr{C}_\dist) \rVert_{\HS}^2$ when the background measure is $\mathbb{P}_\dist$.
Recalling the notation from Lemma~\ref{lem:uniform-convergence-of-quantile-map}, since
%Since  \Rajen{Did we change notation? $q_\alpha$?} \Anton{In the old version where we used resampling to get $p$-values, we didn't account for this step in the proof. Thus $q_\alpha$ is in estimate of $q(\hat{\mathscr{C}})$.}
$$
\mathbb{P}_\dist(\psi_n = 1) = \mathbb{P}_\dist(T_n > q(\hat{\mathscr{C}}))
$$
we need to show that
\[
\lim_{n \to \infty} \sup_{\dist \in \tilde{\mathcal{P}}_0} \left| \mathbb{P}_\dist(T_n  > q(\hat{\mathscr{C}}))  - \alpha \right| = 0,
\]
which amounts to finding, for each $\epsilon > 0$, an $N \in \mathbb{N}$, such that for all $n \geq N$,
\begin{equation}
  \label{eq:uniform-level-upper-bound-proof}
\sup_{\dist \in \tilde{\mathcal{P}}_0} \mathbb{P}_\dist(T_n > q(\hat{\mathscr{C}}))  <  \alpha + \epsilon
\end{equation}
and
\begin{equation}
  \label{eq:uniform-level-lower-bound-proof}
\inf_{\dist \in \tilde{\mathcal{P}}_0} \mathbb{P}_\dist(T_n > q(\hat{\mathscr{C}}))  >  \alpha - \epsilon .
\end{equation}

To show \eqref{eq:uniform-level-upper-bound-proof}, take $\delta > 0$ (to be fixed later). If $|q(\hat{\mathscr{C}}) - q(\mathscr{C}_\dist)| < \delta$ and $T_n  > q(\hat{\mathscr{C}})$, then $T_n  > q(\mathscr{C}_\dist) - \delta$, so
\[
 \mathbb{P}_\dist(T_n  > q(\hat{\mathscr{C}})) \leq  \mathbb{P}_\dist(T_n  > q(\mathscr{C}_\dist) - \delta) +  \mathbb{P}_\dist(|q(\hat{\mathscr{C}}) -  q(\mathscr{C}_\dist)| \geq \delta ).
\]
Taking suprema and rewriting, we get
\begin{align*}
\sup_{\dist \in \tilde{\mathcal{P}}_0} \mathbb{P}_\dist(T_n > q(\hat{\mathscr{C}}_\dist)) &\leq  \overbrace{\sup_{\dist \in \tilde{\mathcal{P}}_0} [ \mathbb{P}_\dist(T_n  > q(\mathscr{C}_\dist) - \delta) - \mathbb{P}_\dist(W > q(\mathscr{C}_\dist) - \delta)]}^{=:\RN{1}_n} \\
&+ \underbrace{\sup_{\dist \in \tilde{\mathcal{P}}_0} [ \mathbb{P}_\dist(W > q(\mathscr{C}_\dist) - \delta) - \alpha]}_{=:\RN{2}_n}   +  \underbrace{\sup_{\dist \in \tilde{\mathcal{P}}_0} \mathbb{P}_\dist(|q(\hat{\mathscr{C}}) -  q(\mathscr{C}_\dist)| \geq \delta )}_{=:\RN{3}_n}  + \alpha.
\end{align*}

We seek to show that, if $n$ is sufficiently large, we can make each of the terms $\RN{1}_n$, $\RN{2}_n$ and $\RN{3}_n$ less than $\epsilon/3$ such that 
\[
\sup_{\dist \in \tilde{\mathcal{P}}_0} \mathbb{P}_\dist(T_n  > q(\hat{\mathscr{C}})) < \alpha + \epsilon,
\]
as desired. 

We note first that
\begin{equation}
	\begin{aligned}
		\label{eq:thm3-rn1}
|\RN{1}_n| &\leq \sup_{\dist \in \tilde{\mathcal{P}}_0} | \mathbb{P}_\dist(T_n^{1/2}  > \{q(\mathscr{C}_\dist) - \delta\}^{1/2}) - \mathbb{P}_\dist(W^{1/2} > \{q(\mathscr{C}_\dist) - \delta\}^{1/2})|\\
&\leq \sup_{\dist \in \tilde{\mathcal{P}}_0} \sup_{x \in \mathbb{R}}| \mathbb{P}_\dist(T_n^{1/2}  > x) - \mathbb{P}_\dist(W^{1/2} > x)|.
	\end{aligned}
\end{equation}
For each $\dist \in \tilde{\mathcal{P}}_0$, $W$ has the same distribution as
\[
\sum_{k=1}^\infty \lambda_k^{\dist} V_k^2,
\]
where $\lambda_k^{\dist}$ is the $k$th eigenvalue of $\mathscr{C}_\dist$ and $(V_k)_{k \in \mathbb{N}}$ is a sequence of independent standard Gaussian random variables. We have assumed that the operator norm of $(\mathscr{C}_\dist)_{\dist \in \tilde{\mathcal{P}}_0}$ is bounded away from zero which implies that $\lambda_1^\dist$ is bounded away from zero. Thus, the family $(\lambda_1^{\dist} V_1^2)_{\dist \in \tilde{\mathcal{P}}_0}$ is uniformly absolutely continuous with respect to the Lebesgue measure by Lemma~\ref{lem:uniform-absolute-continuity-of-gamma-family}. Theorem~\ref{thm:independent-convolution-preserves-uniform-absolute-continuity} yields that $W$ is also uniformly absolutely continuous with respect to the Lebesgue measure and Lemma~\ref{lem:uniform-absolute-continuity-preserved-under-square-root} yields that the same is true for $W^{1/2}$, since $W$ is uniformly tight by the assumed uniform bound on $\mathbb{E}_\dist(\lVert \epsilon_\dist \rVert^2 \lVert \xi_\dist \rVert^2)$. Further, Corollary \ref{corollary:uniform-absolute-continuity-wrt-lebesgue-iff-wrt-gaussian} yields that $W^{1/2}$ is also uniformly absolutely continuous with respect to the standard Gaussian on $\mathbb{R}$. Proposition~\ref{prop:uniform-continuous-mapping-theorem} and Theorem~\ref{thm:asymptotic normality and covariance estimation} $T_n^{1/2} \overset{\mathcal{D}}{\rightrightarrows} W^{1/2}$ since $\lVert \cdot \rVert_{\HS}$ is Lipschitz. Finally since we argued that $W^{1/2}$ is uniformly absolutely continuous with respect to the standard Gaussian on $\mathbb{R}$, Proposition~\ref{prop:uniform-convergence-of-distribution-functions} yields that we can make the bound in \eqref{eq:thm3-rn1} less than $\epsilon/3$ for $n$ sufficiently large.

For the $\RN{2}_n$ term, recall that $\alpha =  \mathbb{P}_\dist(W > q(\mathscr{C}_\dist))$, and thus
\[
 \mathbb{P}_\dist(W > q(\mathscr{C}_\dist) - \delta) - \alpha =  \mathbb{P}_\dist(W \in [ q(\mathscr{C}_\dist) - \delta,  q(\mathscr{C}_\dist)]) .
\]
By the uniform absolute continuity of $W$ with respect to the Lebesgue measure $\lambda$, we may fix $\delta$ such that $\sup_{\dist \in \tilde{\mathcal{P}}_0}  \mathbb{P}_\dist(W \in B) < \epsilon/3$ whenever $\lambda(B) < 2\delta$. This implies that $\RN{2}_n < \epsilon/3$.

For the $\RN{3}_n$ term, Theorem~\ref{thm:asymptotic normality and covariance estimation} yields $\hat{\mathscr{C}} \overset{P}{\rightrightarrows} \mathscr{C}_\dist$ and since Lemma~\ref{lem:uniform-convergence-of-quantile-map} yields that $q$ is uniformly continuous, Proposition~\ref{prop:uniform-continuous-mapping-theorem} yields $q(\hat{\mathscr{C}}) \overset{P}{\rightrightarrows} q(\mathscr{C}_\dist)$. Thus, the third term is less than $\epsilon/3$ when $n$ is large enough.

To show \eqref{eq:uniform-level-lower-bound-proof}, note first that, as before, if $|q(\hat{\mathscr{C}}) - q(\mathscr{C}_\dist)| <\delta$ and $T_n  > q(\mathscr{C}_\dist) + \delta$, then $T_n  > q(\hat{\mathscr{C}}) $ and hence
\begin{equation}
	\begin{aligned}
		\label{eq:thm3-lower-bound}
\mathbb{P}_\dist(T_n  > q(\hat{\mathscr{C}})) &\geq \mathbb{P}_\dist((T_n  > q(\mathscr{C}_\dist) + \delta) \cap (|q(\hat{\mathscr{C}}) -  q(\mathscr{C}_\dist)| <\delta) )\\
 &\geq \mathbb{P}_\dist(T_n  > q(\mathscr{C}_\dist) + \delta) - \mathbb{P}_\dist(|q(\hat{\mathscr{C}}) -  q(\mathscr{C}_\dist)| \geq \delta).
	\end{aligned}
\end{equation}
The final step uses that for any measurable sets $A$ and $B$,
\[
\mathbb{P}(A \cap B) = \mathbb{P}(A) + \mathbb{P}(B) - \mathbb{P}(A \cup B) =  \mathbb{P}(A) - \mathbb{P}(B^c) + 1 - \mathbb{P}(A \cup B) \geq \mathbb{P}(A) - \mathbb{P}(B^c).
\]
This lets us continue using similar arguments as for \eqref{eq:uniform-level-upper-bound-proof}, proving the statement.
\end{proof}

\subsection{Proof of Theorem~\ref{thm: consistency of test}}

\begin{proof}
To argue that the modified GHCM satisfies \eqref{eq:level}, we can repeat the arguments of Theorem~\ref{thm:asymptotic normality and covariance estimation} and Theorem~\ref{thm:level of test} replacing conditioning on $X^{(n)}$ and $Z^{(n)}$ with conditioning on $Z^{(n)}$ and $A$ and conditioning on $Y^{(n)}$ and $Z^{(n)}$ with conditioning on $Z^{(n)}$ and $A$.

For the first claim that $\tilde{\mathscr{T}}_n \overset{\mathcal{D}}{\rightrightarrows} \mathcal{N}(0, \mathscr{C}_\dist)$, we can repeat the decomposition of the proof of Theorem~\ref{thm:asymptotic normality and covariance estimation} and write
\[
\frac{1}{\sqrt{n}} \sum_{i=1}^n (\mathscr{R}_i - \mathscr{K}_\dist) = \underbrace{\frac{1}{\sqrt{n}} \sum_{i=1}^n (\varepsilon_i \otimes \xi_i - \mathscr{K}_\dist)}_{=:U_n} + a_n + b_n + c_n,
\]
where $a_n$, $b_n$ and $c_n$ are as in the proof of Theorem~\ref{thm:asymptotic normality and covariance estimation}. We have $U_n \overset{\mathcal{D}}{\rightrightarrows} \mathcal{N}(0, \mathscr{C}_\dist)$ over $\mathcal{Q}$ by Proposition~\ref{prop:uniform-clt} $a_n {\overset{P}{\rightrightarrows}} 0$ over $\mathcal{Q}$ by the same argument as in the proof of Theorem~\ref{thm:asymptotic normality and covariance estimation}. The argument of the proof of Theorem~\ref{thm:asymptotic normality and covariance estimation} to show that $b_n {\overset{P}{\rightrightarrows}} 0$ and $c_n {\overset{P}{\rightrightarrows}} 0$ will also work here if we replace conditioning as we did for the first claim.

For the second claim that $\lVert \hat{\mathscr{C}} - \mathscr{C} \rVert_{\TR} \overset{P}{\rightrightarrows} 0$, note that by the $\tilde{\mathscr{T}}_n$ result, Proposition~\ref{prop:uniform-slutsky} and Proposition~\ref{prop:uniform-convergence-in-probability-preserved-under-continuous-operations-when-tight},
\[
\frac{1}{n} \sum_{i=1}^n \mathscr{R}_i = \frac{1}{\sqrt{n}} \cdot \frac{1}{\sqrt{n}} \sum_{i=1}^n (\mathscr{R}_i-\mathscr{K}_\dist) + \mathscr{K}_\dist \overset{P}{\rightrightarrows}_\mathcal{Q} \mathscr{K}_\dist .
\]
Hence, by Proposition~\ref{prop:uniform-convergence-in-probability-preserved-under-continuous-operations-when-tight},
\[
\left( \frac{1}{n} \sum_{i=1}^n \mathscr{R}_i \right) \otimes_{\HS} \left( \frac{1}{n} \sum_{i=1}^n \mathscr{R}_i \right) \overset{P}{\rightrightarrows}_\mathcal{Q} \mathscr{K}_\dist \otimes_{\HS} \mathscr{K}_\dist,
\]
since the mapping $(\mathscr{A},\mathscr{B}) \mapsto \mathscr{A} \otimes_{\HS} \mathscr{B}$ is continuous. We can now repeat the remaining arguments of the proof of Theorem~\ref{thm:asymptotic normality and covariance estimation} while again replacing conditioning as we did in the proof of the first claim to yield the desired result.

For the final claim that for large enough $n$ the GHCM has power greater than $\beta$ over alternatives where $\lVert \sqrt{n} \mathscr{K}_\dist \rVert_{\HS} > c$, let $W$ be distributed as $\lVert \mathcal{N}(0, \mathscr{C}_\dist) \rVert_{\HS}^2$ when the background measure is $\mathbb{P}_\dist$ for $\dist \in \mathcal{Q}$.
Let $q$ denote the mapping that sends a covariance operator $\mathscr{C}$ to the $1-\alpha$ quantile of the distribution of $\lVert \mathcal{N}(0, \mathscr{C} )\rVert_{\HS}^2$ as in Lemma~\ref{lem:uniform-convergence-of-quantile-map}. By similar arguments as \eqref{eq:thm3-lower-bound} in the proof of Theorem~\ref{thm:level of test}, we get that for any $\delta > 0$, $c >0$ and $n \in \mathbb{N}$,
\[
\inf_{\dist \in \mathcal{Q}_{c, n}} \mathbb{P}_\dist(T_n > q(\hat{\mathscr{C}})) \geq \inf_{\dist \in \mathcal{Q}_{c, n}}  \mathbb{P}_\dist(T_n > q(\mathscr{C}_\dist) + \delta) - \sup_{\dist \in \mathcal{Q}_{c, n}}  \mathbb{P}_\dist(|q(\hat{\mathscr{C}}) -  q(\mathscr{C}_\dist)| \geq \delta) .
\]
Defining $\tilde{T}_n^{1/2} := \lVert \tilde{\mathscr{T}}_n \rVert_{\HS}$, by the reverse triangle inequality
\[
T_n^{1/2} = \left\lVert \tilde{\mathscr{T}}_n + \sqrt{n} \mathscr{K}_\dist \right\rVert_{\HS}  \geq \left| \tilde{T}_n^{1/2} - \sqrt{n} \lVert \mathscr{K}_\dist \rVert_{\HS} \right|  \geq \sqrt{n} \lVert \mathscr{K}_\dist \rVert_{\HS} -  \tilde{T}_n^{1/2},
\]
and hence
\[
\inf_{\dist \in \mathcal{Q}_{c, n}}  \mathbb{P}_\dist(T_n > q(\mathscr{C}_\dist) + \delta) \geq \inf_{\dist \in \mathcal{Q}_{c, n}}  \mathbb{P}_\dist(\sqrt{n} \lVert \mathscr{K}_\dist \rVert_{\HS} - \tilde{T}_n^{1/2} > \{q(\mathscr{C}_\dist) + \delta\}^{1/2}).
\]
Now since we are taking an infimum over a set where $\sqrt{n} \lVert \mathscr{K}_\dist \rVert_{\HS} > c$, we have
\[
\inf_{\dist \in \mathcal{Q}_{c, n}}  \mathbb{P}_\dist(\sqrt{n} \lVert \mathscr{K}_\dist \rVert_{\HS} - \tilde{T}_n^{1/2} > \{q(\mathscr{C}_\dist) + \delta\}^{1/2}) \geq \inf_{\dist \in \mathcal{Q}_{c, n}}  \mathbb{P}_\dist(c - \tilde{T}_n^{1/2} > \{q(\mathscr{C}_\dist) + \delta\}^{1/2}),
\]
and thus combining all the above yields
\begin{align*}
\inf_{\dist \in \mathcal{Q}_{c, n}} \mathbb{P}_\dist(T_n > q(\hat{\mathscr{C}_\dist}))
& \geq \overbrace{\inf_{\dist \in \mathcal{Q}_{c, n}}  [\mathbb{P}_\dist(c - \tilde{T}_n^{1/2} > \{q(\mathscr{C}_\dist) + \delta\}^{1/2}) - \mathbb{P}_\dist(c - W^{1/2} > \{q(\mathscr{C}_\dist) + \delta\}^{1/2}) ]}^{=:\RN{1}_n} \\
& + \underbrace{\inf_{\dist \in \mathcal{Q}_{c, n}}  \mathbb{P}_\dist(c - W^{1/2} > \{q(\mathscr{C}_\dist) + \delta\}^2)}_{=:\RN{2}_n} - \underbrace{\sup_{\dist \in \mathcal{Q}_{c, n}}  \mathbb{P}_\dist(|q(\hat{\mathscr{C}}) -  q(\mathscr{C}_\dist)| \geq \delta)}_{=:\RN{3}_n} .
\end{align*}
If we can show that for $n$ sufficiently large we can make $\RN{1}_n + \RN{2}_n + \RN{3}_n \geq \beta$, we will be done.

For the $\RN{1}_n$ term, we can write
$$
\RN{1}_n \geq -\sup_{\dist \in \mathcal{Q}_{c, n}} \sup_{x \in \mathbb{R}}  |\mathbb{P}_\dist(\tilde{T}_n^{1/2} < x) - \mathbb{P}_\dist(W^{1/2} < x) |.
$$
By the first claim proven above and Proposition~\ref{prop:uniform-continuous-mapping-theorem}, $\tilde{T}_n^{1/2} \overset{\mathcal{D}}{\rightrightarrows} W^{1/2}$. We can therefore repeat the arguments used to deal with the $\RN{1}_n$ term in the proof of Theorem~\ref{thm:level of test} to see that for $n$ sufficiently large we have $\RN{1}_n \geq -(1-\beta)/3$.

For the $\RN{2}_n$ term, we can write
\[
\RN{2}_n = 1- \sup_{\dist \in \mathcal{Q}_{c, n}}  \mathbb{P}_\dist(W^{1/2} + \{q(\mathscr{C}_\dist) + \delta\}^{1/2} \geq c ).
\]
Hence, by uniform tightness of $(W^{1/2} + \{q(\mathscr{C}_\dist) + \delta\}^{1/2})_{\dist \in \mathcal{Q}}$ we can find $c$ such that $\sup_{\dist \in \mathcal{Q}_{c, n}}  \mathbb{P}_\dist(W^{1/2} + \{q(\mathscr{C}_\dist) + \delta\}^{1/2} \geq c ) < (1-\beta)/3 $ which implies $\RN{2}_n > 1-(1-\beta)/3$.

For the $\RN{3}_n$ term, we can repeat the arguments for the $\RN{3}_n$ term in the proof of Theorem~\ref{thm:level of test} to show that $\RN{3}_n \overset{P}{\rightrightarrows} 0$. Hence, for sufficiently large $n$, we have $\RN{3}_n > -(1-\beta)/3$.

Putting things together, we have for $n$ sufficiently large that
\[
\inf_{\dist \in \mathcal{Q}_{c, n}} \mathbb{P}_\dist(T_n > q(\hat{\mathscr{C}})) \geq \beta. \qedhere
\]
\end{proof}

\subsection{Proof of Theorem~\ref{thm:kernel-regression-ghcm} and related results}
\label{sec:4.3-proofs}

We first prove a representer theorem \citep{Kimeldorf1970,Schoelkopf2001} for scalar-on-function regression which we use to provide bounds on the in-sample error of the Hilbertian linear model in Lemma~\ref{lem:deterministic-regression-mse-bound}.

\begin{lemma}
\label{lem:representer-theorem}
Let $\mathcal{H}$ denote a Hilbert space with norm $\lVert \cdot \rVert$, $x_1, \dots, x_n \in \mathbb{R}$, $z_1, \dots, z_n \in \mathcal{H}$ and $\gamma > 0$ Let $K$ be an $n \times n$ matrix where $K_{i, j}:= \langle z_i, z_j \rangle$ and let $x = (x_1, \dots, x_n)^\top  \in \mathbb{R}^n$. Then $\hat{\beta}$ minimises
\[
  L_1(\beta) = \sum_{i=1}^n (x_i - \langle \beta, z_i \rangle)^2 + \gamma \lVert \beta \rVert^2
\]
over $\beta \in \mathcal{H}$ if and only if $\hat{\beta} = \sum_{i=1}^n \hat{\alpha}_i z_i$ and $\hat{\alpha} = (\hat{\alpha}_1, \dots, \hat{\alpha}_n)^\top  \in \mathbb{R}^n$ minimises
\[
  L_2(\alpha) =  \lVert x - K \alpha \rVert_2^2 + \gamma \alpha^\top  K \alpha
\]
over $\mathbb{R}^n$ where $\lVert \cdot \rVert_2$ denotes the standard Euclidean norm on $\mathbb{R}^n$.
\end{lemma}
\begin{proof}
Assume that $\hat{\beta}$ minimises $L_1$. Write $\hat{\beta} = u + v$ where $u \in \mathcal{U} := \textrm{span}(z_1, \dots, z_n)$ and $v \in \mathcal{U}^{\perp}$. Since 
\[
  \langle \hat{\beta}, z_i \rangle = \langle u , z_i \rangle,
\]
the first term of $L_1$ only depends on the quantity $u$. Also, by Pythagoras' theorem,
\[
\lVert \hat{\beta} \rVert^2 = \| u\|^2 + \|v\|^2 \geq \lVert u \rVert^2.
\]
Thus, $v=0$ by optimality of $\hat{\beta}$, and so $\hat{\beta}$ can be written
\[
  \hat{\beta} = \sum_{i=1}^n \hat{\alpha}_i z_i
\]
for some $\hat{\alpha} \in \R^n$. But now that $\hat{\beta}$ is known to have this form, it can be seen that $\hat{\alpha}^\top  K \hat{\alpha} = \lVert \hat{\beta} \rVert^2$ and 
\[
  \sum_{i=1}^n (x_i - \langle \hat{\beta}, z_i \rangle)^2 = \sum_{i=1}^n \left(x_i - \sum_{j=1}^n \hat{\alpha}_j \langle z_j, z_i \rangle \right)^2 =  \lVert x - K \hat{\alpha} \rVert_2^2,
\]
hence $\hat{\alpha}$ minimises $L_2$. 

Assume now that $\hat{\alpha} \in \R^n$ minimises $L_2$ and $\hat{\beta} = \sum_{i=1}^n \hat{\alpha}_i z_i$. Clearly, $L_2(\hat{\alpha}) = L_1(\hat{\beta})$. For any $\tilde{\beta} \in \mathcal{H}$, we can write $\tilde{\beta} = \tilde{u} + \tilde{v}$ with $\tilde{u} \in \mathcal{U}$ and $\tilde{v} \in \mathcal{U}^{\perp}$ as before. By similar arguments as above,
\[
L_1(\tilde{\beta}) \geq L_1(\tilde{u}) .
\]
However, $\tilde{u} = \sum_{i=1}^n \tilde{\alpha}_i z_i$, hence by optimality of $\hat{\alpha}$, we have
\[
  L_1(\tilde{u}) = L_2(\tilde{\alpha}) \geq L_2(\hat{\alpha}) = L_1(\hat{\beta}),
\]
proving that $\hat{\beta}$ minimises $L_1$ as desired.
\end{proof}

\begin{lemma}
\label{lem:deterministic-regression-mse-bound}
Let $n \in \mathbb{N}$ be fixed. Consider the estimator $\hat{\mathscr{S}}$ \eqref{eq:penalised-likelihood-criterion-hilbertian-linear-model} in the Hilbertian linear model which is a function of $x_1, \dots, x_n, z_1, \dots, z_n$ and let $\sigma^2 > 0$ be such that $\mathbb{E}( \lVert \varepsilon \rVert^2 \cond Z) \leq \sigma^2$ almost surely. Let $K$ be the $n \times n$ matrix where $K_{ij} := \langle z_i, z_j \rangle$ and let $(\hat{\mu}_i)_{i=1}^n$ denote the eigenvalues of $K$.
Then, letting $Z^{(n)} := (z_1, \dots, z_n)$,
\begin{equation} \label{eq:conditional-mse-bound}
\frac{1}{n}  \mathbb{E}\left( \sum_{i=1}^n \lVert \mathscr{S}(z_i) - \hat{\mathscr{S}}(z_i) \rVert^2 \cond Z^{(n)} \right) \leq \frac{\sigma^2}{\gamma } \frac{1}{n} \sum_{i=1}^n \min(\hat{\mu}_i/4, \gamma) + \lVert \mathscr{S} \rVert_{\HS}^2 \frac{\gamma}{4n} 
\end{equation}
almost surely. 
\end{lemma}
\begin{proof}
Let $(e_k)_{k \in \mathbb{N}}$ denote a basis of $\mathcal{H}_X$ and write $\langle \cdot, \cdot \rangle_X$ and $\langle \cdot, \cdot \rangle_Z$ for the inner products and $\lVert \cdot \rVert_X$ and $\lVert \cdot \rVert_Z$ for the norms on $\mathcal{H}_X$ and $\mathcal{H}_Z$, respectively. Then
\begin{align}
  \sum_{i=1}^n \lVert \mathscr{S}(z_i) - \hat{\mathscr{S}}(z_i) \rVert_X^2 &= \sum_{k=1}^\infty \sum_{i=1}^n ( \langle \mathscr{S}(z_i), e_k \rangle_X - \langle \hat{\mathscr{S}}(z_i) , e_k \rangle_X)^2 \notag \\
  &= \sum_{k=1}^\infty \sum_{i=1}^n ( \langle z_i, \mathscr{S}^*(e_k) \rangle_Z - \langle z_i , \hat{\mathscr{S}}^*(e_k) \rangle_Z)^2  \label{eq:mse-decomposition}
\end{align}
and similarly we can rewrite the penalised square-error criterion in \eqref{eq:penalised-likelihood-criterion-hilbertian-linear-model} as
\begin{equation*} 
  \sum_{i=1}^n \lVert x_i - \tilde{\mathscr{S}}(z_i) \rVert_X^2 + \gamma \lVert \tilde{\mathscr{S}} \rVert_{\HS}^2 = \sum_{k=1}^\infty \left[ \sum_{i=1}^n  ( \langle x_i, e_k \rangle_X - \langle z_i , \tilde{\mathscr{S}}^*(e_k) \rangle_Z)^2 + \gamma \lVert \tilde{\mathscr{S}} e_k \rVert^2 \right] .
\end{equation*}
%\Rajen{final term needs to be rewritten}
Since each of the terms in square brackets can be chosen independently of each other, we have
\[
\hat{\beta}_k := \hat{\mathscr{S}}^*_\gamma(e_k) = \argmin_{\beta \in \mathcal{H}_Z} \sum_{i=1}^n (\langle x_i , e_k \rangle_X - \langle z_i, \beta \rangle_Z )^2 + \gamma \lVert \beta \rVert_Z^2 .
\]
A bit of matrix calculus combined with Lemma \ref{lem:representer-theorem} yields that
\[
  (\langle z_1, \hat{\beta}_k \rangle_Z, \dots, \langle z_n , \hat{\beta}_k \rangle_Z)^\top  = K(K+\gamma I)^{-1}X^{(n)}_k,
\]
where $I$ is the $n \times n$ identity matrix and $X^{(n)}_k := (\langle x_1, e_k \rangle_X, \dots, \langle x_n, e_k \rangle_X)^\top $. Defining $\beta_k := \mathscr{S}^{*}(e_k)$, we can write $\beta_k = u_k + v_k$ where $u_k \in \mathcal{U} := \textrm{span}(z_1, \dots, z_n)$ and $v \in \mathcal{U}^{\perp}$. Writing $u_k = \sum_{j=1}^n \alpha_{k,j}z_j$ where $\alpha_k = (\alpha_{k, 1}, \dots, \alpha_{k, n})^\top  \in \mathbb{R}^n$, we have for $i \in \{1, \dots, n\}$,
\[
  \langle z_i , \beta_k \rangle_Z = \langle z_i, u_k \rangle_Z = \left\langle z_i , \sum_{j=1}^n \alpha_{k, j} z_j \right\rangle_Z = \sum_{j=1}^n \alpha_{k, j} \langle z_i, z_j \rangle_Z.
\]
This entails
\[
  (\langle z_1, \beta_k \rangle_Z, \dots, \langle z_n , \beta_k \rangle_Z)^\top  = K\alpha_k .
\]
Let $K = U D U^\top $ be the eigendecomposition of $K$, where $D_{ii} = 
\hat{\mu}_i$, and let $\theta_k := U^\top  K \alpha_k$. Let $\varepsilon^{(n)}_k := (\langle \varepsilon_1 , e_k \rangle_X, \dots ,  \langle \varepsilon_n , e_k \rangle_X)^\top  \in \R^n$ and note that $X^{(n)}_k = K\alpha_k + \varepsilon^{(n)}_k$. Letting $\lVert \cdot \rVert_2$ denote the Euclidean norm, $n$ times the left-hand side of equation (\ref{eq:conditional-mse-bound}) can now be written (using equation \eqref{eq:mse-decomposition})
\begin{align}
  &\mathbb{E} \left[ \sum_{k=1}^\infty \lVert K(K+\gamma I)^{-1}(U\theta_k + \varepsilon^{(n)}_k) - U \theta_k \rVert^2_2 \cond Z^{(n)} \right] \notag \\
  &= \mathbb{E} \left[ \sum_{k=1}^\infty \lVert DU^\top (UDU^\top +\gamma I)^{-1}(U\theta_k + \varepsilon^{(n)}_k) - \theta_k \rVert^2_2  \cond  Z^{(n)}  \right] \notag \\
  &= \mathbb{E} \left[ \sum_{k=1}^\infty \lVert D(D+\gamma I)^{-1}(\theta_k + U^\top \varepsilon^{(n)}_k) - \theta_k \rVert_2^2  \cond Z^{(n)} \right] \notag \\
  &=  \sum_{k=1}^\infty \lVert (D(D+\gamma I)^{-1}-I)\theta_k \rVert_2^2 
  + \mathbb{E} \left[ \sum_{k=1}^\infty \lVert D(D+\gamma I)^{-1} U^\top \varepsilon^{(n)}_k \rVert_2^2 \cond Z^{(n)} \right] \label{eq:decomp}
\end{align}
where the final equality uses that the first term is a function of $Z^{(n)}$ and the conditional expectation of the cross term in the sum of squares is $0$, since $\mathbb{E}(\varepsilon^{(n)}_k \cond Z^{(n)}  ) = 0$.

The second term of \eqref{eq:decomp} may be simplified as follows:
\begin{align*}
  &\mathbb{E} \left[ \sum_{k=1}^\infty \lVert D(D+\gamma I)^{-1} U^\top \varepsilon^{(n)}_k \rVert_2^2 \cond Z^{(n)} \right] \\
  &=  \mathbb{E} \left[ \sum_{k=1}^\infty  \textrm{tr} \left(  D(D+\gamma I)^{-1}  U^\top  \varepsilon^{(n)}_k(\varepsilon^{(n)}_k)^\top  U D(D+\gamma I)^{-1}  \right)  \cond Z^{(n)} \right] \\
  &=   \textrm{tr} \biggl(  D(D+\gamma I)^{-1}  U^\top  \underbrace{\mathbb{E} \left[ \sum_{k=1}^\infty \varepsilon^{(n)}_k(\varepsilon^{(n)}_k)^\top   \cond Z^{(n)} \right]}_{\Sigma_{\varepsilon|Z}} U D(D+\gamma I)^{-1}  \biggr),
\end{align*}
where we have used that only $\varepsilon^{(n)}_k$ is not a function of $Z^{(n)}$ and linearity of conditional expectations and the trace. Note that $\Sigma_{\varepsilon \cond Z}$ is a diagonal matrix with $i$th diagonal entry equal to 
\[
  \mathbb{E} \left [ \sum_{k=1}^\infty \langle \varepsilon_i, e_k \rangle_X^2 \cond Z^{(n)} \right] =   \mathbb{E} \left [ \lVert \varepsilon_i \rVert_X^2 \cond z_i \right] ,
\]
hence we can bound each diagonal term by $\sigma^2$ by assumption. This implies that 
\begin{align*}
  \textrm{tr} \biggl(  D(D+\gamma I)^{-1}  U^\top  \Sigma_{\varepsilon \cond Z} U D(D+\gamma I)^{-1}  \biggr) &\leq \sigma^2\textrm{tr} \biggl(  D(D+\gamma I)^{-1} D(D+\gamma I)^{-1}  \biggr)\\
    &= \sigma^2 \sum_{i=1}^n \frac{\hat{\mu}_i^2}{(\hat{\mu}_i+\gamma)^2}.
\end{align*}

The first term of \eqref{eq:decomp} can be dealt with by noting that
\begin{align*}
  \sum_{k=1}^\infty \lVert (D(D+\gamma I)^{-1}-I)\theta_k \rVert_2^2 &= \sum_{k=1}^\infty \sum_{i=1}^n \frac{\gamma^2 \theta_{k, i}^2}{(\hat{\mu}_i+\gamma)^2} = \sum_{k=1}^\infty \sum_{i: \hat{\mu}_i > 0} \frac{\gamma^2 \theta_{k, i}^2}{(\hat{\mu}_i+\gamma)^2} = \sum_{k=1}^\infty \sum_{i: \hat{\mu}_i > 0} \frac{ \theta_{k, i}^2}{\hat{\mu}_i} \frac{\gamma^2 \hat{\mu}_i}{(\hat{\mu}_i+\gamma)^2}\\
  &\leq \left( \max_{i \in 1, \dots, n} \frac{\gamma^2 \hat{\mu}_i}{(\hat{\mu}_i+\gamma)^2} \right) \sum_{k=1}^\infty \sum_{i: \hat{\mu}_i > 0} \frac{ \theta_{k, i}^2}{\hat{\mu}_i} \leq \frac{\gamma}{4} \sum_{k=1}^\infty \sum_{i: \hat{\mu}_i > 0} \frac{ \theta_{k, i}^2}{\hat{\mu}_i} .
\end{align*}
The second equality uses that $\theta_k = U^\top  K \alpha_k = D U^\top  \alpha_k$, hence $\theta_{k, i} = 0$ whenever $\hat{\mu}_i = 0$ and the final inequality uses that $ab^2/(a+b)^2 \leq b/4$. Let $D^+$ denote the generalised inverse of $D$, i.e. \ $D_{ii}^+ := \hat{\mu}_i^{-1} \ind_{\hat{\mu}_i > 0}$. Then 
\begin{align*}
  \sum_{i: \hat{\mu}_i > 0} \frac{ \theta_{k, i}^2}{\hat{\mu}_i} &= \lVert \sqrt{D^+}\theta_k \rVert_2^2 = \alpha_k^\top  K U D^+ U^\top  K \alpha_k =  \alpha_k^\top  U D D^+ D U^\top  \alpha_k = \alpha_k^\top  K \alpha_k\\
    &= \lVert u_k \rVert_Z^2 \leq  \lVert u_k \rVert_Z^2 +  \lVert v_k \rVert_Z^2 =  \lVert \beta_k \rVert_Z^2 .
\end{align*}
Putting things together, we have
\[
  \sum_{k=1}^\infty \lVert (D(D+\gamma I)^{-1}-I)\theta_k \rVert_2^2 \leq  \frac{\gamma}{4} \sum_{k=1}^\infty \lVert \beta_k \rVert_Z^2 = \frac{\gamma}{4} \lVert \mathscr{S} \rVert_{\HS}^2.
\]
Hence,
\[
  \frac{1}{n}  \mathbb{E}\left( \sum_{i=1}^n \lVert \mathscr{S}(Z_i) - \hat{\mathscr{S}}(Z_i) \rVert_Z^2 \cond Z^{(n)} \right) \leq \frac{\sigma^2}{n} \sum_{i=1}^n \frac{\hat{\mu}_i^2}{(\hat{\mu}_i+\gamma)^2} + \frac{\gamma}{4n} \lVert \mathscr{S} \rVert_{\HS}^2,
\]
and using that
\[
  \frac{\hat{\mu}_i^2}{(\hat{\mu}_i+\gamma)^2} \leq \min(1, \hat{\mu}_i^2/(4d_i \gamma)) = \min(\hat{\mu}_i/4, \gamma)/\gamma,
\]
we have shown equation \eqref{eq:conditional-mse-bound}. 
\end{proof}

To go from a conditional statement to an unconditional result, we first require the following lemma.

\begin{lemma} \label{lem:koltchinskii}
  Let $x_1,\ldots,x_n$ be i.i.d. observations of a centred Hilbertian random variable $X$ with $E \lVert X \rVert^2 < \infty$. Let $\mathscr{C}$ denote the covariance operator of $X$ with eigen-expansion
  \begin{equation} \label{eq:Mercer}
  \mathscr{C} = \sum_{k=1}^\infty \mu_{k} e_{k}\otimes e_{k}
  \end{equation}
  for an orthonormal basis $(e_{k})_{k=1}^\infty$, and summable eigenvalues $\mu_{1} \geq \mu_{2} \geq \cdots \geq 0$. Let the random matrix $K \in \R^{n \times n}$ have entries given by $K_{ij} = \langle x_i, x_j \rangle$ and denote the eigenvalues of $K/n$ by $\hat{\mu}_1 \geq \hat{\mu}_2 \geq \cdots \geq \hat{\mu}_n \geq 0$.  
  
  For all $r > 0$,
  \[
  \E\bigg(\sum_{k=1}^n \min(\hat{\mu}_k , r) \bigg) \leq \sum_{k=1}^\infty \min(\mu_k, r).
  \]
  \end{lemma}
  \begin{proof}
    It suffices to show that given any $\epsilon > 0$, we have
    \[
    \E\bigg(\sum_{k=1}^n \min(\hat{\mu}_k , r) \bigg) \leq \epsilon + \sum_{k=1}^\infty \min(\mu_k, r).
    \]
    Now let $d$ be such that
    \[
    \sum_{k=d+1}^\infty \mu_k < \epsilon/n.
    \]
    Let $\Phi \in \mathbb{R}^{n \times d}$ have entries given by
    $$
      \Phi_{ij} :=  \langle x_i, e_j \rangle,
    $$
    such that 
    $$
      (\Phi \Phi^\top )_{ij} := \sum_{k=1}^d  \langle x_i , e_k \rangle \langle x_j , e_k \rangle .
    $$
    From this, it is clear that
    $$
    (K-\Phi \Phi^\top )_{ij} = \sum_{k=d+1}^\infty \langle x_i , e_k \rangle \langle x_j , e_k \rangle.
    $$
    Thus for $v \in \mathbb{R}^d$
    $$
    v^\top  (K-\Phi \Phi^\top ) v = \sum_{i=1}^d \sum_{j=1}^d v_i   v_j  \sum_{k=d+1}^\infty \langle s_i, e_k \rangle \langle x_j, e_k \rangle = \sum_{k=d+1}^\infty \left\langle \sum_{i=1}^d v_i x_i, e_k \right\rangle^2  \geq 0,
    $$
    showing that $K-\Phi \Phi^\top $ is positive semi-definite. 
  
    Next let $\mathbb{S}^d_+$ be the cone of positive semi-definite $d \times d$ matrices, and for $A \in \mathbb{S}^d_+$ and $k=1,\ldots,d$, let $\lambda_k(A)$ denote the $k$th largest eigenvalue. Let $f : \mathbb{S}^d_+ \to \R$ be given by
    \[
    f(A) = \sum_{k=1}^d \min(\lambda_k(A), r).
    \]
    By Weyl's inequality, noting that the non-zero eigenvalues of $\Phi^\top \Phi$ and $\Phi \Phi^\top $ coincide, we have for all $k$,
    \[
    \hat{\mu}_k \leq \lambda_k(\Phi^\top \Phi/n) + \lambda_1(K - \Phi\Phi^\top /n)
    \]
    and so
    \[
    \min(\hat{\mu}_k , r) \leq \min(\lambda_k(\Phi^\top \Phi /n) , r) + \tr(K - \Phi\Phi^\top )/n.
    \]
    Thus,
    \begin{equation} \label{eq:eq1}
    \E\bigg(\sum_{k=1}^n \min(\hat{\mu}_k , r) \bigg) \leq \E f( \Phi^\top \Phi/n) + \E \tr(K - \Phi\Phi^\top ).
    \end{equation}
    Now by Fubini's theorem,
    \[
    \E \tr(K - \Phi\Phi^\top ) = \sum_{i=1}^n \sum_{k=d+1}^\infty  \E (\langle x_i, e_k \rangle^2 )= n \sum_{k=d+1}^\infty \mu_k  < \epsilon.
    \]
    We now claim that $f$ is concave, from which the result will follow. Indeed, then by Jensen's inequality,
    $\E f( \Phi^\top \Phi/n) \leq f(\E \Phi^\top \Phi/n)$ and
    \[
    \frac{1}{n}\big(\E \Phi^\top \Phi\big)_{kl} = \frac{1}{n} \sum_{i=1}^n \E (\langle x_i, e_k \rangle \langle x_i, e_l \rangle ) = \mu_k \ind_{\{k=l\}}.
    \]
    Thus,
    \[
    f(\E \Phi^\top \Phi/n) = \sum_{k=1}^d \min(\mu_k,r),
    \]
    and so returning to \eqref{eq:eq1} we would have
    \[
    \E\bigg(\sum_{k=1}^n \min(\hat{\mu}_k , r) \bigg) \leq \epsilon + \sum_{k=1}^\infty \min(\mu_k,r).
    \]
    
    We now show that $f$ is concave. Take $t \in (0, 1)$ and $A,B \in\mathbb{S}^d_+$. We will show that
    \begin{equation} \label{eq:concave}
    \sum_{k=1}^d (\lambda_k(tA + (1-t)B)-r)_+ \leq \sum_{k=1}^d \{t(\lambda_k(A)-r)_+ + (1-t)(\lambda_k(B)-r)_+ \},
    \end{equation}
    where $( \cdot)_+$ denotes the positive part. This will prove concavity of $f$ as
    \begin{align*}
    \sum_{k=1}^d \lambda_k(tA +(1-t)B) &= \tr(tA + (1-t)B)  \\
    & = t \tr (A) + (1-t)\tr(B) = \sum_{k=1}^d \{t\lambda_k(A) + (1-t)\lambda_k(B) \}, 
    \end{align*}
    so subtracting \eqref{eq:concave} yields $f(tA + (1-t)B) \geq tf(A) +(1-t)f(B)$ as desired.

Certainly
\eqref{eq:concave} holds when $r \geq \lambda_1(tA + (1-t)B) $.
Now by Lidskii's inequality, for each $j=1,\ldots,d$,
\begin{equation} \label{eq:lidskii}
	\sum_{k=1}^j\lambda_k(tA + (1-t)B) \leq \sum_{k=1}^j \{t\lambda_k(A) + (1-t)\lambda_k(B) \}.
\end{equation}
%	Considering $j=d$, we see that when $r \leq \lambda_d(tA + (1-t)B) $ we have
%	\begin{align*}
	%		\sum_{i=1}^d(\lambda_i(tA + (1-t)B) - r)_+ &= \sum_{i=1}^d(\lambda_i(tA + (1-t)B) - r)  \\
	%		&\leq \sum_{i=1}^j \{t(\lambda_i(A)-r) + (1-t)(\lambda_i(B)-r) \} \\
	%		&\leq \sum_{i=1}^j \{t(\lambda_i(A)-r)_+ + (1-t)(\lambda_i(B)-r)_+ \}.
	%	\end{align*}
For convenience, let us set $\lambda_{d+1}(tA + (1-t)B)=0$. Then for any $j=1,\ldots,d$, if
$\lambda_{j+1}(tA + (1-t)B) \leq r \leq \lambda_j(tA + (1-t)B) $, we have
\begin{align*}
	\sum_{k=1}^d(\lambda_k(tA + (1-t)B) - r)_+ &= \sum_{k=1}^j(\lambda_k(tA + (1-t)B) - r) \\
	&\leq \sum_{k=1}^j \{t(\lambda_k(A)-r) + (1-t)(\lambda_k(B)-r) \} \\
	&\leq \sum_{k=1}^d \{t(\lambda_k(A)-r)_+ + (1-t)(\lambda_k(B)-r)_+ \},
\end{align*}
using \eqref{eq:lidskii} for the first inequality. We thus have that \eqref{eq:concave} holds whatever the value of $r$, and so $f$ is concave, which completes the proof.
  \end{proof}

Combining Lemma~\ref{lem:deterministic-regression-mse-bound} and Lemma~\ref{lem:koltchinskii} now yields the following bound on our regression estimator.

\begin{lemma}
  \label{lem:regression-gamma-hat-upper-bound}
Let $\mathcal{P}$ consist of a family of distributions of $(X, Z) \in \mathcal{H}_X \times \mathcal{H}_Z$ such that
$$
X = \mathscr{S}_\dist Z + \varepsilon_\dist,
$$
where we assume that $\sup_{\dist \in \mathcal{P}} \lVert \mathscr{S}_\dist \rVert_\HS < C$ and $\sup_{\dist \in \mathcal{P}} \E_P \lVert \varepsilon_\dist \rVert^2 < \sigma^2$. Suppose we are given $n$ i.i.d. observations $(x_i, z_i)_{i=1}^n$ of $(X, Z)$ and denote by $(\mu_{k, \dist})_{k \in \mathbb{N}}$ the non-negative eigenvalues of $\Cov_\dist(\varepsilon_P)$. Let $\mathscr{S}_\gamma$ be the estimator in \eqref{eq:penalised-likelihood-criterion-hilbertian-linear-model}. We have for each $P \in \mathcal{P}$, that
\begin{equation} \label{eq:gamma_res}
\frac{1}{n}  \mathbb{E}_P\left( \sum_{i=1}^n \lVert \mathscr{S}(z_i) - \hat{\mathscr{S}}_\gamma(z_i) \rVert^2 \right) \leq \frac{\sigma^2}{\gamma } \frac{1}{n} \sum_{k=1}^\infty \min(\mu_{k, \dist}/4, \gamma) + \lVert \mathscr{S}_P \rVert_{\HS}^2 \frac{\gamma}{4n}  .
\end{equation}
Further, if we use $\hat{\gamma}$ as in \eqref{eq:gamma_hat}, that is,
$$
\hat{\gamma} = \argmin_{\gamma > 0} \left( \frac{1}{\gamma n} \sum_{k=1}^n \min(\hat{\mu}_k/4, \gamma) + \frac{\gamma}{4} \right),
$$
to produce an estimate $\hat{\mathscr{S}} := \hat{\mathscr{S}}_{\hat{\gamma}}$ of $\mathscr{S}_P$, then 
\begin{equation} \label{eq:gamma_hat_res}
\sup_{\dist \in \mathcal{P}} \mathbb{E}_\dist \left( \frac{1}{n} \sum_{i=1}^n \lVert \mathscr{S}_P(Z_i) - \hat{\mathscr{S}}(Z_i) \rVert^2 \right) \leq \max(\sigma^2, C) \sup_{\dist \in \mathcal{P}} \inf_{\gamma > 0} \left( \frac{1}{\gamma n } \sum_{k=1}^\infty \min(\mu_{k, \dist} , \gamma) +  \gamma \right).
\end{equation}
\end{lemma}
\begin{proof}
	Result \eqref{eq:gamma_res} follows immediately from Lemmas~\ref{lem:deterministic-regression-mse-bound} and \ref{lem:koltchinskii}. To show \eqref{eq:gamma_hat_res}, we argue as follows.
Let $(e_k)_{k \in \mathbb{N}}$ denote a basis of $\mathcal{H}_X$. Then
conditioning on $z_1, \dots, z_n$ and applying equation (\ref{eq:conditional-mse-bound}) in Lemma~\ref{lem:deterministic-regression-mse-bound}, we get that
\begin{align*}
&\sup_{\dist \in \tilde{\mathcal{P}}_0} \mathbb{E}_\dist \left( \frac{1}{n} \sum_{i=1}^n \lVert \mathscr{S}_P(z_i) - \hat{\mathscr{S}}(z_i) \rVert^2 \right) \leq \sup_{\dist \in \tilde{\mathcal{P}}_0} \mathbb{E}_\dist \left(\frac{\sigma^2}{\hat{\gamma}} \frac{1}{n} \sum_{k=1}^n \min(\hat{\mu}_k /4 , \hat{\gamma}) + \lVert \mathscr{S}_P \rVert_{\HS}^2 \frac{\hat{\gamma}}{4}  \right) \\
&\leq \max(\sigma^2, C) \sup_{\dist \in \tilde{\mathcal{P}}_0}  \mathbb{E}_\dist \left[ \min_{\gamma > 0} \left(\frac{1}{\gamma n } \sum_{k=1}^n \min(\hat{\mu}_k /4 , \gamma) +  \frac{\gamma}{4}  \right) \right] .
\end{align*}
Using the fact that the expectation of a minimum is less than the minimum of the expectation, we get that
\begin{align*}
\sup_{\dist \in \tilde{\mathcal{P}}_0}  \mathbb{E}_\dist \left[ \min_{\gamma > 0} \left(\frac{1}{\gamma n } \sum_{k=1}^n \min(\hat{\mu}_k /4 , \gamma) +  \frac{\gamma}{4}  \right) \right]  &\leq \sup_{\dist \in \tilde{\mathcal{P}}_0} \inf_{\gamma > 0}  \left[\mathbb{E}_\dist  \left(\frac{1}{\gamma n } \sum_{k=1}^n \min(\hat{\mu}_k /4 , \gamma) +  \frac{\gamma}{4}  \right) \right] \\
 &\leq \sup_{\dist \in \tilde{\mathcal{P}}_0} \inf_{\gamma > 0} \left( \frac{1 }{\gamma n } \sum_{k=1}^\infty \min(\mu_{k, \dist} , \gamma) +  \gamma \right), 
\end{align*}
where the second inequality is due to Lemma~\ref{lem:koltchinskii}.
\end{proof}

Finally, we can prove Theorem~\ref{thm:kernel-regression-ghcm}.

\begin{proof}
By Theorem~\ref{thm:level of test} and the assumptions of the Theorem it is sufficient to show that 
\begin{equation}
  \label{eq:thm-5-sufficient-equation}
\sup_{\dist \in \tilde{\mathcal{P}}_0} \sqrt{n}\mathbb{E}_\dist \left( \frac{1}{n} \sum_{i=1}^n \lVert \mathscr{S}^X_\dist(z_i) - \hat{\mathscr{S}}(z_i) \rVert^2 \right)  \to 0
\end{equation}
and similarly for the regression of $Y$ on $Z$. This can be seen by noting that an application of Cauchy--Schwarz and Markov's inequality yields that $nM_{n, \dist}^f M_{n, \dist}^g \overset{P}{\rightrightarrows} 0$ and, by the upper bound on $u_\dist$ and $v_\dist$ in assumption (ii), $\tilde{M}_{n, \dist}^f \overset{P}{\rightrightarrows} 0$ and $\tilde{M}_{n, \dist}^g\overset{P}{\rightrightarrows} 0$.

Lemma~\ref{lem:regression-gamma-hat-upper-bound} implies that it is sufficient to show that
\[
\sqrt{n} \sup_{\dist \in \tilde{\mathcal{P}}_0}  \inf_{\gamma > 0} \left( \frac{1}{\gamma n } \sum_{k=1}^\infty \min(\mu_{k, \dist} , \gamma) +  \gamma \right) \to 0 
\]
as $n \to \infty$ for \eqref{eq:thm-5-sufficient-equation} to hold. For each $\dist \in \tilde{\mathcal{P}}_0$, we let $\phi_\dist: \mathbb{R}_+ \to \mathbb{R}_+$
be given by
\[
\phi_\dist(\gamma) = \sum_{k=1}^\infty \min(\mu_{k, \dist}, \gamma)  .
\]
By assumption (iii), $\lim_{\gamma \downarrow 0} \sup_{\dist \in \tilde{\mathcal{P}}_0} \phi_\dist(\gamma) = 0$, hence for any $\epsilon > 0$ we can find $N \in \mathbb{N}$ such that for any $n \geq N$,  $\sup_{\dist \in \tilde{\mathcal{P}}_0}  \sqrt{\phi_\dist \left(n^{-1/2} \right)} < \epsilon/2$. Let $\gamma_{n, \dist} = n^{-1/2} \sqrt{\phi_\dist\left(n^{-1/2}  \right)}$. Then,
\begin{align*}
 &\sqrt{n} \sup_{\dist \in \tilde{\mathcal{P}}_0}  \inf_{\gamma > 0} \left( \frac{1}{\gamma n } \sum_{k=1}^\infty \min(\mu_{k, \dist} , \gamma) +  \gamma \right) = \sup_{\dist \in \tilde{\mathcal{P}}_0}  \inf_{\gamma > 0} \left( \frac{\phi_\dist(\gamma)}{\gamma \sqrt{n} }  +  \sqrt{n}\gamma \right)\\
  &\leq \sup_{\dist \in \tilde{\mathcal{P}}_0} \left( \frac{\phi_\dist(\gamma_{n, \dist})}{\gamma_{n, \dist} \sqrt{n} }  +  \sqrt{n}\gamma_{n, \dist} \right) = \sup_{\dist \in \tilde{\mathcal{P}}_0} \left( \frac{\phi_\dist \left(n^{-1/2} \sqrt{\phi_\dist\left(n^{-1/2}  \right)} \right)}{\sqrt{\phi_\dist\left(n^{-1/2}  \right)} }  +  \sqrt{\phi_\dist\left(n^{-1/2}  \right)}  \right).
\end{align*}
Assuming that $\epsilon \leq 2$ and using that $\phi_\dist$ is increasing, we get that for $n \geq N$,
\begin{align*}
&\sup_{\dist \in \tilde{\mathcal{P}}_0} \left( \frac{\phi_\dist \left(n^{-1/2} \sqrt{\phi_\dist\left(n^{-1/2}  \right)} \right)}{\sqrt{\phi_\dist\left(n^{-1/2}  \right)} }  +  \sqrt{\phi_\dist\left(n^{-1/2}  \right)}  \right) <\sup_{\dist \in \tilde{\mathcal{P}}_0} \left( \frac{\phi_\dist \left(n^{-1/2} \epsilon/2 \right)}{\sqrt{\phi_\dist\left(n^{-1/2}  \right)} }  +  \sqrt{\phi_\dist\left(n^{-1/2}  \right)}  \right)\\
 & < \sup_{\dist \in \tilde{\mathcal{P}}_0}  2\sqrt{\phi_\dist\left(n^{-1/2}  \right)}  < \epsilon,
\end{align*}
proving the result.
\end{proof}

\begin{corollary}\label{cor:exponential-decay-mse}
  Consider the setup of Lemma~\ref{lem:regression-gamma-hat-upper-bound} but with the additional assumption that for some $a, b >0$, we have $\mu_{k, \dist} \leq a e^{-b k}$ for all $\dist \in \mathcal{P}$. Then 
  $$
  \sup_{\dist \in \tilde{\mathcal{P}}_0} \mathbb{E}_\dist \left( \frac{1}{n} \sum_{i=1}^n \lVert \mathscr{S}_P(z_i) - \hat{\mathscr{S}}(z_i) \rVert^2 \right) = o(\log n / n)
  $$
\end{corollary}
\begin{proof}
Applying Lemma~\ref{lem:regression-gamma-hat-upper-bound}, we show that 
$$
\sup_{\dist \in \tilde{\mathcal{P}}_0}  \inf_{\gamma > 0} \left( \frac{1}{\gamma n } \sum_{k=1}^\infty \min(\mu_{k, \dist} , \gamma) +  \gamma \right) \leq  \inf_{\gamma > 0} \left( \frac{1}{\gamma n } \sum_{k=1}^\infty \min(a e^{-b k}, \gamma) +  \gamma \right) = o(\log n / n ).
$$
To that end, note that
$$
\frac{1}{\gamma n } \sum_{k=1}^\infty \min(a e^{-b k}, \gamma) +  \gamma  \leq -\frac{1}{ nb } \log(\gamma/a) + \frac{1}{n\gamma} \int_{-\log(\gamma/a)/b}^\infty a e^{-x b} \, \mathrm{d}x +  \gamma = -\frac{1}{ nb } \log(\gamma/a) + \frac{1}{nb} +  \gamma.
$$
The right-hand side is a strictly convex function in $\gamma$ hence it has a unique minimum at the unique root of the derivative function given by $\gamma^* := \frac{1}{nb}$ which yields a minimum of
\begin{equation*}
\frac{1}{ nb } (\log(anb) + 2) =o(\log n / n). \qedhere
\end{equation*}
\end{proof}

\section{Additional numerical results} \label{app: additional-sim-results}

Here we include additional results relating to the setups in Section~\ref{sec:experiments}.
Figures~\ref{fig:FPR-TPR-a-2}, \ref{fig:FPR-TPR-a-6} and \ref{fig:FPR-TPR-a-12} plot rejection rates against nominal significance levels for \texttt{pfr} and the GHCM, for the setups described in~\ref{sec:level-power-sim}.

Figure~\ref{fig:heavy-tailed} plots rejection rates for a subset of null settings considered in Section~\ref{sec:scalar-level-power-sim} but where the noise $N_Y$ in \eqref{eq:sim_alt} is $t$-distributed.

Figure~\ref{fig:rates} plots rejection rates for a subset of null settings considered in Section~\ref{sec:scalar-level-power-sim} but where instead of \eqref{eq:sim_alt}, the regression model for $Y$ is given by
\[
Y = \int_0^1 \alpha_a(t)Z(t) \mathrm{d}t + \sqrt{\frac{100}{n}} \int_0^1 \frac{\alpha_a(t)}{a}X(t) \mathrm{d}t + N_Y. 
\]
Note that when $n=100$, the model is identical to \eqref{eq:sim_alt}. In general however, $\|\E \Cov(X, Y \cond Z)\|_{\HS}$ scales with $1/\sqrt{n}$ here, and so Theorem~\ref{thm: consistency of test} suggests as $n$ changes, the power should not change much. This is confirmed by our empirical results where we observe that the power remains largely unchanged as $n$ changes, suggesting in particular that the GHCM has power against $1/\sqrt{n}$ alternatives.

Figure~\ref{fig:fdboost} plots rejection rates for the same settings considered in Section~\ref{sec:scalar-level-power-sim} but where we use the \texttt{FDboost} package for regressions instead of the \texttt{refund} package. We use default tuning parameters for the regression; it is possible that performance could improve with more careful tuning.

Figure~\ref{fig:irregular} plots rejection rates for the same settings considered in Section~\ref{sec:functional-level-power-sim} but where the $X$ and $Y$ curves are observed on an irregular grid with points sampled independently and uniformly on $[0,1]$. We consider a sparse grid of $4$ points as well as four unequal grid sizes sampled as the maximum of $4$ and a Poisson random variable with mean in $\{10, 25, 50, 100\}$.

Figure~\ref{fig:eeg-sim} plots rejection rates for a simulation based on the real data analysis in Section~\ref{sec:real-data}. For each of the two edges that had Benjamini--Hochberg-corrected $p$-values at most $5\%$ (O-L---PO-L and O-R---PO-R), we created artificial datasets as follows. We added independent Brownian motion noise to each of the estimated regression functions (note there were regression functions estimated for each variable in each of the two groups) thereby simulating a new $X$ and $Y$ conditional on the fixed $Z$. In these simulated datasets, the null of conditional independence does hold, and so we should expect the GHCM to deliver uniformly-distributed $p$-values. The results using the GHCM as described in Section~\ref{sec:real-data} and for varying standard deviation $\sigma$ of the Brownian motion noise for one set of regressions with the other set at $1$, are shown in Figure~\ref{fig:eeg-sim}. We see that even in the low $\sigma$ settings, which are expected to be the most challenging, the GHCM maintains level control.

\begin{figure}[H]
\centering
\includegraphics{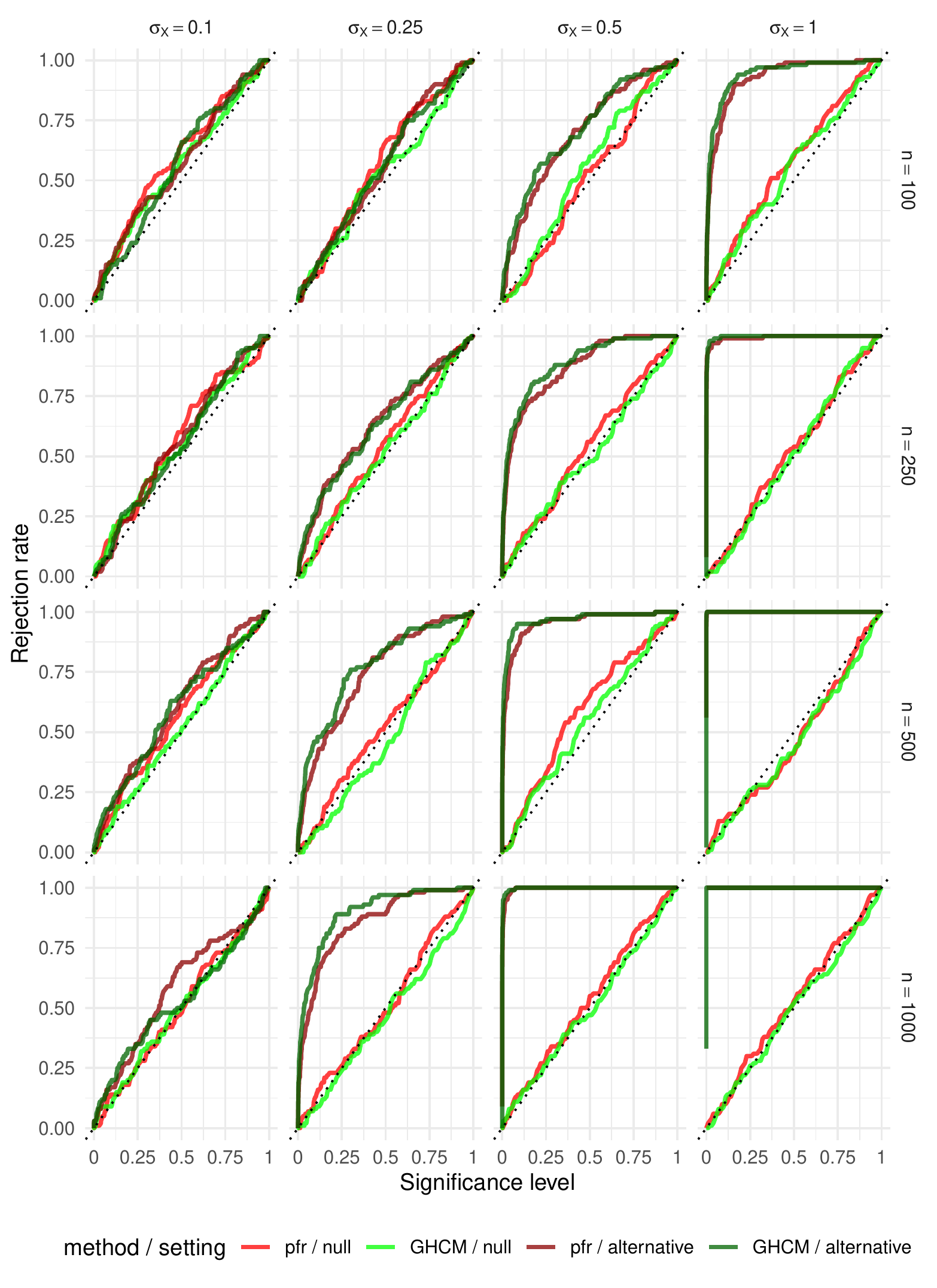}
\caption{Rejection rates against significance level $\alpha$ for the $\texttt{pfr}$ (red) and GHCM (green) tests under null (light) and alternative (dark) settings when $a=2$.}
\label{fig:FPR-TPR-a-2}
\end{figure}

\begin{figure}[H]
\centering
\includegraphics{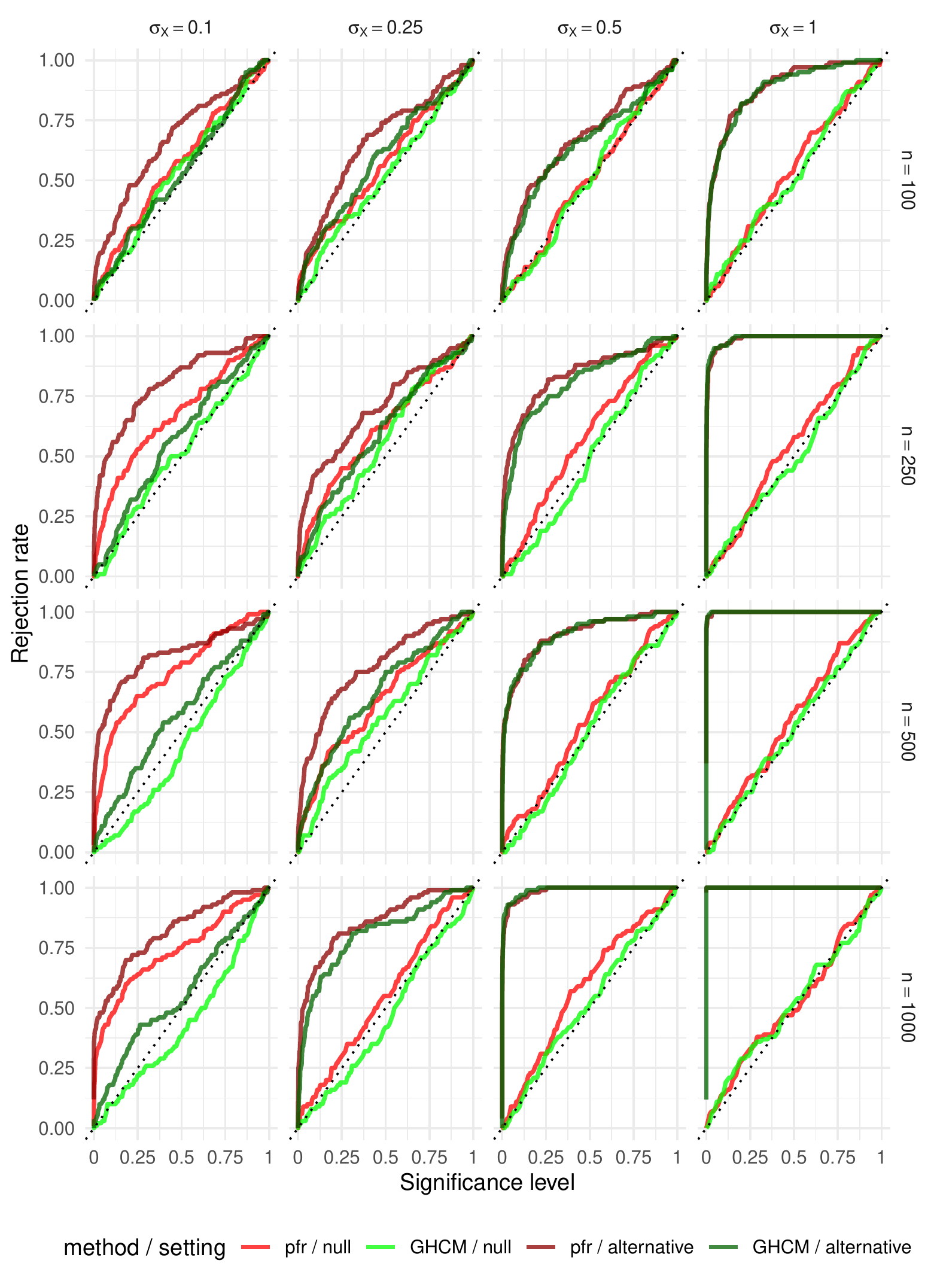}
\caption{Rejection rates against significance level $\alpha$ for the $\texttt{pfr}$ (red) and GHCM (green) tests under null (light) and alternative (dark) settings when $a=6$.}
\label{fig:FPR-TPR-a-6}
\end{figure}

\begin{figure}[H]
\centering
\includegraphics{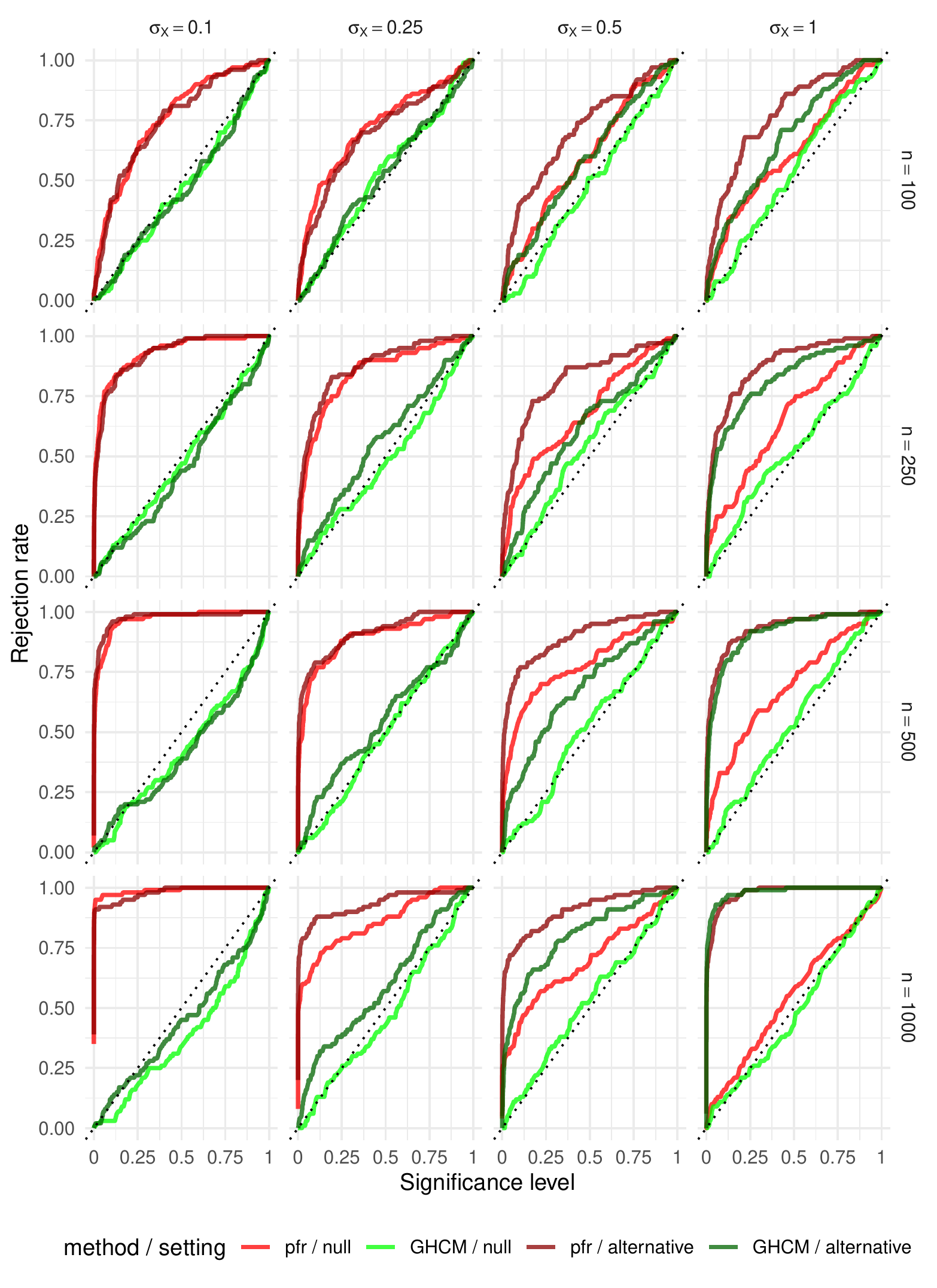}
\caption{Rejection rates against significance level $\alpha$ for the $\texttt{pfr}$ (red) and GHCM (green) tests under null (light) and alternative (dark) settings when $a=12$.}
\label{fig:FPR-TPR-a-12}
\end{figure}

\begin{figure}[H]
\centering
\includegraphics{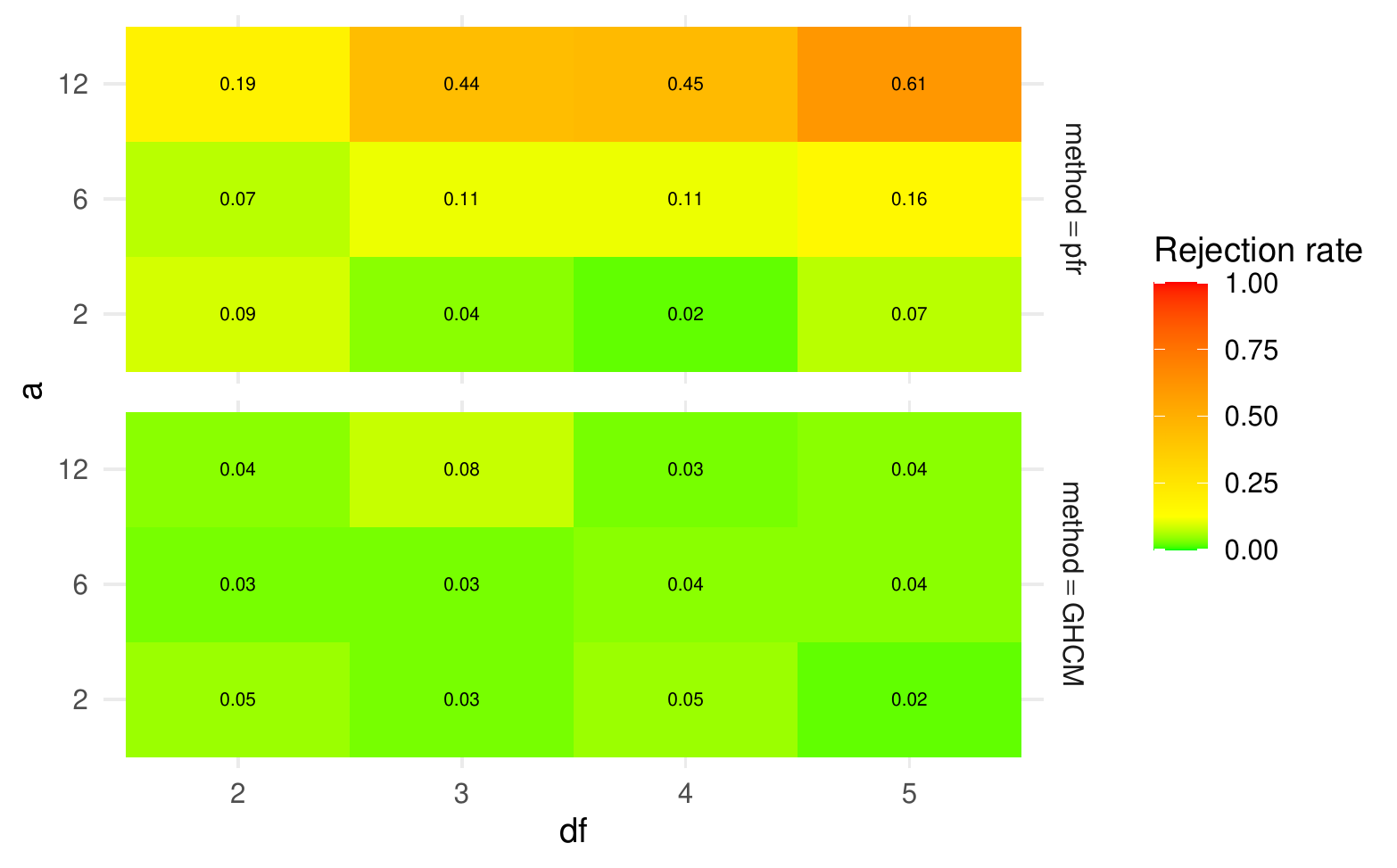}
\caption{Rejection rates in a subset of the null settings considered in Section~\ref{sec:scalar-level-power-sim} for the nominal 5\%-level \texttt{pfr} test (top) and GHCM test (bottom) where $\sigma_X = 0.25$ and $n=500$ and the noise $N_Y$ in \eqref{eq:sim_alt} is $t$-distributed with $\mathrm{df}$ degrees of freedom.}
\label{fig:heavy-tailed}
\end{figure}

\begin{figure}[H]
  \centering
  \includegraphics{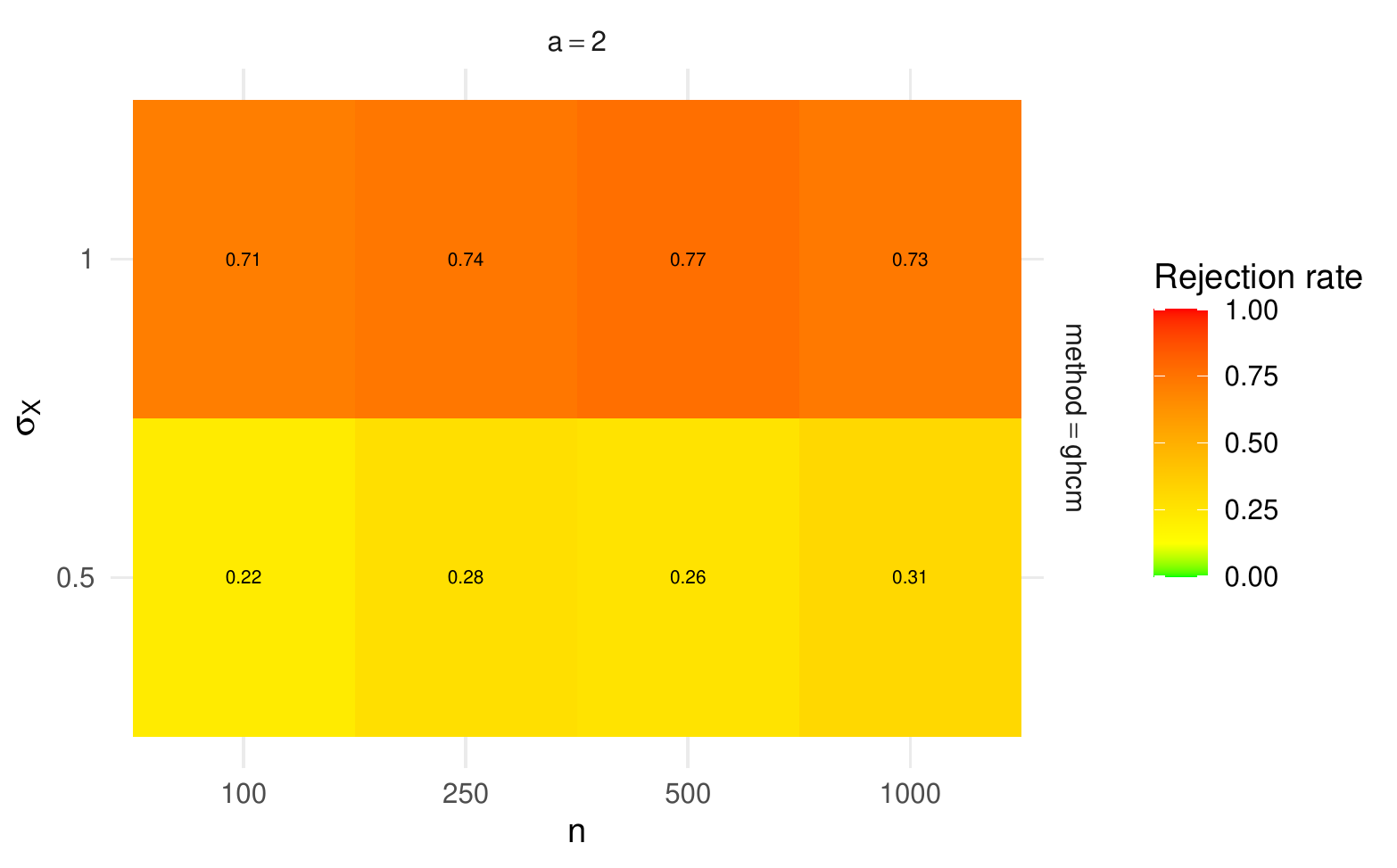}
  \caption{Rejection rates in a subset of the alternative settings considered in Section~\ref{sec:scalar-level-power-sim} for the nominal 5\%-level GHCM test where $a=2$ and $\alpha_a$ has been replaced with $(100/n)^{-1/2}\alpha_a$.}
  \label{fig:rates}
\end{figure}

\begin{figure}[H]
	\centering
	\includegraphics{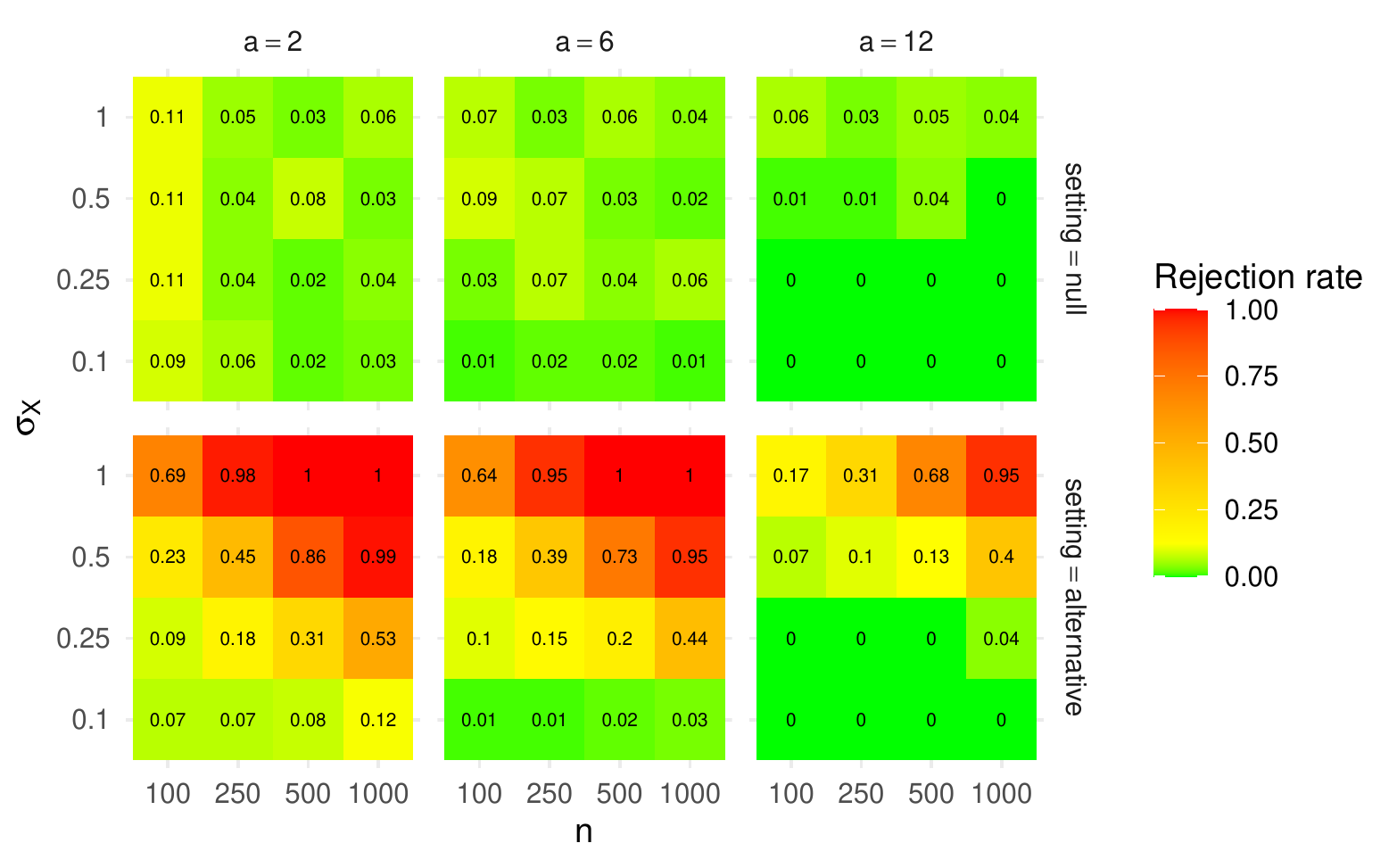}
	\caption{Rejection rates in the setting of Section~\ref{sec:scalar-level-power-sim}, replicating Figures~\ref{fig:level-rejection-rates} and \ref{fig:power-rejection-rates}, for the nominal 5\%-level GHCM test using \texttt{FDboost} package for regressions instead of the \texttt{refund} package.}
	\label{fig:fdboost}
\end{figure}

\begin{figure}[H]
	\centering
	\includegraphics{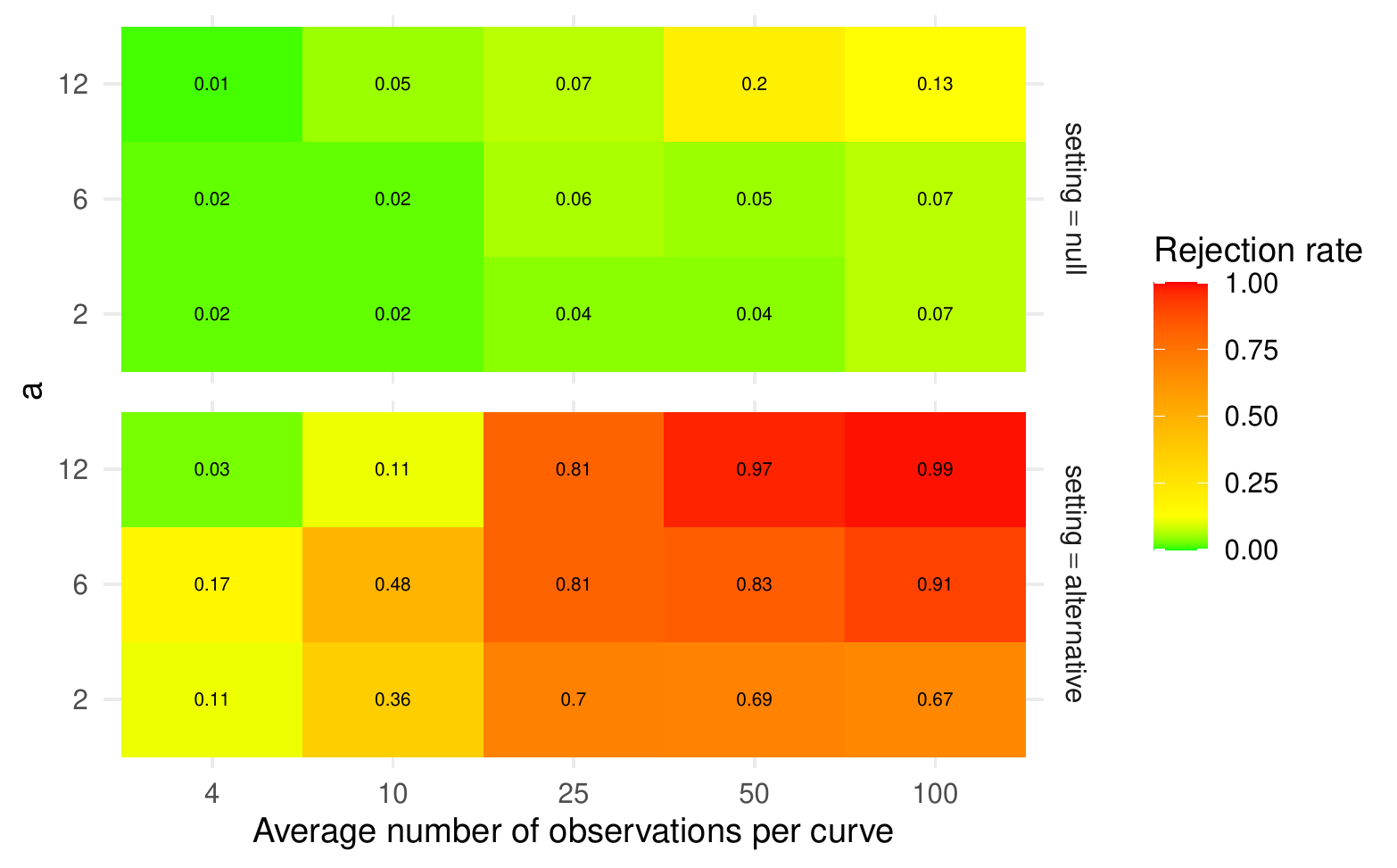}
	\caption{Rejection rates in the setting of Section~\ref{sec:functional-level-power-sim}, replicating Figure~\ref{fig:functional-rejection-rates}, for the nominal 5\%-level GHCM test where the $X$ and $Y$ curves are observed on irregular grids as described in the main text.}
	\label{fig:irregular}
\end{figure}

\begin{figure}[H]
	\centering
	\includegraphics[scale=0.75]{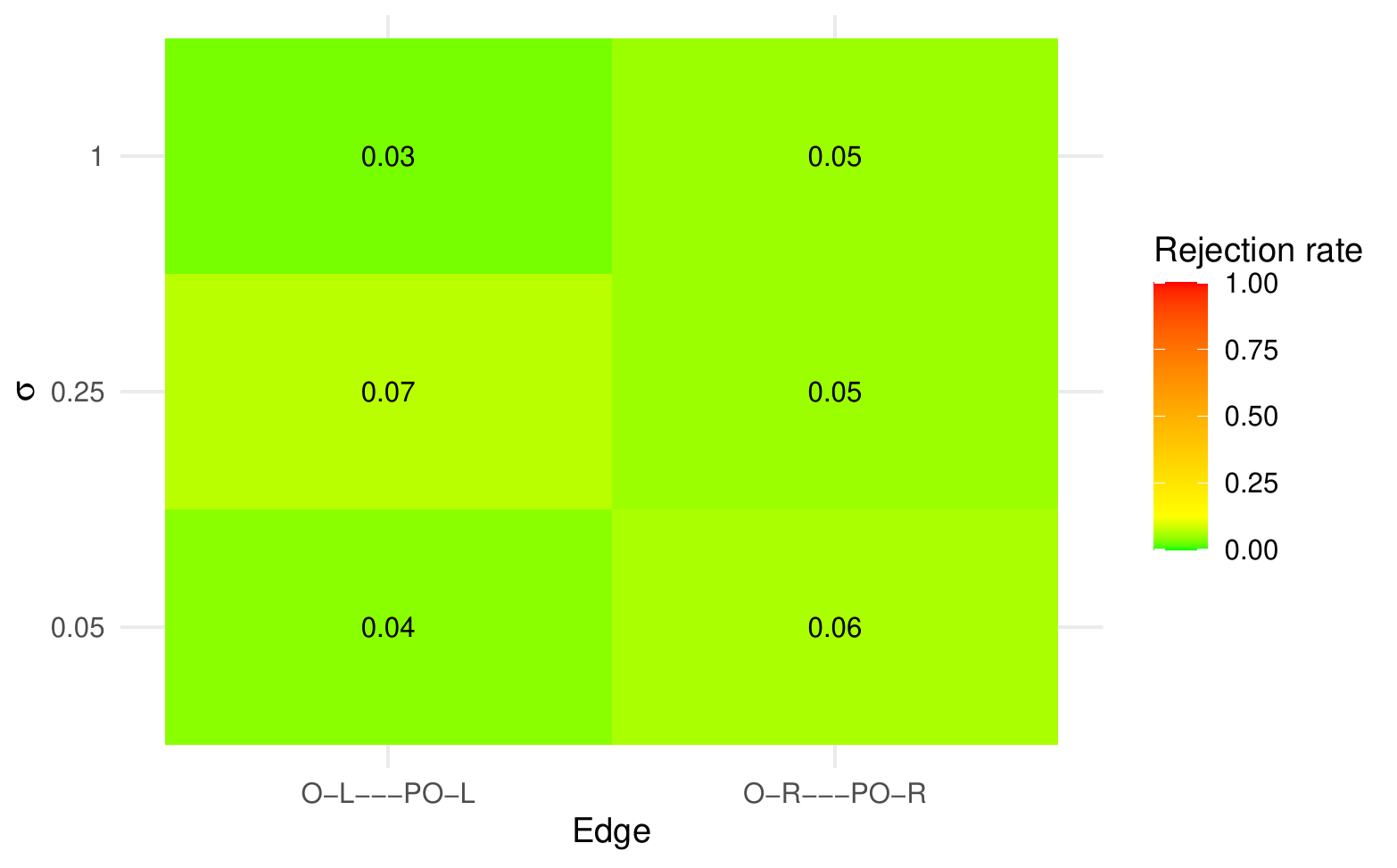}
	\caption{Rejection rates for nominal 5\%-level GHCM tests in simulation settings based on the EEG data studied in Section~\ref{sec:real-data}; see the main text for further details.}
	\label{fig:eeg-sim}
\end{figure}

%\putbib[bibliography.bib]
%\end{bibunit}

\end{cbunit}
\end{document}